%% file: exact_triangle.tex
\documentclass[11pt,reqno]{amsart}
\usepackage{bm,amsmath,amsthm,amssymb,mathtools,verbatim,amsfonts,tikz-cd,mathrsfs}
\usepackage{enumitem}
\usepackage{wrapfig}

\usepackage{caption}

\usepackage[hidelinks,colorlinks=true,linkcolor=blue,citecolor=black,linktocpage=true]{hyperref}
\usepackage{graphicx}
\usepackage[alphabetic]{amsrefs}

\usepackage[top=1.4in, bottom=1.2in, left=1.4in, right=1.4in,marginpar=1in]{geometry}

\usetikzlibrary{matrix,arrows,decorations.pathmorphing}

\usepackage{chngcntr}
\counterwithin{figure}{section}

\newtheorem{thm}{Theorem}[section]
\newtheorem{prop}[thm]{Proposition}
\newtheorem{conj}[thm]{Conjecture}
\newtheorem{cor}[thm]{Corollary}
\newtheorem{lem}[thm]{Lemma}

\theoremstyle{definition}
\newtheorem{define}[thm]{Definition}

\theoremstyle{remark}
\newtheorem{rem}[thm]{Remark}
\newtheorem{example}[thm]{Example}

\newcommand{\ve}[1]{\boldsymbol{\mathbf{#1}}}

\newcommand{\R}{\mathbb{R}}

\newcommand{\Z}{\mathbb{Z}}
\newcommand{\N}{\mathbb{N}}
\newcommand{\Q}{\mathbb{Q}}

\renewcommand{\d}{\partial}
\renewcommand{\subset}{\subseteq}
\renewcommand{\supset}{\supseteq}
\renewcommand{\tilde}{\widetilde}

\newcommand{\iso}{\cong}

\DeclareMathOperator{\Aut}{{Aut}}

\DeclareMathOperator{\can}{{can}}
\DeclareMathOperator{\ch}{{ch}}

\DeclareMathOperator{\lc}{{lc}}

\DeclareMathOperator{\End}{{End}}

\DeclareMathOperator{\gr}{{gr}}

\DeclareMathOperator{\Hom}{{Hom}}

\DeclareMathOperator{\id}{{id}}

\DeclareMathOperator{\im}{{im}}

\DeclareMathOperator{\Mod}{{Mod}}

\DeclareMathOperator{\Spin}{{Spin}}

\DeclareMathOperator{\Tw}{{Tw}}

\DeclareMathOperator{\Span}{{Span}}
\DeclareMathOperator{\Sym}{{Sym}}

\DeclareMathOperator{\Tors}{{Tors}}

\newcommand{\lk}{\mathrm{lk}}

\newcommand{\bA}{\mathbb{A}}
\newcommand{\bB}{\mathbb{B}}

\newcommand{\bE}{\mathbb{E}}
\newcommand{\bF}{\mathbb{F}}

\newcommand{\bH}{\mathbb{H}}
\newcommand{\bI}{\mathbb{I}}

\newcommand{\bT}{\mathbb{T}}

\newcommand{\bX}{\mathbb{X}}
\newcommand{\bY}{\mathbb{Y}}

\newcommand{\cA}{\mathcal{A}}

\newcommand{\cC}{\mathcal{C}}
\newcommand{\cD}{\mathcal{D}}

\newcommand{\cH}{\mathcal{H}}

\newcommand{\cK}{\mathcal{K}}
\newcommand{\cL}{\mathcal{L}}
\newcommand{\cM}{\mathcal{M}}

\newcommand{\cO}{\mathcal{O}}

\newcommand{\cR}{\mathcal{R}}
\newcommand{\cS}{\mathcal{S}}

\newcommand{\cX}{\mathcal{X}}
\newcommand{\cY}{\mathcal{Y}}

\newcommand{\frA}{\mathfrak{A}}

\newcommand{\frD}{\mathfrak{D}}

\newcommand{\frK}{\mathfrak{K}}
\newcommand{\frL}{\mathfrak{L}}

\newcommand{\frR}{\mathfrak{R}}
\newcommand{\frS}{\mathfrak{S}}

\newcommand{\frs}{\mathfrak{s}}
\newcommand{\frt}{\mathfrak{t}}

\newcommand{\frz}{\mathfrak{z}}

\newcommand{\scA}{\mathscr{A}}
\newcommand{\scB}{\mathscr{B}}

\newcommand{\scE}{\mathscr{E}}
\newcommand{\scF}{\mathscr{F}}

\newcommand{\scH}{\mathscr{H}}

\newcommand{\scO}{\mathscr{O}}

\newcommand{\scU}{\mathscr{U}}
\newcommand{\scV}{\mathscr{V}}

\newcommand{\cCFL}{\mathcal{C\!F\!L}}
\newcommand{\cCFK}{\mathcal{C\hspace{-.5mm}F\hspace{-.3mm}K}}

\newcommand{\CF}{\mathit{CF}}

\newcommand{\HF}{\mathit{HF}}

\newcommand{\opp}{\mathrm{opp}}

\newcommand{\PD}{\mathit{PD}}

\newcommand{\CFD}{\mathit{CFD}}
\newcommand{\CFA}{\mathit{CFA}}

\newcommand{\xs}{\ve{x}}
\newcommand{\ys}{\ve{y}}
\newcommand{\zs}{\ve{z}}
\newcommand{\ws}{\ve{w}}

\newcommand{\ps}{\ve{p}}
\newcommand{\qs}{\ve{q}}
\newcommand{\as}{\ve{\alpha}}
\newcommand{\bs}{\ve{\beta}}
\newcommand{\gs}{\ve{\gamma}}

\newcommand{\Xs}{\ve{X}}
\newcommand{\Ys}{\ve{Y}}

\newcommand{\Dt}{\Delta}

\renewcommand{\a}{\alpha}
\renewcommand{\b}{\beta}
\newcommand{\g}{\gamma}
\newcommand{\dt}{\delta}
\newcommand{\veps}{\varepsilon}

\usepackage{leftidx}

\DeclareMathOperator{\Cone}{{Cone}}

\newcommand{\Sss}[1]{\scriptscriptstyle{#1}}

\numberwithin{equation}{section}

\newcommand{\ar}{\mathrm{a.r.}}

\newcommand{\llsquare}{[\hspace{-.5mm}[}
\newcommand{\rrsquare}{]\hspace{-.5mm}]}

\allowdisplaybreaks

\newcommand{\cell}{\mathrm{cell}}

\newcommand{\vecotimes}{\mathrel{\vec{\otimes}}}

\newcommand{\Map}{\scF}

\makeatletter
\DeclareRobustCommand{\cev}[1]{%
  {\mathpalette\do@cev{#1}}%
}
\newcommand{\do@cev}[2]{%
  \vbox{\offinterlineskip
    \sbox\z@{$\m@th#1 x$}%
    \ialign{##\cr
      \hidewidth\reflectbox{$\m@th#1\vec{}\mkern4mu$}\hidewidth\cr
      \noalign{\kern-\ht\z@}
      $\m@th#1#2$\cr
    }%
  }%
}
\makeatother

\newcommand{\cevotimes}{\mathrel{\cev{\otimes}}}

\newcommand{\tildeotimes}{\mathrel{\tilde{\otimes}}}

\newcommand{\lb}{\left \langle }
\newcommand{\rb}{\right \rangle }

\newcommand{\lbmed}{\big\langle }
\newcommand{\rbmed}{\big \rangle }

\newcommand{\lbsm}{\langle }
\newcommand{\rbsm}{\rangle }

\newcommand{\Fil}{\mathrm{Fil} }

\title{A general Heegaard Floer surgery formula}

\author{Ian Zemke}
\address{Department of Mathematics\\Princeton University\\  Princeton, NJ, USA}
\email{izemke@math.princeton.edu}
\thanks{IZ was partially supported by NSF grant DMS-2204375.}

\begin{document}
\maketitle

\begin{abstract} We give several new perspectives on the Heegaard Floer Dehn surgery formulas of Manolescu, Ozsv\'{a}th and Szab\'{o}. Our main result is a new exact triangle in the Fukaya category of the torus which gives a new proof of these formulas. This exact triangle is different from the one which appeared in Ozsv\'{a}th and Szab\'{o}'s original proof. This exact triangle simplifies a number of technical aspects in their proofs and also allows us to prove several new results. A first application is an extensions of the link surgery formula to arbitrary links in closed 3-manifolds, with no restrictions on the link being null-homologous. A second application is a proof that the modules for bordered manifolds with torus boundaries, defined by the author in a previous paper, are invariants. Another application is a simple proof of a version of the surgery formula which computes knot and link Floer complexes in terms of subcubes of the link surgery hypercube. As a final application, we show that the knot surgery algebra is homotopy equivalent to an endomorphism algebra of a sum of two decorated Lagrangians in the torus, mirroring a result of Auroux concerning the algebras of Lipshitz, Ozsv\'{a}th and Thurston.
\end{abstract}

\tableofcontents

\section{Introduction}

In this paper, we study the Heegaard Floer groups of Ozsv\'{a}th and Szab\'{o} \cite{OSDisks}. If $Y$ is a closed, oriented 3-manifold, we focus on the invariant $\ve{\CF}^-(Y)$, which is a finitely generated, free chain complex over the power series ring $\bF\llsquare U\rrsquare$. 

  Ozsv\'{a}th and Szab\'{o} developed a technique for computing the Heegaard Floer homology of Dehn surgeries in terms of the knot Floer complex \cite{OSKnots} \cite{OSIntegerSurgeries}. This was later extended to links by Manolescu and Ozsv\'{a}th \cite{MOIntegerSurgery}. To a null-homologous knot $K\subset Y$ with integral framing $\lambda$,  Ozsv\'{a}th and Szab\'{o} constructed a mapping cone chain complex
\[
\bX_\lambda(Y,K)=\Cone( v+h_\lambda\colon \bA(Y,K)\to \bB(Y,K))
\]
which is homotopy equivalent to
\[
\ve{\CF}^-(Y_\lambda(K)):=\CF^-(Y_\lambda(K))\otimes_{\bF[U]} \bF\llsquare U \rrsquare.
\]

In a different direction, Lipshitz, Ozsv\'{a}th and Thurston \cite{LOTBordered} introduced a Heegaard Floer theory  for 3-manifolds with boundary called \emph{bordered Heegaard Floer homology}. To a closed and oriented surface $Z$ (suitably parametrized) they defined an algebra $\cA(Z)$. To a 3-manifold $Y$ with boundary $Z$, they associated an $A_\infty$-module and a type-$D$ module (i.e. a projective $dg$-module over $\cA$), denoted
\[
\widehat{\CFA}(Y)_{\cA(Z)} \quad \text{and} \quad {}^{\cA(Z)}\widehat{\CFD}(Y),
\]
respectively.
If $Y_1$ and $Y_2$ are 3-manifolds with boundaries $Z$ and $-Z$ (resp.), then $\widehat{\CF}(Y_1\cup Y_2)$ may be recovered by a suitable derived tensor product of the corresponding modules for $Y_1$ and $Y_2$. Lipshitz, Ozsv\'{a}th and Thurston have recently constructed a minus version of their theory for the torus algebra. See \cite{LOT-minus-algebra} \cite{LOT-minus-modules} for recent developments.

 Dehn surgery is itself naturally an operation involving 3-manifolds with torus boundary components. Inspired by the invariants of Lipshitz, Ozsv\'{a}th and Thurston,  the author \cite{ZemBordered} showed that the surgery formulas of Manolescu, Ozsv\'{a}th and Szab\'{o} have natural interpretations in terms of a new algebra, the \emph{knot surgery algebra} $\cK$. This is an algebra over an idempotent ring with two elements 
 \[
 \ve{I}=\ve{I}_0\oplus \ve{I}_1, \quad \text{where} \quad \ve{I}_i\iso \bF=\Z/2.
 \]
  We set
\[
\ve{I}_0\cdot \cK\cdot\ve{I}_0=\bF[\scU,\scV]\quad \text{and} \quad \ve{I}_1 \cdot \cK\cdot \ve{I}_1=\bF[U,T,T^{-1}].
\]
We set $\ve{I}_0\cdot \cK\cdot \ve{I}_1=0$. The subspace $\ve{I}_1 \cdot \cK \cdot \ve{I}_0$ contains two special algebra elements $\sigma$ and $\tau$. These are subject to the relations
\[
\sigma \scU=U T^{-1} \sigma,\quad \sigma \scV= T \sigma,\quad \tau \scU=T^{-1} \tau,\quad  \tau \scV= UT \tau.
\]

To a knot $K\subset S^3$, the author showed that the data of the knot surgery complex $\bX_\lambda(Y,K)$ of Ozsv\'{a}th and Szab\'{o} can naturally be encoded using a type-$D$ module $\cX_\lambda(Y,K)^{\cK}$ over the algebra $\cK$. Furthermore, the author constructed a type-$A$ module ${}_{\cK} \cD_0$ for the 0-framed solid torus and observed
\begin{equation}
\bX_\lambda(Y,K)\iso \cX_\lambda(Y,K)^{\cK}\boxtimes {}_{\cK} \cD_0,\label{eq:tensor-prod-intro}
\end{equation}
where $\boxtimes$ denotes the \emph{box tensor product} (a model of the derived tensor product) of Lipshitz, Ozsv\'{a}th and Thurston.

The bordered perspective from \cite{ZemBordered} generalizes naturally to encode Manolescu and Ozsv\'{a}th's link surgery formula. To an integrally framed link $L\subset S^3$ with $\ell$ components, the author reinterpreted the link surgery complex of Manolescu and Ozsv\'{a}th in terms of a type-$D$ module
\[
\cX_{\Lambda}(S^3,L)^{\cL_\ell},
\]
where 
\[
\cL_\ell:=\overbrace{\cK\otimes_{\bF}\cdots \otimes_{\bF} \cK}^{\ell}.
\]
A central result from \cite{ZemBordered}*{Section~1.3} was a connected sum formula for the modules $\cX_{\Lambda}(S^3,L)^{\cL_\ell}$. The connected sum formula has a natural interpretation in terms of gluing bordered manifolds with torus boundaries together, motivated by the following topological fact: if $K_1$ and $K_2$ are two knots in $S^3$, with integral framings $\lambda_1$ and $\lambda_2$, then
\begin{equation}
S^3_{\lambda_1+\lambda_2}(K_1\# K_2)\iso (S^3\setminus \nu(K_1))\cup_\phi (S^3\setminus \nu(K_2)).
\label{eq:Dehn-surgery-connected-sum}
\end{equation}
 In the above, $\phi$ is the diffeomorphism of boundaries which sends the meridian $\mu_1$ to $\mu_2$, and which sends the longitude $\lambda_1$ to $-\lambda_2$. The above fact generalizes to arbitrary Morse framed (i.e. longitudinally framed) links in any 3-manifold.

In this paper, our main goal is to understand the properties of the modules $\cX_{\Lambda}(S^3,L)^{\cL_n}$ and extend their construction naturally to arbitrary links in general 3-manifolds. A secondary goal is to give efficient proofs of the existing surgery formulas in Heegaard Floer theory, which are easier to generalize.

\subsection{A new equivalence}

 \begin{wrapfigure}{r}{5cm}
 	\vspace{-.5cm}
	\begin{center}
		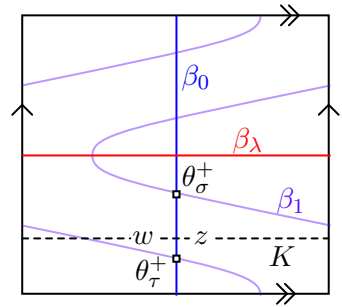
	\end{center}
	\caption{Lagrangians $\b_\lambda$, $\b_0$ and $\b_1$. }
	\label{fig:torus_intro}
\end{wrapfigure}
The main result of this paper is a new proof of the knot and link surgery formulas of Manolescu, Ozsv\'{a}th and Szab\'{o}.
 Our proof follows from a statement about three Lagrangians $\b_{\lambda},$ $\b_0$ and $\b_1$ in the torus. See Figure~\ref{fig:torus_intro}. 

 To $\b_0$ and $\b_1$, we associate the $\bF\llsquare U\rrsquare$-modules $E_0=\bF\llsquare \scU,\scV\rrsquare$ and $E_1=\bF\llsquare U, T,T^{-1}\rrsquare$, respectively. On $E_0$, $U$ acts by $\scU\scV$. We assume that the torus is decorated with a \emph{knot shadow} $K$, by which we mean a simple closed curve intersecting $\b_0$ in a single point. We place two basepoints, $w$ and $z$, along $K$ on either side of $\b_0$. There is a natural way to view $E_0$ and $E_1$ as determining local systems of $\bF\llsquare U\rrsquare$-modules over $\b_0$ and $\b_1$ respectively.  We write $\b_0^{E_0}$ and $\b_1^{E_1}$ for the Lagrangians equipped with the above data.

 We note that unlike $\b_1^{E_1}$, the Lagrangian $\b_0^{E_0}$ is slightly non-traditional as a Lagrangian with a local system. Roughly speaking, the monodromy map for $\b_0^{E_0}$ takes into account both the fundamental groupoid of $\b_0$ (as in a traditional local system) and also multiplicity of a holomorphic polygon at the basepoint $w$. See Section~\ref{sec:knot-surgery-formula} for precise details.
 
  In this setting, a Floer morphism from $\b_0^{E_0}$ to $\b_1^{E_1}$ consists of a pair $\lb \xs, \phi\rb$ where
\[
\xs\in \b_0\cap \b_1\quad \text{and }\quad \phi\in \Hom_{\bF\llsquare U\rrsquare} (E_0,E_1).
\]
 There are two distinguished morphisms
\[
\lb \theta_\sigma^+,\phi^\sigma\rb, \lb \theta_\tau^+, \phi^\tau\rb\colon  \b_0^{E_0}\to \b_1^{E_1}.
\]

We prove the following:
\begin{thm}\label{thm:equivalence-intro}
There is an equivalence in the Fukaya category
\[
\b_{\lambda}\simeq  \Cone\left(\lb \theta_\sigma^+,\phi^\sigma\rb+ \lb \theta_\tau^+, \phi^\tau\rb\colon \b_0^{E_0}\to \b_1^{E_1}\right).
\]
\end{thm}

 \begin{wrapfigure}{r}{5cm}
 	\vspace{-.5cm}
	\begin{center}
		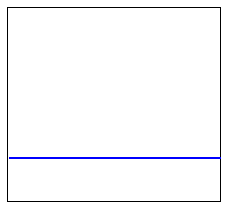
	\end{center}
	\caption{Lagrangians $\g_0,$ $\g_1$ and $\g_2$. }
	\label{fig:54}
\end{wrapfigure}
Theorem~\ref{thm:equivalence-intro} is the most important result of our paper. It is interesting to note that the original surgery exact triangle of Ozsv\'{a}th and Szab\'{o}
\begin{equation*}
\cdots \to \ve{\HF}^-(Y_{\lambda})\to \ve{\HF}^-(Y_{\lambda+1})\to \ve{\HF}^-(Y)\to\cdots\hspace{5.5cm}
\end{equation*}
 may be interpreted as a parallel equivalence $\g_0\simeq \Cone(\theta\colon \g_1\to \g_2)$, for the three Lagrangians $\g_0,\g_1,\g_2$ in $\bT^2$ shown in Figure~\ref{fig:54}. See work of Lekili and Perutz \cite{LPFukaya}*{Theorem~1}, as well as \cite{SeidelTwist} \cite{AbouzaidSurface}*{Lemma~5.3}.

\subsection{A general surgery formula}
\label{sec:intro-gen-surgery-formula}

From Theorem~\ref{thm:equivalence-intro} we may quickly derive a very general version of Ozsv\'{a}th and Szab\'{o}'s surgery formula, valid for any knot in any closed 3-manifold.

 We consider a Heegaard knot diagram $(\Sigma,\as,\bs_0,w,z)$ which has a distinguished knot shadow $K\subset \Sigma$, passing through the points $w$ and $z$. We assume that $K$ intersects a distinguished component $\b_0\in \bs_0$ in a single point, but is otherwise disjoint from $\bs_0$. We let $\bs_{\lambda}$ be obtained by replacing $\b_0$ with a curve $\b_\lambda$ which is parallel to $K$. A neighborhood of $\b_{\lambda}\cup \b_0$ is a punctured torus, and we pick a third set of attaching curves $\bs_1$ by replacing $\b_0$ with a curve $\b_1$ as in Figure~\ref{fig:torus_intro}. We observe that a choice of shadow determines a Morse (longitudinal) framing on the knot $K$.

 In Section~\ref{sec:iterating-background}, we describe how to extend Theorem~\ref{thm:equivalence-intro} to prove another equivalence
 \[
 \bs_{\lambda}\simeq \Cone\left(\lb \theta_\sigma^+,\phi^\sigma\rb+ \lb \theta_\tau^+, \phi^\tau\rb\colon \bs_0^{E_0}\to \bs_1^{E_1}\right)
 \]
 of decorated Lagrangians in $\Sym^g(\Sigma)$. 
 
If we apply the $A_\infty$-functor $\ve{\CF}^-(\as,-)$, we obtain a homotopy equivalence of chain complexes
\[
\ve{\CF}^-(\as, \bs_{\lambda})\simeq \Cone\left(v+h\colon \ve{\CF}^-(\as,\bs_0^{E_0})\to \ve{\CF}^-(\as, \bs_1^{E_1})\right).
\]
In the above, $v$ counts holomorphic triangles with inputs $\theta_\sigma^+$, and $h$ counts triangles with inputs $\theta_\tau^+$. For a null-homologous knot $K\subset Y$, the above complex can be identified with the mapping cone complex $\bX_{\lambda}(Y,K)$ of Ozsv\'{a}th and Szab\'{o}, where $\lambda$ is the longitudinal framing induced by the embedding of $K$ on the surface $\Sigma$.

 In more detail, we observe that
\begin{enumerate}
\item $\ve{\CF}^-(\as,\bs_{\lambda})$ is  $\ve{\CF}^-(Y_{\lambda}(K))$.
\item $\ve{\CF}^-(\as, \bs_0^{E_0})$ is naturally isomorphic to $\bA(K)$, which is a completed version of the knot Floer complex $\cCFK(Y,K)$ over $\bF[\scU,\scV]$.
\item $\ve{\CF}^-(\as, \bs_1^{E_1})$ is identified with $\bB(K)$ when $K$ is null-homologous. More generally, $\ve{\CF}^-(\as, \bs_1^{E_1})$ is isomorphic to a completion of the twisted Floer complex $\underline{\CF}^-(Y;M_{[K]})$, where $M_{[K]}$ denotes the $\bF[H^1(Y)]$-module $\bF[T,T^{-1}]$ with action
$e^{\eta}\cdot T^i=T^{i+\eta([K])}$, where $\eta\in H^1(Y)$.
\end{enumerate}

We describe in Section~\ref{sec:iterating-background} how to iterate the equivalence in Theorem~\ref{thm:equivalence-intro} to obtain a link surgery formula for a Morse framed link $L$ in closed 3-manifold $Y$. This takes the form of a chain complex $\bX_{\Lambda}(Y,L)$ which is homotopy equivalent to $\ve{\CF}^-(Y_{\Lambda}(L))$.
 The algebraic structure of the chain complex $\bX_{\Lambda}(Y,L)$ is most naturally interpreted in terms of a finitely generated type-$D$ module $\cX_\Lambda(Y,L)^{\cL_\ell}$ which encodes $\bX_{\Lambda}(Y,L)$ similarly to Equation~\eqref{eq:tensor-prod-intro}. 
 
 The local nature of the above proof of the surgery formula makes our proof particularly flexible for extensions and refinements of the surgery formula. In a future work, we use it to build a bimodule between the algebra $\cK$ and the torus algebra of Lipshitz, Ozsv\'{a}th and Thurston, which we hope to use to relate the two theories.

\subsection{Invariance of modules}

In this paper, we adopt the following convention:

\begin{define}  A \emph{bordered manifold with torus boundaries} is a compact, oriented 3-manifold $Y$ with boundary, such that each component $Z$ of $\d Y$ is a torus which is equipped with an oriented basis $(\mu,\lambda)$ of $H_1(Z)$. 
\end{define}

In \cite{ZemBordered}, we observed that it is possible to use the link surgery for links in $S^3$ to define a type-$D$ module for a  bordered 3-manifold $Y$ with torus boundaries. The description given therein required $Y$ to be presented as Dehn surgery on a framed link $L$ in the complement of an unlink in $S^3$. The dependence of the resulting module on the choice Dehn surgery presentation of $Y$ was not clear, though we conjectured in \cite{ZemBordered} that the resulting modules were invariants of $Y$.
 
 The exact triangle from Section~\ref{sec:intro-gen-surgery-formula} gives a proof of this conjecture, and also gives a more direct construction of the modules:

\begin{thm}
\label{thm:intro-invariance}
If $Y$ is a bordered 3-manifold with torus boundaries, then the chain homotopy type of the type-$D$ module $\cX(Y)^{\cL_n}$ is an invariant of $Y$.
\end{thm}

\begin{rem} In \cite{ZemBordered}, we also considered an additional decoration, called an \emph{arc system}. The construction in this paper is equivalent to choosing an \emph{alpha-parallel} arc systems, in the terminology from \cite{ZemBordered}.
\end{rem}

Theorem~\ref{thm:intro-invariance} has several concrete applications. The first of which is functoriality of the modules under Dehn surgery on link components.

\begin{thm} Suppose that $(Y',L')$ is a link with Morse framing $\Lambda'$, obtained from a link $(Y,L)$ with Morse framing $\Lambda$ by performing Dehn surgery on one component of $L$ according to its framing. Then
\[
\cX_{\Lambda'}(Y',L')^{\cL_{\ell-1}}\iso \cX_{\Lambda}(Y,L)^{\cL_\ell}\boxtimes{}_{\cK} \cD_0.
\]
\end{thm}

In \cite{ZemBordered}*{Section~1.3}, the author proved a connected sum formula for the link surgery formula. Combining our proof of invariance of the modules with this connected sum formula as well as the topological perspective from Equation~\eqref{eq:Dehn-surgery-connected-sum} relating Dehn surgery on connected sums and gluing along torus boundary components, we obtain a gluing formula for the bordered modules.

  To state our result,  we recall first that in \cite{ZemBordered}*{Section~8.3} we described an operation which transforms a type-$D$ action of $\cK$ into a type-$A$ action of $\cK$.  Using this operation, if $Y$ was a bordered manifold with torus boundaries and one boundary component was distinguished, we constructed a $DA$-bimodule ${}_{\cK} \cX(Y)^{\cL_{n-1}},$
where $n=|\d Y|$.

\begin{thm}
\label{thm:naturality-gluing-intro} Suppose $Y_1$ and $Y_2$ be two bordered manifolds with torus boundaries, each with a distinguished boundary component $Z_1$ and $Z_2$, respectively. We let $\phi\colon Z_1\to Z_2$ be the orientation reversing diffeomorphism sending $\mu_1$ to $\mu_2$ and $\lambda_1$ to $-\lambda_2$. Let $\cX(Y_1)^{\cL_{n_1-1}\otimes \cK}$ be the type-$D$ module for $Y_1$ and let ${}_{\cK} \cX(Y_2)^{\cL_{n_2-1}}$ be the type-$DA$ bimodule for $Y_2$. Here the copies of $\cK$ correspond to the distinguished boundary components. Then
\[
\cX(Y_1)^{\cL_{n_1-1}\otimes \cK}\boxtimes {}_{\cK}\cX(Y_2)^{\cL_{n_2-1}}\simeq \cX(Y_1\cup_\phi Y_2)^{\cL_{n_1+n_2-1}}. 
\]
\end{thm}

\subsection{Sublinks}

Theorem~\ref{thm:intro-invariance} also gives a conceptually simple proof of the following technique for computing knot and link Floer complexes in terms of the link surgery formula. This technique is a folklore result known to experts, though a general proof has not appeared in the literature. 

By definition, if $K\subset Y$ is a framed knot, then
\[
\cX_{\lambda}(Y,K)^{\cK}\cdot \ve{I}_0\iso \cCFK(Y,K),
\]
where we view the knot Floer complex $\cCFK(Y,K)$ as a type-$D$ module over $\bF\llsquare\scU,\scV\rrsquare$. 

In particular, if $(Y,L)$ is a link with Morse framing $\Lambda$ and $K\subset L$ is a distinguished component, then we can tensor $\cX_{\Lambda}(Y,L)^{\cL}$ with the type-$A$ modules ${}_{\cK} \cD_0$ for solid tori for all components of $L$ except for $K$. We are left with a type-$D$ module over $\cK$, which we denote $\cX_{\Lambda}(Y,L;K)^{\cK}$.  Theorem~\ref{thm:intro-invariance}  implies that idempotent 0 of this module coincides with $\cCFK(Y',K)$, where $Y'$ is the 3-manifold obtained by surgering $Y$ along $L\setminus K$.

 In terms of the link surgery formula of Manolescu and Ozsv\'{a}th, our invariance statement from Theorem~\ref{thm:intro-invariance} has the following interpretation.  We consider the complex 
\[
\bX_{\Lambda}(Y,L;K)\subset \bX_{\Lambda}(Y,L)
\]
consisting of the codimension 1 subcube where the coordinate for $K$ is 0. Then the subcube $\bX_{\Lambda}(Y,L;K)$ is naturally a module over $\bF\llsquare \scU,\scV \rrsquare$, and furthermore
\begin{equation}
{}_{\bF\llsquare \scU,\scV\rrsquare}\bX_{\Lambda}(Y,L;K)\simeq {}_{\bF\llsquare \scU,\scV\rrsquare} \cCFK(Y',K).\label{eq:compute-CFK}
\end{equation}
A similar discussion holds when $K$ is a link with several components. See Section~\ref{sec:sublink-surgery-formula}. The above fact is used in \cite{BLZLatticeLink} to compute the link Floer complexes of all algebraic links in $S^3$. 

\begin{rem}  It is likely that one could adapt the work of Manolescu and Ozsv\'{a}th \cite{MOIntegerSurgery} to prove Equation~\eqref{eq:compute-CFK} and similar extensions (at least when $K$ is a null-homologous knot). Our techniques avoid some technical challenges which would appear in such a proof, such as truncation arguments, and are more user-friendly and flexible.
\end{rem}

\subsection{Endomorphism algebras}

Our framework also gives a natural analog of Auroux's description \cite{AurouxBordered} of the Lipshitz--Ozsv\'{a}th--Thurston algebra $\cA(\bT^2)$ in terms of the endomorphism algebra of the Fukaya category of the torus. We consider the endomorphism algebra of $\b_0^{E_0}\oplus \b_1^{E_1}$, where $\b_0^{E_1}$ and $\b_1^{E_1}$ are the decorated Lagrangians appearing in the statement of Theorem~\ref{thm:equivalence-intro}. In Section~\ref{sec:endomorphisms}, we define a restricted subspace of endomorphisms of $\b_0^{E_0}\oplus \b_1^{E_1}$ which we call \emph{filtered} morphisms, and which we denote $\End_{\Fil}(\b_0^{E_0}\oplus \b_1^{E_1})$. We prove:

\begin{thm} There is a homotopy equivalence of $A_\infty$-algebras:
\[
\cK\simeq \End_{\Fil}\left(\b_0^{E_0}\oplus \b_1^{E_1}\right).
\]
\end{thm}

 \subsection{Comparison with existing techniques}

  We now describe some ways in which our Theorem~\ref{thm:equivalence-intro} simplifies previous descriptions of the surgery formulas
of Manolescu, Ozsv\'{a}th and Szab\'{o}.

 If $N>0$, Ozsv\'{a}th and Szab\'{o} \cite{OSIntegerSurgeries} prove an equivalence in the Fukaya category of $\bT^2$ which gives a homotopy equivalence 
\[
\ve{\CF}^-(Y_{\lambda}(K))\simeq \Cone\left(F_W\colon  \ve{\CF}^-(Y_{\lambda+N}(K))\to \underline{\ve{\CF}}^-(Y)\right).
\] Here $\underline{\ve{\CF}}^-$ is a version of Floer homology with twisted coefficients and $F_W$ is a cobordism map. Note that this mapping cone is not the same as $\bX_{\lambda}(Y,K)$.  Instead, Ozsv\'{a}th and Szab\'{o}'s proof requires two additional technical steps:
  \begin{enumerate}
  \item \emph{A large surgery formula}. If $N\gg 0$ they identify $\ve{\CF}^-(Y_{\lambda+N}(K))$ with a subspace of the knot Floer complex of $K$ in $Y$.
  \item \emph{A sequence of truncation arguments}. They describe several algebraic truncation operations. One such operation is quotienting by the image of $U^{\delta}$ for some $\delta \gg 0$. Another reduces $\bX_{\lambda}(Y,K)$ to a finitely generated complex over $\bF\llsquare U\rrsquare$. 
   \end{enumerate}
      They show that all suitably large truncations of $\ve{\CF}^-(Y_{\lambda}(K))$ and $\bX_{\lambda}(Y,K)$ are homotopy equivalent, and conclude that therefore $\ve{\CF}^-(Y_{\lambda}(K))$ and $\bX_{\lambda}(Y,K)$ must be homotopy equivalent. Manolescu and Ozsv\'{a}th's proof of the link surgery formula follows from similar logic, with a substantial increase in complexity of the algebraic operations involved. Note that the large surgery formula only holds for rationally null-homologous knots and hence a direct adaptation of their proof for homologically essential knots seems challenging.

 Our approach to Theorem~\ref{thm:equivalence-intro} avoids both the large surgery formula and the truncation procedure, thereby avoiding the above steps in the original proof of Ozsv\'{a}th and Szab\'{o}.

\subsection{Additional properties and remarks}

Our framework also allows us to prove several new properties about the link surgery formulas. Of particular note are formulas which compute the $H_1(Y)/\Tors$ action using the link surgery complex. See Section~\ref{sec:H1-action}. Additionally, our techniques give simple proofs of Maslov and Alexander grading formulas on the link surgery formula and its refinements. See Section~\ref{sec:gradings}.

Additionally, we investigate in this paper the role of completions within the knot surgery formula. In \cite{ZemBordered}, we equipped $\cK$ with a linear topology from a filtration on $\cK$ by right ideals. The algebra $\cK$ was naturally a \emph{linear topological chiral algebra}, in the terminology of \cite{BeilinsonChiral}. In this paper, we also investigate a different topology, the $U$-adic topology $\cK$, arising from the filtration by the two sided ideals $(U^i)$, $i\ge 0$. We prove that this topology can also be used in the surgery formula. In a future work, we investigate \emph{self-gluing} in the context of these theories. Interestingly, a self-gluing formula seems only to be possible with the $U$-adic topology, and not with respect to the chiral topology.

We also note that that there exist several other interesting relations between the Heegaard Floer surgery formulas and the Fukaya category of the torus. Recently, Hanselman \cite{HanselmanCFK} has interpreted the data of the surgery formula in terms of immersed curves on the torus in a similar spirit to \cite{HRWImmersedCurves}. Compare also the earlier work of Kotelskiy, Watson, and Zibrowius \cite{KWZMnemonic}. It would be interesting to understand the relation between the above perspectives and the work of this paper.

\subsection*{Acknowledgments}

 The author thanks Jonathan Hanselman, Adam S. Levine,  Robert Lipshitz, Ciprian Manolescu and Peter Ozsv\'{a}th for interesting conversations. The author is also very thankful to his collaborators Maciej Borodzik, Kristen Hendricks, Jennifer Hom, Beibei Liu, Matthew Stoffregen  for their insights on related projects.

\section{Preliminaries}

\subsection{Type-\texorpdfstring{$D$}{D} and \texorpdfstring{$A$}{A} modules}

We assume the reader is familiar with the framework of type-$D$ and $A$ modules of Lipshitz, Ozsv\'{a}th and Thurston \cite{LOTBordered} \cite{LOTBimodules}. We recall these notations very briefly. Suppose that $\cA$ is an associative algebra over an idempotent ring $\ve{i}$, which is of characteristic 2. We write $\mu_2$ for multiplication. We recall that a \emph{right type-$D$} module over $\cA$, denoted $X^\cA$, is a pair $(X,\delta^1)$ where $X$ is a right $\ve{i}$-module and
\[
\delta^1\colon X\to X\otimes_{\ve{i}} \cA
\]
is an $\ve{i}$-linear map which satisfies
\[
(\id_X\otimes \mu_2)\circ (\delta^1\otimes \id_\cA)\circ\delta^1=0
\]

A \emph{left type-$A$ module} ${}_{\cA} Y$ in their notation is identical to an $A_\infty$-module. See \cite{KellerNotes} for additional background. Briefly, an $A_\infty$-module is a left $\ve{i}$-module $Y$ with a collection of $\ve{i}$-linear maps 
\[
m_{j+1}\colon \overbrace{\cA\otimes_{\ve{i}}\cdots \otimes_{\ve{i}} \cA}^{j}\otimes_{\ve{i}} X\to X, \quad j\ge 0,
\]
which satisfy
\[
\begin{split}
&\sum_{0\le i\le  n} m_{n-i+1}(a_n,\dots, m_{i+1}(a_{i},\dots, a_1,x))\\
+&\sum_{1\le i\le n-1} m_{n}(a_n,\dots, \mu_2(a_{i+1},a_{i}),\dots, a_1, x)=0
\end{split}
\]

Finally, given a type-$D$ and type-$A$ modules $X^{\cA}$ and ${}_{\cA} Y$, then under suitable boundedness hypotheses there is a model of the derived tensor product, called the \emph{box tensor product}, which is a chain complex
\[
X^{\cA}\boxtimes {}_{\cA} Y=(X\otimes_{\ve{i}} Y, \d_{\boxtimes})
\]
One defines $\delta^j\colon X\to X\otimes \otimes^j \cA$ by iterating $\delta^1$ $j$-times (with $\delta^0=\id_X$). Write $T^* \cA=\prod_{j\ge 0} \otimes^j \cA$ and set $\delta\colon X\to X\otimes T^* \cA$ to be the sum of all $\delta^j$. The differential $\d_{\boxtimes}$ is then given as the composition
\[
\begin{tikzcd} X\otimes Y\ar[r, "\delta\otimes \id"]& X\otimes T^* \cA\otimes  Y \ar[r, "\id\otimes m_*"] & X\otimes Y
\end{tikzcd}
\]
We refer the reader to the work of Lipshitz, Ozsv\'{a}th and Thurston for more details.

\subsection{Hypercubes of chain complexes}

We now recall the formalism of \emph{hypercubes of chain complexes}, due to Manolescu and Ozsv\'{a}th \cite{MOIntegerSurgery}*{Section~5}. We write $\bE_n=\{0,1\}^n$ for the $n$-dimensional cube. If $\veps,\veps'\in \bE_n$, we write $\veps\le \veps'$ the inequality holds for all components.

\begin{define} A \emph{hypercube of chain complexes} $\cC=(C_{\veps},D_{\veps, \veps'})_{\veps\in \bE_n}$ consists of a collection of vector spaces $C_{\veps}$, together with a linear map
\[
D_{\veps,\veps'}\colon C_{\veps}\to C_{\veps'}
\]
whenever $\veps\le \veps'$. Furthermore, we assume that if $\veps\le \veps''$, then
\[
\sum_{\substack{\veps'\in \bE_n\\ \veps\le \veps'\le \veps''}} D_{\veps',\veps''}\circ D_{\veps,\veps'}=0. 
\]
\end{define}

\subsection{Twisted complexes}

We briefly recall the formalism of twisted complexes. See \cite{BondalKapronov} \cite{SeidelFukaya}*{Sections~3k,l}  for further details.  Suppose $\cC$ is an $A_\infty$-category. For convenience, we assume the morphism spaces have characteristic 2. The additive enlargement $\Sigma \cC$ is as follows. Objects of $\Sigma \cC$ consist of collections $(X_i,V_i)_{i\in I}$ such that $I$ is a finite index set, each $X_i$ is an object of $\cC$ and each $V_i$ is a finite dimensional, graded vector space.  If $X=(X_i,V_i)_{i\in I}$ and $Y=(Y_j,W_j)_{j\in J}$ are objects of $\Sigma \cC$, then $\Hom(X,Y)$ is defined to be the direct sum over $(i,j)\in I\times J$ of $\Hom(V_i,W_j)\otimes \Hom_{\cC}(X_i,Y_j).$ It is straightforward to verify that $\Sigma \cC$ is naturally an $A_\infty$-category, where the $A_\infty$ composition $\mu_n^{\Sigma \cC}$ of a sequence of morphisms is computed by applying $\mu_n$ to the $\cC$ factors of the morphisms, and composing (in the normal sense) the linear map components of the morphisms.

A \emph{twisted complex} in $\cC$ consists of an object $X$ of $\Sigma \cC$, together with an endomorphism $\delta_X\in \Hom(X,X)$ of degree $-1$, such that
\[
\sum_{n \ge 1} \mu_{n}^{\Sigma \cC}(\delta_X,\dots, \delta_X)=0. 
\]
Note that some assumption is necessary to ensure that the above sum is finite. Seidel assumes that $\delta_X$ is \emph{strictly lower triangular} with respect a filtration on $X$. Although this is somewhat restrictive, this condition will typically be satisfied in the cases of interest in our paper since we focus on  hypercubes (which are filtered by $\bE_n:=\{0,1\}^n$) and similarly filtered chain complexes. 

Twisted complex $\Tw(\cC)$ naturally form an $A_\infty$-category. Morphisms are the same as in $\Sigma \cC$. Given a composable sequence of morphisms
\[
\begin{tikzcd}
X_0\ar[r, "f_{0,1}"] & \cdots \ar[r, "f_{n-1,n}"] & X_n,
\end{tikzcd}
\]
one defines
\[\begin{split}
&\mu_n^{\Tw}(f_{0,1},\dots, f_{n-1,n})\\
=&\sum_{i_0,\dots, i_n\ge 0} \mu_{n+i_0+\cdots+i_n}^{\Sigma \cC}(\overbrace{\delta_{X_0},\dots, \delta_{X_0}}^{i_0}, f_{0,1}, \overbrace{\delta_{X_1},\dots, \delta_{X_1}}^{i_1}, \dots, f_{n-1,n}, \overbrace{\delta_{X_n},\dots, \delta_{X_n}}^{i_n})
\end{split}
\]
If $f\colon X\to Y$ is a morphism of twisted complexes which satisfies $\mu_1(f)=0$, then we may naturally construct the mapping cone $\Cone(f)$, which is also a twisted complex. The underlying space of the cone is the union of the elements of $X$, and $\delta_{\Cone(f)}$ is the sum of $\delta_X$, $\delta_Y$ and $f$.

The most important examples are as follows, which are standard in the literature:
\begin{enumerate}
\item If $\cA$ is an $A_\infty$-algebra over an idempotent ring $\ve{i}$, then we may view $\cA$ as an $A_\infty$ category whose objects are idempotents $i\in \ve{i}$, and such that $\Hom(i,j)=i\cdot \cA\cdot j$. The category of twisted complexes $\Tw(\cA)$ is the same as the category of type-$D$ modules of Lipshitz, Ozsv\'{a}th and Thurston.
\item If $(W,\omega)$ is a symplectic manifold (satisfying suitable conditions for the Fukaya category to be defined, e.g. exactness), a \emph{hypercube of Lagrangians} $(L_\veps, \theta_{\veps,\veps'})_{\veps\in \bE_n}$ in the Fukaya category of $W$ is the same as a twisted complex of Lagrangians which is filtered by $\bE_n$. Concretely, this consists of a collection of Lagrangians $L_{\veps}$ indexed by $\bE_n$, together with morphisms $\theta_{\veps,\veps'}\in \CF(L_{\veps},L_{\veps'})$ ranging over $\veps<\veps'$. Furthermore, the following compatibility condition is satisfied whenever $\veps<\veps'$:
\[
\sum_{\veps=\veps_0<\cdots<\veps_n=\veps'} \mu_n(\theta_{\veps_0,\veps_1},\dots, \theta_{\veps_{n-1},\veps_n})=0.
\]
\end{enumerate}

\subsection{Lagrangians with local systems}
\label{sec:local-systems}

In Lagrangian Floer theory, one often considers Lagrangians with local systems. We recall that a local system over a Lagrangian consists of a vector space $E$ together with a monodromy representation of the fundamental groupoid of $L$ into $\Aut(E)$.
 We are interested in a special case of this construction where we have a choice of oriented, codimension 1 hypersurface $S$ in the Lagrangian $L$ such that the monodromy map is determined by the intersection number of a path with $S$, via a morphism of algebras
 \[
 \rho\colon \bF[\Z]\to \End_{\bF[U]}(E).
 \]
 We write elements of $\bF[\Z]$ as $e^{n}$, and refer $\rho$ as the monodromy.

 We will be focused on the case that the Lagrangian is the torus $\bT_{\g}=\g_1\times \cdots \times \g_n\subset \Sym^n(\Sigma)$ for a set of attaching curves $\gs\subset \Sigma$, and the hypersurface is 
\[
(K\times \Sym^{n-1}(\Sigma))\cap \bT_\g
\]
where $K$ is a simple closed curve on the Heegaard diagram.  If $L=\bT_{\g}$ for a set of attaching curves $\gs\subset \Sigma$, then we frequently abbreviate $\gs^{E}$ for a triple $(\bT_{\g}, E, \rho)$.

A Floer morphism from $\gs^{E}$ to $\gs^{E'}$ consists of a pair $\lb\xs,\phi\rb$ where $\xs\in \bT_{\g}\cap \bT_{\g'}$ and $\phi\colon E\to E'$ is an $\bF[U]$-equivariant map.

\begin{rem}
In practice, some of the local systems we consider are more naturally viewed as determining a monodromy map $\rho$ which takes values in $\End_{\bF[U]}(E,U^{-1}\cdot E)$ (i.e. negative powers of $U$ are allowed). See specifically the monodromy on $E_0$ in Section~\ref{sec:knot-surgery-formula}. 
\end{rem}

If $\gs_1^{E_1},\dots, \gs_n^{E_n}$ are attaching curves with local systems, with monodromy maps $\rho_1,\dots, \rho_n$, respectively, we now explain the holomorphic polygon maps.  If
\[
\lb \xs_{i,i+1}, \phi_{i,i+1}\rb \colon \gs_i^{E_i}\to \gs_{i+1}^{E_{i+1}}
\] 
is a sequence of Floer morphisms, for $1\le i\le n-1$, then the holomorphic polygon counts holomorphic $n$-gons of index $3-n$ with the inputs $\xs_{i,i+1}$. The outputs will consist of pairs $\lb \ys, f\rb$ where $f\colon E_1\to E_n$. If $\psi$ is a class of polygons, the output $f$ will be as follows. Let $m_i:=\# (\d_{\g_i}(\phi)\cap K)$. The morphism contributed by the polygon counting map is the composition
\begin{equation}
U^{n_w(\phi)}\left( \rho_n(e^{m_{n}})\circ \phi_{n-1,n}\circ \rho_{n-1}(e^{m_{n-1}})\circ \cdots \circ \phi_{1,2}\circ \rho_1(e^{m_1})\right).
\label{eq:local-systems}
\end{equation}
Schematically, we think of each Floer morphism $[\xs_{i,i+1}, \phi_{i,i+1}]$ as determining an $\bF[U]$-module map from $E_i$ to $E_{i+1}$ at each vertex of the polygon, and each $\rho_i(e^{m_i})$ as determining an $\bF[U]$-module endomorphism of $E_i$ for each edge of the polygon. The $\bF[U]$-module map in the output is the composition of all of these maps. See Figure~\ref{fig:45} for a schematic.

\begin{figure}[h]
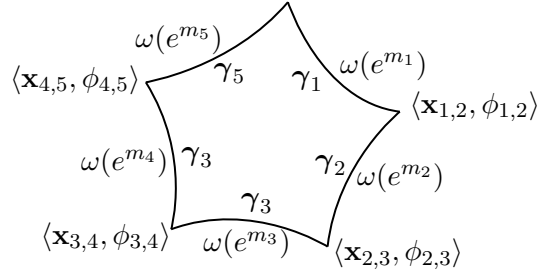
\caption{The $\bF[U]$-module map in the output is obtained by composing the shown morphisms in clockwise order, and then multiplying by $U^{n_w(\phi)}$.}
\label{fig:45}
\end{figure}

\begin{rem} 
Given an attaching curve $\gs$, we may always view it as having a trivial local system by setting $E=\bF[U]$ and $\rho(e^n)=\id$ for all $n$. 
\end{rem}

We will also be interested in the case that we have a collection of curves $K_1,\dots, K_n\subset \Sigma$ and an $\bF[U_1,\dots, U_n]$-module $E$. In this case, we are interested in a map
\[
\rho\colon \bF[\Z^n]\to \End_{\bF[U_1,\dots, U_n]}(E).
\]
 In the holomorphic polygon maps, we  use
\[
\rho\left(e^{(\#\d_{\g}(\phi)\cap K_1,\dots, \#\d_{\g}(\phi)\cap K_n)}\right),
\]
instead of $\rho(e^{\# \d_\g(\phi)\cap K})$.

See Section~\ref{sec:knot-surgery-formula} for the examples of local systems that appear in the surgery formula.

\subsection{Linear topological spaces}

\label{sec:linear-space-background}

We now recall some basic material on completions. A \emph{linear topological vector space} $X$ is a vector space $X$ together with a decreasing filtration by subspaces $X_{i}$ indexed by $i\in I$ for some partially ordered, directed set $I$. Such a filtration induces a topology on $X$, which has a basis of open sets at a point $x\in X$ of the form $x+X_{i}$, ranging over $i\in I$.

  Equivalently, one can consider vector spaces $X$ equipped with topologies such that addition is continuous and $0$ has a basis of open sets consisting of subspaces of $X$. For such a topology, we obtain a filtration consisting of the open subspaces, ordered by reverse inclusion.
  
  The \emph{completion} of a linear topological vector space $X$ is given by the inverse limit
  \[
\ve{X}:=  \varprojlim_{i\in I} X/X_{i}.
  \]
  Alternatively, one can identify the above completion with the set of Cauchy nets in $X$ (or Cauchy sequences if $X$ is first countable). Note that $\ve{X}$ is also a linear topological space.
  
  For our purposes, we define the set of morphisms between two linear topological spaces $X$ and $Y$ to be the set of continuous linear maps between completions, $\Hom(\ve{X},\ve{Y})$, equipped with the uniform topology. In particular, in this formulation of the category of linear topological spaces, every $X$ is isomorphic to its completion $\ve{X}$.

 \begin{example} The \emph{product topology} in the category of linear topological spaces coincides with the product topology in the category of topological spaces, i.e. if $(X_i)_{i\in I}$, are linear topological spaces, then a basis of open subspaces in $\prod_{i\in I} X_i$ consists of $\prod_{i\not \in S} X_i\times \prod_{i\in S} U_i$, where $S$ ranges over finite subsets of $I$ and $U_i\subset X_i$ are open subspaces. 
 \end{example}
  
  We recall that given two linear topological spaces $X$ and $Y$, there are multiple ways to topologize the tensor product $X\otimes Y$. For our purposes, the following two tensor products are the most important:
  \begin{enumerate}
  \item The \emph{standard} tensor product, $X\otimes^! Y$. A subspace $E\subset X\otimes Y$ is open if and only if there are open subspaces $U\subset X$ and $V\subset Y$ such that
  \[
  U\otimes Y+X\otimes V\subset E.
  \]
  \item The \emph{chiral} tensor product, $X\vecotimes Y$. A subspace $E\subset X\otimes Y$ is open if and only if there is an open subspace $U\subset X$ so that $U\otimes Y\subset E$, and for each $x\in X$, there is an open subspace $V_x\subset Y$ so that $ x \otimes V_x\subset E$. 
  \end{enumerate}
  The chiral tensor product is from Beilinson \cite{BeilinsonChiral}. See \cite{Positelski-Linear}*{Section~12} for additional background and exposition. Of course we can switch the roles of $X$ and $Y$ and consider $X\cevotimes Y$.
  
 \begin{rem} The motivation for these two tensor product operations is as follows. If $\cA$ is an associative algebra which is topologized by a family of 2-sided ideals $I_n$, then the map
  \[
  \mu_2\colon \cA\otimes^! \cA\to \cA
  \]
  is continuous. On the other hand, if $\cA$ satisfies the following:
  \begin{enumerate}
  \item $\cA$ is topologized by a family of right ideals $I_n$ (i.e. $I_n\cdot \cA\subset I_n$); and
  \item For each $x\in \cA$, the map $\mu_2\colon x\otimes \cA\to \cA$ is continuous,
  \end{enumerate}
  then the map 
  \[
  \mu_2\colon \cA\vecotimes \cA\to \cA
  \]
  is continuous. 
\end{rem}

Note that if $\cA$ is an achiral linear topological algebra, we can consider achiral linear topological type-$A$ modules ${}_{\cA} M$. These are linear topological $\ve{i}$-modules $M$ with a collection of maps
\[
m_{j+1}\colon \underbrace{\cA\otimes^!\cdots \otimes^! \cA}_{j}\otimes^! M\to M, \quad j\ge 0,
\]
which are continuous and satisfy the $A_\infty$-module relations.

\begin{rem} If unspecified, a linear topological algebra or module will mean an achiral linear topological algebra or module.
\end{rem}

If $\cA$ is a linear topological chiral algebra, we can also consider the category ${}_{\cA}\Mod_{\ch}$ \emph{chiral type-$\cA$ modules}, which are collections as above such that $m_{j+1}$ is continuous as a map
\[
m_{j+1}\colon \cA\vecotimes \cdots \vecotimes \cA\vecotimes M\to M.
\]
Naturally, we can also consider chiral type-$D$ modules, which are pairs $X^{\cA}=(X,\delta^1)$ where
\[
\delta^1\colon X\to X\vecotimes \cA
\]
is continuous and satisfies the type-$D$ structure relations. 

\begin{rem}
For any linear topological spaces $X$ and $Y$, there is a continuous map
\[
X\vecotimes Y\to X\otimes^! Y.
\]
In particular, any achiral algebra may naturally be viewed as a chiral linear topological algebra.
\end{rem}

\section{The knot and link surgery complexes}
\label{sec:describe-surgery-complexes}

In this section, we describe  the link surgery complex for a Morse framed link $L$ in a closed 3-manifold $Y$.
For null-homologous knots and links, we will see in Theorem~\ref{thm:MO=our-version} that the complex is isomorphic to the knot and link surgery formulas of Manolescu, Ozsv\'{a}th, and Szab\'{o} \cite{OSIntegerSurgeries} \cite{MOIntegerSurgery}. We note, however, that our description extends without complication to any link in a closed 3-manifold (in particular, we make no assumption that $L$ is even rationally null-homologous).

\subsection{Heegaard diagrams}

We begin by describing the Heegaard diagrams which appear in our description of the link surgery formula. We pick a Heegaard link diagram $(\Sigma,\as,\bs,\ws,\zs)$ for $(Y,L)$. By definition \cite{OSLinks}*{Section~1.2}, this consists of the following:
\begin{enumerate}
\item A Heegaard splitting $\Sigma$ of $Y$ such that $L$ intersects $\Sigma$ transversely at the finite  set $\ws\cup \zs$. Write $U_{\a}$ and $U_{\b}$ for the two components of $Y\setminus \Sigma$. 
\item We assume that $L\cap U_{\a}$ and $L\cap U_{\b}$ consist of boundary parallel tangles.
\item $\as$ consists of a collection of pairwise disjoint, simple closed curves on $\Sigma$ which cut $\Sigma$ into a union of planar surfaces, each component of which contains exactly one point from $\ws$ and one point from $\zs$. Furthermore, the two points of $\ws\cup \zs$ contained in a component of $\Sigma\setminus \ws\cup \zs$ are from the same component of $L$. The same condition holds for $\bs$. 
\end{enumerate}

\begin{define}
We say that a Heegaard link diagram $(\Sigma,\as,\bs,\ws,\zs)$ is a \emph{meridianal link diagram} if for each component $K_i$ of $L$, there is a special beta curve $\b_{0,i}$ such that $w_i$ and $z_i$ are on opposite sides of $\b_{0,i}$, as in Figure~\ref{fig:46}. 
\end{define}

\begin{figure}[h]
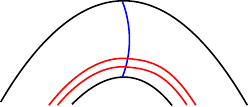
\caption{A special beta curve $\b_0$ of a meridianal Heegaard link diagram.}
\label{fig:46}
\end{figure}

\subsection{The knot and link Floer complexes}

In this section, we recall some background on knot Floer homology \cite{OSKnots} \cite{RasmussenKnots} and link Floer homology \cite{OSLinks}. 

To a knot $K\subset Y$, there is a group
\[
\cCFK(Y,K)
\]
defined using a Heegaard knot diagram $(\Sigma,\as,\bs,w,z)$. This complex is freely generated over intersection points $\xs\in \bT_{\a}\cap \bT_{\b}$ by the polynomial ring $\bF[\scU,\scV]$. The differential counts pseudo-holomorphic disks via the formula:
\begin{equation}
\d \xs=\sum_{\substack{\phi\in \pi_2(\xs,\ys)\\ \mu(\phi)=1}} \# (\cM(\phi)/\R) \scU^{n_w(\phi)} \scV^{n_z(\phi)} \cdot \ys. \label{eq:knot-floer-differential}
\end{equation}

If $L\subset Y$ is a link, there is an analogous complex $\cCFL(Y,L)$. Given a Heegaard link diagram $(\Sigma,\as,\bs,\ws,\zs)$, we define $\cCFL(Y,L)$ to be free generated over $\bF[\scU_1,\scV_1,\dots, \scU_\ell, \scV_\ell]$ by intersection points $\xs\in \bT_{\a} \cap \bT_{\b}$. The differential is similar to Equation~\eqref{eq:knot-floer-differential}, except weighting a disk by the product of $\scU_i^{n_{w_i}(\phi)} \scV_i^{n_{z_i}(\phi)}$. 

We say that any collection $\ve{p}\subset \ws\cup \zs$ which contains one point from each link component is a \emph{complete collection} of basepoints. Given a complete collection $\ve{p}$, there is a map
\[
\frs_{\ps}\colon \bT_{\a}\cap \bT_{\b}\to \Spin^c(Y).
\]
See \cite{OSDisks}*{Section~2.6}. We note that
\[
\frs_{\ws}(\xs)-\frs_{\zs}(\xs)=\PD[L].
\]

There is a well-defined grading $\gr_{\ve{p}}$, defined on intersection points where $\gr_{\ve{p}}(\xs)$ is torsion. This is extended over the ring $\bF[\scU_1,\scV_1,\dots, \scU_\ell, \scV_\ell]$ by declaring $\scU_i$ to have $\gr_{\ve{p}}$-grading $-2$ if $w_i\in \ve{p}$ and grading 0 otherwise. Similarly $\scV_i$ is given grading $-2$ if $z_i\in \ve{p}$ and 0 otherwise.

If $L$ is a null-homologous link in $S^3$, then there is an $\ell$-component Alexander grading $A=(A_1,\dots, A_\ell)$ which takes values in the lattice
\[
\bH(L)=\prod_{i=1}^\ell \Z+\lk(L_i,L\setminus L_i). 
\]

\subsection{The knot surgery complex}
\label{sec:knot-surgery-formula}

We begin by formulating the knot surgery complex for a knot $K\subset Y$.  We pick a meridianal Heegaard diagram $(\Sigma,\as,\bs,w,z)$ for $(Y,K)$.

 We pick a \emph{shadow} of the knot $K$ on $\Sigma$, by which we mean a concatenation of two embedded arcs connecting $w$ to $z$, such that one arc avoids $\as$ and the other arc avoids $\bs$. Pushing the interior of the arc which avoids $\as$ (resp. $\bs$) into the alpha-handlebody (resp. beta-handlebody) yields a copy of the knot $K$. Abusing notation, we write $K\subset \Sigma$ also for the shadow. Observe that a choice of a shadow of $K$ on $\Sigma$ induces a Morse framing of the knot $K$ (i.e. the framing which is parallel to $T\Sigma$). 
 
 By construction, the shadow $K$ intersects only the special meridianal beta curve $\b_0$, and no others. In particular, the boundary of a neighborhood of $K\cup \b_0$ is a punctured torus which is disjoint from the beta curves except for $\b_0$. We call this torus the \emph{special toroidal region for $K$}. In this region, we let $\b_1$ denote the curve shown in Figure~\ref{fig:47}, obtained by winding $\b_0$ as shown. We form a collection of curves $\bs_1$ by taking $\b_1$, and adjoining small translates of the curves from $\bs\setminus \b_0$.

We will decorate $\bs_0$ and $\bs_1$ with the following local systems of $\bF[U]$-modules:
\begin{enumerate}
\item We decorate $\bs_0$ with the module $E_0:=\bF[\scU,\scV]$. The monodromy 
\[
\rho_0\colon \bF[\Z]\to \Hom_{\bF[U]}(E_0, U^{-1} E_0)
\]
sends $e^n$ to $\scV^{n}\cdot \id$.
\item We decorate $\bs_1$ with the module $E_1:=\bF[U,T,T^{-1}]$. The monodromy 
\[
\rho_1\colon \bF[\Z]\to \Hom_{\bF[U]}(E_1,E_1)
\]
 sends $e^n$ to $T^n\cdot \id$.
\end{enumerate}

\begin{rem}
\label{rem:negative-powers}
\begin{enumerate}
\item As in Equation~\eqref{eq:local-systems}, the output $\bF[U]$-module morphism is determined by the $\bF[U]$-module morphism inputs, the evaluation of the monodromy $\rho$ along the boundary of a polygon, and also an overall factor of $U^{n_w(\phi)}$. 
\item  Since the action on $E_0$ involves $\scV^n\cdot \id$ for $n\in \Z$, \emph{a-priori} negative powers of $\scV$ may appear.  Hence it is not immediately clear from our initial descriptions that these maps are well-defined. We verify in Lemma~\ref{lem:positivity-simple}, below, that  no negative powers of $U$, $\scU$ or $\scV$ appear in surgery complex. See also Section~\ref{sec:positivity}.  
\end{enumerate}
\end{rem}

We now describe two morphisms
\[
\lb\theta^+_\sigma,\phi^\sigma\rb, \lb\theta^+_\tau, \phi^\tau\rb\colon \b_0^{E_0}\to \b_1^{E_1}.
\]
In the above, we define
\[
\phi^\sigma, \phi^\tau\colon \bF[\scU,\scV]\to \bF [U,T,T^{-1}]
\]
via the formulas
\begin{equation}
\phi^\sigma(\scU^i\scV^j)=U^i T^{j-i}\quad \text{and}\quad \phi^\tau(\scU^i \scV^j)=U^j T^{j-i}.
\label{eq:phi-sigma-tau-maps}
\end{equation}
The intersection points of $\theta^+_{\sigma}$ and $\theta^+_\tau$ are shown in Figure~\ref{fig:47}. Outside of the special genus 1 region they coincide with top degree intersection points between curves of $\bs$ and their translates.

\begin{figure}[h]
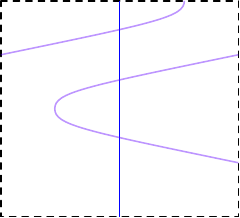
\caption{The diagram $(\bT^2,\b_0,\b_1,w,z)$. }
\label{fig:47}
\end{figure}

The diagram $(\bT^2,\b_0,\b_1,w,z)$ is a knot diagram for $S^1\times \{pt\}\subset S^1\times S^2$. Furthermore, there are two $\Spin^c$ structures represented by intersection points on the diagram. We define $\theta_\sigma^+$ denotes the top degree generator in the torsion $\Spin^c$ structure with respect to the $w$ basepoint. The intersection point $\theta_\tau^+$ denotes the top degree generator in the torsion $\Spin^c$ structure with respect to the $z$ basepoint.

\begin{lem}\label{lem:theta-tau/sigma-cycles}
 The morphisms $\lb\theta_\sigma^+,\phi^\sigma\rb$ and $\lb\theta^+_\tau,\phi^\tau\rb$ are both cycles.
\end{lem}
\begin{proof} By destabilizing, it is sufficient to consider the case when the genus is 1. We focus on $\lb\theta_\sigma^+,\phi^\sigma\rb$, since the argument for the $\tau$-generator is not substantially different. There are two index 1 bigons from $\theta_\sigma^+$ to $\theta_\sigma^-$, which both have a unique holomorphic representative by the Riemann mapping theorem. One bigon has boundary which is disjoint from the knot shadow $K$. This contributes $\lb\theta_\sigma^-,\phi^\sigma\rb$ to $\d \lb\theta_\sigma^+,\phi^\sigma\rb$. The other bigon class $\psi$ has
\[
\#(\d_{\b_0}(\psi)\cap K)=1\quad \text{and} \quad \# (\d_{\b_1}(\psi)\cap K)=-1.
\]
Hence, the contribution to $\d\lb\theta_\sigma^+,\phi^\sigma\rb$ is
\[
\lb\theta_\sigma^-, T^{-1}\circ \phi^\sigma\circ \scV\rb=\lb\theta_\sigma^-,\phi^\sigma\rb.
\]
This cancels the other term in $\d\lb\theta_\sigma^+,\phi^\sigma\rb$, proving the claim.
\end{proof}

We write $\lambda$ for the above Morse framing of $K$. The mapping cone complex $\bX_{\lambda}(K)$ is defined to be the $\bF[U]$-chain complex
\[
\bX_{\lambda}(K)=\Cone\left(v+h\colon \ve{\CF}^-(\as, \bs_0^{E_0})\to \ve{\CF}^-(\as,\bs_1^{E_1})\right),
\]
where $v$ counts holomorphic triangles with inputs $\lb\theta_\sigma^+,\phi^\sigma\rb$ and $h$ counts holomorphic triangles with inputs $\lb\theta_\tau^+,\phi^\tau\rb$. Phrased another way, the mapping cone complex is $\ve{\CF}^-(\as, \scB_\lambda)$, where $\scB_\lambda$ is the twisted complex of attaching curves
\[
\scB_\lambda=\begin{tikzcd}[column sep=3cm] \bs_0^{E_0}\ar[r, "{\lb\theta_\sigma^+,\phi^\sigma\rb+\lb\theta_\tau^+,\phi^\tau\rb}"]& \bs_1^{E_1}.
\end{tikzcd}
\]

\begin{rem} Note that in Ozsv\'{a}th and Szab\'{o}'s framework, the map $h$ depended on the framing $\lambda$, whereas $v$ was independent. In our framework, both $v$ and $h$ depend on the framing. See Section~\ref{sec:reformulation} for the relation between the two descriptions.
\end{rem}

\begin{lem}\label{lem:positivity-simple} The differential on $\bX_{\lambda}(K)$ involves no negative powers of $\scU$, $\scV$ or $U$.
\end{lem}
\begin{proof} Consider holomorphic disks from in $\ve{\CF}^-(\as, \bs_0^{E})$. By construction, holomorphic disks are counted with weights $\scV^{n_z(\phi)-n_w(\phi)} U^{n_w(\phi)}$. Since $U=\scU\scV$, this is $\scU^{n_w(\phi)} \scV^{n_z(\phi)}$, and hence no negative powers appear because $n_w(\phi),n_z(\phi)\ge 0$.

We now consider holomorphic triangles counted by $v$ (i.e. with input $\lb \theta^+_\sigma, \phi^\sigma\rb$). A holomorphic triangle is weighted by
\[
U^{n_w(\phi)}T^{\#\d_1(\phi) \cap K} \circ \phi^\sigma \circ \scV^{n_z(\phi)-n_w(\phi)}= T^{\#\d_1(\phi) \cap K}\circ \phi^\sigma \circ \scU^{n_w(\phi)} \scV^{n_z(\phi)}
\]
which clearly maps $E_0$ to $E_1$ with no negative $U$-powers. 
\end{proof}

 \subsection{A symmetry}
 \label{sec:symmetry}
 Our construction of the local systems $E_0$ and $E_1$ is asymmetric with respect to the basepoints $w$ and $z$ since the monodromy map associated with $E_0$ is $\scV^{\# \d_{\b_0}(\phi)\cap K}$, and we multiply all of the maps by an overall factor of $U^{n_w(\phi)}$. We call these maps the \emph{$w$-pointed} maps. We could symmetrically define \emph{$z$-pointed} maps, where we equip $E_0$ with the morphism $\rho(e^n)=\scU^{-n}$, and multiply an output morphism by an overall factor of $U^{n_z(\phi)}$.

 It turns out that the $w$-pointed and the $z$-pointed maps coincide in a fairly general context, as we now describe. Suppose that $\Sigma$ is a Heegaard surface with a distinguished knot trace $K\subset \Sigma$, with $w,z\in K$ distinct points. Suppose that $\gs_{\veps_1}^{E_{\veps_1}},\dots, \gs_n^{E_{\veps_n}}$ are attaching curves with local systems on a Heegaard diagram, such that each $E_{\veps_i}$ is one of $E_0$, $E_1$ or the trivial local system (i.e. $\bF[U]$). Suppose that the curves labeled $E_0$ intersect only the component of $K\setminus \{w,z\}$ oriented from $w$ to $z$, and all other attaching curves are disjoint from this subarc.

 \begin{lem}  
In the above situation, the $w$-pointed maps coincide with the $z$-pointed maps for any choices of Floer morphisms.
 \end{lem}
 \begin{proof} The $w$-pointed and the $z$-pointed maps count the same holomorphic curves, so it suffices to address the associated $\bF[U]$-module maps. We consider a homology class $\phi$ of polygons. Let $\delta_1,\dots, \delta_j$ denote the multiplicities of the curves along $0$-subarc from $w$ to $z$, normalized in sign so that 
 \[
 \delta_1+\dots+\delta_j=n_z(\phi)-n_w(\phi).
 \]
 (We recall that by construction the only attaching curves which intersect the short arc from $w$ to $z$ are small translates of $\b_0$). The $w$-pointed maps count
 curves representing $\phi$ with composite factors of $\scV^{\delta_i}$ and an overall factor of $U^{n_w(\phi)}$. The $z$-pointed maps will instead have composite factors of $\scU^{-\delta_i}$ and an overall factor of $U^{n_z(\phi)}$.
 We observe 
 \[
 \scV^{\delta_i}=U^{\delta_i} \scU^{-\delta_i}.
 \]
  Since all maps are $\bF[U]$-equivariant, we may factor out $U^{\delta_i}$ from the $w$-pointed map which has factors of $\scU^{-\delta_i}$, and with an overall factor of
 \[
 U^{n_w(\phi)} U^{\delta_1+\cdots +\delta_j}=U^{n_z(\phi)}.
 \] 
 The $U$-power of the $z$-pointed map, so the proof is complete.
 \end{proof}

\subsection{The link surgery complex}

\label{sec:background-link-surgery}

The construction of the link surgery complex is obtained by iterating the construction of the knot surgery complex for each link component. We now describe the construction using the framework described in the previous section.

We take a meridianal Heegaard link diagram $(\Sigma,\as,\bs,\ws,\zs)$ for $(Y,L)$ and we assume that we pick embedded shadows of the components $K_1,\dots, K_n$ of $L$ which are pairwise disjoint. The union of the shadow of $K_i$ with its special beta meridianal curve is a punctured torus. Write $\b_{0,i}$ for the special meridian of $K_i$, and write $\b_{1,i}$ for a copy of the $\b_1$ curve from Figure~\ref{fig:47}, placed into this special genus 1 region.

We now construct curves $\bs_{\veps}$ indexed by $\veps\in \bE_\ell$. We write the original beta curves as $\bs=\bs'\cup \{\b_{0,1},\dots, \b_{0,n}\}$. We form $\bs_{\veps}$ by taking small translates of $\bs'$, as well as a small translate of $\b_{0,i}$ for $i$ such that $\veps_i=0$, and a small translate of $\b_{1,i}$ for $i$ such that $\veps_i=1$. We apply small Hamiltonian translations to the curves in each $\bs_{\veps}$ to achieve admissibility and ensure that if two copies of a given beta curve appear in two collections, then they intersect in a pair of points.

To the attaching curve $\bs_{\veps}$, we associate the $\bF[U_1,\dots, U_n]$-module
\[
E_{\veps}:=E_{\veps_1}\otimes_{\bF} \cdots \otimes_{\bF} E_{\veps_n}. 
\]
 We set
\[
\phi_i^{\sigma}=\underbrace{\id\otimes \cdots\otimes \id}_i \otimes  \phi^\sigma\otimes \underbrace{\id \otimes  \cdots \otimes  \id}_{n-i-1}.
\]
We define $\phi_i^\tau$ similarly.

If $\veps<\veps'$ and $|\veps'-\veps|_{1}=1$ (where $|-|_{1}$ denotes the $L^1$-norm) we define two cycles
\[
\lbmed\theta_{\sigma,\veps,\veps'}^+, \phi^\sigma_i\rbmed \quad \text{and} \quad \lbmed\theta_{\tau,\veps,\veps'}^+, \phi^\tau_i \rbmed.
\]
Here, $\theta_{\sigma,\veps,\veps'}^+$ denotes the top $\gr_{\ws}$-graded intersection point  $\xs\in \bT_{\b_{\veps}}\cap \bT_{\b_{\veps'}}$ such that $\frs_{\ws}(\xs)$ is torsion.
 Similarly $\theta_{\tau,\veps,\veps'}^+$ is the top $\gr_{\zs}$-graded intersection point over intersection points $\xs\in \bT_{\b_{\veps}}\cap \bT_{\b_{\veps'}}$ such that $\gr_{\ve{z}}(\xs)$ is torsion. If $\veps'$ and $\veps$ differ only in coordinate $i$, we  write $\theta^+_{\sigma,i}$ and $\theta^+_{\tau,i}$.

\begin{lem} 
The diagram $\scB_{\Lambda}$ is a hypercube of attaching curves with local systems.
\end{lem}
\begin{proof}  We observe firstly that all compositions $\mu_{i}$ of sequences of chains in the cube vanish when $i>2$. Indeed given any sequence $\theta_1,\dots, \theta_n$ of composable morphisms in the cube, there is a complete collection of basepoints $\ve{p}\subset \ws\cup \zs$ such that each $\theta_i$ is the top degree Floer cycle in the torsion $\Spin^c$ structure with respect to $\ve{p}$. Composing morphisms gives a chain in $\gr_{\ve{p}}$-grading $i-2$ (where we view the top degree element of homology as being in $0$). Hence, the composition must vanish unless $i\in \{1,2\}$.

The length 1 hypercube relations amount to the claim that $[\theta_{\sigma,i}^+,\phi_i^\sigma]$ and $[\theta_{\tau,i}^+,\phi_i^\tau]$ are cycles, which is proven similarly to the genus one case in Lemma~\ref{lem:theta-tau/sigma-cycles}. The length 2 hypercube relations amount to the claim that for each $i\neq j$ and $\circ,\circ'\in \{\sigma,\tau\}$, we have
\[
\mu_2\left(\lbmed\theta_{\circ,i}^+, \phi_i^{\circ}\rbmed, \lbmed\theta_{\circ',j}^+,\phi_i^{\circ'}\rbmed\right)+\mu_2\left(\lbmed\theta_{\circ',j}^+,\phi_i^{\circ'}\rbmed,\lbmed\theta_{\circ,i}^+, \phi_i^{\circ}\rbmed\right)=0
\]
To see the above, observe that if we ignore the local systems, both of the products above must represent the top degree cycle in the torsion $\Spin^c$ structure in $\ve{\CF}^-(\bs_{\veps}, \bs_{\veps+e_i+e_j}, \ve{p})$ for some complete collection of basepoints $\ve{p}$. We claim the associated $\bF[U]$-module morphism in the output of both pairings is $\phi_i^\circ \circ \phi_j^{\circ'}$. Observe that each triangle counted by the two pairings will output a morphism of the form $T_1^{m_1}\cdots T_n^{m_n}\cdot \phi_i^\circ \circ \phi_j^{\circ'}$ for some $m_1,\dots, m_n\in \Z$. To see $m_1=\cdots =m_n=0$ we observe that the maps $\mu_i$ preserve the $n$-component Alexander grading on the Floer complexes under consideration. This is because the intersection of any component of $\d D(\psi)$ with a knot shadow $K_\ell$, $\ell\in \{1,\dots, n\}$ contributes $\pm 1$ (depending on the sign of the intersection) to the Alexander grading of the output morphism. However, $K_\ell \cap \d D(\psi)=0$ since $\d D(\psi)$ is a boundary and $K_\ell$ is closed.
 Hence the output Alexander grading will vanish and the proof is complete.
\end{proof}

 The link surgery complex is obtained by pairing
\[
\bX_{\Lambda}(Y,L):=\ve{\CF}^-(\as, \scB_{\Lambda}),
\]
where we view $\as$ as being decorated with the trivial local system $\bF[U_1,\dots, U_n]$ (with $\rho$ acting by the identity, in the notation of Section~\ref{sec:local-systems}). In Section~\ref{sec:reformulation}, we verify that $\bX_{\Lambda}(Y,L)$ is isomorphic to the construction of Manolescu and Ozsv\'{a}th for links in $S^3$. As a consequence, Manolescu and Ozsv\'{a}th's main theorem gives the following:

\begin{thm}[\cite{MOIntegerSurgery}]\label{thm:MOlink-surg} If $L\subset S^3$ is given integral framing $\Lambda$, then there is a homotopy equivalence
\[
\bX_{\Lambda}(S^3,L)\simeq \ve{\CF}^-(S^3_{\Lambda}(L))
\]
of chain complexes over $\bF[U]$. On the left, $U$ acts by some $U_i$ (all have chain homotopic action). 
\end{thm}

Note that $\bX_{\Lambda}(Y,L)$ naturally has a filtration by the cube $\bE_\ell$, and is an $\ell$-dimensional hypercube of chain complexes, where $\ell=|L|$. If $\vec{M}$ is an oriented sublink of $L$, we write $\Phi^{\vec{M}}$ for the component of the hypercube differential on $\bX_{\Lambda}(L)$ which counts holomorphic polygons with inputs $\theta_i^{\sigma}$ for $i$ such that $+K_i\subset \vec{M}$, and $\theta_i^{\tau}$ for $i$ such that $-K_i\subset \vec{M}$, and which does not increase any other components of the cube. If $\veps<\veps'$ are points of the cube, and $(\veps_i'-\veps_i)=1$ if and only if $\pm K_i\in \vec{M}$, sometimes we will also write $\Phi^{\vec{M}}_{\veps,\veps'}$ for the summand of $\Phi^{\vec{M}}$ moving from cube point $\veps$ to $\veps'$.

Note that although the diagram of attaching curves $\scB_{\Lambda}$ has only length 1 morphisms, the complex $\bX_{\Lambda}(Y,L)$ will typically have morphisms of longer length, obtained by the $A_\infty$-pairing of $\as$ and $\scB_{\Lambda}$.

\section{The knot surgery algebra}

In this section, we consider the knot surgery algebra from \cite{ZemBordered}.

\subsection{The surgery algebra}

 We now recall the knot surgery algebra $\cK$ from \cite{ZemBordered}. We recall that $\cK$ is itself an algebra over an idempotent ring of two elements, denoted
\[
\ve{I}=\ve{I}_0\oplus \ve{I}_1
\]
where $\ve{I}_{\veps}\iso \bF$. We set
\[
\ve{I}_0\cdot \cK\cdot \ve{I}_0\iso \bF[\scU,\scV],\quad \ve{I}_1\cdot \cK\cdot \ve{I}_1\iso \bF[U,T,T^{-1}]\quad \text{and} \quad \ve{I}_0\cdot \cK\cdot \ve{I}_1=0.
\]
We set $\ve{I}_1\cdot \cK\cdot \ve{I}_0$ to be two copies of $\bF[U,T,T^{-1}]$, generated as a vector space by monomials of the form
\[
U^i T^j \sigma\quad \text{and }\quad U^i T^j \tau,
\]
for $i\ge 0$ and $j\in \Z$.
These elements are subject to the relations that
\[
\sigma \cdot a=\phi^\sigma(a)\cdot \sigma\quad \text{and} \quad \tau \cdot a=\phi^\tau(a)\cdot \tau,
\]
for $a\in \ve{I}_0\cdot \cK\cdot \ve{I}_0$. Here $\phi^\sigma$ and $\phi^{\tau}$ are the maps defined in Equation~\eqref{eq:phi-sigma-tau-maps}, viewed as maps from $\ve{I}_0\cdot \cK\cdot \ve{I}_0$ to $\ve{I}_1\cdot \cK\cdot \ve{I}_1$. 

\subsection{Topologies on the surgery algebra}

In \cite{ZemBordered}*{Section~6} we described a topology on $\cK$. This topology was generated by a sequence of right ideals
\[
\cK=J_0\supset J_1\supset J_2\supset\cdots.
\]
We call this the \emph{chiral topology} (the terminology is inspired by \cite{BDChiral} \cite{BeilinsonChiral}). We recall $J_n$ presently. We write 
\[
J_n^0\subset \ve{I}_0\cdot \cK\cdot \ve{I}_0=\bF[\scU,\scV]
\] for the subspace spanned by $\scU^i\scV^j$ where $\max(i,j)\ge n$. We write 
\[
J_n^1\subset \ve{I}_1\cdot \cK\cdot \ve{I}_1=\bF[U,T,T^{-1}].
\]
Then $J_n$ is the subspace spanned over $\bF$ by $J_n^0$, $J_n^1$ and
\[
J_n^1\cdot \Span_{\bF}(\sigma,\tau)+\Span_{\bF}(\sigma,\tau)\cdot J_n^0.
\]

Multiplication is continuous with respect to the chiral topology
\[
\mu_2\colon \cK\vecotimes \cK\to \cK.
\]
See \cite{ZemBordered}*{Proposition~6.4}.

\begin{rem}
Multiplication on $\cK$ does give a continuous map $\mu_2\colon \cK\otimes^! \cK\to \cK$, so $\cK$ cannot be viewed as an achiral linear topological algebra. For example, the sequence $x_n=T^{-n}\otimes T^n \sigma$ approaches 0 in $\cK\otimes^! \cK$, while $\mu_2(x_n)=\sigma\not\to 0$. 
\end{rem}

\begin{rem} The topology $\cK\vecotimes \cK$ is not first countable.
\end{rem}

There is additionally the $U$-adic topology on the algebra, whose topology is generated by the 2-sided ideals $(U^i)$ ranging over $i \ge 0$. We will think of this topology as determining a different linear topological algebra, for which we write
$\frK.$ The map
\[
\mu_2\colon \frK\otimes^! \frK\to \frK
\]
is continuous. 

\begin{rem}
Although the definition of the chiral topology appears more complicated than the definition of the $U$-adic topology, type-$D$ modules over $\cK$ are typically easier to work with than $\frK$. Over $\cK$ there are more convergent series in the algebra, so there are more type-$D$ morphisms. As a concrete example, we encourage the reader to compare the proofs of the surgery exact triangles over $\cK$ \cite{ZemBordered}*{Section~18.3} and $\frK$  (Section~\ref{sec:U-adic-surgery}.)
\end{rem}

\subsection{Type-\texorpdfstring{$D$}{D} modules over \texorpdfstring{$\cK$}{K}}

In \cite{ZemBordered}, we described how to repackage the knot surgery formula in terms of a type-$D$ algebra over the algebra $\cK$. We recall this construction.

We suppose that $K\subset Y$ is a knot with Morse framing $\lambda$.  We can repackage the information of $\bX_{\lambda}(Y,K)$ as a type-$D$ module $\cX_{\lambda}(Y,K)^{\cK}$, as we now describe.

 In idempotent 0, we view $\cX_{\lambda}(Y,K)$ as being generated over $\bF$ by elements of $\bT_{\a}\cap \bT_{\b_0}$. In idempotent 1, the generators of $\cX_{\lambda}(Y,K)$ coincide with elements of $ \bT_{\a}\cap \bT_{\b_1}$. We can reinterpret the differential on $\ve{\CF}^-(\as, \scB_{\lambda})$ as giving a structure map
\[
\delta^1\colon \cX_{\lambda}(Y,K)\to \cX_{\lambda}(Y,K)\otimes_{\mathrm{I}} \cK,
\]
as follows:
\begin{enumerate} 
\item If $\xs_0,$ $\ys_0$ are in idempotent 0 and $\d(\xs)$ contains a summand of $\ys_0 \scU^i \scV^j$, then $\delta^1(\xs_0)$ contains a summand of $\ys_0\otimes \scU^i\scV^j$. 
\item If $\xs_0$ and $\ys_1$ are in idempotents 0 and 1 respectively, and $v(\xs_0):=f_{\a,\b_0,\b_1}(\xs_0, \theta_\sigma^+)$ (resp. $h(\xs_0):=f_{\a, \b_0,\b_1}(\xs_0, \theta_\tau^+)$) contains a summand of $\ys_1 U^i T^j$, then $\delta^1(\xs_0)$ contains a summand of $\ys_1\otimes U^i T^j\sigma$ (resp. $\ys_1\otimes U^i T^j\tau$).
\item If $\xs_1,$ and $\ys_1$ are in idempotent 1 and $\d(\xs)$ contains a summand of $\ys_0 U^i T^j$, then $\delta^1(\xs_1)$ contains a summand of $\ys_1\otimes U^i T^j$. 
\end{enumerate}

There is a chiral type-$A$ module ${}_{\cK} \cD_0$ for the solid torus. This module has $m_j=0$ for $j\neq 2$. In idempotent 0, $\cD_0$ is isomorphic to $\bF[\scU,\scV]$ and in idempotent 1 it is isomorphic to $\bF[U,T,T^{-1}]$. The elements of $\ve{I}_0\cdot \cK\cdot \ve{I}_0\iso \bF[\scU,\scV]$ and $\ve{I}_1\cdot \cK\cdot \ve{I}_1\iso \bF[U,T,T^{-1}]$ act on idempotents 0 and 1 of $\cD_0$ in the obvious way. The elements $\sigma$ and $\tau$ of the algebra act by the maps $\phi^\sigma$ and $\phi^\tau$. The knot surgery complex is related to the above modules as follows:
\[
\bX_{\lambda}(Y,K)\iso \cX_{\lambda}(Y,K)^{\cK}\boxtimes {}_{\cK} \cD_0.
\]

If we view ${}_{\cK} \cD_0$ as the direct sum of a copy of $\bF$ over each Alexander and Maslov grading supported by the module, then the chiral topology coincides with the product topology. The completion is given by the direct product of a copy of $\bF$ in each of these gradings, or equivalently
\[
\bF\llsquare \scU,\scV\rrsquare \oplus \bF\llsquare U, T,T^{-1}\rrsquare.
\]
It is shown in \cite{ZemBordered}*{Lemma~8.1} that multiplication gives a continuous map
\[
m_2\colon \cK\vecotimes \cD_0\to \cD_0.
\]

We define the $U$-adic type-$A$ module ${}_{\frK} \frD_0$ similarly. As a linear topological vector space, we define it is $\bF[\scU,\scV]\oplus \bF[U,T,T^{-1}]$ equipped with the $U$-adic topology (where we view $U$ as acting by $\scU\scV$ on $\bF[\scU,\scV]$). The action $m_2$ is the same as for $\cD_0$.

\subsection{On completions}
\label{sec:different-topologies}

We now briefly discuss completions in the context of the surgery formula. In particular, we compare the $U$-adic completion and the chiral completion in the context of the knot surgery formula, focusing on concrete examples and the knot surgery complex.

 Let $K$ be a knot in $S^3$ and let $A_s$ denote the subspace of $\cCFK(K)$ in Alexander grading $s$, and let $B_s$ denote a copy of $\CF^-(S^3)$. The topology on $\bX_{\lambda}(K)=\Cone(v+h\colon \bA(K)\to \bB(K))$ corresponding to the chiral theory is given by equipping each $A_s$ with the $U$-adic topology, and then equipping $\bA(K)=\prod_{s\in \Z} A_s$ with the product topology. Similarly, we equip $\bB(K)=\prod_{s\in \Z} B_s$ with the product topology. The completion of this space coincides with the description given in \cite{MOIntegerSurgery}. The $U$-adic completion of $\bX_\lambda(K)$ is different. For this topology, we set $\bA(K)=\bigoplus_{s\in \Z} A_s$ and $\bB(K)=\bigoplus_{s\in \Z} B_s$ and complete using the submodules $U^i\cdot \bA(K)$ and $U^i \cdot \bB(K)$, ranging over $i\in \N$.

 Note that the completions of $\cD_0$ and $\frD_0$ are \emph{not} naturally isomorphic as vector spaces. The validity of both topologies on $\bX_{\lambda}(K)$ to compute $\ve{\CF}^-(S^3_{\lambda}(K))$ is proven in Section~\ref{sec:admissibility}. 

As an example, we consider the knot surgery complexes of the $+1$ and $-1$ framed unknots.  These complexes take the following form
\[
\bX_{+1}(U)=
\begin{tikzcd}[labels=description]
\cdots
	\ar[dr]
&\xs_{-2}
	\ar[d, "U^2"]
	\ar[dr, "1"]
& \xs_{-1}
	\ar[d, "U"]
	\ar[dr, "1"]
&\xs_{0}
	\ar[d, "1"]
	\ar[dr, "1"]
& \xs_1
	\ar[d, "1"]
	\ar[dr, "U"]
&\xs_2
	\ar[d, "1"]
	\ar[dr, "U^2"]
&\cdots\\
\cdots&\ys_{-2}& \ys_{-1}& \ys_0& \ys_{1}& \ys_2&\cdots
\end{tikzcd}
\]
\[
\bX_{-1}(U)=
\begin{tikzcd}[labels=description]
\cdots
&\xs_{-2}
	\ar[d, "U^2"]
	\ar[dl, "1"]
& \xs_{-1}
	\ar[d, "U"]
	\ar[dl, "1"]
&\xs_{0}
	\ar[d, "1"]
	\ar[dl, "1"]
& \xs_1
	\ar[d, "1"]
	\ar[dl, "U"]
&\xs_2
	\ar[d, "1"]
	\ar[dl, "U^2"]
&\cdots
	\ar[dl]
\\
\cdots&\ys_{-2}& \ys_{-1}& \ys_0& \ys_{1}& \ys_2&\cdots
\end{tikzcd}
\]
In the above, each generator denotes a copy of $\bF\llsquare U\rrsquare$. Elements in the completion with respect to the chiral topology consist of infinite sums
\[
\sum_{i\in \Z} (\a_i\cdot \xs_i+ \b_i\cdot \ys_i)
\]
 where $\a_i,\b_i\in \bF\llsquare U\rrsquare$. In the $U$-adic topology, elements instead consist of infinite sums of the form 
\[
\sum_{i\ge 0} U^i\cdot \ve{X}_i,
\]
where $\ve{X}_i$ is a finite sum of $\xs_j$ and $\ys_j$ generators.

We observe that in both topologies, $H_*(\bX_{+1}(U))$ (resp. $H_*(\bX_{-1}(U))$) is isomorphic to $\bF\llsquare U\rrsquare$ (resp. $\bF\llsquare U\rrsquare$), spanned by $\xs_0+\sum_{i\ge 1} U^{i(i-1)/2}(\xs_{i}+\xs_{-i})$ (resp. $\ys_0$).

 \subsection{Examples}
\label{sec:examples}

 We now consider the type-$D$ module for an $n$-framed solid torus for $n\in \Z$, (i.e. the module for the complement of an $n$-framed unknot). See Figure~\ref{fig:37} for the $-2$-framed solid torus. There are two generators $\xs_0$ and $\xs_1$, which are concentrated in idempotents 0 and 1, respectively. There are two  holomorphic triangles which contribute to $\delta^1$. We obtain the type-$D$ module
 \[
 \begin{tikzcd}[column sep=2cm] \xs_0 \ar[r, "\sigma+T^{-2} \tau"] &\xs_1.
 \end{tikzcd}
 \]
 The computation easily generalizes to the $n$-framed solid torus.

 \begin{figure}[h]
 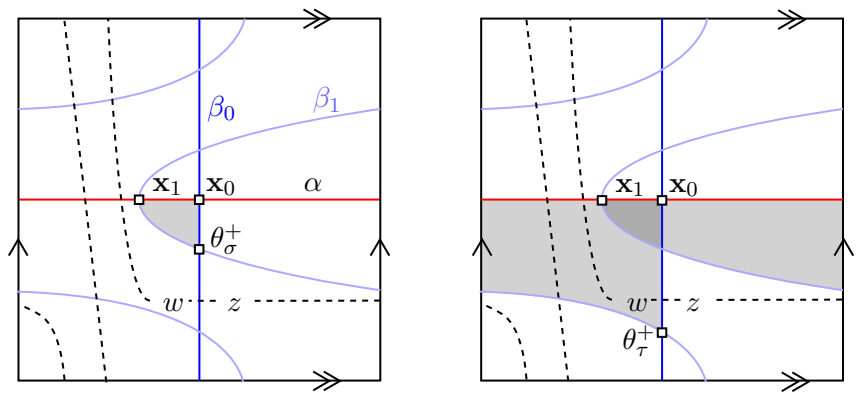
 \caption{Triangles on the Heegaard triple which computes $\cD_{-2}^{\cK}$, the $-2$-framed unknot complement. The dashed line is the shadow of $K$.  The left triangle contributes to the $\sigma$ summand of $\delta^1$, and the right triangle contributes to the $\tau$ summand.}
 \label{fig:37}
 \end{figure}
 
 We now consider the $\infty$-framed solid torus (i.e. the knot surgery complex for $S^1\times pt\subset S^1\times S^2$).  A diagram computing this is shown in Figure~\ref{fig:38}. We write $\cD_\infty^{\cK}$ for this complex.
 
 \begin{figure}[h]
 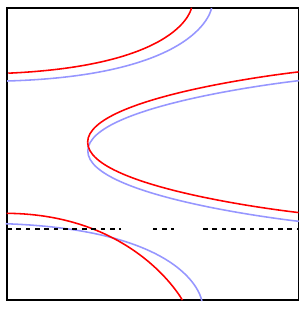
 \caption{A Heegaard triple which computes $\cD_{\infty}^{\cK}$, the type-$D$ module for an $\infty$-framed solid torus. Shaded is a triangle which contributes $\zs_1^+\otimes \sigma$ to $\delta^1(\ys_0^+)$. }
 \label{fig:38}
 \end{figure}

 We now compute the complex. In idempotent $0$, the complex has four generators with differential as follow:
 \[
 \begin{tikzcd} \xs_0^+\ar[r, "1+\scU"] &\xs_0^-
 \end{tikzcd}\qquad  \begin{tikzcd} \ys_0^+\ar[r, "1+\scV"] &\ys_0^-.
  \end{tikzcd}
 \]
 The differential counts only bigons. In idempotent $1$, the complex takes the form
 \[
 \begin{tikzcd} \zs^+_1\ar[r, "1+T"] & \zs^-_1
 \end{tikzcd}.
 \]
We now count the summands of $\delta^1$ which are weighted by $\sigma$ or $\tau$.  In our diagram, there are four such triangles. Combining all of the information, we get the following diagram:
 \begin{equation}
\cD^{\cK}_\infty= \begin{tikzcd}
\xs_0^+
	\ar[r, "1+\scU"]
	\ar[dr, "T^{-1}\tau"]
&
\xs_0^-
	\ar[dr, "\tau", pos=.2,swap]
&
\ys_0^+
	\ar[r, "1+\scV"]
	\ar[dl, "\sigma", pos=.2, crossing over]
	&
\ys_0^-
	\ar[dl, "\sigma"]
\\[1cm]
& \zs_1^+
	\ar[r, "1+T"]
&\zs_1^-
\end{tikzcd}
 \label{eq:infinity-framed-solid-torus}
 \end{equation}

 \begin{rem} 
  Over $\cK$, the above complex can be simplified further since $1+\scU$ and $1+\scV$ are units in $\ve{I}_0\cdot \cK\cdot \ve{I}_0$.  Hence, $\cD_\infty^{\cK}$ is homotopy equivalent to $ \begin{tikzcd} \zs^+_1\ar[r, "1+T"] & \zs^-_1
   \end{tikzcd}$.
 
 On the other hand, $1+\scU$ and $1+\scV$ are not units in $\ve{I}_0\cdot \frK\cdot \ve{I}_0$ (the algebra with $U$-adic completions). Hence the corresponding module over the $U$-adic topology $\frD_{\infty}^{\frK}$ does not admit the above simplification.
 \end{rem}

\subsection{Operator topologies}
\label{sec:operator-topologies}

We may view the knot surgery algebra as being a set of operators (i.e. $\bF[U]$-module endomorphisms) on $E_0\oplus E_1$. Ignoring completions, we define an action of the knot surgery algebra on $E_0\oplus E_1$ as follows. Algebra elements in $\ve{I}_0 \cdot \cK \cdot \ve{I}_0$ act on $E_0$ by polynomial multiplication. Elements of $\ve{I}_1\cdot \cK \cdot \ve{I}_1$ act on $E_1$ similarly. The elements $\sigma$ and $\tau$ map $E_0$ to $E_1$ via the maps $\phi^\sigma$ and $\phi^\tau$. In particular, if we ignore completions, we may identify
\[
\cK\subset \End_{\bF[U]}(E_0\oplus E_1)
\]

 Recall that if $X$ and $Y$ are linear topological spaces, then we can endow $\Hom(X,Y)$ with the \emph{uniform topology}. In this topology, the fundamental open subspaces in $\Hom(X,Y)$ take the form $\cO_U:=\{\phi\colon X\to Y: \phi(X)\subset U\}$, ranging over open subspaces $U\subset Y$.  In particular, depending on the topology on $E_0\oplus E_1$, we will naturally obtain a different topology on the endomorphism algebra.

 We are interested in the following topologies:
\begin{enumerate}
\item (Chiral) The topology obtained by viewing $E_0\oplus E_1$ a direct sum over each Alexander and Maslov grading supported, and then using the product topology.
\item ($U$-adic) The $U$-adic topology on $E_0\oplus E_1$.
\end{enumerate}

The topologies that the knot surgery algebra obtains via the subspace topology are familiar:

\begin{prop}
\,
\begin{enumerate}
\item If we equip $E_0\oplus E_1$ with the product topology, then the knot surgery algebra obtains (as a subspace of $\End_{\bF[U]}(E_0\oplus E_1)$) the chiral topology $\cK$.
\item If we equip $E_0\oplus E_1$ with the $U$-adic topology, then the knot surgery algebra obtains the $U$-adic topology $\frK$.
\end{enumerate}
\end{prop}
\begin{proof}
We focus first on the product topology. Write $\cK_{\End}$ for the topology induced from the subspace topology with respect to the product topology. The open subspaces of $E_0$ in the product topology are given by $W_n=\Span(\scU^i\scV^j: \max(i,j)\ge n)$. By definition, this means that the open subspaces of $\ve{I}_0\cdot \cK_{\End}\cdot \ve{I}_0$ have a basis of opens given by $\{\phi: \im(\phi)\subset W_n\}$. It is clear that $\scU^i\scV^j\cdot \id$  has image in $W_n$ if and only if $\max(i,j)\ge n$.   Similarly, the fundamental open subspaces of $E_1$ in the product topology are $W_n'=\Span(U^i T^j: i\ge n \text{ or } |j|\ge n)$. It is clear that $U^i T^j\cdot \id$  has image in $W_n'$ if and only if $i\ge n$, so the two topologies coincide on $\ve{I}_1\cdot \cK\cdot \ve{I}_1$. The analysis of $\ve{I}_1\cdot \cK\cdot \ve{I}_0$ is entirely analogous.

 A nearly identical argument holds for the $U$-adic topology, and we leave the details to the reader.
\end{proof}

\begin{rem}\label{rem:continuity-of-composition}
The perspective of the algebras as spaces of linear maps naturally explains the appearance of the chiral topology on $\cK$, via the following general fact. If $A$, $B$ and $C$ are linear topological spaces, then composition always gives a continuous map
\[
\circ\colon\Hom(B,C)\vecotimes \Hom(A,B)\to \Hom(A,C),
\]
if we equip the $\Hom$-spaces with the uniform topology.  In relation to the $U$-adic topology, we observe that if $A$, $B$ and $C$ are $\bF[U]$-modules equipped with the $U$-adic topology, then
\[
\circ\colon \Hom_{\bF[U]}(B,C)\otimes^! \Hom_{\bF[U]}(A,B)\to \Hom_{\bF[U]}(A,C)
\]
is also continuous. We will extend the perspective of the algebras $\cK$ and $\frK$ as endomorphism algebras in Section~\ref{sec:endomorphisms}.  
\end{rem}

\subsection{Module categories and functors}

In this section, we describe several module categories and functors which are important in this paper.

We recall from Section~\ref{sec:linear-space-background} that there are the categories of chiral and achiral type-$D$ modules over $\frK$, and the category of chiral type-$D$ modules over $\cK$:
\[
\Mod^{\frK},\quad \Mod^{\frK}_{\ch}, \quad \Mod^{\cK}_{\ch}.
\]
There are also type-$A$ and $DA$ versions of the above categories as well.

Next, we observe that inclusion $\bI\colon \frK\to \cK$
is continuous, and hence there is a bimodule
\[
{}_{\frK}[\bI]^{\cK}.
\]
Note that when tensoring the above bimodule, we cannot naturally mix chiral and achiral tensor products. Hence, we view the above only as a chiral $DA$-bimodule. In particular, the above gives functors
\[
\Mod^{\frK}_{\ch}\to \Mod^{\cK}_{\ch} \quad \text{and} \quad  {}_{\cK}\Mod_{\ch}\to {}_{\frK} \Mod_{\ch}.
\]
\begin{rem}
\label{rem:chiralizing-functor}
There is one setting where we may obtain a functor from achiral to chiral modules. This is by restricting to \emph{linearly compact spaces} (cf. \cite{LefschetzAT}). Recall that a space $\cX$ is called \emph{linearly compact} if $\cX/U$ is finite dimensional for every open subspace $U\subset \cX$. For such spaces, it is straightforward to see that the map natural map
\[
\cX\vecotimes Y\to \cX\otimes^! Y
\]
is a homeomorphism for any $Y$. We can phrase this by saying there is a \emph{chiralizing} functor
\[
 \Mod_{\lc}^{\frK}\to \Mod_{\lc,\ch}^{\frK}
\]
where $\Mod_{\lc}^{\frK}$ denotes the category of linearly compact type-$D$ modules over $\frK$. Note that finite dimensional spaces are linearly compact, so $\Mod_{\lc,\ch}^{\frK}$ contains all finitely generated type-$D$ modules. Composing with the functor from $\Mod_{\ch}^{\frK}$ to $\Mod_{\ch}^{\cK}$ gives a functor
\begin{equation}
\Mod^{\frK}_{\lc}\to \Mod^{\cK}_{\lc,\ch}, \label{eq:functor-achiral-to-chiral}
\end{equation}
 which  transforms linearly compact, achiral type-$D$ modules over $\frK$ into chiral type-$D$ modules over $\cK$. 
\end{rem}

 Note also that the link surgery modules $\cX(Y,L)^{\frL}$ (as constructed in this paper)  are finitely generated.

\begin{rem} There does not seem to be a natural analog of Remark~\ref{rem:chiralizing-functor}  for type-$A$ modules. We note that \[{}_{\cK}\cD_0=
\bF\llsquare \scU, \scV \rrsquare \oplus \bF\llsquare U,T,T^{-1}\rrsquare
\]
whereas ${}_{\frK} \frD_0$ (the $U$-adic solid torus module) is given by the $U$-adic completion of $\bF[\scU,\scV]\oplus \bF[U,T,T^{-1}]$.
\end{rem}

\section{A local surgery formula}
\label{sec:mapping-cone}

In this section, we give our new proof of the knot and link surgery formulas.

\subsection{The statement}
\label{sec:preliminary-defs-morphisms}

We now introduce the Lagrangians and morphisms which appear in our description of the mapping cone formula. We begin by considering three simple closed curves
\[
\b_{\lambda},\b_0,\b_1\subset \bT^2.
\] 
 We additionally put two basepoints, $w,z\in \bT^2$. See Figure~\ref{fig:17}.

\begin{figure}[ht]
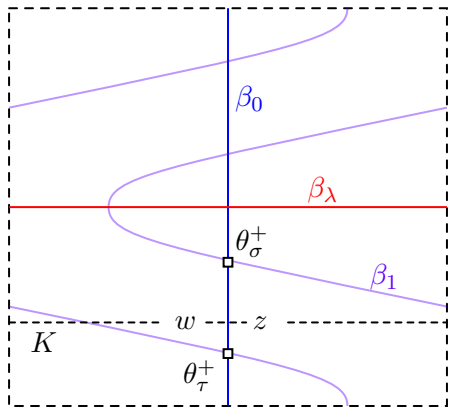
\caption{The curves $\b_\lambda$, $\b_0$ and $\b_1$, as well as the intersection points $\theta_\sigma^+$ and $\theta_\tau^+$. Also shown is the curve $K$.}
\label{fig:17}
\end{figure}

 We give $\b_\lambda$, $\b_0$ and $\b_1$ the following local systems. We associate $\b_{\lambda}$ with the trivial local system $\bF[U]$. We give $\b_0$ and $\b_1$ the local systems described in Section~\ref{sec:knot-surgery-formula}, which we recall briefly. We associate $\b_0$ with $E_0:=\bF[\scU,\scV]$. We associate $\b_1$ with $E_1:=\bF[U, T,T^{-1}]$. We define monodromy maps $\rho_{\lambda}$, $\rho_0$ and $\rho_1$ as follows. We set $\rho_{\lambda}(e^n)=\id$ for all $n$. We set $\rho_0(e^n)=\scV^{n}\cdot \id$ for all $n$. We set $\rho_1(e^n)=T^n\cdot\id$ for all $n$. To apply the monodromy maps, we use a closed curve $K$ parallel to $\b_{\lambda}$ which passes through $w$ and $z$.

 The main theorem of this section is the following:
 
 \begin{prop}
 \label{prop:general-HE} There is a homotopy equivalence in the Fukaya category of $\bT^2$:
 \[
 \b_\lambda\simeq  \Cone\left(\lb\theta_\tau^+,\phi^\tau\rb+\lb\theta_\sigma^+,\phi^\sigma\rb\colon \b_0^{E_0}\to \b_1^{E_1} \right).
 \]
 (Here, we view $\b_{\lambda}$ as having the trivial local system $\bF[U]$). 
 \end{prop}

Holomorphic polygons are counted as in Section~\ref{sec:local-systems}. In a bit more detail, the $\bF[U]$-module morphism associated to the boundary edge of a holomorphic polygon mapping to $\b_\veps$ will be $\rho_\veps(e^{\# (\d_{\b_\veps}(\phi)\cap K)})$. We also multiply the output morphism by an overall factor of $U^{n_w(\phi)}$.

 \begin{rem}The orientation of $K$ is so that if $\a$ is a Lagrangian on $\bT^2$ (not passing through the short arc connecting $w$ and $z$ which intersects $\b_0$), a holomorphic disk on a diagram $(\bT^2,\a,\b_0,w,z)$ will be weighted by 
\[
U^{n_w(\phi)} \scV^{\# (\d_{\b_0}(\phi)\cap K)}=U^{n_w(\phi)} \scV^{n_z(\phi)-n_w(\phi)}=\scU^{n_w(\phi)} \scV^{n_z(\phi)}.
\] 
 \end{rem}
 
Our proof will be to construct morphisms of twisted complexes of Lagrangians which fit into the following diagram:
\begin{equation}
\begin{tikzcd}[column sep=3cm]
\b_{\lambda}
	\ar[r, "\Phi", shift left=1.1mm]
	& \Cone(\lb \theta_\tau^+,\phi^\tau\rb+\lb \theta_\sigma^+, \phi^\sigma\rb) \ar[l, "\Psi", shift left=1.1mm] \ar[loop below, "Z"]
\end{tikzcd}
\label{eq:overview-diagram-HE}
\end{equation}
We will show that $\d(\Psi)=0$, $\d(\Phi)=0$ and
\[
\Psi\circ \Phi=\id_{\b_{\lambda}}\quad \text{and} \quad \Phi\circ \Psi=\id_{\Cone}+\d(H).
\]
In the above, of course $\Psi\circ \Phi$ means $\mu_2(\Phi,\Psi)$, and $\d(H)$ means $\mu_1(H)$. (The difference in ordering between $\Psi\circ \Phi$ and $\mu_2(\Phi,\Psi)$ is because we are thinking of $\Psi$ and $\Phi$ as morphisms of beta attaching curves).

 \subsection{Morphisms}

 In this section, we define the morphisms $\Phi$ and $\Psi$ which appear in Equation~\eqref{eq:overview-diagram-HE}. 

There is a natural Floer chain
\[
\lb\theta_{\b_1,\b_{\lambda}},\Pi\rb\colon \b_1^{E_1}\to \b_\lambda.
\]
Here, $\Pi$ is the map given by
\[
\Pi(U^i T^j)=\begin{cases} U^i& \text{ if } j=0,\\
0& \text{ otherwise}.
\end{cases}
\]

Next, there is a Floer chain
\[
\lb \theta_{\b_{\lambda},\b_0},\Delta\rb \colon \b_\lambda\to \b_0^{E_0}.
\]
Here, $\Delta$ is the map $\Delta(U^i)= (\scU\scV)^i$.  It is straightforward to see that $\lb\theta_{\b_1,\b_\lambda},\Pi\rb$ and $\lb\theta_{\b_\lambda,\b_0},\Delta\rb$ are both cycles in their respective Floer complexes.

We now define our morphisms between  $\b_{\lambda}$ and the twisted complex $\Cone(\lb\theta_\tau^+,\phi^\tau\rb+\lb\theta_\sigma^+,\phi^\sigma\rb)$. For our morphism $\Phi$ from $\b_{\lambda}$ to the cone, we use the morphism $\lb\theta_{\b_\lambda,\b_0}, \Delta\rb$ (and the 0 morphism from $\b_\lambda$ to $\b_1$). This is shown in the diagram below (solid arrows denote the morphism, and dashed arrows denote internal differentials):
\[
\Phi=\begin{tikzcd}[column sep=2cm]
\b_\lambda \ar[d, "{\lb\theta_{\b_\lambda,\b_0},\Delta\rb}"]
\\
\b_0^{E_0}\ar[r,dashed]&\b_1^{E_1}
\end{tikzcd}
\]

Similarly, we define a $\Psi$ morphism in the opposite direction as follows. It is helpful to view the morphism in the opposite direction as going from $\Cone(\lb\theta_\sigma^+,\phi^\sigma\rb+\lb\theta_\tau^+,\phi^\tau\rb)$ to $\b'_\lambda$, where $\b'_\lambda$ is a small Hamiltonian translation of $\b_\lambda$, intersecting it in two points. (We also equip $\b'_{\lambda}$ with the trivial local system). We define $\Psi$ to be the solid arrow in the diagram below:
\[
\Psi=
\begin{tikzcd}[column sep=2cm]
\b_0^{E_0}\ar[r,dashed]&\b_1^{E_1}
	\ar[d, "{\lbmed\theta_{\b_1,\b_\lambda'},\Pi\rbmed}"]
\\
&\b_\lambda'
\end{tikzcd}
\]

\begin{lem}
 The morphisms $\Phi$ and $\Psi$ satisfy $\mu_1(\Phi)=0$ and $\mu_1(\Psi)=0$. 
\end{lem}
\begin{proof}  It suffices to show that $\mu_1(\lb \theta_{\b_\lambda,\b_0}, \Dt \rb)=0$ and that
\begin{equation}
\mu_2\left( \lb\theta_{\b_\lambda,\b_0},\Delta \rb,\lb\theta_\tau^+,\phi^\tau\rb+\lb\theta_\sigma^+,\phi^\sigma\rb \right)=0.\label{eq:composition-mu_1-Phi=0}
\end{equation}
The relation $\mu_1(\lb \theta_{\b_\lambda,\b_0}, \Dt \rb)=0$ is straightforward. For Equation~\eqref{eq:composition-mu_1-Phi=0}, observe that there are two holomorphic triangles which contribute.  They are shown in Figure~\ref{fig:18}. It remains to show that the linear maps from the two triangles coincide.
\begin{figure}[ht]
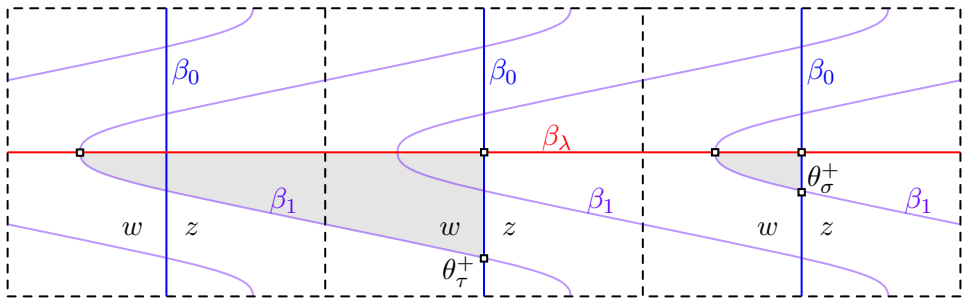
\caption{Triangles on the diagram $(\bT^2,\b_\lambda,\b_0,\b_1)$. The computation is shown on a cover of the torus.}
\label{fig:18}
\end{figure}

For the $\sigma$-labeled triangle, the linear map is
\[
\phi^\sigma\circ \Delta=(U^i\mapsto U^iT^0).
\]
For the $\tau$-labeled triangle, the linear map is
\[
U (T\circ \phi^\tau\circ \scV^{-1}\circ \Delta)=\phi^\tau\circ \Delta=(U^i\mapsto U^i T^0).
\]
Hence the $\sigma$ and $\tau$ labeled triangles make canceling contribution, completing the proof. 

We now consider the morphism $\Psi$. It is sufficient to show that
\[
\mu_2\left(\lb\theta_\tau^+,\phi^\tau\rb+\lb\theta_\sigma^+,\phi^\sigma\rb, \lbmed\theta_{\b_1, \b_\lambda'},\Pi\rbmed\right)=0.
\]
In this case, the holomorphic triangles are essentially the same as considered for the map $\Phi$. It remains to consider the vector space morphisms associated to them. The $\sigma$-labeled triangle now has morphism-weight
\[
\Pi\circ \phi^\sigma.
\]
The $\tau$-labeled triangle has weight
\[
U(\Pi\circ T\circ \phi^\tau\circ \scV^{-1})=\Pi\circ \phi^\tau.
\]
We observe that
\[
\Pi\circ \phi^\sigma=\Pi\circ \phi^\tau
\]
since both maps send $\scU^i\scV^j$ to $U^i$ if $i=j$ and to 0 if $i\neq j$. Hence the two triangles make canceling contribution, so $\mu_1(\Psi)=0$.
\end{proof}

\subsection{Homotopy equivalence}

We now show that the morphisms $\Phi$ and $\Psi$, described in the previous section form a homotopy equivalence. One composition is easier than the other:

\begin{prop}\label{prop:PhiPsi}
 There is an equality
\[
\mu_2(\Phi, \Psi) =\lbmed\theta^+_{\b_\lambda,\b_\lambda'},\id\rbmed.
\]
\end{prop}
\begin{proof} The composition of the morphisms $\Phi$ and $\Psi$ has a single contribution, which is the holomorphic quadrilateral count:
\[
\mu_3\left(\lb\theta_{\b_\lambda,\b_0}, \Delta\rb, \lb\theta_\tau^+,\phi^\tau\rb+\lb\theta_\sigma^+,\phi^\sigma\rb, \lbmed\theta_{\b_1,\b_\lambda'}, \Pi\rbmed\right).
\]
This holomorphic quadrilateral count is performed in Figure~\ref{fig:19}. For the configuration shown, there is exactly one index $-1$ holomorphic quadrilateral with the $\tau$ input, and there are none with the $\sigma$ input. The morphism associated to the unique holomorphic quadrilateral is
\[
U(   \Pi \circ T\circ \phi^{\tau}\circ \scV^{-1}\circ\Delta)=\id.
\]
Hence, the output of the quadrilateral count is $[\theta^+_{\b_\lambda,\b_\lambda'},\id]$, completing the proof.
\end{proof}

\begin{figure}[ht]
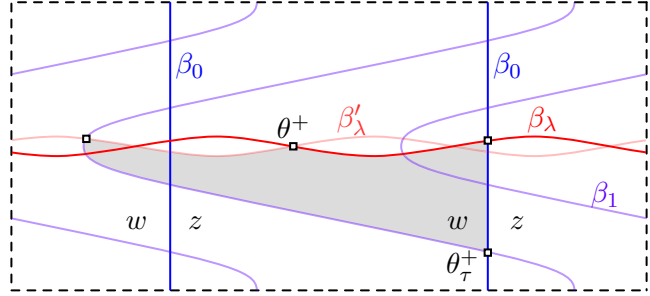
\caption{A holomorphic rectangle on the diagram $(\bT^2,\b_\lambda,\b_0,\b_1,\b_\lambda')$.}
\label{fig:19}
\end{figure}

We now consider the composition $\mu_2(\Psi, \Phi)$. In contrast to Proposition~\ref{prop:PhiPsi}, we have the following:

\begin{prop}\label{prop:PsiPhi}There is a morphism 
\[
H\colon \Cone(\b_0^{E_0}\to \b_1^{E_1})\to \Cone(\b_0'^{E_0}\to \b_1'^{E_0})
\]
 so that
\[
\mu_2(\Psi,\Phi)+ \lbmed\theta_{\b_0,\b_0'}^+,\id\rbmed+\lbmed\theta_{\b_1,\b_1'}^+,\id\rbmed=\mu_1 (H).
\]
\end{prop}

The proof of Proposition~\ref{prop:PsiPhi} is slightly  more involved than Proposition~\ref{prop:PhiPsi}, because we must compute the composition and construction a homotopy.

\begin{lem} The composition $\mu_2(\Psi, \Phi)$ is equal to the solid arrows in the diagram below:
\[
\begin{tikzcd}[column sep=2cm, row sep=1.5cm, labels=description]
\b_0^{E_0}\ar[r,dashed] \ar[dr,"X"]& \b_1^{E_1}\ar[d,"Y"]\\
& \b_0'^{E_0}\ar[r,dashed]& \b_1'^{E_1}
\end{tikzcd}
\]
where:
\begin{enumerate}
\item $X=\lbmed\theta^+_{\b_0,\b_0'},\delta_0\rbmed$, where $\delta_0\colon \bF[\scU,\scV]\to \bF[\scU,\scV]$ denotes projection to Alexander grading 0. 
\item $Y=\lb\theta_\sigma^-,\Delta\circ \Pi\rb+\lb\theta_\tau^-,\scU\circ \Delta\circ \Pi\circ T\rb$.
\end{enumerate}
\end{lem}
\begin{proof} The two claims amount to computing the  following compositions:
\begin{enumerate}
\item\label{comp-1} $\mu_3(\lbsm\theta_{\sigma}^+,\phi^\sigma\rbsm+\lbsm\theta_\tau^+,\phi^\tau\rbsm,\lbsm\theta_{\b_1,\b_\lambda},\Pi\rbsm, \lbmed\theta_{\b_\lambda,\b_0'},\Delta\rbsm)$. This is $X$.
\item\label{comp-2} $\mu_2( \lbsm\theta_{\b_1,\b_\lambda},\Pi\rbsm, \lbsm\theta_{\b_\lambda,\b_0},\Delta\rbsm )$. This is $Y$.
\item\label{comp-3}
 $\mu_3(\lbsm\theta_{\b_1,\b_\lambda},\Pi\rbsm, \lbsm\theta_{\b_\lambda,\b_0'},\Delta\rbsm,\lbsm\theta_{\sigma}^+,\phi^\sigma\rbsm+\lbsm\theta_\tau^+,\phi^\tau\rbsm)$.
\item\label{comp-4} $\mu_4(\lbsm\theta_{\sigma}^+,\phi^\sigma\rbsm+\lbsm\theta_\tau^+,\phi^\tau\rbsm,\lbsm\theta_{\b_1,\b_\lambda},\Pi\rbsm, \lbsm\theta_{\b_\lambda,\b_0'},\Delta\rbsm,\lbsm\theta_{\sigma}^+,\phi^\sigma\rbsm+\lbsm\theta_\tau^+,\phi^\tau\rbsm)$.
\end{enumerate}
We claim that the first two compositions contribute $X$ and $Y$, and that the second two compositions make trivial contribution.

The composition in~\eqref{comp-1} is computed in Figure~\ref{fig:20}. There is a single holomorphic rectangle which contributes. The $\bF[U]$-module map in the output is
\[
U(\Delta\circ \Pi \circ T\circ \phi^\tau\circ \scV^{-1})=\delta_0,
\]
as claimed.

\begin{figure}[ht]
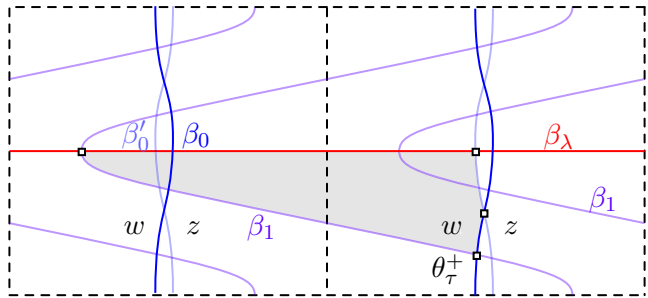
\caption{A quadrilateral on the diagram $(\b_0,\b_1,\b_\lambda,\b_0')$. This rectangle has the $\tau$-labeled generator as input. There are no rectangles with $\sigma$-labeled input on this diagram.}
\label{fig:20}
\end{figure}

Next, we consider the composition in ~\eqref{comp-2}. The two holomorphic triangles which contribute are the same ones that have appeared in our earlier computations, as shown in Figure~\ref{fig:18}. Note that since we have cyclically permuted $\b_0,\b_1,\b_\lambda$ the output is $\theta_\sigma^-$ and $\theta_\tau^-$ (as opposed to $\theta_\sigma^+$ and $\theta_\tau^+$). The vector space map output for the $\tau$-labeled triangle is
\[
U(\scV^{-1} \circ  \Delta \circ \Pi \circ T)= \scU\circ \Delta\circ \Pi\circ T.
\]
The vector space output for the $\sigma$-labeled triangle is
\[
\Delta\circ \Pi,
\]
as claimed.

We now consider the compositions~\eqref{comp-3} and~\eqref{comp-4}. For ~\eqref{comp-3}, the relevant diagram is $(\bT^2, \b_1, \b_\lambda, \b_0',\b_1')$ shown in Figure~\ref{fig:21}. There are no quadrilaterals which are counted. 

 The composition~\eqref{comp-4}  counts holomorphic pentagons on the diagram $(\bT^2, \b_0,\b_1,\b_\lambda,\b_0',\b_1')$. We leave it as a straightforward (though slightly tedious) exercise to the reader to verify that there are no holomorphic pentagons which contribute. 
 \end{proof}

\begin{figure}[ht]
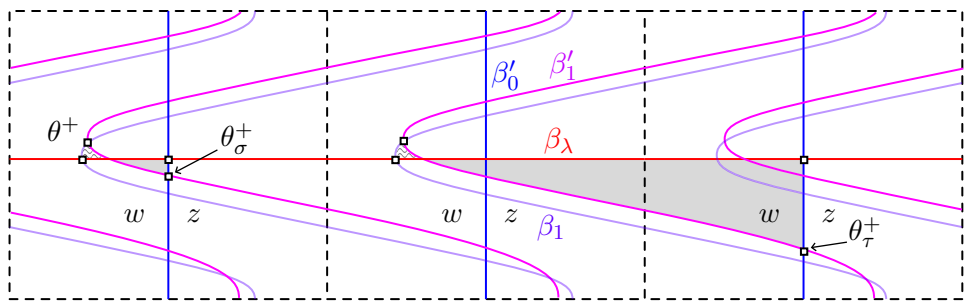
\caption{The diagram $(\bT^2,\b_1,\b_\lambda,\b_0',\b_1')$. Two index $-1$ rectangle classes are shown (solid regions indicate multiplicity 1, while wavy filled regions have multiplicity $-1$). }
\label{fig:21}
\end{figure}

We now observe that the morphism 
\[
Y\colon \b_1^{E_1}\to \b_0'^{E_0}
\]
 is a boundary:
 
 \begin{lem}
 \label{lem:d(Z)}
  Define the morphism 
  \[
  Z=Z_\sigma^++Z_\sigma^+\colon \b_1^{E_1}\to \b_0'^{E_0}
  \]
   where $Z_\sigma^+=\lbmed\theta_\sigma^+,\eta_\sigma\rbmed$ and $Z_{\tau}^+:= \lbmed\theta_\tau^+,\eta_\tau\rbmed$ and
\[
\eta_\sigma=\sum_{i=0}^\infty\scV^i \circ \Delta\circ \Pi\circ T^{-i}\quad \text{and} \quad \eta_\tau=\sum_{i =1}^\infty \scU^i \circ \Delta\circ \Pi\circ T^i.
\]
 As morphisms from $\b_1^{E_1}$ to $\b_0'^{E_0}$, 
 \[
 \mu_1( Z)=Y.
 \]
 \end{lem}

\begin{proof} Note that $\eta_\sigma$ and $\eta_\tau$ are continuous in both the chiral and $U$-adic topologies, when viewed as $\bF[U]$-module morphisms from $E_1$ to $E_0$. For the proof that $\mu_1(Z)=Y$, we observe similarly to Lemma~\ref{lem:theta-tau/sigma-cycles} that if $f\colon E_1\to E_0$ is an $\bF[U]$-module map, then
\[
\mu_1 \lb \theta^+_\tau, f\rb=\lb\theta_\tau^-, U(\scV^{-1}\circ  f \circ T) \rb+\lb\theta_\tau^-, f\rb=\lb\theta_\tau^-, f+\scU\circ f\circ T\rb.
\]
Similarly, counting bigons gives
\[
\mu_1 \lb\theta^+_\sigma,f\rb=\lb\theta_\sigma^-, f+\scV\circ f \circ T^{-1}\rb.
\]
From the above two computations, the equation $\mu_1 (Z)=Y$ follows from the easily verified equations
\[
\eta_\sigma+\scV \circ \eta_\sigma\circ T^{-1}=\Delta\circ \Pi\quad \text{and} \quad \eta_\tau+\scU\circ \eta_\tau\circ T=\scU \circ \Delta\circ \Pi\circ T.
\]
\end{proof}

We now define the morphism $H$ from $\Cone(\b_0^{E_0}\to \b_1^{E_1})$ to $\Cone(\b_0'^{E_0}\to \b_1'^{E_1})$ via the solid arrows in the diagram below
\[
H=\begin{tikzcd}[column sep=2cm, row sep=1.5cm, labels=description]
\b_0^{E_0}\ar[r,dashed] & \b_1^{E_1}\ar[d,"Z"]\\
& \b_0'^{E_0}\ar[r,dashed]& \b_1'^{E_1}
\end{tikzcd}
\]

\begin{lem} We have
\[
\mu_2(\Psi, \Phi)+\mu_1(H)=\lbmed\theta_{\b_0,\b_0'}^+,\id\rbmed+\lbmed\theta_{\b_1,\b_1'}^+,\id\rbmed.
\]
\end{lem}
\begin{proof}The claim follows from the following four claims:
\begin{enumerate}
\item\label{comp-b-1} $\mu_1(Z)=Y$, where $\mu_1$ here denotes the internal differential of $\ve{\CF}^-(\b_1^{E_1},\b_0'^{E_0})$.
\item\label{comp-b-2} The holomorphic triangle count
\[
\mu_2\left(\lb\theta_{\sigma}^+,\phi^\sigma\rb+\lb\theta_\tau^+,\phi^\tau\rb,Z\right)=\sum_{i\in \Z\setminus \{0\}}\lbmed\theta_{\b_0,\b_0'},  \delta_i\rbmed.
\]
\item\label{comp-b-3} The holomorphic triangle count
\[
\mu_2\left(Z,\lb\theta_{\sigma}^+,\phi^\sigma\rb+\lb\theta_\tau^+,\phi^\tau\rb\right)=\lbmed\theta_{\b_1,\b_1'}^+,\id\rbmed.
\]
\item\label{comp-b-4} The holomorphic quadrilateral count
\[
\mu_3\left(\left[\theta_{\sigma}^+,\phi^\sigma\right]+\left[\theta_\tau^+,\phi^\tau\right],Z,\left[\theta_{\sigma}^+,\phi^\sigma\right]+\left[\theta_\tau^+,\phi^\tau\right]\right)=0.
\]
\end{enumerate}

The component $\mu_1(Z)$ from the internal differential is computed in Lemma~\ref{lem:d(Z)}.

The component in~\eqref{comp-b-2} is verified by the triangle count shown in Figure~\ref{fig:22}. Using the Figure, we compute that if $f\colon E_1\to E_0$ is a map, then 
\[
\mu_2\big(\lbmed\theta_{\sigma,\b_0,\b_1}^+,\phi^\sigma\rbmed, \lbmed\theta_{\sigma, \b_1, \b_0'}^+, f\rbmed\big)=\lbmed\theta_{\b_0,\b_0'}^+,  \scV \circ f\circ T^{-1}\circ \phi^\sigma\rbmed 
\] 
and
\[
\mu_2\big(\lbmed\theta_{\tau,\b_0,\b_1}^+,\phi^\tau\rbmed, \lbmed\theta_{\tau,\b_1,\b_0'}^+, f\rbmed\big)=\lbmed\theta_{\b_0,\b_0'}^+ , U\left(f \circ  T \circ \phi^\tau\circ  \scV^{-1}\right) \rbmed=\lbmed\theta_{\b_0,\b_0'}^+,f\circ \phi^\tau\rbmed.
\]
In particular, we compute that
\[
\mu_2(\lb \theta_\sigma^+,\phi^\sigma\rb, Z_\sigma^+)=\sum_{i \ge 0}\mu_2\big(\lb\theta_\sigma^+, \phi^\sigma \rb, \lbmed\theta_\sigma^+, \scV^i \circ \Delta \circ \Pi\circ T^{-i}\rbmed\big ) =\sum_{i \ge 1} \lbmed\theta_{\b_0,\b_0'}^+, \scV^i\circ \Delta \circ \Pi\circ T^{-i}\rbmed.
\]
We observe that the above is equal to 
\[
\sum_{i\ge 1} \lbmed\theta_{\b_0,\b_0'}^+,\delta_i\rbmed.
\]
Finally, we observe that the pairing of $\theta_\tau^+$ with $\theta_{\sigma}^+$ is trivial, because they lie in different $\Spin^c$ structures.
\begin{figure}[ht]
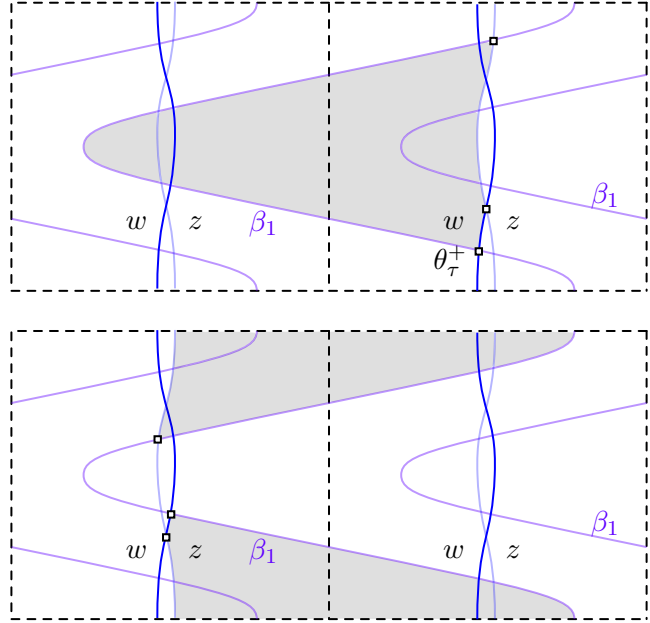
\caption{Two holomorphic triangles on $(\b_0,\b_1,\b_0')$.}
\label{fig:22}
\end{figure}

Similarly, we compute that
\[
\mu_2\big(\lb \theta_\tau^+,\phi^\tau\rb,Z_\tau^+\big)=\sum_{i \ge 1} \lbmed\theta_{\b_0,\b_0'}^+, \scU^i\circ  \Delta \circ \Pi\circ T^i\rbmed=\sum_{i\le -1} \lbmed\theta_{\b_0,\b_0'}^+, \delta_i\rbmed.
\]

We obtain that the sum of all arrows in $\mu_2(\Psi, \Phi)+\mu_1(Z)$ from $\b_0$ to $\b_0'$ is
\[
\sum_{i\in \Z}\lbmed\theta_{\b_0,\b_0'}^+,\delta_i\rbmed=\lbmed\theta_{\b_0,\b_0'}^+,\id\rbmed,
\]
as claimed.

Next, we move on to claim~\eqref{comp-b-3} and the component of $\mu_2(\Psi, \Phi)+\mu_1(Z)$ from $\b_1$ to $\b_1'$. We perform the relevant holomorphic triangle count in Figure~\ref{fig:23}.

\begin{figure}[ht]
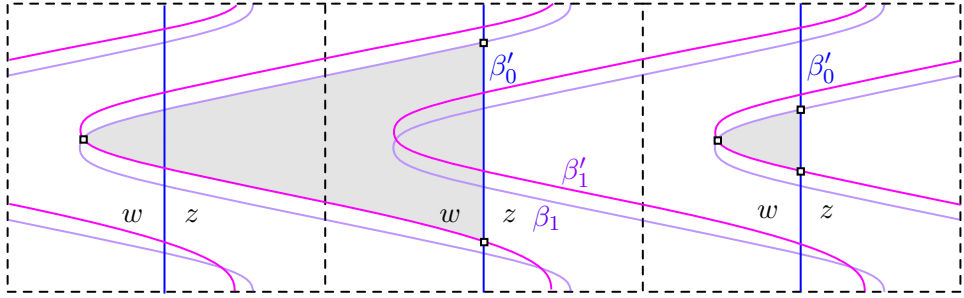
\caption{Holomorphic triangles on the diagram $(\b_1,\b_0',\b_1'$).}
\label{fig:23}
\end{figure}

Using Figure~\ref{fig:23}, we conclude that if $f\colon E_1\to E_0$ is an $\bF[U]$-linear map, then
\[
\mu_2\big(\lbmed\theta^+_{\sigma,\b_1,\b_0'},f\rbmed,\lbmed\theta_{\sigma,\b_0',\b_1'}^+,\phi^\sigma\rbmed \big)=\lbmed \theta_{\b_1,\b_1'}^+,\phi^\sigma \circ f\rbmed
\]
and
\[
\mu_2\big(\lbmed\theta^+_{\tau,\b_1,\b_0'},f\rbmed,\lbmed \theta_{\tau,\b_0',\b_1'}^+,\phi^\tau \rbmed\big)=\lbmed\theta_{\b_1,\b_1'}^+, U(T\circ \phi^\tau\circ \scV^{-1}\circ f)\rbmed=\lbmed\theta_{\b_1,\b_1'}^+, \phi^\tau\circ f\rbmed.
\]
We conclude that
\begin{equation}
\mu_2\big(Z_\sigma^+, \lbmed \theta_{\sigma,\b_0',\b_1'}^+,\phi^\sigma\rbmed\big)=\sum_{i \ge 0} \lbmed\theta_{\b_1,\b_1'}^+, T^{i}  \circ \phi^\sigma \circ \Delta\circ \Pi\circ T^{-i}\rbmed=\sum_{i \ge 0} \lbmed\theta_{\b_1,\b_1'}^+,\delta_i'\rbmed
\label{eq:comp-mu-2-Z-sigma}
\end{equation}
where $\delta_i'\colon E_1\to E_1$ is projection onto Alexander grading $i$.
Similarly, we compute that
\begin{equation}
\mu_2\big(Z_\tau^+, \lbmed\theta_{\tau,\b_0',\b_1'}^+, \phi^\tau\rbmed\big)=\sum_{i\ge 1}\lbmed\theta_{\b_1,\b_1'}^+, T^{-i} \circ \phi^\tau \circ \Delta\circ \Pi\circ T^i\rbmed=\sum_{i\le -1} \lbmed\theta_{\b_1,\b_1'}^+, \delta_i'\rbmed.
\label{eq:comp-mu-2-Z-tau}
\end{equation}
Adding Equations~\eqref{eq:comp-mu-2-Z-sigma} and ~\eqref{eq:comp-mu-2-Z-tau}, we conclude that the morphism from $\b_1$ to $\b_1'$ in $\mu_2(\Psi, \Phi)+\mu_1(H)$ coincides with $\lbmed\theta_{\b_1,\b_1'}^+, \id\rbmed$.

We now investigate ~\eqref{comp-b-4}. This is verified by analyzing the diagram shown in Figure~\ref{fig:25} and observing that there are no holomorphic quadrilaterals of index $-1$ with the given inputs. Note that there are eight total choices of inputs for a polygon (depending on the $\Spin^c$ structures for the inputs), however no choice admits a holomorphic polygon of index $-1$. We leave the verification of this final claim as an exercise to the reader. 
\end{proof}

\begin{figure}[ht]
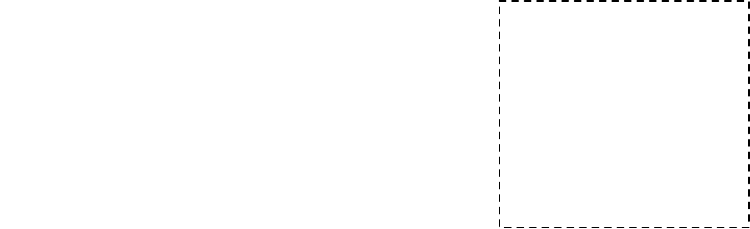
\caption{The Heegaard quadruple $(\bT^2,\b_0,\b_1,\b_0',\b_1')$}
\label{fig:25}
\end{figure}

\section{Iterating the mapping cone formula}

\label{sec:iterating-background}

We now describe how to iterate the exact triangle from Section~\ref{sec:mapping-cone} to prove the link surgery formula for all Morse framed links in 3-manifolds. In Section~\ref{sec:background-link-surgery}, we gave background on the link surgery formula. Let $L\subset Y$ with Morse framing $\Lambda$. We consider a meridianal Heegaard diagram $(\Sigma,\as,\bs_0,\ws,\zs)$ of $(Y,L)$ with a choice of knot shadows $\cS$ on $\Sigma$. Recall that in Section~\ref{sec:background-link-surgery} we described an $|\ell|$-dimensional hypercube of attaching curves $\scB_{\Lambda}$, by iterating the construction for knots in Section~\ref{sec:knot-surgery-formula}. Additionally, we can define a single collection of attaching curves $\bs_{\Lambda}\subset \Sigma$, by removing the meridianal components from $\bs_0$, and adjoining copies of the knot shadows $\cS$. Note that $(\Sigma,\as,\bs_{\Lambda},\ws)$ is an $\ell$-pointed Heegaard diagram for $Y_{\Lambda}(L)$. 

We prove the following:
\begin{thm}\label{thm:iterated-cone}
 There is a homotopy equivalence of twisted complexes of attaching curves with local systems
\[
\bs_{\Lambda}\simeq \scB_{\Lambda}.
\]
\end{thm}

As a consequence, we have the following generalization of Manolescu and Ozsv\'{a}th's link surgery formula (stated in Theorem~\ref{thm:MOlink-surg}):

\begin{cor} For any link $L$ in a closed 3-manifold $Y$ with Morse framing $\Lambda$, there is a homotopy equivalence of chain complexes
\[
\ve{\CF}^-(Y_{\Lambda}(L))\simeq \bX_{\Lambda}(Y,L).
\]
\end{cor}

After recalling some basic formalism, we will show how Theorem~\ref{thm:iterated-cone} quickly follows from the genus 1 case.

\subsection{The 1-handle \texorpdfstring{$A_\infty$}{A-infty}-functor}

To increase the genus, or connect different components of a disconnected Heegaard diagram, we may attach a 1-handle to the surface. We may think of this as a functor between the Fukaya categories of the two surfaces. Given a set of attaching curves $\bs\subset \Sigma$, if $S(\Sigma)$ is obtained by attaching a 1-handle to the surface $\Sigma$, then we construct a set of attaching curves $S(\bs)$ by adjoining the meridian of the 1-handle tube. To a morphism $\Theta\in \ve{\CF}^-(\bs,\bs')$, we define
\[
S(\Theta):=\Theta\times \theta^+\in \ve{\CF}^-(S(\bs),S(\bs')).
\]
Here, we assume that the curves of $S(\bs)$ and $S(\bs')$ are perturbed in the 1-handle region to intersect transversely in two points $\theta^+$ and $\theta^-$.

\begin{prop}\label{prop:1-handle-functor}
 Assuming attaching curves are chosen admissibly, the map $S$ on attaching curves and morphisms is an $A_\infty$-functor for sufficiently stretched almost complex structure. 
\end{prop}
The above result generalizes Ozsv\'{a}th and Szab\'{o}'s original proof that their 1-handle map is a chain map \cite{OSTriangles}*{Section~4.3}.  Proposition~\ref{prop:1-handle-functor} follows  from the stabilization theorem in \cite{HHSZNaturality}*{Proposition~5.5} (cf. Example~5.7 therein), which has the following statement. Let $(\Sigma,\gs_0,\dots, \gs_n,\ws)$ be a Heegaard diagram. Adjoin a tube to $\Sigma$ (disjoint from $\gs_i$). In this tube, add $n$ copies of the meridian $c$, suitably translated to intersect pairwise in two points and achieve admissibility. Then, for any sequence of intersection points $\xs_i\in \bT_{\g_{i-1}}\cap \bT_{\g_i}$, we have
\[
f_{\g_0c,\dots, \g_n c}(\xs_1\times \theta^+,\dots, \xs_n\times \theta^+)=f_{\g_0,\dots, \g_n}(\xs_1,\dots, \xs_n)\otimes \theta^+.
\]
See \cite{ZemGraphTQFT}*{Theorem~8.8} for the case of holomorphic triangles.

\subsection{Disconnected surfaces}

We now recall some formalities about hypercubes of attaching curves on disconnected surfaces. If $\scB=(\bs_{\veps},\theta_{\veps,\veps'})_{\veps\in \bE_n}$ is a hypercube of attaching curves on $\Sigma$, and $\scB'
=(\bs'_{\nu},\theta'_{\nu,\nu'})_{\nu\in \bE_m}$ is a hypercube of attaching curves on $\Sigma'$, then we may construct a diagram of attaching curves $\scB\times \scB'$ of dimension $n+m$, on $\Sigma\sqcup \Sigma'$. We write $(\gs_{(\veps,\nu)}, \phi_{(\veps,\nu),(\veps',\nu')})_{(\veps,\nu)\in \bE_{n+m}}$ for $\scB\times \scB'$, whose construction we now recall. We set
\[
\gs_{(\veps,\nu)}\approx\bs_\veps\sqcup \bs'_{\nu},
\]
where $\approx$ indicates that we translate the attaching curves to achieve admissibility, and so that the copies of $\bs_{\veps}$ from $\gs_{(\veps,\nu)}$ and $\gs_{(\veps,\nu')}$ are Hamiltonian translates of each other which intersect in the minimal possible number of points while achieving admissibility. We make the parallel assumption on the curves $\bs'_{\nu}$ in different collections with the same $\nu$-coordinate.
We set 
\[
\phi_{(\veps,\nu),(\veps',\nu')}=
\begin{cases} \theta_{\veps,\veps'}\otimes \theta^+_{\nu}& \text{ if } \nu=\nu'\\
\theta^+_{\veps}\otimes \theta_{\nu,\nu'} &\text{ if } \veps=\veps'\\
0& \text{ if } \veps<\veps' \text{ and } \nu<\nu'.
\end{cases}
\]

The above construction is based on a similar construction of Lipshitz, Ozsv\'{a}th and Thurston \cite{LOTDoubleBranchedII}*{Definition~3.40}. They show in \cite{LOTDoubleBranchedII}*{Proposition~3.42} that if $\scB$ and $\scB'$ are suitably admissible and the approximations are suitably small, then $\scB\times \scB'$ is a hypercube of attaching curves. See also \cite{ZemBordered}*{Proposition~10.1} for a similar argument in our present notation. By ``suitably admissible'' we require a strong version of admissibility (such as the notion of strong $\frS$-admissibility \cite{OSDisks}*{Section~8.4.2} for a collection $\frS$ of $\Spin^c$ structures) to ensure finiteness of nonnegative homology classes of polygons with given inputs and a given Maslov index. This is satisfiable, for example, when all of the attaching curves of $\scB$ are pairwise related by handleslides and isotopies, and similarly for $\scB'$. The strong admissibility requirement is needed because the argument of Lipshitz, Ozsv\'{a}th and Thurston uses a limiting argument which applies only to a finite number of polygons classes at a time.

A special case of the proof is the following:

\begin{lem}
\label{lem:tensor-product-hypercube}
Suppose $\scB_0, \dots, \scB_n$ is a sequence of hypercubes and that we have a sequence of morphisms of hypercubes of attaching curves
\[
\begin{tikzcd}
\scB_0\ar[r, "\theta_{0,1}"]&\scB_1\ar[r, "\theta_{1,2}"]&\cdots\ar[r, "\theta_{n-1,n}"]& \scB_n.
\end{tikzcd}
\] 
Suppose that $\scB'$ is another hypercube of attaching curves. Then the composition of the corresponding morphisms between the $\scB_i\times \scB'$ is given by
 \[
 \mu_{n}^{\Tw}(\theta_{0,1}\otimes \theta^+,\dots, \theta_{n-1,n}\otimes \theta^+)=\mu_{n}^{\Tw}(\theta_{0,1},\dots, \theta_{n-1,n})\otimes \theta^+.
 \]
\end{lem}
The above follows immediately from \cite{ZemBordered}*{Lemma~10.6}.

\begin{rem} Morally, the construction $\scB\times \scB'$ is an instance of a more general operation. Firstly, if $\scA$ and $\scA'$ are two strictly unital $A_\infty$-categories (for example, strictly unital $A_\infty$-algebras), we can form another $A_\infty$-category $\scA\otimes_{\Gamma} \scA'$ as follows. As the notation indicates, this construction requires a choice of cellular diagonal $\Gamma$ of the associahedron. Tersely, this has the following meaning. Write $C_*^{\cell}(K_n)$ for the standard cell complex of the associahedron (freely generated by planar trees with $n$ inputs and no valence 2 vertices). Then $\Gamma$ is a sequence of grading preserving chain maps 
\[
\Gamma_n\colon C_*^{\cell}(K_n)\to C_*^{\cell}(K_n)\otimes C_*^{\cell}(K_n)
\]
for $n\ge 2$, which are compatible with splicing of trees. Here, we are using the notation of Lipshitz, Ozsv\'{a}th and Thurston \cite{LOTDiagonals}. See  \cite{SaneblidzeUmble} \cite{Loday} for earlier work. Objects of $\scA\otimes_{\Gamma} \scA'$ consist of pairs $(i,j)$ where $i$ is an object of $\scA$, and $j$ is an object of $\scA'$. One sets $\Hom( (i,j), (i',j'))=\Hom(i,j)\otimes_{\bF} \Hom(i',j')$. The composition maps on $\scA\otimes_{\Gamma} \scA'$ are computed using the diagonal $\Gamma$. Now, if $(X,\delta)$ and $(X',\delta')$ are two twisted complexes in $\scA$ and $\scA'$, respectively, then the pair 
\[
(X\otimes X', \delta\otimes \id+\id\otimes \delta')
\] is a twisted complex in $\scA\otimes_{\Gamma} \scA'$ for any choice of diagonal $\Gamma$. This follows easily from strict unality of $\scA$ and $\scA'$, as well as a basic fact about associahedron diagonals (see \cite{ZemBordered}*{Lemma~3.10}, though presumably this is well known). Morally, the operation $\scB\times \scB'$ is in instance of the above construction, and the limiting argument using ``sufficiently small'' translations is a stand-in for strict unality. 
\end{rem}

\subsection{Proof of Theorem~\ref{thm:iterated-cone}}

Theorem~\ref{thm:iterated-cone} is now a formal consequence of Proposition~\ref{prop:1-handle-functor} and Lemma~\ref{lem:tensor-product-hypercube}, and the genus 1 case, Proposition~\ref{prop:general-HE}. In a bit more detail, Proposition~\ref{prop:general-HE} handles the case when there is one special beta curve, and the genus of the underlying Heegaard surface is 1. Lemma~\ref{lem:tensor-product-hypercube} proves the theorem when we have $n$ special beta curves, and the Heegaard surface consists of $n$ components which are each genus 1. Proposition~\ref{prop:1-handle-functor} allows us to connect different components of the Heegaard surface and then further increase the genus by attaching 1-handles, while keeping the number of special beta curves fixed.

\section{Equivalence with the Manolescu-Ozsv\'{a}th-Szab\'{o} construction}
\label{sec:reformulation}

In this section, we prove that our construction is equivalent to the original construction of Ozsv\'{a}th and Szab\'{o} \cite{OSIntegerSurgeries} \cite{OSRationalSurgeries} and Manolescu and Ozsv\'{a}th \cite{MOIntegerSurgery}, when these are defined. We focus on the presentation given by Manolescu and Ozsv\'{a}th for links in $S^3$.

In this section, we adopt the following notation:
\begin{enumerate}
\item We write $\cX_{\Lambda}(L)^{\cL}$ and $\bX_{\Lambda}(L)$ for the type-$D$ modules and link surgery complexes constructed using the Manolescu-Ozsv\'{a}th-Szab\'{o} construction. We write $\Phi^{\vec{M}}$ for the hypercube structure maps on $\bX_{\Lambda}(L)$. 
\item We write $\cY_{\Lambda}(L)^{\cL}$ and $\bY_{\Lambda}(L)$ for the type-$D$ modules and link surgery complexes, constructed in this paper. We write $\Psi^{\vec{M}}$ for the hypercube structure maps on $\bY_{\Lambda}(L)$. 
\end{enumerate}

\begin{thm}\label{thm:MO=our-version}
 If $L\subset S^3$ is an integrally framed link, then there is a homotopy equivalence of type-$D$ modules
\[
\cX_{\Lambda}(L)^{\cL}\simeq \cY_{\Lambda}(L)^{\cL}.
\]
\end{thm}

We remark that essentially the same analysis can be made to compare our construction for Morse framed knots in rational homology 3-spheres to the one in \cite{OSRationalSurgeries}. 

To simplify the notation, we will focus on proving the isomorphism
\begin{equation}
\bX_{\Lambda}(L)\iso \bY_{\Lambda}(L).
\label{eq:surgery-complexes-isomorphic}
\end{equation}
 Note that the existence of a homotopy equivalence in Equation~\eqref{eq:surgery-complexes-isomorphic} also follows from the fact that both are homotopy equivalent to $\ve{\CF}^-(S^3_{\Lambda}(L))$, though our argument gives a canonical chain isomorphism and shows the stronger statement in Theorem~\ref{thm:MO=our-version} concerning type-$D$ modules.

\subsection{Equivalence of groups}
\label{sec:equivalence-groups}

As a first step, we describe an isomorphism of groups $\bX_{\Lambda}(L)\iso \bY_{\Lambda}(L)$. Recall that both complexes are $\ell=|L|$ dimensional hypercubes of chain complexes. We write $\bX^\veps_{\Lambda}(L)$ and $\bY^\veps_{\Lambda}(L)$ for the chain complexes at a given $\veps\in \bE_\ell$.

\begin{lem}
\label{lem:complexes-isomorphic} For each $\veps\in \bE_\ell$,  there is an isomorphism of chain complexes over $\bF[U_1,\dots, U_\ell]$
\[
\bX^{\veps}_{\Lambda}(L)\iso \bY^{\veps}_{\Lambda}(L)
\]
which is canonical up to overall multiplication by $T^{\a_1}_{i_1}\cdots T^{\a_{k}}_{i_{k}}$ for $\a_{1},\dots, \a_{k}\in \Z$, where $i_{1},\dots, i_{k}$ are the indices for which $\veps_i=1$. 
\end{lem}
\begin{proof} 
We recall Manolescu and Ozsv\'{a}th's definition of $\bX_{\Lambda}^{\veps}(L)$. To define it, write $L_{\veps}$ for the sublink of $L$ consisting of components $K_i$ for $i$ such that $\veps_i=0$. They picked a Heegaard diagram for $(S^3,L_{\veps})$ which had $|\veps|$ additional free basepoints (i.e. basepoints which are not on $L_\veps$). Given $\ve{s}\in \bH(L_{\veps})$, they defined a chain complex $\frA(L_\veps,\ve{s})$ which was generated by monomials $\xs \cdot  U_1^{j_1}\cdots U_\ell^{j_\ell}$ satisfying
\begin{equation}
A_{i}^{\veps}(\xs)-j_i\le s_i, \label{eq:Alex-grading-well-defined}
\end{equation}
for each $i$ such that $\veps_i=0$. Here, $A_i^{\veps}(\xs)$ denotes the Alexander multi-grading of the intersection point $\xs$, viewed as an element of the link Floer homology of $L_\veps$.

 There is a related complex $\cCFL(L_{\veps})$, which is freely generated by intersection points on this diagram, multiplied by products of monomials in $\bF[\scU_i,\scV_i]$ (ranging over $i$ such that $\veps_i=0$) and $\bF[U_i]$ (ranging over $i$ such that $\veps_i=1$).  There is a natural inclusion
\[
\frA(L_{\veps},\ve{s})\hookrightarrow \cCFL(L_\veps)
\]
obtained by mapping $\xs \cdot U_1^{j_1}\cdots U_n^{j_n}$ to the product of $\xs$ with $\scU^{j_i}_i \scV^{s_i-A^{\veps}_{i}(\xs)+j_i}_i$ for $i$ such that $\veps_i=0$, as well as $U_i^{j_i}$ for $i$ such that $\veps_i=1$. Equation~\eqref{eq:Alex-grading-well-defined} implies that only positive powers of $\scV_i$ appear. Furthermore, the image of this inclusion is the subspace of $\cCFL(L_{\veps})$ in Alexander grading $\ve{s}$. Therefore, we may view
\[
\cCFL(L_\veps)\iso \bigoplus_{\ve{s}\in \bH(L_\veps)} \frA(L_{\veps}, \ve{s}). 
\]

We recall that Manolescu and Ozsv\'{a}th \cite{MOIntegerSurgery}*{Section~3.7} define an affine map
\[
\psi_{\vec{0}, \veps}\colon \bH(L)\to \bH(L_\veps),
\]
whose fiber is $\Z^{|\veps|}$. (In their notation, we are considering $\psi^{M}$, where $M=L\setminus L_{\veps}$, oriented consistently with $L$). This gives a non-canonical isomorphism $\bH(L)$ with $\Z^{|\veps|}\times \bH(L_{\veps})$. They define
\[
\bX_{\Lambda}^{\veps}(L)=\bigoplus_{\ve{s}\in \bH(L)} \frA(L_{\veps}, \psi_{\vec{0},\veps}(\ve{s})). 
\]
Therefore, we obtain a chain isomorphism
\[
\bX_{\Lambda}^{\veps}(L)\iso \cCFL(L_{\veps})\otimes \bF[\Z^{|\veps|}]
\] 
where $\bF[\Z^{|\veps|}]$ is generated by products of $T_i^j$, for $j\in \Z$ and $i$ such that $\veps_i=1$. Furthermore, the isomorphism is canonical up to powers of $T_i$ for $i$ such that $\veps_i=1$. On the other hand, a parallel analysis as above shows that this is the same as our definition of $\bY_{\Lambda}^{\veps}(L)$. 
\end{proof}

\begin{rem} The completions that Manolescu and Ozsv\'{a}th consider (working over the power series ring $\bF\llsquare U_1,\dots, U_\ell\rrsquare$ and taking the direct product over Alexander gradings) coincide with the completions induced by the chiral topologies on $\cK$ and ${}_{\cK} \cD_0$. Cf. Section~\ref{sec:different-topologies}.
\end{rem}

\subsection{Alexander gradings}
\label{sec:intro-Alexander-gradings}
There are a number of Alexander multi-gradings that we can assign to each $\bX_{\Lambda}^{\veps}(L)$ and $\bY_{\Lambda}^{\veps}(L)$. We recall that the Manolescu-Ozsv\'{a}th complex $\bX_{\Lambda}^{\veps}(L)$ is defined as a direct sum of complexes 
\[
\bX_{\Lambda}^{\veps}(L)=\bigoplus_{\ve{s}\in \bH(L)} \frA(L_{\veps}, \psi_{\vec{0},\veps}(\ve{s})).
\]
In particular, we $\bX_{\Lambda}^{\veps}(L)$ is graded by $\ve{s}\in \bH(L)$. We write $A$ for this grading.

We now observe that the complexes $\bY^{\veps}_{\Lambda}(L)$ admit a parallel $\Z^n$-valued grading $A'$, as we now define. For each knot component $K_i\subset S^3$, we pick a 2-chain $S_i^\veps$ on our Heegaard surface $\Sigma$, such that $\d S_i^\veps$ is a sum of $-K_i$ (the shadow of $K_i$ on $\Sigma$, oriented negatively) as well as a linear combination of small push-offs of the curves from $\as$ and $\bs_{\veps}$. This allows us to define a $\Z^n$-valued grading $A'$ on each $\bY^{\veps}_{\Lambda}(L)$. We declare, for each intersection point
\[
A_i'(\xs)=n_{S_i^\veps}(\xs).
\]
Additionally, we declare 
\begin{equation}
A_i'(\scU_j)=-\delta_{i,j}\quad A_i'(U_j)=0\quad \text{and} \quad A_i'(\scV_j)=A_i'(T_j)=\delta_{i,j},\label{eq:declare-gradings}
\end{equation}
where $\delta_{i,j}$ is the Kronecker delta function. Note that $A_i'$ is not independent of the choice of $S_i^\veps$, and hence it is more natural to view $A_i'$ as determining a relative grading on each $\bY^{\veps}_{\Lambda}(L)$. The reader may compare Hedden and Levine's formula for the Alexander grading \cite{HeddenLevineSurgery}*{Section~2.3}.

It is also helpful to define 2-chains $\hat{S}_i$ by capping the boundary components of $S_i^\veps$ which lie along the alpha and beta circles. The resulting 2-chains do not depend on $\veps$ (up to homology) and also satisfy $\d \hat{S}_i=-K_i$.

The following lemma concerns the gradings of the hypercube structure maps $\Psi^{\vec{M}}$ on $\bY_{\Lambda}(L)$:

\begin{lem}\label{lem:grading-descent-maps}\, Suppose that $L\subset S^3$ is a link with framing $\Lambda$. 
\begin{enumerate}
\item If $\vec{M}\subset L$ is an oriented sublink, then $\Psi^{\vec{M}}$  is homogeneously graded with respect to $A'$.
\item If $\vec{M}$ and $\vec{M}'$ are two  different orientations of a sublink $M\subset L$, then
\[
A_i'\left(\Psi^{\vec{M}}\right)-A_i'\left(\Psi^{\vec{M}'}\right)=-\frac{\lk(\vec{M}, K_i)-\lk(\vec{M}',K_i) }{2}.
\]
In the above, if $K_i\subset M$, then $\lk(\vec{M}, K_i)$ and $\lk(\vec{M}',K_i)$ are computed by pushing $M$ slightly off of itself in the direction of the framing $\Lambda$.
\end{enumerate}
\end{lem}

\begin{rem}By construction, Manolescu and Ozsv\'{a}th's link surgery maps $\Phi^{\vec{M}}$ satisfy the statement of Lemma~\ref{lem:grading-descent-maps} with respect to the Alexander grading $A$. 
\end{rem}

Before proving Lemma~\ref{lem:grading-descent-maps}, we recall some standard constructions. Given a Heegaard diagram $(\Sigma,\as,\bs)$ for $Y$, we may pick compressing disks for $\as$ and $\bs$ in $Y$. We also pick identifications of these disks with a unit complex disk. Given an intersection point $\xs\in \bT_{\a}\cap \bT_{\b}$, we may form a 1-chain $\g_{\xs}\subset Y$ by coning the points of $\xs$ radially into the compressing disks for the $\as$ and $\bs$ curves. Similarly, given a class of disks $\phi\in \pi_2(\xs,\ys)$, we may cone the domain $D(\phi)$ along these compressing disks to get an integral 2-chain $\hat{D}(\phi)$ which satisfies
\[
\d \hat{D}(\phi)=\g_{\ys}-\g_{\xs}.
\]
See Figure~\ref{fig:28}.
\begin{figure}[h]
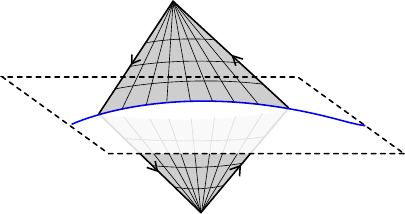
\caption{The cone $\hat{D}(\phi)$ (shaded gray) of a disk class $\phi$ from $\xs$ to $\ys$. }
\label{fig:28}
\end{figure}

The coning construction extends to classes of polygons counted on a tuple 
\[
(\Sigma,\as_m,\dots, \as_1, \bs_1,\dots, \bs_n)
\]
whenever $\Sigma\subset Y$ is a Heegaard surface such that all $\as_i$ bound compressing disks in a component $U_\a\subset Y\setminus \Sigma$, and similarly for all $\bs_j$.  For this construction, we pick compressing disks for each $\as_i$ and $\bs_j$, and we cone $\d_{\a_i}(\psi)$ into the alpha-handlebody of $Y$, and we cone each $\d_{\b_j}(\psi)$ into the beta handlebody for each $j$.

\begin{proof}[Proof of Lemma~\ref{lem:grading-descent-maps}]
We consider the claim that the maps are homogeneously graded with respect to $A'_i$.  Consider one of the Heegaard tuples $(\as,\bs_{\veps_0},\dots, \bs_{\veps_{m}})$ used to compute $\Psi^{\vec{M}}$. Here, $\veps_0< \cdots< \veps_m$ where $|\veps_{i+1}-\veps_i|_{L^1}=1$, and $m=|M|$. 
Suppose
\[
\psi\in \pi_2(\xs,\theta_1,\dots, \theta_m,\ys)\quad \text{and} \quad \psi'\in \pi_2(\xs',\theta_1',\dots, \theta_m',\ys'),
\]
are two classes of polygons on this Heegaard diagram which contribute to $\Psi^{\vec{M}}$. We observe that
\begin{equation}
\d \hat{D}(\psi)=\g_{\ys}-\g_{\theta_m}-\cdots -\g_{\theta_1}-\g_{\xs}.\label{eq:boundary-polygon}
\end{equation}

We observe firstly that the Alexander grading change associated to the algebra weights for the polygon classes $\psi$ and $\psi'$ are 
\[
\#(\d_{\b_{\veps_0}+\cdots +\b_{\veps_{m}}}D(\psi)\cap K_i)\quad \text{and} \quad \#(\d_{\b_{\veps_0}+\cdots +\b_{\veps_{m}}}D(\psi')\cap K_i)
\]
respectively. In the above, we are viewing $D(\psi)$, $D(\psi')$, and $K_i$ as being simplicial chains on $\Sigma$. (These are the contributions from the $\rho$ maps from our local systems).

Note that we may identify
\[
\#(\d_{\b_{\veps_0}+\cdots +\b_{\veps_{m}}}D(\psi)\cap K_i)=-\#(\hat{D}(\psi)\cap \d \hat{S}_i)=-\# (\d  \hat{D}(\psi) \cap \d \hat{S}_i),
\]
and similarly for $\psi'$. This is because $\d \hat{S}_i=-K_i$, pushed slightly into the $U_{\beta}$ handlebody. 

Abusing notation slightly, write $A'_i(\Psi^{\vec{M}}(\xs))$ for the $A'_i$ grading of the $\ys$ term of the output of $\Psi^{\vec{M}}(\xs)$ contributed by $\psi$. Therefore
\[
A'_i(\Psi^{\vec{M}}(\xs))=n_{S_i^{\veps_{m}}}(\ys)-\# \d \hat{D}(\psi)\cap \hat{S}_i.
\]

We therefore have
\begin{equation}
\begin{split}
&\left(
A_i'(\Psi^{\vec{M}}(\xs))-A'_i(\xs)
\right)-\left(
A_i'(\Psi^{\vec{M}}(\xs'))-A'_i(\xs')\right)\\
=&\#(\g_{\ys}-\g_{\ys'})\cap S_i^{\veps_{m}} -\#(\g_{\ys}-\g_{\ys'})\cap \hat{S}_i -\# (\g_{\xs}-\g_{\xs'})\cap S_i^{\veps_0}+\# (\g_{\xs}-\g_{\xs'})\cap \hat{S}_i
\\
&+\sum_{j=1}^m \# (\g_{\theta_j}-\g_{\theta_j'})\cap \hat{S}_i\\
&=\sum_{j=1}^m \# (\g_{\theta_j}-\g_{\theta_j'})\cap \hat{S}_i
\end{split}
\label{eq:Alexander-grading-change}
\end{equation}

If $\vec{M}=\vec{M}'$, then $\theta_j=\theta_j'$ for all $j$, so the above equation vanishes. Hence each $\Psi^{\vec{M}}$ is homogeneously graded.

We observe via the model computation in Figure~\ref{fig:24} that $\g_{\theta_j^\sigma}-\g_{\theta_j^\tau}$ is isotopic to $+K_j$. Furthermore, this isotopy may be supported in the complement of each $\d \hat{S}_i$, which is a copy of $-K_i$, pushed very slightly into $U_{\b}$. For the case of a link, we observe that
\[
\g_{\theta_j}-\g_{\theta_j'}=\begin{cases} 0& \text{ if } \theta_j=\theta_j'\\
K_j& \text{ if } K_j\in \vec{M}\text{ and } K_j\in \vec{M}'\\
-K_j& \text{ if } -K_j\in \vec{M}\text{ and } -K_j\in \vec{M}'. 
\end{cases}
\]
In particular, since $\d \hat{S}_i=-K_i$, Equation~\eqref{eq:Alexander-grading-change} reduces to the linking number description in the statement.
\end{proof}

\begin{figure}[ht]
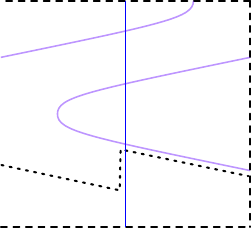
\caption{The class $\g_{\theta^\sigma}-\g_{\theta^\tau}$ (dotted line), which is homologous to $K$.}
\label{fig:24}
\end{figure}

There are several additional Alexander gradings that we can define on the link surgery complexes. We may additionally define a $\Q^{|L|-|\veps|}$-valued grading $A^{\veps}$ on $\bY_{\Lambda}^{\veps}(L)$ and $\bX_{\Lambda}^{\veps}(L)$ by pulling back the ordinary $\Q^{|L_{\veps}|}$-valued Alexander grading under the map
\[
\bX_\Lambda^\veps(L)\to \bX_\Lambda^\veps(L)/\bF[\Z^{|\veps|}]\iso \cCFL(L_\veps),
\]
where we view $\bF[\Z^{|\veps|}]$ as being generated in products of $T_i^{j}$, $j\in \Z$, ranging over $i$ such that $\veps_i=1$.  (Recall $L_{\veps}$ denotes the sublink of $L$ consisting of components $K_i$ for which $\veps_i=0$). With respect to this grading, $T_i$ to has $A^\veps$ grading 0 if $\veps_i=1$. 

\begin{lem}\label{lem:relative-grading} For each $i$ such that $\veps_i=0$, the gradings $A'_i$ and $A^{\veps}_{i}$ coincide up to an overall constant.
\end{lem}
\begin{proof} For simplicity, consider $\veps=\vec{0}$.   By definition $A_{i}(\xs)-A_i(\ys)=n_{z_i}(\phi)-n_{w_i}(\phi)$, for any $\phi\in \pi_2(\xs,\ys)$.
Hence
\[
A_i(\ys)-A_i(\xs)=n_{w_i}(\phi)-n_{z_i}(\phi)= - (\hat{D}(\phi)\cap K)=\#(\hat{D}(\phi)\cap \d 
\hat{S}_i),
\]
since $\d \hat{S}_i=-K$, by our convention above, and $K$ is oriented to intersect $\Sigma$ positively at $z_i$ and negatively at $w_i$. On the other hand, $\d \hat{D}(\phi)=\g_{\ys}-\g_{\xs}$, so
\[
\#(\hat{D}(\phi) \cap \d \hat S_i)=\#( \d\hat{D}(\phi)\cap \hat{S}_i)=\# (\g_{\ys}-\g_{\xs})\cap \hat S_i=n_{S_i^\veps}(\ys)-n_{S_i^\veps}(\xs)=A'_i(\ys)-A'_i(\xs),
\]
completing the proof.
\end{proof}

Manolescu and Ozsv\'{a}th define a reduction map $\psi^{\vec{M}}\colon \bH(L_{\veps})\to \bH(L_{\veps'})$  \cite{MOIntegerSurgery}*{Section~3.7}. For our purposes, we only need to consider the map when $\vec{M}$ is oriented consistently with $L$, and we write
 \[
 \psi_{\veps,\veps'}\colon \bH(L_{\veps})\to \bH(L_{\veps'})
 \]
 for this map. 
 
 By definition, the Alexander grading $A$ on $\bX_{\Lambda}^{\veps}(L)$ has the property that
\[
\psi_{\vec{0},\veps}\circ A=A^\veps.
\]
By Lemma~\ref{lem:relative-grading}, we may normalize $A'$ on $\bY_{\Lambda}^\veps(L)$ by declaring
\[
\psi_{\vec{0},\veps}\circ A'=A^{\veps}.
\]
Note that the grading $A'$ still has an indeterminacy on $\bY^{\veps}_{\Lambda}(L)$ by $\Q^{|L|-|L_\veps|}$ on the factors correspond to $L\setminus L_\veps$. We further normalize $A'$ on $\bY^{\veps}_{\Lambda}(L)$ so that the maps $\Psi^{\vec{M}}_{\veps,\veps'}$ preserve $A'$ whenever $\vec{M}$ is a positively oriented sublink of $L$.

\subsection{Completion of the proof of Theorem~\ref{thm:MO=our-version}}

In this section, we complete the proof of Theorem~\ref{thm:MO=our-version}.

\begin{lem}
\label{lem:hypercube-differentials-isomorphic} For each oriented sublink $\vec{M}\subset L$ and each $\veps,\veps'\in \bE_\ell$ such that $\veps<\veps'$, the hypercube differentials $\Phi_{\veps,\veps'}^{\vec{M}}$ and $\Psi_{\veps,\veps'}^{\vec{M}}$ (from $\bX_{\Lambda}(L)$ and $\bY_{\Lambda}(L)$, respectively) are equal up to overall multiplication by some $T_{i_1}^{\a_{i_1}}\cdots T_{i_j}^{\a_{i_j}}$, where each $\a_{i_k}\in \Z$, and $i_1,\dots, i_j$ denote the indexes $i$ such that $\veps'_i=1$. 
\end{lem}
\begin{proof}We observe firstly that the diagrams $(\Sigma,\as,\bs_{\veps},\ws,\zs)$ satisfy the definition of a complete system of Heegaard diagrams \cite{MOIntegerSurgery}*{Section~8.6}. See \cite{MOIntegerSurgery}*{Section~9} for more details about the construction of the maps $\Phi^{\vec{M}}_{\veps,\veps'}$. 

For the purposes of the lemma statement, we first compose the maps $\Phi_{\veps,\veps'}^{\vec{M}}$ and $\Psi^{\vec{M}}_{\veps,\veps'}$ 
with the quotient map $Q\colon Y^{\veps'}_{\Lambda}(L)\to Y^{\veps'}_{\Lambda}(L)/(T_i^\a: \veps'_i=1)$. 

We need the following facts about their map $Q\circ \Phi^{\vec{M}}$, which are immediate from the construction
\begin{enumerate}
\item If $\veps'_i=0$, then $\Phi^{\vec{M}}$ and $\Psi^{\vec{M}}$ will be $\scU_i$ and $\scV_i$ equivariant. The $\scU_i$ and $\scV_i$ powers of the two maps are given by $n_{w_i}(\psi)$ and $n_{z_i}(\psi)$, if $\psi$ is the class of a holomorphic polygon counted by the two maps.
\item If $+K_i\subset L$, then the maps $Q\circ \Phi^{\vec{M}}_{\veps,\veps'}$ and $Q\circ \Psi^{\vec{M}}_{\veps,\veps'}$ will intertwine $\scU_i$ with $U_i$. These maps intertwine $\scV_i$ with $1$. The $U_i$ power of these maps will be $n_{w_i}(\psi)$. 
\item If $-K_i\subset L$, the maps intertwine $\scV_i$ with $U_i$, and intertwine $\scU_i$ with $1$. The $U_i$ power of the summand of both maps corresponding to a polygon $\psi$ is $n_{z_i}(\psi)$.
\end{enumerate}
The above imply, in particular, that
\[
Q\circ \Phi^{\vec{M}}_{\veps,\veps'}=Q\circ \Psi^{\vec{M}}_{\veps,\veps'}.
\]
To see that the two maps differ by multiplication by a single monomial $T_{i_1}^{\a_1}\cdots T_{i_k}^{\a_k}$, we recall that $\Phi^{\vec{M}}_{\veps,\veps'}$ is homogeneously graded by the $A$-grading, and $\Psi^{\vec{M}}_{\veps,\veps'}$ is homogeneously graded by the $A'$-grading. Furthermore, by Lemma~\ref{lem:grading-descent-maps}, they have the same grading shifts, so differ by a monomial as stated.
\end{proof}

\begin{proof}[Proof of Theorem~\ref{thm:MO=our-version}]
Lemma~\ref{lem:complexes-isomorphic} gives an isomorphism 
\[
F_{\veps} \colon \bX_{\Lambda}^{\veps}(L)\to \bY_{\Lambda}^{\veps}(L),
\]
which is well-defined up to multiplication by overall powers of $T_i^{\a_i}$, $\a_i\in \Z$, ranging over $i$ such that $\veps_i=1$. The maps $F_{\veps}$ may be normalized by requiring that they intertwine the $A$ and $A'$ Alexander gradings. We claim that with respect to this convention, the map $F=\sum_{\veps\in \bE_\ell} F_{\veps}$ is a chain isomorphism from $\bX_{\Lambda}^{\veps}(L)$ to $\bY_{\Lambda}^{\veps}(L)$. 

By Lemma~\ref{lem:hypercube-differentials-isomorphic}, we have that $F_{\veps'}\circ \Phi_{\veps,\veps'}^{\vec{M}}$
and $\Psi_{\veps,\veps'}^{\vec{M}}\circ F_{\veps}$ are equal up to overall powers of $T_i^{\a_i}$ ranging over $i$ such that $\veps'_i=1$. By Lemma~\ref{lem:grading-descent-maps}, we see that the two maps $F_{\veps'}\circ \Phi_{\veps,\veps'}^{\vec{M}}$
and $\Psi_{\veps,\veps'}^{\vec{M}}\circ F_{\veps}$ have the same grading shift with respect to $A$ and $A'$, and hence must be equal on the nose. The proof is complete. 
\end{proof}

\section{Endomorphism algebras}
\label{sec:endomorphisms}

 Auroux \cite{AurouxBordered} gave a description of the bordered strands algebra of Lipshitz, Ozsv\'{a}th and Thurston  as the endomorphism algebra of a twisted complex in the partially wrapped Fukaya category of a punctured surface. In this section we prove an analog for the knot surgery algebra. We will show that our knot surgery algebra $\cK$ is $A_\infty$-homotopy equivalent to a subalgebra of the endomorphism algebra of 
$\b_0^{E_0}\oplus \b_1^{E_1}.$

In Section~\ref{sec:filtered_endomorphisms} we define a subspace of \emph{filtered endomorphisms} 
\[
\End_{\Fil}(\b_0^{E_0}\oplus \b_1^{E_1})\subset\End(\b_0^{E_0}\oplus \b_1^{E_1}),
\] which is an $A_\infty$-algebra. Note that we may define versions of $\End_{\Fil}(\b_0^{E_0}\oplus \b_1^{E_1})$ using either chiral or $U$-adic completions.

Our main theorem is the following:

\begin{thm}
\label{thm:equivalence-algebras} If we equip $\b_0^{E_0}\oplus \b_1^{E_1}$ with the chiral topology, then there is a homotopy equivalence of topological chiral $A_\infty$-algebras
\[
\End_{\Fil}(\b_0^{E_0}\oplus \b_1^{E_1})^{\opp}\simeq \cK.
\] If we equip morphism spaces with the $U$-adic topology, then there is a homotopy equivalence of achiral topological algebras $\End_{\Fil}(\b_0^{E_0}\oplus \b_1^{E_1})^{\opp}\simeq \frK$. 
\end{thm}

\subsection{Filtered endomorphisms}
\label{sec:filtered_endomorphisms}

In this section, we define  $\End_{\Fil}(\b_0^{E_0}\oplus \b_1^{E_1})$.

\begin{define}
\label{def:filtered-morphism} Fix either the chiral or $U$-adic topology on $E_0\oplus E_1$ (see Section~\ref{sec:operator-topologies}). If $\veps,\nu\in \{0,1\}$, we say a morphism $\lb\xs,f\rb\colon \b_{\veps}^{E_{\veps}}\to \b_{\nu}^{E_{\nu}}$ is \emph{filtered} if the following hold:
\begin{enumerate}
\item $\veps\le \nu$.
\item $f$ is a continuous $\bF[U]$-linear map from $E_\veps$ to $E_{\nu}$.
\item If $\veps=\nu=0$, then the map $f$ is non-increasing in the $\gr_w$ and $\gr_z$-gradings. 
\end{enumerate}
\end{define}

If $\veps,\nu\in \{0,1\}$, we write $\ve{\CF}_{\Fil}^-(\b_{\veps}^{E_{\veps}},\b_{\nu}^{E_{\nu}})$ for the Floer complex of filtered morphisms.  Similarly, we write $\End_{\Fil}(\b_0^{E_0}\oplus \b_1^{E_1})$ for the space of all filtered endomorphisms.

We equip each $\Hom_{\bF[U]}(E_\veps,E_{\veps'})$ with the uniform topology (cf. Section~\ref{sec:operator-topologies}), which endows $\End_{\Fil}(\b_0^{E_0}\oplus \b_1^{E_1})$ itself with a linear topology (depending on which topology we choose for $E_0\oplus E_1$).

We now prove an algebraic lemma about filtered morphisms, which appears several times in our paper:

\begin{lem}\label{lem:filtered-morphism-compose-Vs} Suppose that $f\colon E_0\to E_0$ is a filtered $\bF[U]$-module morphism. Then $\scV\circ f\circ \scV^{-1}$ and $\scV^{-1}\circ f\circ \scV$ are both filtered morphisms.
\end{lem}
\begin{proof} A morphism $f\colon E_0\to E_0$ is filtered if and only if it preserves the principal ideal $(\scU^{i}\scV^j)$ for all $i,j \ge 0$. Note that if $f$ is a filtered map, then it also preserves the $\bF[\scU,\scV]$-submodule $(\scU^i\scV^j)\subset U^{-1} \bF[\scU,\scV]$ for all $i,j\in \Z$. 

We observe that if $i,j\ge 0$, then $\scV^{-1}$ maps $(\scU^i\scV^j)$ to $(\scU^i\scV^{j-1})$, which $f$ maps to $(\scU^i\scV^{j-1})$. This is mapped by $\scV$ to $(\scU^i \scV^j)$, so $\scV\circ f\circ \scV^{-1}$ is filtered. The same argument works for $\scV^{-1} \circ f\circ \scV$. 
\end{proof}

\subsection{Gradings and \texorpdfstring{$\Spin^c$}{Spin-c} structures}

We now discuss gradings and $\Spin^c$ structures, focusing on the Heegaard diagram $(\bT^2,\b_0,\b_1,w,z)$. This diagram $(\bT^2,\b_0,\b_1,w,z)$ represents the homologically essential knot $S^1\times pt\subset S^1\times S^2$. 

We say that $\xs$ and $\ys$ are \emph{$\Spin^c$-equivalent} if there is a homotopy class of disks from $\xs$ to $\ys$. On the diagram $(\bT^2, \b_0,\b_1,w,z)$, there are two $\Spin^c$ equivalence classes, consisting of $\theta_\sigma^{\pm}$ and $\theta_{\tau}^{\pm}$. We write $\frs_\sigma$ and $\frs_\tau$ for these two classes.

For a generator $\lb\xs, \phi\rb$ in any $\Spin^c$ structure on $(\bT^2, \b_0, \b_1)$, we define the Alexander grading of a generator by the formula
\[
A(\lb\xs,\phi\rb)=A(\phi),
\]
where $A(\phi)$ denotes the Alexander grading shift of $\phi$, viewed as a map from $\bF[\scU,\scV]$ to $\bF[U,T,T^{-1}]$. (Recall that we set $A(\scU)=-1$, $A(\scV)=1$, $A(T)=1$ and $A(U)=0$).

If $\lb \xs, \phi\rb$ 
 represents $\Spin^c$ structure $\frs_\sigma$, we define
\[
\gr(\lb\xs,\phi\rb)=\gr_w(\xs)+\gr_w(\phi),
\]
where $\gr_w(\phi)$ is the grading shift of $\phi$, using the definition that $\gr_w(\scU)=\gr_w(U)=-2$ and $\gr_w(\scV)=\gr_w(T)=0$.
Similarly, for generators $\lb\xs,\phi\rb$ representing $\frs_\tau$, we define
\[
\gr(\lb\xs,\phi\rb)=\gr_z(\xs)+\gr_z(\phi). 
\]

It is straightforward to see that the differential preserves the Alexander grading $A$, and decreases the Maslov grading by $1$. 

\begin{rem} For our purposes, it is not as natural to identify $\Spin^c$ equivalence classes with elements of $\Spin^c(S^1\times S^2)$.   We recall that Ozsv\'{a}th and Szab\'{o} define two natural maps
\[
\frs_{w},\frs_z\colon \bT_{\a}\cap \bT_{\b}\to \Spin^c(Y).
\]
These satisfy
\begin{equation}
\frs_w(\xs)-\frs_z(\xs)=\PD[K],\label{eq:diff-spin-c-structure}
\end{equation}
and hence give an asymmetry for any identification of $\Spin^c$ classes with $\Spin^c(S^1\times S^2)$. 
\end{rem}

\subsection{Admissibility and positivity}

Before proving Theorem~\ref{thm:equivalence-algebras}, we prove that the endomorphism algebras $\End_{\Fil}(\b_0^{E_0}\oplus \b_1^{E_1})$, as defined above, are indeed linear topological $A_\infty$-algebras for either choice of topology. In particular, the composition maps $\mu_i$ are continuous, and do not introduce any negative powers of $U$. 

 We recall the following admissibility condition from \cite{OSDisks}:

\begin{define}  If $(\Sigma,\gs_0,\dots, \gs_n,w)$ is a Heegaard diagram, a \emph{periodic domain} $P$ is an integral 2-chain such that $\d P$ is a linear combination of the curves $\gs_i$ and $n_w(P)=0$. A diagram is called \emph{weakly admissible} with respect to $w$ if every non-zero periodic domain has both positive and negative multiplicities.
\end{define}

Weak admissibility is useful since it implies that for each $N>0$, there are only finitely many polygon classes $\psi\in \pi_2(\xs_1,\dots, \xs_n,\ys)$ with nonnegative domain such that $n_w(\psi)\le N$. Compare \cite{OSDisks}*{Section~4.2.2}. 

Our main result is the following:

\begin{prop}
\label{prop:finite-ness-endomorphisms} Suppose that $(\bT^2,\b_{\veps_0},\dots, \b_{\veps_n},w,z)$ is a Heegaard tuple where each $\b_{\veps_i}$ is a copy of a small translate of $\b_0$ or $\b_1$. Suppose that the diagram is also weakly admissible at $w$ and weakly admissible at $z$.
\begin{enumerate}
\item If $\lb \xs_1,f_1\rb,\dots, \lb \xs_n, f_n\rb$ are are filtered in the sense of Definition~\ref{def:filtered-morphism}, then
\[
\mu_n(\lb \xs_1,f_1\rb,\dots ,\lb \xs_n, f_n\rb)
\]
involves a finite count of holomorphic curves. Furthermore, the output is also a filtered morphism. (In particular, no negative powers of $U$ appear).  
 \item More generally, the map
\[
\mu_n\colon \ve{\CF}^-_{\Fil}\left(\b_{\veps_0}^{E_{\veps_0}},\b_{\veps_1}^{E_{\veps_1}}\right)\otimes \cdots \otimes \ve{\CF}^-_{\Fil}\left(\b_{\veps_{n-1}}^{E_{\veps_{n-1}}},\b_{\veps_n}^{E_{\veps_n}}\right)\to \ve{\CF}^-_{\Fil}\left(\b_{\veps_0}^{E_{\veps_0}}, \b_{\veps_n}^{E_{\veps_n}}\right).
\]
is continuous with respect to both the chiral and $U$-adic topologies. For the chiral topology, we use the chiral tensor product $\cevotimes$ in the statement (note direction of arrow), and for the $U$-adic topology, we use the $\otimes^!$ tensor product.
\end{enumerate}
\end{prop}
The statement about positivity (i.e. no negative $U$-powers) will be proven in Section~\ref{sec:positivity-endomorphism-algebras} after building up some basic related results. We now prove the remaining claims:

\begin{proof}
 The first statement about finiteness follows from the fact that if $\lb\xs_1,f_1\rb,\dots, \lb \xs_n,f_n\rb$ is a composable sequence of filtered intersection morphisms, then at least one $p\in \{w,z\}$ has the property that $\frs_p(\xs_i)$ is the torsion $\Spin^c$ structure  $\frs_0\in \Spin^c(S^1\times S^2)$ for all $i$. 
(The choice of $p$ is determined by whether $\theta_\sigma^{\pm}$ or $\theta_\tau^{\pm}$ is present in the sequence). In particular, if $\psi\in \pi_2(\xs_1,\dots, \xs_n, \ys)$ is a class of polygons, then 
\[
\gr_{p}(\ys)=\gr_p(\xs_1)+\cdots+\gr_p(\xs_n)+2n_p(\psi). 
\]
There are only finitely many intersection points $\ys$ on the diagram. Given $\xs_1,\dots, \xs_n$, there are therefore only finitely values of $n_p(\psi)$ which are possible. Therefore by weak admissibility, there are only finitely many nonnegative classes of polygons which have $\xs_1,\dots, \xs_n$ as inputs. Therefore, $\mu_n(\lb \xs_1,f_1\rb,\dots, \lb \xs_n, f_n\rb)$ involves a finite sum of holomorphic curves. 

The second claim follows from the analogous fact for the continuity of composition of $\Hom$ spaces from Remark~\ref{rem:continuity-of-composition}, as well as the finiteness of holomorphic curve counts.
\end{proof}

\subsection{Proof of Theorem~\ref{thm:equivalence-algebras}}

  We now consider $\End_{\Fil} (\b_0^{E_0}\oplus \b_1^{E_1})$ as a chain complex. We will, in particular, compute the homology of its completion. Subsequently we partially compute the compositions $\mu_j$ for $j\ge 2$.

\begin{lem}\label{lem:homology-I0-I0} Fix either the chiral or $U$-adic topology on $E_0$.
For an appropriate Hamiltonian translation $\b_0'$ of $\b_0$ (intersecting in two points), we have the following:
\begin{enumerate}
\item  The complex $\ve{\CF}_{\Fil}^-(\b_0^{E_0},\b_0'^{E_0})$  is spanned by pairs $\lb \theta^+,f\rb$ and $\lb \theta^-,f\rb$, where $f\colon E_0\to E_0$ is an $\bF[U]$-linear map which is continuous and filtered in the sense of Definition~\ref{def:filtered-morphism}. Furthermore,
\[
\d\lb\theta^+,f\rb=\lb\theta^-,f\rb+\lb\theta^-, \scV\circ f\circ \scV^{-1}\rb. 
\]
\item The space $E_0$ is a strong deformation retract of $\ve{\CF}^-(\b_0^{E_0}, \b_0'^{E_0})$. 
\end{enumerate}
\end{lem}
\begin{proof}
 The computation of the differential is straightforward and involves only counting bigons. (We can use the diagrams from Figure~\ref{fig:57} below). 
 We focus now on the claim about the deformation retraction. Firstly, define
\[
\delta_{> 0}\colon \bF[\scU,\scV]\to \bF[\scU,\scV]
\]
to be projection onto positive Alexander gradings, and let $\delta_{\le 0}$ denote projection onto nonpositive Alexander gradings. We define our deformation retraction
\[
\begin{tikzcd}\ar[loop left,looseness=3, "H"]\ve{\CF}^-\left(\b_0^{E_0},\b_0'^{E_0}\right)\ar[r,shift left, "\Pi"] &\bF\llsquare \scU,\scV \rrsquare \ar[l, shift left, "I"] 
\end{tikzcd}
\]
via the equations:
\[
\begin{split} I(\scU^i \scV^j)&=\lbsm \theta^+, \scU^i \scV^j\cdot \id\rbsm\\
\Pi(\lbsm \theta^+, f\rbsm)&=f(1)\\
\Pi(\lbsm \theta^-, f\rbsm)&=0\\
H(\lbsm \theta^+, f\rbsm)&=0\\
H(\lbsm \theta^-, f\rbsm)&=\lbsm \theta^+,h(f)\rbsm
\end{split}
\]
 where 
\[
h(f)=\sum_{n\ge 0} \scV^n \circ f\circ \delta_{> 0}\circ \scV^{-n}+\sum_{n\ge 1} \scV^{-n}\circ   f\circ \delta_{\le 0}\circ \scV^n.
\]

We claim that $h(f)$ is continuous map from $E_0$ to $E_0$, and that $h$ is itself continuous as a map from $\Hom_{\Fil}(E_0,E_0)$ to itself. To establish these claims, we observe the following facts:
\begin{enumerate}[label=($h$-\arabic*), ref=$h$-\arabic*]
\item\label{h-1} The map $\scV^n \circ f\circ \delta_{> 0} \circ \scV^{-n}$ vanishes on Alexander gradings below $n$, and the map $\scV^{-n} \circ f\circ \delta_{\le 0} \circ \scV^n$ vanishes on Alexander gradings above $-n$.
\item\label{h-2} If $f$ is filtered, then $h(f)$ is also filtered.
\item\label{h-3} Let $I_m$ be the ideal $(\scU^m,\scV^m)\subset \bF[\scU,\scV]$. If $f$ is filtered and has image in $ I_m$, then $\scV^n \circ f \circ \delta_{> 0} \circ \scV^{-n}$ and $\scV^{-n}\circ f\circ \delta_{\le 0} \circ \scV^n$ both have image in $I_{\max(n,m)}$. 
\item\label{h-4} $\scV\circ  h(f)\circ \scV^{-1}+h(f)=f$ and $h(\scV \circ f\circ \scV^{-1}+f)+f(1)\cdot \id=f$.
\end{enumerate}
Claim \eqref{h-1} is straightforward. Claim ~\eqref{h-2} follows from Lemma~\ref{lem:filtered-morphism-compose-Vs}. Claim~\eqref{h-3} is proven as follows. By ~\eqref{h-1}, we know that $\scV^n\circ f \circ \delta_{>0} \circ \scV^{-n}$ is non-vanishing only on Alexander gradings $n$ or higher. These are spanned by monomials $\scU^i\scV^j$ with $i\ge 0$ and $j\ge i+n$. On this subspace $\delta_{>0} \circ \scV^{-n}$ has image in the span of $\scU^i\scV^j$ with $i,j\ge 0$, so $f\circ \delta_{>0} \circ \scV^{-n}$ has image in $I_m$. We observe that $\scV^n$ applied to $I_m$ has image in both $I_{m}$ and $I_n$, so the composition has image in $I_{\max(m,n)}$. For the other map, we observe that 
\[
\scV^{-n}\circ f \circ \delta_{\le 0}\circ \scV^n=\scU^{n}\circ f\circ \delta_{\le 0} \circ \scU^{-n},
\]
 so an identical argument as above may be used for this summand.
Claim~\eqref{h-4} is a direct computation, which we leave to the reader.

 Note that \eqref{h-2}  implies that $h(f)$ has image in $\bF[\scU,\scV]$ (no negative powers of $U$ appear). Note that \eqref{h-3} implies that the infinite sum defining $h(f)$ is a convergent series of continuous maps, which implies the limit is a continuous map. (Recall that if $X$ and $Y$ are linear topological vector spaces and $F_i\colon X\to Y$ is a sequence of continuous maps, then $\sum_{i\ge 0}F_i$ converges to a continuous map (in the uniform topology) if and only if $F_i\to 0$ in the uniform topology). Finally, \eqref{h-3} implies that $h$ is continuous as an endomorphism of $\Hom_{\Fil}(E_0,E_0)$ because it preserves $\Hom_{\Fil}(E_0, I_m)$. 

It is straightforward to verify that $\d (\Pi)=0$, $\d (I)=0$, and $\Pi\circ I=\id_{E_0}$. We observe additionally that \eqref{h-4} implies that
\[
I\circ \Pi+\id=\d(H).
\]

 Next, we consider the claim for the $U$-adic topology. In this case, the argument is nearly identical to the above, except \eqref{h-3} is not relevant, and we do not claim that the sum defining $h$ is a convergent sum. Rather, we claim that evaluated on any element $x\in E_0=\bF[\scU,\scV]$, only finitely many of the terms of the sum defining $h$ are non-vanishing. This follows from \eqref{h-1}. Continuity of the map $h$ is automatic in the $U$-adic topology from its $U$-equivariance, and hence the rest of the claims follow easily.
\end{proof}

\begin{rem} The homotopy equivalence in Lemma~\ref{lem:homology-I0-I0} can be visualized somewhat more easily if we consider the subspace $\Hom_{\Fil}(E_0,E_0)$ of maps which have homogeneous Alexander grading $s$, for some fixed $s$. For concreteness, consider $s=0$. In this case, an arbitrary map may be written as $\sum_{i\in \Z} \a_i U^{n_i} \delta_i$, where $\a_i\in \bF$, $n_i\ge 0$ and $\delta_i$ denotes projection onto Alexander grading $i$. Note that $\scV \circ \delta_i\circ \scV^{-1}= \delta_{i+1}$. We may view this subcomplex of $\Hom_{\Fil}(E_0,E_0)$ as the infinite staircase
\[
\begin{tikzcd}\cdots \ar[dr]& \delta_{-2}\ar[d]\ar[dr]& \delta_{-1} \ar[d]\ar[dr]& \delta_0  \ar[d]\ar[dr]& \delta_1 \ar[d]\ar[dr]& \delta_2 \ar[d]\ar[dr]&\cdots \\
\cdots&\delta_{-2}& \delta_{-1}& \delta_0& \delta_1& \delta_2&\cdots
\end{tikzcd}
\]
The top row is the span of $\theta^+$ and the bottom row is the span of $\theta^-$.
Arbitrary elements of $f\in \Hom_{\Fil}(E_0,E_0)$ which preserve Alexander grading may be thought of as infinite $\bF\llsquare U\rrsquare$ linear combinations of the $\delta_i$. The homology of this chain complex is spanned by the generator $\sum_{i\in \Z}\lb\theta^+, \delta_i\rb=\lb \theta^+,\id\rb$.  The maps $\Pi$, $I$ and $H$ restrict to give a deformation retract of the above staircase complex onto the $\bF\llsquare U\rrsquare$-span of $1\in \ve{I}_0\cdot \cK \cdot \ve{I}_0$.  
\end{rem}

Next we consider $\ve{\CF}_{\Fil}^-(\b_0^{E_0},\b_1^{E_1})$:

\begin{lem}\label{lem:homology-I1-I0}
\begin{enumerate}
\item   The complex $\ve{\CF}_{\Fil}^-(\b_0^{E_0},\b_1^{E_1},\frs_\sigma)$ is spanned by $\lb \theta^{\pm}_\sigma,f\rb$ where $f\colon E_0\to E_1$, and has differential
\[
\d\lb\theta^+_\sigma,f\rb=\lb\theta^-_\sigma,f\rb +\lb\theta^-_\sigma, T^{-1}\circ f\circ \scV\rb.
\]
If we equip $E_0$ and $E_1$ with the chiral topology, then there is a strong deformation retraction of $\ve{CF}^-_{\Fil}(\b_0^{E_0},\b_1^{E_1},\frs_\sigma)$ onto the principal ideal $
 (\sigma)\subset  \ve{I}_1\cdot \cK\cdot \ve{I}_0.$
If we equip $E_0$ and $E_1$ with the $U$-adic topology, then there is a strong deformation retraction onto   $(\sigma)\subset \frK$. In both cases, the homology of the completed complex is isomorphic to a completion of $\bF[U,T,T^{-1}]$, and has a dense subspace spanned by $\lb\theta_\sigma^+, U^i T^j  \phi^\sigma \rb$, where $i\ge 0$ and $j\in \Z$.
\item Similarly, the complex $\ve{\CF}_{\Fil}^-(\b_0^{E_0},\b_1^{E_1},\frs_\tau)$ is spanned by generators $\lb \theta^{\pm}_\sigma, f\rb$ and has differential
\[
\d \lb\theta^+_\tau,f\rb =\lb\theta^-_\tau,f\rb +\lb\theta^-_\tau, T\circ f\circ \scU\rb.
\]
The complex deformation retracts onto the $(\tau)\subset \ve{I}_1\cdot \cK\cdot \ve{I}_0$. The homology of the completed complex is isomorphic to a completion of $\bF[U,T,T^{-1}]$, with dense subspace spanned by $\lb \theta_\tau^+, U^i T^j\phi^\tau\rb$, where $i\ge 0$ and $j\in \Z$. 
\end{enumerate}
\end{lem}

The proof of Lemma~\ref{lem:homology-I1-I0} follows from nearly identical reasoning to that of Lemma~\ref{lem:homology-I1-I0}, so we leave the verification to the reader. A similar argument also yields the following:

\begin{lem} \label{lem:homology-I1-I1}
The complex  $\ve{\CF}_{\Fil}^-(\b_1^{E_1},\b_1'^{E_1})$ is generated by $\lb \theta^{\pm},f\rb$ where $f\colon E_1\to E_1$ is $\bF[U]$-equivariant and continuous. (Note that $E_1$ is given the same topology in the $U$-adic and chiral topologies). The differential is given by
\[
\d\lb\theta^+,f\rb =\lb\theta^-,f\rb +\lb\theta^-, T^{-1}\circ f\circ T\rb.
\]
Furthermore, $\ve{\CF}^-_{\Fil}(\b_1^{E_1}, \b_1'^{E_1})$ deformation retracts onto the $U$-adic completion of $\bF[U,T,T^{-1}]$. The homology is spanned by $\lb \theta^+, \a \cdot \id \rb$ where $\a$ is in the $U$-adic completion of $\bF[U,T,T^{-1}]$. 
\end{lem}

Summarizing Lemmas~\ref{lem:homology-I0-I0}, ~\ref{lem:homology-I1-I0} and ~\ref{lem:homology-I1-I1}, we have the following:

\begin{cor}\label{cor:isomorphism-homology-algberas}
\begin{enumerate}
\item If we equip $E_0$ and $E_1$ with chiral topology, then there is a deformation retraction of chain complexes from $\End_{\Fil}(\b_0^{E_0}\oplus \b_1^{E_1})$ onto $\cK$.
\item If we equip $E_0$ and $E_1$ with $U$-adic topology, then there is a deformation retraction of chain complexes from $\End_{\Fil}(\b_0^{E_0}\oplus \b_1^{E_1})$ onto $\frK$. 
\end{enumerate}
\end{cor}

We now compute several actions $\mu_j$ for $j\ge 0$:

\begin{figure}[h]
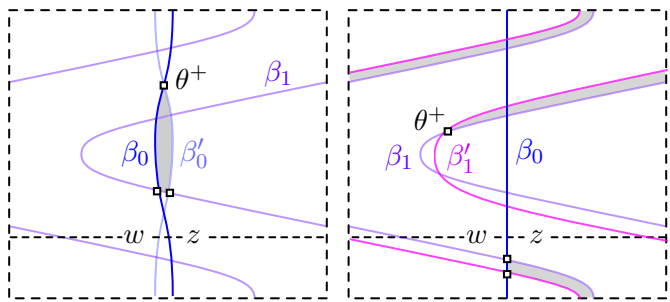
\caption{Diagrams $(\bT^2, \b_0,\b_0', \b_1)$ and $(\bT^2, \b_0, \b_1,\b_1')$ as well as some holomorphic triangle classes contributing to $\mu_2$.}
\label{fig:57}
\end{figure}

\begin{lem}
\label{lem:model-computation-End}
 For suitably chosen small translations of $\b_0$ and $\b_1$ in the torus, we have the following actions:
\begin{enumerate}
\item \label{model-comp-1} $\mu_2(\lb \theta^+, f\rb , \lb \theta_{\sigma}^+, g\rb )=\lb \theta_\sigma^+,g\circ f\rb $ and $\mu_2(\lb \theta^+, f\rb , \lb \theta_{\tau}^+, g\rb )=\lb \theta_\tau^+,g\circ f\rb $.
\item \label{model-comp-2} $\mu_2(\lb \theta_{\sigma}^+,f\rb , \lb \theta^+, g\rb )=\lb \theta_\sigma^+, g\circ f\rb $ and  $\mu_2(\lb \theta_{\tau}^+,f\rb , \lb \theta^+, g\rb )=\lb \theta_\tau^+, g\circ f\rb $.
\item \label{model-comp-3} Suppose $n\ge 3$ and $\lb \xs_1,f_1\rb,\dots, \lb \xs_n,f_n\rb$ are filtered and have the property that for all $i$,  $\frs_w(\xs_i)$ is torsion and $\xs_i$ is the top $\gr_w$-grading intersection point in its respective Floer complex. Then
\[
\mu_n(\lb \xs_1,f_1\rb,\dots, \lb \xs_n,f_n\rb)=0.
\]
The same holds if instead $\frs_z(\xs_i)$ are torsion for all $i$.
\end{enumerate}
\end{lem}
\begin{proof} Claims \eqref{model-comp-1} and ~\eqref{model-comp-2} follow from model computations in $\bT^2$, which are shown in Figure~\ref{fig:57}.

Claim~\eqref{model-comp-3} follows from grading considerations. To streamline the formulas,  we normalize gradings so that the top degree generator of $\widehat{\HF}(S^1\times S^2)$ has $\gr_w$-grading $0$. We claim that there are no holomorphic polygons of index $3-n$ with the inputs $\xs_1,\dots, \xs_n$. We observe that if $\lb \ys, g\rb$ is a summand of  $\mu_n(\lb \xs_1,f_1\rb,\dots, \lb\xs_n,f_n\rb)$, contributed by some homology class of $(n+1)$-gons $\psi$, then
\[
\gr_w(\ys)=\gr_w(\xs_1)+\cdots+\gr_w(\xs_n)+(n-2)+2n_w(\psi)\ge n-2.
\] 
The output will therefore vanish if $n\ge 3$ since there are no chains $\ys$ of that grading.
\end{proof}

We now complete the proof of Theorem~\ref{thm:equivalence-algebras}:

\begin{proof}[Proof of Theorem~\ref{thm:equivalence-algebras}] Lemmas~\ref{lem:homology-I0-I0}, ~\ref{lem:homology-I1-I0} and~\ref{lem:homology-I1-I1} give a strong deformation retraction of chain complexes
\[
\begin{tikzcd}\ar[loop left, "H"]
\End_{\Fil}(\b_0^{E_0}\oplus \b_1^{E_1})^{\opp}\ar[r,shift left, "\Pi"] &\cK \ar[l, shift left, "I"] .
\end{tikzcd}
\]
 Homological perturbation theory shows that $\cK$ may be equipped with an $A_\infty$-module structure which is homotopy equivalent to $\End_{\Fil}(\b_0^{E_0}\oplus \b_1^{E_1})^{\opp}$. We use the version of the homological perturbation lemma stated by Kontsevich and Soibelman \cite{KontsevichSoibelman}*{Proposition~6} in terms of trees.    (This is a helpful special case of the homotopy transfer lemma for $A_\infty$-algebras, due to Kadeishvili \cite{Kadeishvili_Ainfinity}).

 We recall from Kontsevich and Soibelman's presentation that there is an induced $A_\infty$-algebra structure on $\cK$, such that $\mu_n(a_1,\dots, a_n)$ is a sum over all planar, rooted, connected trees $T$ with $n$ inputs and no interior vertices of valence 2 or less. We first apply $I$ to each input $a_1,\dots, a_n$.  At an interior vertex $v$ of $T$ with valence $j$ we apply $\mu_{j-1}$. Along any edge connecting two interior vertices, we apply the homotopy $H$. At the root we apply $\Pi$. See Figure~\ref{fig:hom-pert} for an example.  It follows immediately from Lemma~\ref{lem:model-computation-End} that the only non-trivial composition map induced by the homological perturbation is $\mu_2$, and the induced map coincides with ordinary multiplication on $\cK$, completing the proof.  
\end{proof}

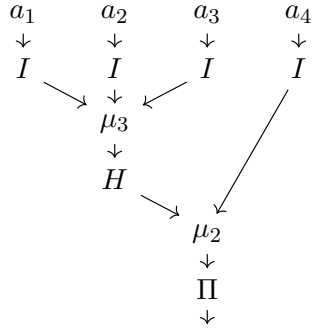
\begin{figure}[h]
\begin{center}
\begin{tikzcd}[row sep=.2cm, column sep=.5cm]
a_1\ar[d]&a_2\ar[d]& a_3\ar[d]& a_4\ar[d]\\
I\ar[dr]&I\ar[d]&I\ar[dl]&I\ar[dddl]\\
&\mu_3\ar[d]\\
&H\ar[dr]&\\
&&\mu_2\ar[d]\\
&&\Pi\ar[d]\\
&&\,
\end{tikzcd}
\end{center}
\caption{An example of a tree counted by $\mu_4(a_1,\dots, a_4)$ using the homological perturbation lemma}
\label{fig:hom-pert}
\end{figure}

\subsection{Non-filtered morphisms}

 It is natural to investigate the full endomorphism algebra 
\[
\End(\b_0^{E_0}\oplus \b_1^{E_1})
\] 
generated by morphisms $\lb\xs,\phi\rb$ where $\phi$ is only required to be $\bF[U]$-equivariant and map $E_{\veps}$ to $E_{\veps'}$. Firstly, we observe that it is not possible to relax the condition that maps $f\colon E_0\to E_0$ are non-increasing in the $\gr_w$ and $\gr_z$ gradings. Compare Lemma~\ref{lem:filtered-morphism-compose-Vs} and the following example: 

\begin{example}Consider $\d\lb\theta^+, \scV^{-1} \delta_{1}\rb$ in $\ve{\CF}^-(\b_0^{E_0}, \b_0^{E_0})$, where $\delta_1\colon E_0\to E_0$ is projection onto Alexander grading 1. Observe that $\scV^{-1}\delta_{1}$ maps $E_0$ to $E_0$ with no negative $U$ powers. It follows from Lemma~\ref{lem:homology-I0-I0} that  $\d\lb \theta^+, \scV^{-1} \delta_1\rb $ has a summand of $\lb \theta^-, \scV^{-1} \scV^{-1} \delta_{1} \scV\rb $ which is the same as $\lb \theta^-, \scV^{-1} \delta_0\rb $. Note that $\scV^{-1} \delta_0$ maps $1$ to $\scV^{-1}$, and hence involves negative $U$ powers.
\end{example}

Nonetheless, it might be possible to relax the condition that morphisms cannot go from $\b_1^{E_1}$ to $\b_0^{E_0}$. We define the \emph{weakly-filtered} endomorphism space 
\[
\End_{w\Fil}(\b_0^{E_0}\oplus \b_1^{E_1})
\]
 to be the space of endomorphisms which satisfy conditions (2) and (3) from Definition~\ref{def:filtered-morphism} (i.e. allowing morphisms from $\b_1^{E_1}$ to $\b_0^{E_0}$). We make the following conjecture:

\begin{conj} For both the chiral and $U$-adic topologies, the space $\End_{w\Fil}(\b_0^{E_0}\oplus \b_1^{E_1})$ is an $A_\infty$-algebra. Furthermore, it is $A_\infty$-homotopy equivalent to either $\cK$ (in the chiral case) or $\frK$ (in the $U$-adic case).
\end{conj}

We note that the challenge in proving the above conjecture is not actually constructing the homotopy equivalence, but rather in showing that $\End_{w\Fil}(\b_0^{E_0}\oplus \b_1^{E_1})$ is a well-defined linear topological $A_\infty$-algebra.  Indeed, the challenge is showing that the composition maps introduce no negative powers of $U$, and induce continuous maps in the two topologies. We observe that neither Proposition~\ref{prop:finite-ness-endomorphisms} nor the arguments of Appendices~\ref{sec:positivity} or~\ref{sec:admissibility} are sufficient for these cases.

 We observe that assuming these admissibility and positivity results, the above conjecture would following since one may compute directly that $\ve{\CF}^-(\b_1^{E_1},\b_0^{E_0})$ is acyclic, so the same homological perturbation argument we used for $\End_{\Fil}$ would apply to show that $\End_{w \Fil}(\b_0^{E_0}\oplus \b_1^{E_1})$ is homotopy equivalent to $\cK$.

\subsection{Additional computations}

It will be helpful to extend a few of the computations from the previous setting to the setting of links. We are interested in a Heegaard diagram  $(\Sigma,\bs_0,\ws,\zs)$ with only one set of attaching curves. Here, $\bs_0$ is a collection of attached curves such that each component of $\Sigma\setminus \bs_0$ contains one point from $\ws$ and one point from $\zs$. We assume also that there are choices of oriented knot shadows (embedded, oriented, simple closed curves) $K_1,\dots, K_\ell$ on $\Sigma$, which are pairwise disjoint. Furthermore, we assume that $\ws$ and $\zs$ are collections of $n$ points such that each $K_i$ contains one point from each of $\ws$ and $\zs$. We assume the subarc of $K_i$ oriented from $z_i$ to $w_i$ is disjoint from $\bs_0$, and that the subarc from $w_i$ to $z_i$ intersects a single curve of $\bs_0$, geometrically once. Furthermore, each $K_i$ only intersects a single curve of $\bs_0$.

The union of $K_i$ with the corresponding meridianal curve of $\bs_0$ determines a punctured torus in $\Sigma$. For $\veps\in \bE_\ell$, we let $\bs_{\veps}^{\can}$ denote the collection of curves obtained by winding each component of $\bs_0$ corresponding to the components $K_i$ with $\veps_i=1$. We assume the corresponding genus 1 regions are as in Figure~\ref{fig:47}. Furthermore, it is helpful to write $D_i$ for a small disk containing $w_i$ and $z_i$, which also contains the intersection point of $\bs_0$ and $K_i$.

 The diagram $(\Sigma,\bs_{\veps}^{\can}, \bs_{\nu}^{\can}, \ws,\zs)$ represents a link $L_{\veps,\nu}$ in a three manifold $Y$. The manifold $Y$ is a connected sum of $g(\Sigma)$ copies of $S^1\times S^2$. The $L_{\veps,\nu}$ consists of $n-|\nu-\veps|_{1}$ unknots (unlinked from other components) and $|\nu-\veps|_1$ fibers of different copies of $S^1\times S^2$.

There are $2^{|\nu-\veps|_1}$ $\Spin^c$ structures of interest on $Y$. 
Given a sequence of symbols $\scO\in \{\sigma,\tau\}^{|\nu-\veps|_1}$, there is a $\Spin^c$ structure $\frs_\scO$ on $Y$. We define $\frs_{\sigma,\dots, \sigma}$ to be the torsion $\Spin^c$ structure on $Y$, and 
\[
\frs_{\scO}=\frs_{\sigma,\dots,\sigma}-\sum_{\{i| \scO_i=\tau\}} \PD[K_i].
\]

\begin{lem}
\label{lem:generalized-top-generator}
 Suppose that $(\Sigma,\bs_0,\ws,\zs)$ is a partial Heegaard diagram with knot traces $K_1,\dots, K_\ell$, as above. Suppose also that $\veps,\nu\in \bE_\ell$, $\veps\le \nu$ and $\bs_{\veps}$, $\bs_{\nu}$ are attaching curves on $(\Sigma,\ws,\zs)$ which are obtained from $\bs_{\veps}^{\can}$ and $\bs_{\nu}^{\can}$ by a sequence of handleslides and isotopies which are supported in the complement of the disks $D_i$. Assuming the diagrams are weakly admissible with respect to each complete collection $\ve{p}\subset \ws\cup \zs$, the subspace of
\[
\HF^-_{\Fil}\left(\bs_{\veps}^{E_{\veps}}, \bs_{\nu}^{E_{\nu}},\frs_\scO\right)
\]
in Alexander grading $(0,\dots, 0)$ is canonically isomorphic to $\bF\llsquare U\rrsquare \otimes_{\bF}  \Lambda_{g(\Sigma)-\ell}$, where $\Lambda_{g(\Sigma)-\ell}$ is the exterior algebra on $g(\Sigma)-\ell$ generators (each viewed as having Maslov grading $-1$). The top degree subspace is spanned by the generator $\lb \theta^+_{\scO}, \phi^{\scO}_{\veps,\nu} \rb$.
 For other $\Spin^c$ structures, 
\[
\HF^-_{\Fil}\left(\bs_{\veps}^{E_{\veps}}, \bs_{\nu}^{E_{\nu}},\frs\right)\iso 0.
\]
\end{lem}
\begin{proof} The ordinary proof of invariance of Heegaard Floer homology \cite{OSDisks} adapts to show that the above complexes are invariant up to graded homotopy equivalence under Heegaard moves, so it suffices to prove the isomorphism for sufficiently simple diagrams. Lemmas~\ref{lem:homology-I0-I0}, ~\ref{lem:homology-I1-I0}, and ~\ref{lem:homology-I1-I1}  provide proofs in the case that $n=1$ and $g(\Sigma)=1$.

To handle the general case, we argue by induction on $n$ and the genus. Observe first that if $\cH_1$ and $\cH_2$ are two Heegaard diagrams, we may join them by attaching a 1-handle connecting two points on the Heegaard surface to form a new Heegaard diagram $\cH_1\# \cH_2$. Along the connected sum 1-handle, we add two beta curves $\b,\b'$, which are isotopic and intersect in two points. See Figure~\ref{fig:50}.  An adaptation of Ozsv\'{a}th and Szab\'{o}'s proof of stabilization invariance \cite{OSLinks}*{Section~6.1} shows that
\[
\ve{\CF}^-_{\Fil}(\cH_1 \# \cH_2)\iso \ve{\CF}^-_{\Fil}(\cH_1)\tildeotimes_{\bF\llsquare U\rrsquare} \ve{\CF}^-_{\Fil}(\cH_2)
\]
\[
:=\left(
\begin{tikzcd}\ve{\CF}^-_{\Fil}(\cH_1)\otimes^!_{\bF} \ve{\CF}^-_{\Fil}(\cH_2)\ar[r, "U_1+U_2"]&\ve{\CF}^-_{\Fil}(\cH_1)\otimes^!_{\bF} \ve{\CF}^-_{\Fil}(\cH_2)\end{tikzcd}
\right), 
\]
where $U_1$ and $U_2$ are $U$-actions from two of the link components from $\cH_1$ and $\cH_2$ (respectively). In the above $\tilde{\otimes}_{\bF\llsquare U\rrsquare}$ denotes derived tensor product over $\bF\llsquare U\rrsquare$, which is defined to be the mapping cone in the second line. See \cite{HMZConnectedSum}*{Lemma~5.7} for a detailed proof. The exact choice of link components corresponding to $U_1$ and $U_2$ is not important but can be determined from the placement of the 1-handle connecting the two diagrams. By inducting on the main statement, we see that the action of $U_1+U_2$ on 
\[
H_*(\ve{\CF}^-_{\Fil}(\cH_1)\otimes^!_{\bF} \ve{\CF}^-_{\Fil}(\cH_2))\iso \HF^-_{\Fil}(\cH_1)\otimes^! \HF^-_{\Fil}(\cH_2).
\] 
is injective. Therefore the snake lemma shows that the homology in Alexander grading 0 coincides with the quotient of the Alexander grading zero part of $\HF^-(\cH_1)\otimes^! \HF^-(\cH_2)$ by $U_1+U_2$. By induction, this is  $\Lambda_{g(\Sigma_1)-\ell_1}\otimes_{\bF} \Lambda_{g(\Sigma_2)-\ell_2}\otimes_{\bF} \bF\llsquare U\rrsquare$, with generators as claimed (where $\ell_i$ denote the number of link components  from $\cH_i$). 

Next, we observe that increasing $g(\Sigma)$ while keeping $n=|L|$ fixed may be achieved by adding a 1-handle to the surface. Using the well-known effect of attaching a 1-handle to a Heegaard surface \cite{OSTriangles}*{Section~4.3} we see that the effect is to replace $\ve{\CF}^-_{\Fil}(\cH)$ by $\ve{\CF}^-_{\Fil}(\cH)\otimes_{\bF} \Lambda_1$.  This completes the proof.
\end{proof}

\begin{figure}[h]
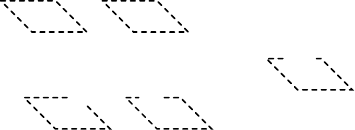
\caption{The stabilizations from Lemma~\ref{lem:generalized-top-generator}. The curves in the 1-handle regions are denotes $\b_1$ and $\b_2$ in the proof.}
\label{fig:50}
\end{figure}

\begin{rem}\label{rem:refined-hypercube}
Lemma~\ref{lem:generalized-top-generator} has a natural extension which is helpful for filling hypercubes later on. There is a subcomplex 
\[
\ve{\CF}_{\phi^\scO}^-(\bs_{\veps}^{E_{\veps}}, \bs_{\nu}^{E_{\nu}}, \frs_{\scO})\subset \ve{\CF}_{\Fil}^-(\bs_{\veps}^{E_{\veps}}, \bs_{\nu}^{E_{\nu}}, \frs_{\scO}),
\]
  generated over $\bF\llsquare U_1,\dots, U_\ell\rrsquare$ by Floer generators $\lb \xs, \phi^{\scO}_{\veps,\nu}\rb$ (i.e. we restrict all morphisms to be in the $\bF\llsquare U_1,\dots, U_\ell\rrsquare$-span of $\phi_{\veps,\nu}^{\scO})$. Then 
\[
\ve{\HF}^-_{\phi^{\scO}}(\bs_{\veps}^{E_{\veps}}, \bs_{\nu}^{E_{\nu}}, \frs_{\scO})\iso \Lambda_{g}\otimes_{\bF} \bF\llsquare U \rrsquare.
\] 
Furthermore, the canonical map
\[
\ve{\HF}^-_{\phi^{\scO}}(\bs_{\veps}^{E_{\veps}}, \bs_{\nu}^{E_{\nu}}, \frs_{\scO})\to \HF^-_{\Fil}(\bs_{\veps}^{E_{\veps}}, \bs_{\nu}^{E_{\nu}}, \frs_{\scO})
\]
is an isomorphism on the top Maslov degree. To prove this statement, we observe first that the genus 1 case holds because by Lemma~\ref{lem:homology-I1-I0}, the differential on $\ve{\CF}^-_{\phi^{\scO}}(\b_0^{E_0}, \b_1^{E_1},\frs_\sigma)$ vanishes. Hence in Alexander grading 0 the complex consists of the $\bF\llsquare U\rrsquare$-span of $\lb \theta^+, \phi^\sigma\rb $ and $\lb \theta^-, \phi^{\sigma}\rb$. The general case is proven by an inductive argument identical to the proof of Lemma~\ref{lem:generalized-top-generator}. 
\end{rem}

\section{Invariance}

In this section, we prove Theorem~\ref{thm:intro-invariance} of the introduction, which states that the bordered link surgery modules give invariants of bordered 3-manifolds. This follows immediately from the following theorem, which we prove in this section:

\begin{thm} \label{thm:invariance}
The bordered link surgery modules are invariants in the following sense:
\begin{enumerate}
\item \label{thm:invariance-1} If $L\subset Y$ is a link with Morse framing $\Lambda$, then the homotopy type of $\cX_{\Lambda}(Y,L)^{\frL}$ is an invariant of the tuple $(Y,L,\Lambda)$. (In particular, independent of the Heegaard diagram). 
\item \label{thm:invariance-2} Suppose that $(Y,L,\Lambda)$ is a link with Morse framing, and $(Y',L',\Lambda')$ is obtained by performing Dehn surgering on one  component of $L$, according to the framing $\Lambda$. Then
\[
\cX_{\Lambda}(Y,L)^{\frL}\boxtimes {}_{\frK} \frD_0\simeq \cX_{\Lambda'}(Y',L')^{\frL'}
\]
\item  \label{thm:invariance-3}The same statements hold if we replace the $\frL$ with $\cL$ (i.e. the statement holds in both the chiral and $U$-adic topologies). Furthermore, the $\cL$ modules are related to the $\frL$ modules via
\[
\cX_{\Lambda}(Y,L)^{\frL}\boxtimes{}_{\frL}[i]^{\cL}\simeq \cX_{\Lambda}(Y,L)^{\cL}.
\]
\end{enumerate}
\end{thm}

Theorem~\ref{thm:invariance} will have a number of consequences, which include the sublink surgery formula (Section~\ref{sec:sublink-surgery-formula}), as well as invariance of the bordered modules with respect to gluing along torus boundary components (Section~\ref{sec:gluing}). We organize this section by first describing the consequences of Theorem~\ref{thm:invariance}, and then subsequently proving it in Section~\ref{sec:proof-invariance}.

\subsection{The sublink surgery formula}

\label{sec:sublink-surgery-formula}
In this section, we discuss a surgery formula for link Floer homology. This will follow as a special case of Theorem~\ref{thm:invariance}. 

We now suppose that $J\cup L$ is link in a 3-manifold $Y$, whose components are partitioned into two collections $J$ and $L$. Let $\Lambda$ be a Morse framing on $J$. Write $n=|L|$ and define
\[
R_n:=\bF[\scU_1,\scV_1,\dots, \scU_n,\scV_n].
\]
If computed using a Heegaard link diagram, the link Floer complex $\cCFL(Y_{\Lambda}(J),L)$, is a finitely generated, free chain complex over $R_n$. In particular, we may think of $\cCFL(Y_{\Lambda}(J),L)$ as a type-$D$ module over $R_n$.
 In this section, we describe
\[
\cCFL(Y_{\Lambda}(J), L)^{R_n}
\]
using the link surgery formula.

Completions play an important role in the theory, and there are two natural ways of completing $R_n$:
\begin{enumerate}
\item $\cR_n$, the $I$-adic completion with respect to the ideal $I=(\scU_1,\scV_1,\dots, \scU_n,\scV_n)$. Note that the completion is $\bF\llsquare \scU_1,\scV_1,\dots, \scU_n,\scV_n\rrsquare$. This is $\ve{E}_0\cdot \cL_n\cdot \ve{E}_0$.
\item $\frR_n$, the $(U_1,\dots, U_n)$-adic completion of $R_n$. This is $\ve{E}_0\cdot \frL_n\cdot \ve{E}_0$. 
\end{enumerate}
For concreteness, we focus phrasing the statements over $\cR_n$. 

  We define a chain complex $\bX_{\Lambda}(Y,J;L)$ to be the subcube of the link surgery complex for $M=J\cup L$ which has $L$-coordinate 0 everywhere. Note that to define the full link surgery complex for $M$, one must pick a framing on all of $M=J\cup L$, however the framing on $L$ has no effect on the complex $\bX_{\Lambda}(Y,J;L)$. The module $\bX_{\Lambda}(Y, J;L)$ has a natural action of $\bF[\scU_1,\scV_1,\dots, \scU_n,\scV_n]$ (the variables for $L$).

Similarly, we may construct a type-$D$ module $\cX_{\Lambda}(Y,J;L)^{\cR_n}$, where $n=|L|$, by tensoring the type-$D$ module for $(Y,J\cup L)$ with a copy of ${}_{\cK} \cD_0$ for each algebra factor corresponding to a component of $J$, and then restricting to cube components which have $\veps_i=0$ for $K_i\in L$. 

It is helpful to rephrase the operation of restricting to idempotent 0 as tensoring with a bimodule ${}_{\cK}[\ve{I}_0]^{\bF[\scU,\scV]}$, which has underlying $\ve{I}$-module equal to $\ve{I}_0$, and has structure map $\delta_2^1(a,i_0)=i_0\otimes a$ when $a\in \ve{I}_0\cdot \cK\cdot \ve{I}_0$.

As a corollary to Theorem~\ref{thm:invariance}, we have the following:
\begin{cor}
\label{cor:Manolescu-Ozsvath-subcube}
 Suppose that $J\cup L$ is a partitioned link in $Y$, and that $\Lambda$ is a Morse framing on $J$. Write $n=|L|$.
 \begin{enumerate}
 \item There is a homotopy equivalence of type-$D$ modules
 \[
\cCFL(Y_{\Lambda}(J),L)^{\cR_n}\simeq \cX_{\Lambda}(Y,J;L)^{\cR_n}.
 \]
 \item There is a homotopy equivalence of type-$A$ modules
 \[
{}_{\cR_n}\ve{\cCFL}(Y_{\Lambda}(J),L)\simeq  {}_{\cR_n}\bX_{\Lambda}(Y,J;L).
 \]
 \end{enumerate}
\end{cor}

\begin{proof}

The second claim follows from the first by tensoring with the $AA$-bimodule ${}_{\cR_n|\cR_n} \cR_n $, gotten by restricting the merge module ${}_{\cK|\cK} M$ to idempotent 0. Recall that this has a single structure map $\delta_2^1\colon (\cR_n|\cR_n)\otimes \cR_n\to \cR_n$ given by 
\[
\delta_2^1( (a_1|a_2)\otimes b)=a_1\cdot a_2\cdot b.
\] 

 For the first claim, we first pick a Morse framing $\Lambda'$ on $J\cup L$ which extends $\Lambda$. Write $\Lambda''$ for the restriction to $L$. By Theorem~\ref{thm:invariance}, we have that $\cX_{\Lambda''}(Y_{\Lambda}(J),L)^{\cL}$  is obtained by tensoring $\cX_{\Lambda'}(Y,L\cup J)^{\cL}$ with a copy of ${}_{\cK} \cD_0$ for each component of $J$.  Note that tensoring $\cX_{\Lambda'}(Y, L\cup J)^{\cL}$ with copies of ${}_{\cK} \cD_0$ for each component of $J$ and a copy of ${}_{\cK} [\ve{I}_0]^{\bF[\scU,\scV]}$ for each copy of $L$ gives exactly $\cX_{\Lambda}(Y,J;L)^{R_n}$. On the other hand tensoring $\cX_{\Lambda''}(Y_{\Lambda}(J),L)^{\cL}$ with a copy of ${}_{\cK} [\ve{I}_0]^{\bF[\scU,\scV]}$ for each component of $L$ gives $\cCFL(Y_{\Lambda}(J),L)^{\cR_n}$ since by our construction from Section~\ref{sec:background-link-surgery}, if $Z$ is a closed 3-manifold and $N\subset Z$ is a link of $n$-components, then
\begin{equation}
\cCFL(Z,N)^{\cR_n}\simeq \cX_{\Lambda}(Z,N)^{\cL}\cdot \ve{E}_0,
\label{eq:idempotent-0}
\end{equation}
where $\ve{E}_0$ denotes the idempotent corresponding to $(0,\dots, 0)\in \bE_\ell$. This completes the proof.
\end{proof}

 We record a few simplifying remarks:

\begin{rem} 
\begin{enumerate}
\item The above module $\cCFL(Y_{\Lambda}(J),L)^{\cR_n}$ is a chiral type-$D$ module over $\cR_n$, i.e. an object of $\Mod_{\ch}^{\cR_n}$. We recall that if $X$ and $Y$ are linearly topological spaces and $X$ is linearly compact, then $X\vecotimes Y\iso X\otimes^! Y$. We observe that both $\cX_{\Lambda}(Y,J;L)^{\cR_n}$ and the algebra $\cR_n$ itself are linearly compact. Hence the category of linearly compact, chiral type-$D$ modules over $\cR_n$ is equivalent to the category of linearly compact achiral type-$D$ modules, and furthermore, the category of chiral type-$A$ modules over $\cR_n$ is equivalent to the category of achiral type-$A$ modules over $\cR_n$.
\item Corollary~\ref{cor:Manolescu-Ozsvath-subcube} also holds if we work over $\frR_n$ as either a chiral or achiral algebra. Note however that the previous comment does immediately apply if we work with chiral modules over $\frR_n$, since $\frR_n$ is not linearly compact. 
\end{enumerate}
\end{rem}

\subsection{Invariance under gluing}
\label{sec:gluing}

In this section we describe the proof of Theorem~\ref{thm:naturality-gluing-intro}, which states that the bordered modules are functorial with respect to gluing. We restate the theorem below:

\begin{thm}
\label{thm:naturality-gluing}
Suppose $Y_1$ and $Y_2$ be two bordered manifolds with torus boundary components, each with a distinguished boundary component $Z_1$ and $Z_2$, respectively. We let $\phi\colon Z_1\to Z_2$ be the orientation reversing diffeomorphism sending $\mu_1$ to $\mu_2$ and $\lambda_1$ to $-\lambda_2$. Let $\cX(Y_1)^{\cL_{n_1-1}\otimes \cK}$ be the type-$D$ module for $Y_1$ and let ${}_{\cK} \cX(Y_2)^{\cL_{n_2-1}}$ be the type-$DA$ bimodule for $Y_2$. Here the copies of $\cK$ correspond to the distinguished boundary components. Then
\begin{equation}
\cX(Y_1)^{\cL_{n_1-1}\otimes \cK}\boxtimes {}_{\cK}\cX(Y_2)^{\cL_{n_2-1}}\simeq \cX(Y_1\cup_\phi Y_2)^{\cL_{n_1+n_2-1}}. \label{eq:naturality-gluing}
\end{equation}
The same holds over the $U$-adic topologies using the algebras $\frK$ and $\frL$.
\end{thm}

After recalling the various definitions, the theorem will follow quickly from the connected sum formulas from \cite{ZemBordered}, as well as the functoriality of the modules under Dehn surgery proved in Theorem~\ref{thm:invariance}. We now describe the argument.

Before proving Theorem~\ref{thm:naturality-gluing}, we recall the construction of the type-$DA$ module ${}_{\cK} \cX(Y_2)^{\cL_{m-1}}$ in terms of $\cX(Y_2)^{\cL_{m}}$.  The type-$DA$ bimodule ${}_{\cK}\cX(Y_2)^{\cL_{m-1}}$ is defined as a tensor product of $\cX(Y_2)^{\cL_2}$ together with the \emph{type-$A$ identity} module ${}_{\cK|\cK}[\bI^{\Supset}]$ from \cite{ZemBordered}*{Section~8.4}. Ignoring completions, this is the same as a type-$A$ module over $\cL_2=\cK\otimes \cK$. The notation $\cK|\cK$ denotes a specific behavior with respect to completions called the \emph{split Alexander condition} \cite{ZemBordered}*{Section~6.4}.

For completeness, we recall the definition of ${}_{\cK|\cK}[\bI^{\Supset}]$. As a vector space, it is the direct sum of $\bF[\scU,\scV]$ and $\bF[U,T,T^{-1}]$, with summands concentrated in idempotents $i_0| i_0$ and  $i_1|i_1$, respectively. If $a|a'$ are concentrated in idempotent $i_0|i_0$ and $x\in \bF[\scU,\scV]$, then
\[
m_2(a|a', x)=aa'x\in \bF[\scU,\scV].
\]
We use the same formula for $a|a'$ concentrated in idempotent $i_1|i_1$. We declare $m_2$ to vanish on tensors involving $\sigma$ or $\tau$. Next, we declare 
\[
m_3(\sigma|1, 1|\sigma, x)=\phi^\sigma(x)\quad \text{and} \quad m_3(\tau|1,1|\tau, x)=\phi^\tau(x),
\]
(extended linearly over algebra elements which are concentrated in $i_0|i_0$ or $i_1|i_1$). We set $m_3(1|\sigma, \sigma|1, x)=m_3(1|\tau, \tau|1,x)=0$ and $m_j=0$ for $j>3$.

\begin{rem} Over the $U$-adic topology, the type-$A$ identity module $[\bI^{\Supset}]$ is an ordinary type-$A$ module over $\frL_2=\frK\otimes^! \frK$. The underlying space of the $[\bI^{\Supset}]$ is topologized using the $U$-adic topology. 
\end{rem}

Finally, the type-$DA$ bimodule ${}_{\cK} \cX(Y_2)^{\cL_{m-1}}$ is defined by tensoring $\cX(Y_2)^{\cL_m}$ with ${}_{\cK|\cK} [\bI^{\Supset}]$ along one copy of $\cK$. More precisely, we extend ${}_{\cK|\cK} [\bI^{\Supset}]$ to a bimodule ${}_{\cL_{m-1}|\cK} [\bI^{\Supset}]^{\cL_{m-1}}$ by taking the external tensor product with the identity bimodule ${}_{\cL_{m-1}}[\bI]^{\cL_{m-1}}$. Similarly we take the external $\cX(Y_2)^{\cL_m}$ with the identity module ${}_{\cK} [\bI]^{\cK}$. The external tensor product operation is standard, though the reader may consult \cite{ZemBordered}*{Section~3.5} for background in our present notation. Finally, we form ${}_{\cK} \cX(Y_2)^{\cL_{m-1}}$ by taking the ordinary box tensor product of the above two bimodules.

  In particular, the bimodule on the left hand side of Equation~\eqref{eq:naturality-gluing} is the same as
\[
\left(\cX(Y_1)^{\cL_{n}}\otimes_{\bF} \cX(Y_2)^{\cL_m}\right)\boxtimes {}_{\cK|\cK} [\bI^{\Supset}].
\]
In the above, $\otimes_{\bF}$ denotes external tensor product, i.e. the differential is defined using the Leibniz rule.

\begin{proof}[Proof of Theorem~\ref{thm:naturality-gluing}]
Let $Y_1$ and $Y_2$ be two bordered manifolds with torus boundaries, as in the statement. Write both $Y_i$ as Morse surgery on a link $J_i$ in the complement of an unlink $O_{i}\subset S^3$ of $n_i$ components. Write $j_i=|J_i|$.

Theorem~\ref{thm:invariance} implies that the modules $\cX(Y_i)^{\cL_{n_i}}$ may be obtained from the module $\cX(S^3,J_i\cup O_{i})^{\cL_{j_i+n_i}}$ by tensoring the type-$A$ solid torus modules ${}_{\cK} \cD_0$ to the components corresponding to $J_i$. On the other hand, gluing as in Theorem~\ref{thm:naturality-gluing} is achieved by taking the connected sum $O_{1}\# O_{2}$, and then performing Dehn surgery on the component where the connected sum is formed (cf. Equation~\eqref{eq:Dehn-surgery-connected-sum}).  Finally, by using the connected sum formulas of \cite{ZemBordered}*{Theorem~15.2, Proposition 15.4} (cf. also \cite{ZemBordered}*{Theorem~12.1}) we obtain that
\[
\begin{split}
&\cX\left((J_1\cup O_1)\#(J_2\cup O_2)\right)^{\cL_{n_1+j_1+n_2+j_2-1}}\\
\simeq& \left(\cX(J_1\cup O_1)^{\cL_{j_1+n_1}}\otimes_{\bF} \cX(J_2, O_2)^{\cL_{j_2+n_2}}\right)\boxtimes {}_{\cK|\cK}[\bI^{\Supset}],
\end{split}
\]
from which Theorem~\ref{thm:naturality-gluing} follows immediately.
The same argument works over the $U$-adic topology. 
\end{proof}

\subsection{Proof of Theorem~\ref{thm:invariance}}
\label{sec:proof-invariance}

In this section, we prove Theorem~\ref{thm:invariance}. We prove the three claims separately, since they have distinct flavors.

We begin by observing that part~\eqref{thm:invariance-3} follows from the fact that the type-$D$ modules over $\frL$ and $\cL$ are defined using identical Heegaard diagrams and holomorphic curve counts.  The main point is the existence of the functor ${}_{\frL}[i]^{\cL}\colon \Mod^{\frL}_{\lc}\to \Mod^{\cL}_{\lc,\ch}$ from Equation~\eqref{eq:functor-achiral-to-chiral}, which transforms linearly compact, achiral
type-$D$ modules over $\frL$ into chiral type-$D$ modules over $\cL$. Since the modules $\cX_{\Lambda}(Y,L)^{\frL}$ are finitely generated (i.e. finitely generated over $\ve{I}$), they are linearly compact. 

We now show that the type-$D$ module $\cX_{\Lambda}(Y,L)^{\frL}$ is independent from the choice of Heegaard diagram:

\begin{proof}[Proof of Theorem~\ref{thm:invariance} part~\eqref{thm:invariance-1}]
 The main ideas of the proof are due to Manolescu and Ozsv\'{a}th \cite{MOIntegerSurgery}*{Proposition~8.35}. Since our construction has some differences, we sketch the argument for the benefit of the reader. Let us denote a meridianal Heegaard diagram with knot shadows by a tuple $(\Sigma,\as,\bs'\cup\bs_{\mu}, \ws, \zs, \cS)$, where $\bs'$ are the non-meridianal beta curves, and $\bs_{\mu}$ are the meridianal beta curves. Here $\cS$ denotes the knot shadows. Furthermore, by definition of a meridianal  Heegaard diagram, there are distinguished disks $D_1,\dots, D_\ell\subset \Sigma$ such that $D_i\subset \Sigma$ and $w_i,z_i\in D_i$. Furthermore,  $(\as\cup \bs')\cap D_i=\emptyset$, $\b\in \bs_\mu$ intersects only one $D_i$, and each shadow in $\cS$ intersects a single $D_i$. 

 We claim that any two meridianal diagrams for $(Y,L,\Lambda)$ may be related by the following moves:
 \begin{enumerate}
 \item \label{moves-1} Ambient isotopy in $(Y,L)$.
 \item \label{moves-2}  Handleslides and isotopies of curves in $\as$, not crossing any $D_i$.
 \item \label{moves-3} Isotopies and handleslides of curves in $\bs'$ amongst each other, not crossing any $D_i$.
 \item \label{moves-4} Isotopies and handleslides, fixed in $D_i$, of curves in $\bs_{\mu}$ over curves of $\bs'$.
 \item \label{moves-5} Index (1,2)-stabilizations, changing $\Sigma$ to $\Sigma\# \bT^2$ and adding a pair of transverse curves  $\a,\b\subset \bT^2$ such that $|\a\cap \b|=1$.
  \item \label{moves-6} Isotopies of the shadows $\cS$ on $\Sigma$ (fixed in $D_i$), as well as handleslides of components of $\cS$ over beta curves $\bs'$.
 \end{enumerate}
Compare \cite{MOIntegerSurgery}*{Section~8.9}. One can prove the above using a standard Morse theoretic argument. It can be phrased naturally in the language of Heegaard diagrams of sutured manifolds as follows. We refer the reader to \cite{JDisks} on background on sutured Heegaard diagrams. We turn the link complement $Y\setminus N(L)$ into a sutured manifold by placing a single null-homotopic suture on each component of $\d N(L)$. This suture bounds a 2-disk. We let $R^-$ denote the union of the 2-disks, and we let $R^+$ denote these 2-disks in $\d N(L)$ (i.e. a disjoint union of $g$ punctured tori). Write $M_L$ for this sutured manifold.

 We pick a meridian and longitude on each component of $\d N(L)$ which have one point of intersection. We assume this point of intersection is contained in the 2-disks $R^-$.  

Note that the resulting sutured Heegaard manifold is not \emph{balanced} in the sense of \cite{JDisks}*{Definition~2.11}, though we can still consider sutured Heegaard diagrams for this manifold. Picking a self-indexing sutured Morse function and gradient like vector field on $M_L$ gives a sutured Heegaard diagram $(\Sigma_0,\as,\bs')$ for $M_L$. The Heegaard surface $\Sigma_0$ has $\ell=|L|$ boundary components. We can flow the meridian and longitude on each component of $\d N(L)$ onto $\Sigma_0$ using the downward flow of the gradient like vector field. These give $2\ell$ arcs on the Heegaard surface, for which we write $\mu_1,\dots, \mu_\ell$ and $\lambda_1,\dots, \lambda_\ell$. They have boundary on $\d \Sigma_0$ and  are disjoint from the beta curves $\bs'$.

To get a meridianal Heegaard diagram for $(Y,L)$, we first fill in the boundary components of $\Sigma_0$ with $\ell$ disks to get a Heegaard surface $\Sigma$ of $Y$. In each disk, we place two basepoints, $w_i$ and $z_i$. Furthermore, we add a beta curve $\b_{0,i}$ which follows the image of $\mu_i$ on $\Sigma$. Finally, the knot shadow is the image of the longitude $\lambda_i$. See Figure~\ref{fig:56}. Note that any meridianal Heegaard link diagram $(\Sigma,\as,\bs'\cup \bs_\mu,\ws,\zs,\cS)$ may be obtained by the above procedure.

\begin{figure}[h]
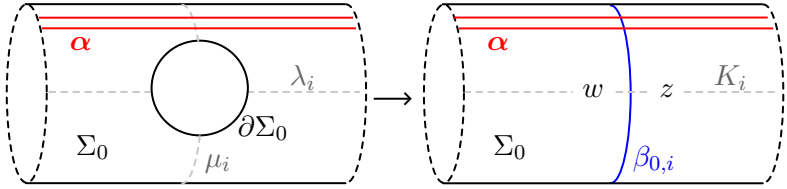
\caption{Left: A Heegaard diagram for the sutured manifold $M_L$. Right: A meridianal Heegaard link diagram for $(Y,L)$. }
\label{fig:56}
\end{figure}

We now consider the effect on the surgery complexes. For all but move~\eqref{moves-6}, the argument is essentially the same as the argument of Manolescu-Ozsv\'{a}th. Invariance under move~\eqref{moves-1} is obvious.  If $(\as,\scB_{\Lambda},\ws,\zs)$ and $(\as', \scB_{\Lambda}', \ws,\zs)$ related by moves \eqref{moves-2}--\eqref{moves-4}, and the total diagram is weakly admissible for each complete collection $\ve{p}\subset \ws\cup \zs$, then we 
construct a homotopy equivalence $b^+\colon \scB_{\Lambda}\to \scB_{\Lambda}'$ by using the hypercube filling procedure of Manolescu and Ozsv\'{a}th \cite{MOIntegerSurgery}*{Lemma~8.6} to iteratively construct chains of increasing length so that the hypercube relations are satisfied. By Lemma~\ref{lem:generalized-top-generator} we can fill each length 1 arrow of the morphism $\scB_{\Lambda}\to \scB_{\Lambda'}$ uniquely up to chain homotopy by requiring them to be top degree cycles which preserve the Alexander grading. We also pick a top degree generator $a^+\in \ve{\CF}^-(\as',\as)$. Using these maps, we construct a homotopy equivalence
\begin{equation}
\begin{tikzcd}
\bX_{\Lambda}(\as, \scB_{\Lambda})\ar[r,equals]&[-.7cm]\ve{\CF}^-(\as,\scB_{\Lambda})\ar[r, "{\mu_2(a^+,\cdot)}"]&\ve{\CF}^-(\as',\scB_{\Lambda})\ar[r, "{\mu_2(\cdot,b^+)}"]&  \ve{\CF}^-(\as',\scB_{\Lambda}')\ar[r,equals]&[-.7cm]\bX_{\Lambda}(\as', \scB_{\Lambda}').
\end{tikzcd}\label{eq:homotopy-equivalence-surgery-complex}
\end{equation}
In the above, $\mu_2$ denotes the composition map $\mu_2^{\Tw}$ for twisted complexes.
  Furthermore, using Remark~\ref{lem:generalized-top-generator}, we may assume the component morphism of $b^+$ corresponding to a sequence $\scO$ of the symbols $\tau,\sigma$ has algebra element which is weighted by an $\bF\llsquare U_1,\dots, U_\ell\rrsquare$ multiple of the map $\phi^{\scO}$. This ensures that the map in Equation~\eqref{eq:homotopy-equivalence-surgery-complex} induces a type-$D$ morphism
\[
\cX_{\Lambda}(\as,\scB_{\Lambda})^{\frL}\to \cX_{\Lambda}(\as',\scB_{\Lambda}')^{\frL}. 
\]
Invariance under stabilizations, move~\eqref{moves-5} follows from a standard stabilization argument for holomorphic polygons, as described in \cite{MOIntegerSurgery}*{Section~8.4}.

We finally consider move~\eqref{moves-6},  isotopies of the shadow $\cS$, and handleslides over the curves in $\bs'$. Here our argument departs slightly from that of Manolescu and Ozsv\'{a}th. Compare \cite{MOIntegerSurgery}*{Section~12.6}. Note that isotopies of a knot shadow $K_i$ may be decomposed as a composition of isotopies of the alpha and beta curves (moves~\eqref{moves-2}--\eqref{moves-4}) together with ambient isotopies (move~\eqref{moves-1}).

 Hence, it suffices to consider handleslides of $\cS$ over the curves in $\bs'$.  In this case, we can construct a choice of $\scB_{\Lambda}$ so that each $\b\in \bs'$ appears (very slightly translated) in each beta curve of $\scB_{\Lambda}$. For example, the hypercubes constructed in Section~\ref{sec:background-link-surgery} satisfy this property. If $\cS'$ denotes the handleslid shadows, then we observe the intersection of $\cS$ and $\cS'$ with the collection of beta curves are identical, so the surgery complex constructed with $\cS$ is identical to the surgery complex constructed with $\cS'$. This completes the proof. 
\end{proof}

Finally, we verify  Part~\eqref{thm:invariance-2} of Theorem~\ref{thm:invariance}, which states that if $(Y,L)$ is a link with Morse framing $\Lambda$, and $(Y',L')$ is the link with Morse framing $\Lambda'$, obtained by performing Dehn surgery on a component $K\subset L$ using the framing from $\Lambda$, then
\[
\cX_{\Lambda}(Y,L)^{\frL}\boxtimes {}_{\frL} \frD_0\simeq \cX_{\Lambda'}(Y',L')^{\frL'}.
\]

\begin{proof}[Proof of Theorem~\ref{thm:invariance} Part~\eqref{thm:invariance-2}] We pick a meridianal Heegaard link diagram $(\Sigma,\as,\bs_0,\ws,\zs)$ for $(Y,L)$, and consider the associated hypercube of attaching curves $\scB_{\Lambda}$ of dimension $|L|$, constructed in Section~\ref{sec:background-link-surgery}. We may construct a Heegaard link diagram for $(Y',L')$ be replacing the meridianal beta curve $\b_0$ for $K$ with a beta curve which is the longitude $\b_\lambda$. We can construct a corresponding hypercube of beta attaching curves $\scB'_{\Lambda'}$, which is of dimension $|L|-1$. Using the techniques of Theorem~\ref{thm:iterated-cone}, we obtain a homotopy equivalence $b^+\colon \scB_{\Lambda}\to \scB'_{\Lambda'}$. Applying the $A_\infty$-functor $\ve{\CF}^-(\as,-)$, we obtain a homotopy equivalence
\[
\ve{\CF}^-(\as, \scB_{\Lambda})\simeq \ve{\CF}^-(\as, \scB'_{\Lambda'}).
\]
This establishes a homotopy equivalence $\bX_{\Lambda}(Y,L)\simeq \bX_{\Lambda'}(Y',L')$. To see that this map is induced by a homotopy equivalence of type-$D$ modules 
\[
\cX_{\Lambda}(Y,L)^{\frL}\boxtimes {}_{\frK} \frD_0\to \cX_{\Lambda'}(Y',L')^{\frL'},
\]
we focus on the morphism $b^+$. We may assume the attaching curves are chosen so that if $\bs_\veps$ and $\bs'_{\veps'}$ are sets of attaching curves in $\scB_{\Lambda}$ and $\scB'_{\Lambda'}$, respectively, and $\ve{p}\subset \ws\cup \zs$ is a complete collection, then there are no intersection points $\xs\in \bT_{\b_{\veps}}\cap \bT_{\b_{\veps'}'}$ which have $\frs_{\ps}(\xs)$ torsion and have $\gr_{\ps}$-grading higher than the top degree generator of $\HF^-(\bs_{\veps}, \bs_{\veps'}')$. Consequently, $b^+$ has only length 1 components, each of which is of the form
\[
\lb \theta^+_{\b_{(\veps_1,\dots, \veps_{\ell-1},1)}, \b'_{(\veps_1,\dots, \veps_{\ell-1})}}, \id_{E_{\veps_1}}\otimes \cdots \otimes \id_{E_{\veps_{\ell-1}}}\otimes \Pi\rb.
\]
The map $\mu_2^{\Tw}(-, b^+)$ clearly induces a type-$D$ module map. The maps in the opposite direction and the chain homotopies appearing in the proof of the homotopy equivalence are analyzed by essentially the same argument. 
\end{proof}

\section{Absolute gradings}
\label{sec:gradings}

We now describe absolute grading formulas on the link and sublink surgery formulas. For knots, Ozsv\'{a}th and Szab\'{o} proved an absolute grading formula \cite{OSIntegerSurgeries}*{Theorem~4.1}. A relative Maslov grading was described by Manolescu and Ozsv\'{a}th \cite{MOIntegerSurgery}*{Section~9.3} on the link surgery formula.   We focus on the case of a link $L\subset S^3$ with integral framing $\Lambda$. We also derive formulas for the Alexander and Maslov multigradings on the sublink surgery formula from Corollary~\ref{cor:Manolescu-Ozsvath-subcube}. 

For links in other integer homology 3-spheres, the same formulas may be used; for links in more general 3-manifolds, a similar analysis may be performed, but we do not investigate it here.

\begin{rem}
 It is likely possible to perform extra bookkeeping in Manolescu and Ozsv\'{a}th's proof of the link surgery formula to also derive these formulas. Due to the technical nature of truncations in their proof, implementation turns out to be more intricate than one might expect. A special case has been implemented in \cite{HHSZDuals} (cf. \cite{HeddenLevineSurgery} \cite{ZhouSurgery} which are similar in spirit). Our proof is simpler since it completely avoids truncations and uses standard grading formulas for cobordisms.
\end{rem}

\subsection{Closed 3-manifolds}

Write $W_{\Lambda}(L)$ for the 2-handle cobordism from $S^3$ to $S^3_{\Lambda}(L)$. Let $\ve{s}=(s_1,\dots, s_\ell)\in \bH(L)$. We write $\frz_{\ve{s}}\in \Spin^c(W_{\Lambda}(L))$ for the $\Spin^c$ structure which satisfies
\begin{equation}
\frac{\langle c_1(\frz_{\ve{s}}), \Sigma_i\rangle -\Sigma\cdot \Sigma_i}{2}=-s_i.
\label{eq:definition-z-s}
\end{equation}
Here, $\Sigma_1,\dots, \Sigma_\ell\in H_2(W_{\Lambda})$ are the classes of the cores of the 2-handles used to construct $W_{\Lambda}(L)$, capped with Seifert surfaces of $L\subset S^3$. We orient the cores of the 2-handles by viewing them as a link cobordism from $L$ to the empty knot. Also $\Sigma:=\Sigma_1+\cdots+\Sigma_\ell$. 

We recall that $\bH(L)/\im \Lambda\iso \Spin^c(S^3_{\Lambda}(L))$. If $\frs\in \Spin^c(S^3_{\Lambda}(L))$, we write
\[
\bX_{\Lambda}(L,\frs)=\bigoplus_{\substack{
\ve{s}\in \bH(L)\\
[\ve{s}]=\frs}
} \bX_{\Lambda}(L,\ve{s}). 
\]
We write $\bX_{\Lambda}^\veps(L,\ve{s})\subset \bX_{\Lambda}(L,\ve{s})$ for the subspace in cube grading $\veps\in \bE_\ell$.

\begin{thm}\label{thm:grading} Suppose $L\subset S^3$ is a link with integer framing $\Lambda$ and $\frs\in \Spin^c(S^3_{\Lambda}(L))$ is torsion.   Then the isomorphism 
$\ve{\CF}^-(S^3_{\Lambda}(L),\frs)\simeq \bX_{\Lambda}(L,\frs)$ 
is absolutely graded if on each $\bX_{\Lambda}^\veps(L,\ve{s})\subset \bX_{\Lambda}(L)$  we use the Maslov grading
\[
\tilde{\gr}:=\gr_{\ws}+\frac{c_1(\frz_{\ve{s}})^2-2\chi(W_{\Lambda}(L))-3\sigma (W_{\Lambda}(L))}{4}+|L|-|\veps|
\]
where $\gr_{\ws}$ is the internal Maslov grading from $\cCFL(L)$. 
\end{thm}

Before proving Theorem~\ref{thm:grading}, we prove several intermediate results.

 We recall that if $(\Sigma,\as,\bs,\ws,\zs)$ is a Heegaard link diagram, we call a subset $\ve{p}\subset \ws\cup \zs$ a \emph{complete collection of basepoints} if it contains exactly one basepoint from each link component.  We can view a complete collection $\ve{p}$ as determining an orientation on $L$ by declaring $L$ to intersect $\Sigma$ negatively at $\ve{p}$. Equivalently, we can consider the orientation induced by the Heegaard link diagram $(\Sigma,\as,\bs,\ps,\qs)$ where $\qs=(\ws\cup \zs)\setminus \ps$. Note that changing which basepoints are labeled as $\ws$ versus $\zs$ does not change the Floer complex $\cCFL(Y,L)$ (up to relabeling some variables $\scU_i$ versus $\scV_i$). In particular, given a complete collection $\ps$, there is an Alexander multi-grading $A^{\ve{p}}$ on $\cCFL(Y,L)$. The standard multi-grading $A$ coincides with $A^{\ws}$. 

\begin{lem}\label{lem:Alexander-grading-change-signs}
 Suppose that $(\Sigma,\as,\bs,\ws,\zs)$ is a Heegaard link diagram for a rationally null-homologous link in a 3-manifold $Y$.
 \begin{enumerate}
 \item Suppose $w_i,z_i$ are basepoints on a link component $K_i$ and assume that both $\ve{p}=\ve{p}_0\cup \{w_i\}$ and $\ve{p}'=\ve{p}_0\cup \{z_i\}$ are complete collections of basepoints. Then
\[
A_j^{\ve{p}}=\begin{cases} A_j^{\ve{p}'}& \text{ if } i\neq j\\
-A_j^{\ve{p}'}& \text{ if } i=j.
\end{cases}
\]
\item Suppose that $\ve{p},\ve{p}'\subset \ws\cup \zs$ are two complete collections of basepoints. Let $M$ denote the oriented sublink of $L$ consisting of components where $\ve{p}$ and $\ve{p}'$ differ. Orient $L$ and $M$ to intersect $\Sigma$ negatively at $\ve{p}$.  Then
\[
\gr_{\ve{p}'}-\gr_{\ve{p}}=\lk(L\setminus M,M)-2\sum_{K_i\subset M} A_i^{\ve{p}}.
\]
\end{enumerate}
\end{lem}

See \cite{ZemBordered}*{Lemma~7.4} for a proof of the above lemma. Note that therein, the second statement is proven when $M$ consists of a single component. The general case is a straightforward consequence.

By pairing the morphism $\scB_{\Lambda}\to \bs_{\Lambda}$ from Theorem~\ref{thm:iterated-cone} (which was a homotopy equivalence in the Fukaya category) with a  Lagrangian $\as$ (viewed as having trivial local system), we obtain our homotopy equivalence
\[
\Gamma\colon\bX_{\Lambda}(L)=\ve{\CF}^-(\as, \scB_{\Lambda})\to \ve{\CF}^-(\as, \bs_{\Lambda})=\ve{\CF}^-(S^3_{\Lambda}(L)).
\]

 Write $\bs_\veps$ for the attaching curves of $\scB_{\Lambda}$, ranging over $\veps\in \bE_\ell$. The map $\Gamma$ decomposes as a sum 
\[
\Gamma=\sum_{\vec{M}\subset L}\Gamma^{\vec{M}}
\] ranging over oriented sublinks $\vec{M}\subset L$ (the empty link is allowed). Here, $\Gamma^{\vec{M}}$ is only non-trivial on the cube point $\veps$ determined by
\begin{equation}
\veps_i=0\quad \iff \quad \pm K_i\subset \vec{M}.\label{eq:unique-non-trivial-epsilon}
\end{equation}
  The map $\Gamma^{\vec{M}}$ is counts holomorphic $(|\vec{M}|+3)$-gons which have a special input for each component of $\vec{M}$, and a final input of $\lbmed\theta_{\b_{(1,\dots, 1)}, \b_{\Lambda}}^+,\Pi\rbmed$, where $\Pi$ is $U_i$ equivariant and satisfies 
  \[
  \Pi(T_1^{i_1}\dots T_\ell^{i_\ell})=\begin{cases} 1 & \text{ if }i_1=\cdots =i_\ell=0\\
  0& \text{ otherwise}.\end{cases}
  \]
   The special 
  inputs corresponding to components of $\vec{M}$ are length 1 chains in $\scB_{\Lambda}$. If $+K\subset \vec{M}$, then the corresponding generator is $\sigma$-labeled, and if $-K\subset \vec{M}$, then the corresponding generator is $\tau$-labeled.

Write 
\[
X_{M}=X_{\a, \b_{\veps_1},\dots, \b_{\veps_n}, \b_{\Lambda}}
\]
for the 4-manifold constructed in \cite{OSDisks}*{Section~8}. This manifold has boundary components $Y_{\b_{\veps_1},\b_{\veps_2}},\cdots, Y_{\b_{\veps_n}, \b_{\Lambda}}$, which are connected sums of $S^1\times S^2$. Filling in these boundary components with 4-dimensional 1-handlebodies gives $W_{\Lambda}(L)$. 

If $\ve{p}\subset \ws\cup \zs$ is a collection complete collection of basepoints, we have a map
\[
\frs_{\ve{p}}\colon \pi_2(\xs,\theta_{\veps_1,\veps_2},\dots, \theta_{\veps_{n-1},\veps_n},\theta_{\veps_n, \Lambda})\to \Spin^c(X_{M}),
\]
as defined in \cite{OSDisks}*{Proposition~8.4}.

\begin{define}
 If $\vec{M}\subset L$, we say that $\ve{p}$ is \emph{compatible} with $\vec{M}$ if $z_i\in \ve{p}$  whenever $-K_i\subset \vec{M}$ and $w_i\in \ve{p}$ whenever $+K_i\subset \vec{M}$. 
\end{define}

When $\ve{p}$ is compatible with $\vec{M}$ and $\psi$ is a class of $(n+2)$-gons with inputs compatible with $\vec{M}$, then the restriction of $\frs_{\ve{p}}(\psi)$ to the pure beta boundary components of $X_M$ is torsion, so we may view $\frs_{\ve{p}}(\psi)$ as being an element of $\Spin^c(W_{\Lambda}(L))$. In particular, we obtain a decomposition of the map $\Gamma^{\vec{M}}$ over $\Spin^c$-structures
\[
\Gamma^{\vec{M}}=\sum_{\frt\in \Spin^c(W_{\Lambda}(L))} \Gamma^{\vec{M}}_{\frt}.
\]
Where $\Gamma^{\vec{M}}_{\frt}$ counts holomorphic polygons with $\frs_{\ve{p}}(\psi)=\frt$. Note that this decomposition is independent of the choice of $\ve{p}$ which is compatible with $\vec{M}$. This is because the map $\Gamma^{\vec{M}}$ is only non-trivial on the $\veps$ satisfying Equation~\eqref{eq:unique-non-trivial-epsilon}, and on the Heegaard multi-diagrams used to define $\Gamma^{\vec{M}}$, the points $w_i$ and $z_i$ are immediately adjacent if $\veps_i=1$. 

As we have considered earlier, for each complete collection $\ve{p}\subset \ws \cup \zs$, there is an Alexander multi-grading $A^{\ps}$ on the complex $\cC^{\vec{0}}_{\Lambda}(L)\iso \cCFL(L)$, taking values in $\bH(L)$.  We can extend each $A^{\ve{p}}$ to a multi-grading on each $\cC^{\veps}_{\Lambda}(L)$ by declaring that $\Phi^{\vec{M}}$ to preserve $A^{\ve{p}}$ whenever $\vec{M}$ and $\ve{p}$ are compatible. (Note that we do not need $\Phi^{\vec{M}}$ to be non-vanishing to make this definition; we can use homology classes of polygons instead of holomorphic polygons for the purposes of gradings, cf. Ozsv\'{a}th and Szab\'{o}'s original definition of Absolute gradings \cite{OSIntersectionForms}*{Section~7}).

We will write $\Sigma_i\in H_2(W_{\Lambda}(L))$ for the class obtained by capping the core of the 2-handle attached along $K_i$. We define
\[
\Sigma_i^{\ve{p}}=\begin{cases} \Sigma_i& \text{ if } w_i\in \ve{p}\\
-\Sigma_i & \text{ if } z_i\in \ve{p}.
\end{cases}
\]
We define $\Sigma^{\ve{p}}=\Sigma_1^{\ve{p}}+\cdots+\Sigma_\ell^{\ve{p}}$.

\begin{lem}
\label{lem:grading-shift-Gamma-M} If $\ve{p}$ and $\vec{M}$ are compatible, and we view $\ve{\CF}^-(S^3_{\Lambda}(L))$ as being supported in Alexander multi-grading $(0,\dots, 0)$, then the map $\Gamma^{\vec{M}}_{\frt}$ has Alexander grading
\[
A^{\ve{p}}_i(\Gamma^{\vec{M}}_\frt)=\frac{\left \langle c_1(\frt), \Sigma_i^{\ve{p}}\right\rangle-\Sigma^{\ve{p}}\cdot \Sigma_i^{\ve{p}} }{2}.
\]
\end{lem}

\begin{proof} 
We consider first the case that $|\vec{M}|=|L|$. The map $\Gamma^{\vec{M}}$ will count holomorphic polygons on diagrams of the form 
\[
(\Sigma, \as,  \bs_{\vec{0}}, \bs_{\veps_1},\dots, \bs_{\veps_{\ell-1}}, \bs_{\vec{1}})
\]
where $\vec{0}<\veps_1<\dots<\veps_{\ell-1}<\vec{1}$ are points in $\bE_\ell$. A polygon class
\[
\psi\in \pi_2(\as, \bs_{\vec{0}}, \bs_{\veps_1},\dots, \bs_{\veps_{\ell-1}}, \bs_{\vec{1}}, \bs_{\Lambda})
\]
may always be decomposed into the splice of two classes
\[
\psi_1\in \pi_2(\as, \bs_{\vec{0}}, \bs_{\vec{1}}) \qquad \text{and} \qquad  \psi_2\in \pi_2(\bs_{\vec{0}}, \bs_{\veps_1},\dots, \bs_{\veps_n}, \bs_{\vec{1}}, \bs_{\Lambda}).
\]
 The Alexander grading shift of $\psi_1$ can be computed as the Alexander grading shift of a link cobordism \cite{ZemAbsoluteGradings}*{Theorem~2.14 (2)}. Using the aforementioned grading shift formula, we see that the Alexander grading shift of $\psi_1$ coincides with the formula in the statement. The Alexander grading shift of $\psi_2$ is zero, since all the input morphisms (such as $\lb\theta^\sigma_i, \phi^\sigma_i\rb$, and so forth) preserve Alexander grading, and the net $A^{\ve{p}}_i$ Alexander grading shift is the intersection of the beta boundary components of this class with the shadow of $\pm K_i$. This intersection number is always zero for $\psi_2$ since there are only beta curves on the Heegaard diagram $(\Sigma, \bs_{\vec{0}}, \bs_{\veps_1},\dots, \bs_{\veps_{\ell-1}},\bs_{\Lambda})$. This proves the claim when $|\vec{M}|=|L|$. 

We now consider the claim for a general $\vec{M}$. This may be derived from the above as follows. We let $\vec{L}$ be the orientation of $L$ induced by a complete collection $\ve{p}$ compatible with $\vec{M}$. Write $\vec{J}=\vec{L}\setminus \vec{M}$. Clearly $\Gamma^{\vec{M}}\circ \Phi^{\vec{J}}$ has the same $A^{\ve{p}}$ grading as $\Gamma^{\vec{L}}$. On the other hand, $\Phi^{\vec{J}}$ preserves $A^{\ve{p}}$ by definition, so $\Gamma^{\vec{M}}$ and $\Gamma^{\vec{L}}$ have the same $A^{\ve{p}}$ grading, completing the proof. 
\end{proof}

 Since the codomain of $\Gamma^{\vec{M}}$ is supported only in Alexander grading $(0,\dots, 0)$, we obtain the following corollary from Lemma~\ref{lem:grading-shift-Gamma-M}:

\begin{cor}\label{cor:non-vanishing}
 If $\vec{M}$ and $\ve{p}$ are compatible, then the only $\Spin^c$ structure $\frt$ such that the map $\Gamma^{\vec{M}}_{\frt}$ is non-vanishing on $A^{\ve{p}}$ grading $\ve{s}=(s_1,\dots, s_\ell)$ is $\frt=\frz_{\ve{s}}^{\ve{p}}$ defined by
\begin{equation}
-s_i=\frac{\left \langle c_1(\frz_{\ve{s}}^{\ve{p}}), \Sigma_i^{\ve{p}}\right\rangle-\Sigma^{\ve{p}}\cdot \Sigma_i^{\ve{p}} }{2},
\label{eq:defining-equation-frz-s}
\end{equation}
for all $i$. 
\end{cor}

We observe that on $\cC^{\vec{0}}_{\Lambda}(L)$, Lemma~\ref{lem:Alexander-grading-change-signs} implies that 
\begin{equation}
A^{\ve{p}}=\delta^{\ve{p}}\circ A^{\ws} \label{eq:Alexander-grading-shift}
\end{equation}
where $\delta^{\ve{p}}\colon \Q^n\to \Q^n$ is 
\[
(x_1,\dots, x_\ell)\mapsto (\delta_1\cdot x_1,\dots, \delta_\ell \cdot x_\ell)
\]
where $\delta_i$ is $1$ is $w_i\in \ve{p}$ and $-1$ if $z_i\in \ve{p}$.

\begin{lem}\label{lem:gr_p-shift}
 If $\xs$ is homogeneously graded element of $\bX_{\Lambda}^{\vec{0}}(L,\ve{s})=\cCFL(L,\ve{s})$, then the quantity
\begin{equation}
\gr_{\ve{p}}(\xs)+\frac{c_1(\frz^{\ve{p}}_{\delta^{\ve{p}}(\ve{s})})^2}{4} \label{eq:grw-c1^2}
\end{equation}
is independent of the choice of $\ve{p}$. 
\end{lem}

\begin{proof}
We consider the effect on Equation~\eqref{eq:grw-c1^2} of changing one $w_i\in \ve{p}$ to a $z_i$. Write $\ve{p}'$ for the $\{z_i\}\cup \ve{p}\setminus \{w_i\}$. By Lemma~\ref{lem:Alexander-grading-change-signs}, we have
\begin{equation}
\gr_{\ve{p}'}(\xs)=\gr_{\ve{p}}(\xs)-2A_i^{\ve{p}}(\xs)-\lk(L^{\ve{p}}\setminus K_i, K_i). \label{eq:grW'-v-grW}
\end{equation}
(Here $L^{\ve{p}}$ denotes $L$ oriented to intersect the Heegaard surface negatively at $\ve{p}$). 
By direct computation applied to Equation~\eqref{eq:defining-equation-frz-s}, we see that
\[
\frz_{\delta^{\ve{p}'}(\ve{s})}^{\ve{p}'}=\frz_{\delta^{\ve{p}}(\ve{s})}^{\ve{p}}-\Sigma_i. 
\]
Hence
\begin{equation}
\begin{split}
\frac{c_1(\frz_{\delta^{\ve{p}'}(\ve{s})}^{\ve{p}'})^2}{4}=&\frac{c_1(\frz_{\delta^{\ve{p}}(\ve{s})}^{\ve{p}})^2}{4}-\langle c_1(\frz_{\delta^{\ve{p}}(\ve{s})}^{\ve{p}}), \Sigma_i \rangle +\Sigma_i\cdot \Sigma_i\\
&=\frac{c_1(\frz_{\delta^{\ve{p}}(\ve{s})}^{\ve{p}})^2}{4}-2\frac{\langle c_1(\frz_{\delta^{\ve{p}}(\ve{s})}^{\ve{p}}), \Sigma_i^{\ve{p}} \rangle-\Sigma^{\ve{p}}\cdot \Sigma_i^{\ve{p}}}{2} -(\Sigma^{\ve{p}}-\Sigma_i^{\ve{p}})\cdot \Sigma_i^{\ve{p}}\\
&=\frac{c_1(\frz_{\delta^{\ve{p}}(\ve{s})}^{\ve{p}})^2}{4}+2s_i-\lk(L^{\ve{p}}\setminus K_i, K_i).
\end{split}
\label{eq:c1W'-v-c1W}
\end{equation}
Adding Equations~\eqref{eq:grW'-v-grW} and \eqref{eq:c1W'-v-c1W} yields the main result.
\end{proof}

We are now in position to prove Theorem~\ref{thm:grading}:

\begin{proof}[Proof of Theorem~\ref{thm:grading}] The remainder of the proof is similar to the proof of Lemma~\ref{lem:grading-shift-Gamma-M}. We first establish the formula for cube point $\veps=\vec{0}$, and then subsequently demonstrate for more general $\veps$. Let $\vec{L}$ be an orientation on $L$ and $\ve{p}$ a compatible complete collection of basepoints. Corollary~\ref{cor:non-vanishing} and Equation~\eqref{eq:Alexander-grading-shift} imply that on 
$A^{\ws}$-grading $\ve{s}$, only $\Spin^c$ structure $\frz_{\delta^{\ve{p}}_{\ve{s}}}^{\ve{p}}$ contributes to $\Gamma^{\vec{L}}$. 

Therefore, the standard absolute grading formula of Ozsv\'{a}th and Szab\'{o} implies that on $A^{\ws}$-grading $\ve{s}$, we have
\[
\gr_{\ve{p}}(\Gamma^{\vec{L}},\ve{s})=\frac{c_1(\frz_{\delta^{\ve{p}}_{\ve{s}}}^{\ve{p}})^2-2 \chi(W_{\Lambda}(L))-3 \sigma(W_{\Lambda}(L))}{4}+|L|.
\]
(Here, we are using the notation that if $F$ is a map and $\ve{s}\in \bH(L)$ then $\gr_{\ws}(F,\ve{s})$ is the grading shift of $F$ when restricted to Alexander grading $\ve{s}$).
Noting that $\gr_{\ws}=\gr_{\ps}$ on $\ve{\CF}^-(S^3_\Lambda(L))$, the above equation combined with Lemma~\ref{lem:gr_p-shift} gives
\[
\gr_{\ws}(\Gamma^{\vec{L}},\ve{s})=\frac{c_1(\frz_{\ve{s}})^2-2 \chi(W_{\Lambda}(L))-3 \sigma(W_{\Lambda}(L))}{4}+|L|.
\]

Next, let $\veps\in \bE_{\ell}$ be arbitrary and let $\vec{M}$ be an orientation on the components $K_i$ of $L$ such that $\veps_i=0$.  Let $\vec{J}$ be the components of $L$ which are not in $M$ (i.e. $K_i$ for which $\veps_i=1$), oriented the same as $L$. Note that $\Phi^{\vec{J}}$ preserves the Alexander grading $A^{\ws}$ and increases the $\gr_{\ws}$-grading by $|\veps|-1$. Hence, 
\[
\gr_{\ws}(\Gamma^{\vec{M}}, \ve{s})=\gr_{\ws}(\Gamma^{\vec{M}}\circ \Phi^{\vec{J}},\ve{s})+1-|\veps|=\gr_{\ws}(\Gamma^{\vec{L}},\ve{s})-|\veps|. 
\]
Combined with the above equation for $\gr_{\ve{p}}(\Gamma^{\vec{L}},\ve{s})$, the proof is complete.
\end{proof}

\subsection{Gradings on the sublink surgery formula}

We now discuss the case of the sublink surgery formula. We suppose that $J\cup L\subset S^3$ is a partitioned link and $J$ is equipped with Morse framing $\Lambda$ such that $S^3_{\Lambda}(J)$ is a rational homology 3-sphere. If $K_i\in J$, write $\Sigma_i\subset W_{\Lambda}(L)$ for the core of the 2-handle attached along $K_i$. If $K_i\subset L$, write $\hat{\Sigma}_i$ for the surface obtained by capping the shadow of $K_i$ in $W_{\Lambda}(J)$ with a rational Seifert surface of the image of $K_i$ in $S_{\Lambda}^3(J)$. We write $\Sigma$ for the sum of all $\Sigma_i$ and $\hat{\Sigma}_i$. If $\ve{s}\in \bH(L)$, we write $\frz_{\ve{s}}^J\in \Spin^c(W_\Lambda(L))$ for the $\Spin^c$ structure which satisfies Equation~\eqref{eq:definition-z-s} for all components of $J$.  

\begin{thm}
Suppose that $J\cup L\subset S^3$ and $\Lambda$ is an integral framing on $J$ such that $S^3_{\Lambda}(L)$ is a rational homology 3-sphere. Then the isomorphism $\ve{\cCFL}(S^3_{\Lambda}(J),L)\simeq \bX_{\Lambda}(J,L)$ from Corollary~\ref{cor:Manolescu-Ozsvath-subcube} is absolutely Maslov and Alexander graded if we equip $\bX_{\Lambda}^{\veps}(J,L,\ve{s})$ with the Maslov and Alexander gradings
\[
\begin{split}
\tilde{\gr}_{\ws}&=\gr_{\ws}+\frac{c_1\left(\frz_{\ve{s}}^{J}\right)^2-2\chi(W_{\Lambda}(J))-3\sigma(W_{\Lambda}(J))}{4}+|J|-|\veps|\\
\tilde{A}_i&=s_i+\frac{\langle c_1\left(\frz_{\ve{s}}^{J}\right),\hat{\Sigma}_i\rangle -\Sigma \cdot \hat \Sigma_i}{2}.
\end{split}
\]
where $i\in \{1,\dots, |L|\}$ and $\gr_{\ws}$ denotes the internal $\gr_{\ws}$-grading on $\cCFL(L\cup J)$.
\end{thm}

The proof is similar to the proof of Theorem~\ref{thm:grading}, and we leave the details to the reader.

\section{The \texorpdfstring{$H_1(Y)/\Tors$}{H1/Tors action}}
\label{sec:H1-action}

Suppose $L\subset S^3$ is a link with framing $\Lambda$. We recall that Ozsv\'{a}th and Szab\'{o} defined an action of $\Lambda^* (H_1(Y)/\Tors)$ on $\HF^-(Y,\frs)$ \cite{OSDisks}*{Section~4.2.5}. In this section, we describe how to compute this action using the link surgery formula.

 We first describe an action of $\Lambda^* (H_1(S_\Lambda^3(L))/\Tors)$ on $\ve{\CF}^-(S^3_{\Lambda}(L))$. We recall that $H_1(S_\Lambda^3(L))\iso \Z^n/\im \Lambda$. In particular, the meridians $\mu_1,\dots, \mu_\ell$ of the link components $K_1,\dots, K_\ell$ span  $H_1(S_{\Lambda}^3(L))$. We define endomorphisms of $\bX_{\Lambda}(L)$ via the formulas
\[
\frA_{\mu_i}=\sum_{\substack{\vec{M}\subset L\\ +K_i\subset \vec{M}}}\Phi^{\vec{M}}.
\]
 This action on the link surgery complex appeared in \cite{ZemHFLattice}, though it was not proven to coincide with the action of $H_1(S^3_\Lambda(L))/\Tors$. 

In the case of 0-surgery on a knot in $S^3$, the computation of the homology action is equivalent to the computation of the twisted Floer complex, a proof of which can be found in \cite{NiPropertyG}*{Theorem~3.4}.
 
 \begin{rem} It is straightforward to see that as endomorphisms of $\bX_\Lambda(L)$, we have
 \[
 \sum_{\substack{\vec{M}\subset L\\ +K_i\subset \vec{M}}}\Phi^{\vec{M}}\simeq  \sum_{\substack{\vec{M}\subset L\\ -K_i\subset \vec{M}}}\Phi^{\vec{M}}.
 \]
It is also possible to see that the action $\frA_{\mu_i}$ naturally descends to an action of $\Lambda^* (H_1(S_\Lambda^3(L))/\Tors)$.  See \cite{ZemHFLattice}*{Section~4.1}.
 \end{rem}
 
 \begin{thm}Let $[\mu_1],\dots, [\mu_\ell]$ denote the classes of the meridians of $K_1,\dots,K_\ell$, viewed as elements of $H_1(S^3_\Lambda(L))$. Under the canonical homotopy equivalence
 \[
 \ve{\CF}^-(S^3_\Lambda(L))\simeq \bX_{\Lambda}(L),
 \]
 the map $A_{[\mu_i]}$ of $\Lambda^*\left( H_1(S^3_{\Lambda}(L))/\Tors\right)$ on $\ve{\CF}^-$ is intertwined with the map $\frA_{\mu_i}$ on $\bX_{\Lambda}(L)$. 
 \end{thm}
 \begin{proof} We begin with a model computation in the genus one case, from which the main result will follow in general. Let $[\mu]$ first denote a translate of $\b_0$ on $\bT^2$.  We build the following generalized diagram of attaching curves on $\bT^2$
 \begin{equation}
 \begin{tikzcd}[column sep=2.5cm, row sep=1cm, labels=description]
  \b_\lambda \ar[d, "\theta_{\lambda,0}"] \ar[r, "\mu"]& \b_\lambda \ar[d, "\theta_{\lambda,0}"] \\
 \b_0 \ar[drr, "\theta^+_\sigma"] \ar[dr, "\theta_\tau^++\theta_\sigma^+",swap] & \b_0 \ar[dr, "\theta_\tau^++\theta_\sigma^+"]\\
 &\b_1 & \b_1
 \end{tikzcd}
 \label{eq:mu-hypercube}
 \end{equation} 
 The diagram is generalized in the sense that we allow the morphism $\mu$, which we interpret as generalized morphism from $\b_{\lambda}$ to $\b_{\lambda}$. Here, $\mu$ denotes a simple closed curve on $\bT^2$ obtained by translating $\b_0$. See Figure~\ref{fig:29}.

 \begin{figure}[ht]
 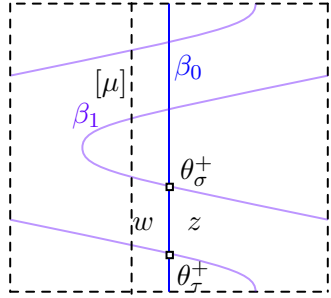
 \caption{The translate $[\mu]$ (dashed) of the meridian on $\bT^2$.}
 \label{fig:29}
 \end{figure}

  For our purposes, we interpret associativity as described in \cite{HHSZNaturality}*{Section~7.2}, as we recall briefly. Given a sequence $\theta_1,\dots, \theta_{j-1},\mu,\theta_{j+1},\dots,\theta_n$ of composable morphisms, we define the holomorphic polygon map on the above sequence to be obtained by counting ordinary holomorphic polygons with inputs $\theta_1,\dots, \theta_{j-1}, \theta_{j+1},\dots, \theta_n$, weighted by the number $\#(\d_{\g_j}(\psi)\cap \mu)$, where $\psi$ is the class of the polygon and $\gs_j$ is the Lagrangian which contains both $\theta_{j-1}$ and $\theta_{j+1}$. Standard $A_\infty$ associativity relations hold. See \cite{HHSZNaturality}*{Lemma~7.3}.

   We now claim that diagram in Equation~\eqref{eq:mu-hypercube} satisfies the hypercube relations.    There is one non-trivial composition to be checked, namely that
    \[
    \mu_{3}([\mu], \theta_{\lambda,0}, \theta_\tau^++\theta_\sigma^+)=\mu_2(\theta_{\lambda,0},\theta_\sigma^+). 
    \]
    This amounts to the fact that holomorphic triangle $u_\tau$ in Figure~\ref{fig:18} with $\theta_\tau^+$ input has $\d_{\b_{\lambda}}(u_\tau)\cap \mu\equiv 0$, while the triangle $u_\sigma$ with $\theta^+_\sigma$ input has  $\#\d_{\b_{\lambda}}(u_\sigma)\cap [\mu]\equiv 1$.

    We remark that the hypercube relations depend on which direction we translate $\mu$ off of $\b_0$. See Figure~\ref{fig:29}.  If we translated in the opposite direction, we would need to put $\theta_\tau^+$ along the diagonal of the bottom face instead of $\theta_\sigma^+$.

 From here, the main theorem follows quickly. We take the hypercube in Equation~\eqref{eq:mu-hypercube}, and tensor with the genus one hypercubes for the other axis directions.
 The result is the hypercube
 \[
  \begin{tikzcd}[column sep =3cm,labels=description, row sep=1.2cm]
   \b_{\lambda_i}\times\bs_{\Lambda'} \ar[d, "\theta_{\lambda,0}\otimes \theta_{\Lambda',0}"] \ar[r, "\mu"]& \b_{\lambda_i}\times\bs_{\Lambda'} \ar[d, "\theta_{\lambda,0}\otimes\theta_{\Lambda',0}"] \\
  \b_0\times\scB_{\Lambda'} \ar[drr, "\theta^+_\sigma\otimes\theta^+"] \ar[dr, "\theta_\tau^+\otimes\theta^+ + \theta_\sigma^+\otimes\theta^+",swap] & \b_0\times\scB_{\Lambda'} \ar[dr, "\theta_\tau^+|\theta^++\theta_\sigma^+|\theta^+"]\\
  &\b_1\times\scB_{\Lambda'} & \b_1\times\scB_{\Lambda'}
  \end{tikzcd}
 \]
 This hypercube takes place on the disjoint union of $n$-copies of $\bT^2$. Here, $\lambda$ denotes the longitude of the component corresponding to $\mu$ and $\Lambda'$ denotes the other longitudes. By Proposition~\ref{prop:1-handle-functor}, attaching 1-handles to connect different components and to increase the genus preserves the hypercubes relations. Pairing with attaching curves $\as$ gives the statement.
 \end{proof}

 \begin{rem} Instead of viewing $\mu$ as a morphism from $\b_{\lambda}$ to itself, we could replace one of the $\b_{\lambda}$ Lagrangians above with a small translate $\b_{\lambda}'$ of itself, and replace $\mu$ with the homology action $A_{\mu}(\theta^+)$, where $\theta^+\in \b_\lambda\cap \b_{\lambda}'$ is the top graded intersection point.
 \end{rem}

\section{The  surgery exact triangle in the \texorpdfstring{$U$}{U}-adic topology}

\label{sec:U-adic-surgery}

In \cite{ZemBordered}*{Section~18.3}, we proved that the type-$D$ modules for solid tori over $\cK$ (with chiral topology) satisfied the surgery exact triangle, in the sense that for any $n$, there is a type-$D$ morphism $f^1$ so that
\[
\cD_\infty^\cK\simeq \Cone(f^1\colon \cD_n^{\cK}\to \cD_{n+1}^{\cK}).
\]
 In this section prove the analogous statement for $U$-adic topology $\frK$. Note that the argument from \cite{ZemBordered} made some properties which hold for $\cK$ but not $\frK$. For example, we used the fact that $1+\scU$ and $1+\scV$ are units in $\ve{I}_0\cdot \cK\cdot \ve{I}_0$, which is not true for $\frK$.

\begin{thm}
\label{thm:exact-triangle} There is a homotopy equivalence
\[
\frD_\infty^\frK\simeq \Cone(f^1\colon \frD_n^{\frK}\to \frD_{n+1}^{\frK}),
\]
for some type-$D$ morphism $f^1$. 
\end{thm}

We now describe $f^1$. Write $\ve{X}_0$, $\ve{X}_1$ for the two generators of $\frD_n^{\frK}$, and write $\ve{Y}_0$, $\ve{Y}_1$ for the two generators of $\frD_{n+1}^{\frK}$.

In idempotent 0, our map $f^1$ satisfies
\[
f^1(\ve{X}_0)=\ve{Y}_0\otimes (1+\a),
\]
where
\[
\a=\sum_{s\ge 1} (\scU^s+\scV^s)U^{s(s-1)/2}.
\]
In idempotent 1, the map $f^1$ satisfies
\[
f^1(\xs_1)=\ys_1\otimes (1+\a'),
\]
where $\a'=\phi^\sigma(\a)$. 
Therefore
\[
\Cone(f^1)=\begin{tikzcd}[row sep=2cm, column sep=1.5cm]
\ve{X}_0
	\ar[d, "\sigma+T^n \tau"]
	\ar[r, "1+\a"]
 & \ve{Y}_0
	\ar[d, "\sigma+T^{n+1} \tau"]
 \\
\ve{X}_1
	\ar[r, "1+\a'"]
& \ve{Y}_1
\end{tikzcd}
\]

We recall from Section~\ref{sec:examples} that
\[
\frD_\infty^{\frK}=\begin{tikzcd}
\xs_0^+
	\ar[r, "1+\scU"]
	\ar[dr, "T^{-1}\tau"]
&
\xs_0^-
	\ar[dr, "\tau", pos=.2,swap]
&
\ys_0^+
	\ar[r, "1+\scV"]
	\ar[dl, "\sigma", pos=.2, crossing over]
	&
\ys_0^-
	\ar[dl, "\sigma"]
\\[1cm]
& \zs_1^+
	\ar[r, "1+T"]
&\zs_1^-
\end{tikzcd}
\]

We prove Theorem~\ref{thm:exact-triangle} in the subsequent sections.

\subsection{Factorizations} 
\label{sec:factorizations}

The proof of Theorem~\ref{thm:exact-triangle} reduces to several algebraic identities which we state and prove in this section. 

We begin by defining the elements of interest:
\begin{equation}
\begin{split}
\b_1&=
\sum_{s\ge 1} (\scV^{-s+1}+\cdots+\scV^{-1}+1+\scV+\cdots +\scV^{s-1}) U^{s(s-1)/2}
\\
\b_{-1}&=\sum_{s\ge 1} (\scU^{-s+1}+\cdots+\scU^{-1}+1+\scU+\cdots +\scU^{s-1}) U^{s(s-1)/2}\\
\delta_1&=\sum_{s\ge 2} \frac{\scV^{-s+1}+\cdots+\scV^{-1}+\scV+\cdots +\scV^{s-1}}{1+\scV} U^{s(s-1)/2}\\
\delta_{-1}&=\sum_{s\ge 2} \frac{\scU^{-s+1}+\cdots+\scU^{-1}+\scU+\cdots +\scU^{s-1}}{1+\scU} U^{s(s-1)/2}\\
\epsilon&=\sum_{s\ge 1} (1+\scU+\cdots +\scU^{s-1})(1+\scV+\dots +\scV^{s-1})U^{s(s-1)/2}\\
\b'&=\sum_{s\ge 1} (T^{-s+1}+\cdots+T^{-1}+1+T+\cdots +T^{s-1}) U^{s(s-1)/2}\\
\delta'&=\sum_{s\ge 2} \frac{T^{-s+1}+\cdots+T^{-1}+T+\cdots +T^{s-1}}{1+T} U^{s(s-1)/2}.
\end{split}
\label{eq:definitions-series}
\end{equation}

\begin{lem}
\label{lem:factorizations} The formulas in Equation~\eqref{eq:definitions-series} determine convergent power series in the $U$-adic topology. Furthermore, the following relations hold:
\begin{enumerate}
\item $1+\a=(1+\scV) \b_1$;
\item $1+\a=(1+\scU) \b_{-1}$;
\item $\epsilon$ is a unit.
\item $1+\a=(1+\scU)(1+\scV) \epsilon$;
\item $\b_{-1}=(1+\scV) \epsilon$;
\item $\b_{1}=(1+\scU) \epsilon$;
\item $\b_1+\sum_{s\ge 1} U^{s(s-1)/2}=(1+\scV) \delta_1$;
\item $\b_{-1}+\sum_{s\ge 1} U^{s(s-1)/2}=(1+\scU) \delta_{-1}$.
\end{enumerate}
\end{lem}
\begin{proof} All of the above computations are straightforward to verify (cf. Remark~\ref{rem:polynomials-and-tiles}, below). We remark that $\epsilon$ is a unit in $\frK$ since we may write it as $1+Uf$, for some infinite series $f$. The inverse is given by $\sum_{i=0}^\infty (Uf)^i$. 
\end{proof} 

There is an additional factorization result which is more challenging to prove:

\begin{lem} There is a unique $\kappa$ in the $U$-adic completion of $\bF[\scU,\scV]$ satisfying
\begin{equation}
\b_{1}+\b_{-1}+\sum_{s\ge 1} U^{s(s-1)/2}=\kappa(1+\a).
\label{eq:kappa-factorization}
\end{equation}
\end{lem}
\begin{proof}Uniqueness is clear: the $U$-adic completion of $\bF[\scU,\scV]$ is an integral domain.

Existence is proven as follows. We write 
\[
1+\a=(1+\scU)(1+\scV) \epsilon.
\]
Since $\epsilon$ is a unit, it suffices to show the left-hand side of Equation~\eqref{eq:kappa-factorization} is divisible by $(1+\scU)(1+\scV)$ in the $U$-adic topology, since then we can take
\[
\kappa=\epsilon^{-1}\frac{\b_1+\b_{-1}+\sum_{s\ge 1}U^{s(s-1)/2}}{(1+\scU)(1+\scV)}.
\]

We now show that left-hand side of Equation~\eqref{eq:kappa-factorization} is divisible by $(1+\scU)(1+\scV)$. We first compute that
\begin{equation}
\begin{split}
&\b_{1}+\b_{-1}+\sum_{s\ge 1} U^{s(s-1)/2}
\\
=
&\sum_{s \ge 1} \left(1+(1+\scU^s)(1+\cdots +\scV^{s-1})+(1+\cdots+\scU^{s-1})(1+\scV^s)\right) U^{s(s-1)/2}.
\end{split}
\label{eq:b1b1sum}
\end{equation}
We define two helpful polynomials. If $s\ge 1$, we define
\[
\begin{split}
B_s^+&=(1+\cdots +\scU^s)+(\scV+\cdots+\scV^s)+\scU^s(1+\dots +\scV^{s-1})+\scV^s(1+\cdots +\scU^{s})\\
B_s^-&=(\scU+\cdots +\scU^s)+(\scV+\cdots+\scV^s)+\scU^s(1+\dots +\scV^{s-1})+\scV^s(1+\cdots +\scU^{s-1}).
\end{split}
\]
We observe that Equation~\eqref{eq:b1b1sum} can be rewritten as follows:
\begin{equation}
\b_{1}+\b_{-1}+\sum_{s\ge 1} U^{s(s-1)/2}=\sum_{\substack{s\ge 1\\ s\text{ odd}}} B_s^+\cdot U^{s(s-1)/2}+\sum_{\substack{s\ge 2\\ s\text{ even}}} B_s^-\cdot U^{s(s-1)/2}.
\label{eq:b1b1-sum}
\end{equation}
We make the following claim, which is straightforward to verify: if $s$ is odd, then $B_s^+$ is a multiple of $(1+\scU)(1+\scV)$; and if $s$ is even, then $B_s^-$ is a multiple of $(1+\scU)(1+\scV)$.  To see this, we observe that if $s$ is odd, then
\begin{equation}
B_s^+=(1+\scU)(1+\scV)\sum_{\substack{0\le i,j\le s-1\\  i+j\in 2\Z}} \scU^{i} \scV^{j}
\label{eq:Bs-factor-1}
\end{equation}
and if $s$ is even, then
\begin{equation}
B_s^-=(1+\scU)(1+\scV)\sum_{\substack{0\le i,j\le s-1\\  i+j\in 2\Z+1}} \scU^{i} \scV^{j}.
\label{eq:Bs-factor-2}
\end{equation}
In light of Equation~\eqref{eq:b1b1-sum}, we conclude that $\b_1+\b_{-1}+\sum_{s\ge 1} U^{s(s-1)/2}$ is divisible by $(1+\scU)(1+\scV)$, so the proof is complete.
\end{proof}

\begin{rem}
\label{rem:polynomials-and-tiles} The relations from Lemma~\ref{lem:factorizations} and
 Equations~\eqref{eq:b1b1-sum}, ~\eqref{eq:Bs-factor-1} and~\eqref{eq:Bs-factor-2} have simple graphical proofs and interpretations, as we now describe. We think of a polynomial in $\bF[\scU,\scV]$ as a union of unit squares $[a,a+1]\times [b,b+1]$ contained in $[0,\infty)\times [0,\infty)$, so that $a,b\in \N$. Here, $[i,i+1]\times [j,j+1]$ corresponds to $\scU^i\scV^j$. (Call such a region a \emph{tile}). Then $B_s^+$ consists of the boundary tiles of the region $[0,s]\times [0,s]$, and $B_s^-$ consists of the boundary tiles of $[0,s]\times [0,s]$ with the corners $[0,1]\times [0,1]$ and $[s-1,s]\times [s-1,s]$ also removed. Note that $(1+\scU)(1+\scV)$ corresponds to the region $[0,2]\times [0,2]$. These regions are shown in Figure~\ref{fig:49}.  Equation~\eqref{eq:b1b1-sum} corresponds to the fact that the left region shown therein can be written as a union of the $B_s^-$ and $B_s^+$ regions as shown.
\end{rem}

\begin{figure}[h]
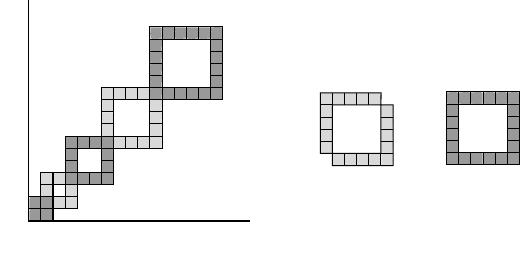
\caption{Regions from Remark~\ref{rem:polynomials-and-tiles}.}
\label{fig:49}
\end{figure}

\begin{lem}
We have the following algebraic relations involving $\phi^\sigma$ and $\phi^\tau$:
\begin{enumerate}
\item $\phi^\tau(1+\a)=T^{-1}(1+\a');$
\item $\phi^\sigma(\b_1)=\phi^\tau(\b_{-1})=\b'$;
\item  $1+\a'=(1+T)\b'$;
\item $\b'$ is a unit;
\item There is a unique $\kappa'$ in the $U$-adic completion of $\bF[U,T,T^{-1}]$ such that
\[
\kappa'(1+\a')=\b'+\sum_{s\ge 1} U^{s(s-1)/2}.
\]
\item There is a unique $\delta'$ in the $U$-adic completion of $\bF[U,T,T^{-1}]$ such that
\[
\delta'(1+T)=\b'+\sum_{s\ge 1} U^{s(s-1)/2}
\]
\end{enumerate}
\end{lem} 
\begin{proof} The first three claims are straightforward. For the third claim, we observe that $(1+\a')=\b'(1+T)$, and $\b'$ is a unit, so it suffices to show that $\b'+\sum_{s\ge 1} U^{s(s-1)/2}$ is divisible by $(1+T)$. To this end, we compute that
\[
\b'+\sum_{s\ge 1} U^{s(s-1)/2}=\sum_{s\ge 2} (T^{s-1}+\cdots +T+T^{-1}+\cdots +T^{1-s}) U^{s(s-1)/2},
\]
which is clearly divisible by $1+T$. The existence of $\delta'$ follows from the same reasoning.
\end{proof}

\subsection{Proof of Theorem~\ref{thm:exact-triangle}}

We will describe type-$D$ morphisms
\[
F^1\colon \Cone(f^1\colon \frD_0^{\frK}\to \frD_1^{\frK})\to \frD_\infty^{\frK}\quad \text{and} \quad G^1\colon \frD_\infty^{\frK}\to \Cone(f^1\colon \frD_0^{\frK}\to \frD_1^{\frK})
\]
and show that 
\[
\d(F^1)=0, \quad \d (G^1)=0,\quad  F^1\circ G^1\simeq  \sum_{s \ge 1} U^{s(s-1)/2}\id, \quad \text{and} \quad G^1\circ F^1\simeq  \sum_{s \ge 1} U^{s(s-1)/2}\id.
\] 
Since $\sum_{s \ge 1} U^{s(s-1)/2}$ is a unit in the $U$-adic topology, this will complete the proof.
The above equations will turn out to be a consequence of the factorizations from Section~\ref{sec:factorizations}.

\begin{figure}
\[
\begin{tikzcd}[column sep=1.3cm, row sep=1.8cm, labels=description]
\Xs_0
	\ar[r,"1+\a"]
	\ar[d, "\sigma+T^n\tau"]
	\ar[rrr,dashed, bend left, pos=.9, "\b_{-1}"]
	\ar[rrrrr,dashed, bend left=35, "\b_1", pos=.93]
&
\Ys_0 
	\ar[d, "\sigma+T^{n+1} \tau"]
	\ar[r,bend left, crossing over, dashed, "1", pos=.7]
	\ar[rrr, bend left, crossing over, dashed, "1", pos=.9]
	\ar[drrr,dashed, " \tfrac{1+T^{n+1}}{1+T} \tau",sloped, pos=.3]
&[2cm]
\xs_0^-
	\ar[from=r, "1+\scU"]
	\ar[dr,crossing over, "\tau"]
&
\xs_0^+
	\ar[dr, "T^{-1}\tau",pos=.3]
&
\ys_0^-
	\ar[from=r, "1+\scV"] 
	\ar[dl, "\sigma",crossing over,pos=.3]
	&
\ys_0^+
	\ar[dl, "\sigma"]
\\
\Xs_1
	\ar[r, "1+\a'"]
	\ar[rrrr, bend right=35,dashed, "\b'"]
&
\Ys_1
	\ar[rr, bend right, dashed, "1"]
&&
\zs_1^-
	\ar[from=r, "1+T"]
&
\zs_1^+
\end{tikzcd}
\]
\[
\begin{tikzcd}[column sep=1.3cm, row sep=1.8cm, labels=description]
\xs_0^-
	\ar[from=r, "1+\scU"]
	\ar[dr, "\tau"]
	\ar[rrrrr,dashed, bend left=35, "\b_{-1}", pos=.08]
	\ar[drrrr,dashed, bend right=10, "(1+UT^{-1})^{-1} \sigma+\tfrac{1+T^{n+1}}{1+T}\tau",sloped, pos=.7]
&
\xs_0^+
	\ar[dr, "T^{-1}\tau",pos=.3]
	\ar[rrr,bend left,dashed, "1", pos=.1]
&
\ys_0^-
	\ar[from=r, "1+\scV"] 
	\ar[dl, "\sigma",crossing over,pos=.3]
	\ar[rrr,bend left,crossing over, dashed, "\b_1", pos=.2]
	\ar[drr,dashed, "T^n (1+UT)^{-1} \tau",sloped]
&
\ys_0^+
	\ar[dl,crossing over, "\sigma"]
	\ar[r,bend left,dashed, "1", pos=.3]
&[2.2cm]
\Xs_0
	\ar[r,"1+\a"] 
	\ar[d, "\sigma+T^n\tau"]
&
\Ys_0 
	\ar[d, "\sigma+T^{n+1} \tau"]
\\
&
\zs_1^-
	\ar[from=r, "1+T"]
	\ar[rrrr,bend right=35,dashed, "\b'"]
&
\zs_1^+
	\ar[rr,bend right, dashed, "1"]
&&
\Xs_1
	\ar[r, "1+\a'"]
&
\Ys_1
\end{tikzcd}
\]
\caption{The maps $F^1$ (top) and $G^1$ (bottom).}
\label{fig:F1G1}
\end{figure}
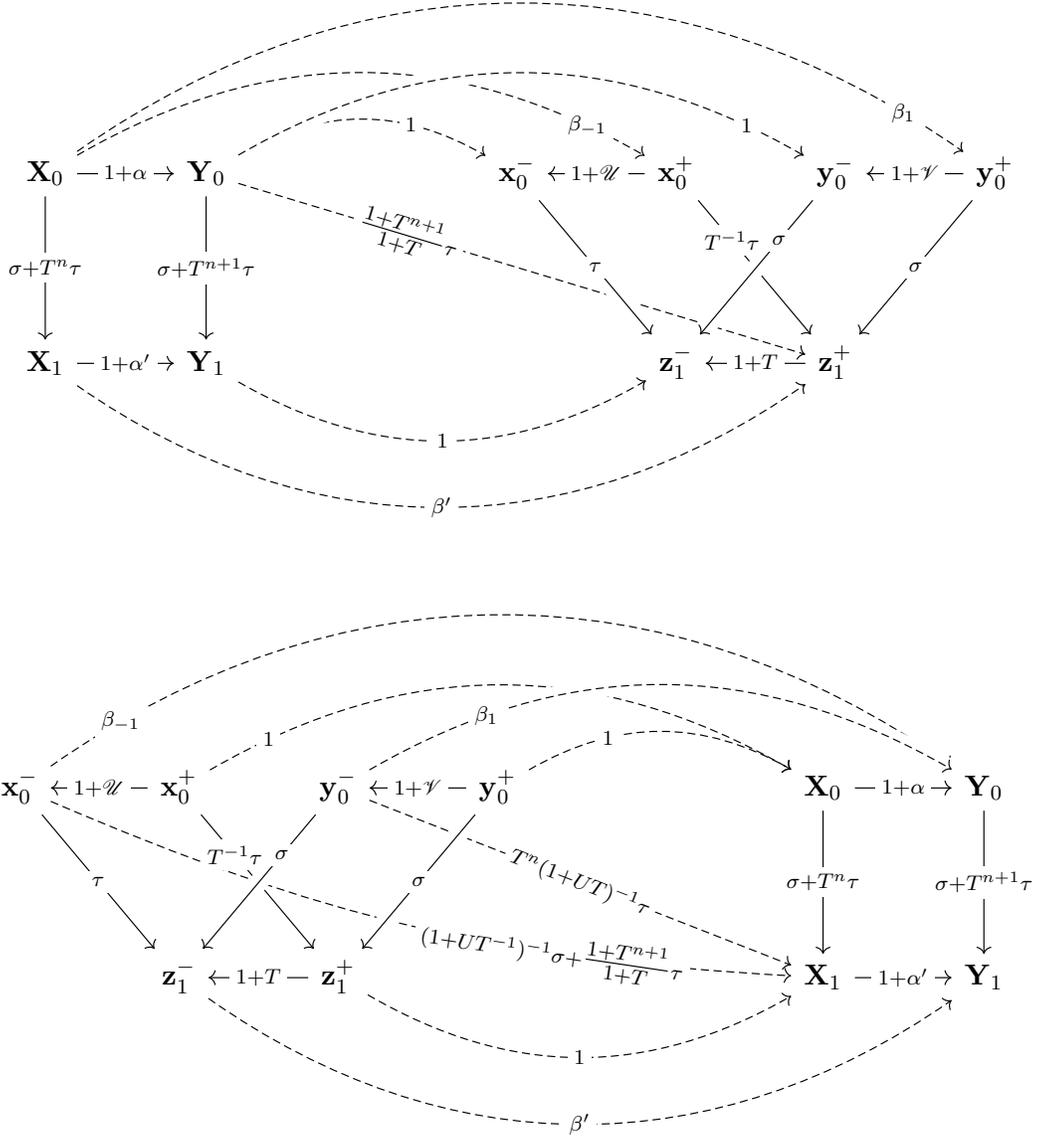

The maps $F^1$ and $G^1$ are shown in Figure~\ref{fig:F1G1}. Their compositions $G^1\circ F^1$ and $F^1\circ G^1$ are shown in Figure~\ref{fig:comps}. One may check that
\[
G^1\circ F^1+\cdot \sum_{s \ge 1} U^{s(s-1)/2}\id=\d(h^1),
\]
 there $h^1$ sends $\ve{Y}_0$ to $\ve{X}_0\otimes \kappa$ and $\ve{Y}_1$ to $\ve{X}_1\otimes \kappa'$.
Additionally, one may check that
\[
F^1\circ G^1+ \sum_{s\ge 1} U^{s(s-1)/2}\id=\d(j^1)
\]
where 
\[
\begin{split} j^1(\xs_0^-)&=\xs^+_0\otimes \delta_{-1}+\ys_0^+\otimes \epsilon\\
j^1(\ys_0^-)&=\ys_0^+\otimes \delta_{1}+\xs_0^+\otimes \epsilon\\
j^1(\zs_1^-)&=\zs_1^+\otimes \delta'.
\end{split}
\]
This completes the proof.

\begin{figure}[h]
\[
G^1\circ F^1=\hspace{-.5cm}\begin{tikzcd}[row sep=3cm, column sep=3cm,labels=description]
\Xs_0
	\ar[r, "1+\a"]
	\ar[d, "\sigma+T^n \tau"]
	\ar[loop left, looseness=30, "\b_1+\b_{-1}",dashed]
& 
\Ys_0
	\ar[d, "\sigma+T^{n+1} \tau"]
	\ar[loop right, looseness=30,crossing over, "\b_1+\b_{-1}",dashed]
	\ar[dl,dashed, "{ \begin{array}{c}
	\frac{T^n}{1+UT} \tau\\
	+\frac{1}{1+UT^{-1}} \sigma \end{array}}"]
\\
\Xs_1
	\ar[r, "1+\a'"]
	\ar[loop left, looseness=30,dashed, "\b'"]
& 
\Ys_1
	\ar[loop right, looseness=30,dashed, crossing over, "\b'"]
\end{tikzcd}
\]

\[
F^1\circ G^1=\hspace{-1cm}\begin{tikzcd}[column sep=2.6cm, row sep=4cm, labels=description]
\xs_0^+
	\ar[r, "1+\scU"]
	\ar[dr, "T^{-1}\tau", pos=.7,crossing over]
	\ar[loop above, looseness=30,dashed, crossing over, "\b_{-1}"]
	\ar[rr,bend left,dashed, "\b_1", pos=.3]
&
\xs_0^-
	\ar[dr, "\tau",pos=.7,crossing over]
	\ar[loop above, looseness=30,dashed, crossing over, "\b_{-1}"]
	\ar[rr,bend left,dashed,crossing over,"\b_1", pos=.7]
	\ar[d, dashed, "\frac{\b'}{(1+UT^{-1})} \sigma", pos=.45]
&
\ys_0^+
	\ar[r, "1+\scV"] 
	\ar[dl, "\sigma",crossing over,pos=.7]
		\ar[loop above, looseness=30,dashed, crossing over, "\b_{1}"]
		\ar[ll,bend left,dashed,crossing over, "\b_{-1}", pos=.7]
&
\ys_0^-
	\ar[dl, "\sigma", pos=.7]
		\ar[loop above, looseness=30,dashed, crossing over, "\b_{1}"]
		\ar[ll,bend left,dashed,crossing over,"\b_{-1}", pos=.3]
		\ar[dll,bend left=10, crossing over,dashed, pos=.35,"\frac{\b'}{(1+UT)} \tau"]
\\
& 
\zs_1^+
	\ar[r, "1+T"]
	\ar[loop below,dashed, looseness=30, "\b'"]
&
\zs_1^-
\ar[loop below,dashed, looseness=30, "\b'"]
\end{tikzcd}
\]
\caption{The compositions $G^1\circ F^1$ and $F^1\circ G^1$. Here, $\delta^1$ are solid arrows, while the $G^1\circ F^1$ and $F^1\circ G^1$ are  dashed arrows. }
\label{fig:comps}
\end{figure}

\subsection{More exact triangles}

There are several additional exact triangles satisfied by the bimodules reflecting the other exact triangles proven by Ozsv\'{a}th and Szab\'{o}. The simplest is an analog of the exact sequence
\[
\cdots \to\bB\to \Cone(v+h\colon\bA\to \bB)\to \bA\to\cdots.
\]
The natural analog for our type-$D$ modules is to define type-$D$ modules $\ve{i}_0^{\cK}$ and $\ve{i}_1^{\cK}$ which each have a single generator over $\bF$, and vanishing $\delta^1$ map. We declare $\ve{i}_0^{\cK}$ to be concentrated in idempotent 0 and $\ve{i}_1^{\cK}$ to be concentrated in idempotent 1. There is a morphism
\[
f^1\colon \ve{i}_0^{\cK}\to \ve{i}_1^{\cK}
\]
given by
\[
f^1(i_0)=i_1\otimes (\sigma+T^n \tau)
\]
and clearly
\[
\cD_n^{\cK}=\Cone(f^1\colon \ve{i}_0^{\cK}\to \ve{i}_1^{\cK}). 
\]

There is another exact triangle reflected by our type-$D$ modules. Ozsv\'{a}th and Szab\'{o} prove an additional exact triangle  \cite{OSIntegerSurgeries}*{Section~3} which takes the form
\[
\cdots \to\underline{\ve{\CF}}^-(Y)\to \ve{\CF}^-(Y_n(K))\to \ve{\CF}^-(Y_{n+m}(K))\to \cdots.
\] 
Here, $\underline{\ve{\CF}}^-(Y)$ is a version of Heegaard Floer homology with twisted coefficients. 

There is an analog for our type-$D$ modules. We define $\underline{\cD}_{\infty,\Z/m}^{\cK}$ 
to have two generators $\zs^+_1$ and $\zs_1^-$, both in idempotent 1, and to have differential
\[
\begin{tikzcd}
\zs_1^+\ar[r, "1+T^m"] & \zs_1^-.
\end{tikzcd}
\]

\begin{prop} If $n\in \Z$ and $m$ is a positive integer, then there is a type-$D$ morphism $f^1$ and a homotopy equivalence
\[
\underline{\cD}_{\infty,\Z/m}^\cK\simeq \Cone(f^1_m\colon \cD_n^{\cK} \to \cD_{n+m}^{\cK}).
\]
\end{prop}
\begin{proof}
The map $f^1_m$ is given by the following diagram:
\[
\begin{tikzcd}[labels=description, column sep=2cm, row sep =2cm]
\ve{x}_0
	\ar[r,dashed, "1+\a_m"]
	\ar[d, "\sigma+T^n \tau"]
& \ys_0
	\ar[d, "\sigma+T^{n+m} \tau"]
\\
\xs_1 \ar[r,dashed, "1+\phi^\sigma(\a_m)"] & \ys_1,
\end{tikzcd}
\]
where 
\[
\a_m=\sum_{s\ge 1} (\scU^{ms}+\scV^{ms}) U^{ms(s-1)/2}. 
\]
(Here, solid arrows denote $\delta^1$ and dashed arrows denote $f^1_m$). 

The homotopy equivalence may be derived from the $m=1$ case as follows. We define an algebra morphism $\phi_{n,m}\colon \cK\to \cK$, as follows. On algebra elements which are concentrated in a single idempotent, we set
\[
\phi_{n,m}(\scU^i \scV^j)=\scU^{im} \scV^{jm} \quad \text{and} \quad \phi_{n,m}(U^i T^j)=U^{mi} T^{mj}.
\]
Additionally, we set
\[
\phi_{n,m}(\sigma)=\sigma \quad \text{and} \quad \phi_{n,m}(\tau)=T^n \tau. 
\]
It is straightforward to check that $\phi_{n,m}(a\cdot b)=\phi_{n,m}(a)\cdot \phi_{n,m}(b)$ so there is an induced bimodule ${}_{\cK} [\phi_{n,m}]^{\cK}$.  We observe
\[
\cD_{p}^{\cK}\boxtimes {}_{\cK} [\phi_{n,m}]^{\cK}=\cD_{pm+n}^{\cK}\quad \text{and} \quad \cD_{\infty}^{\cK}\boxtimes {}_{\cK}[\phi_{n,m}]^{\cK}\simeq \underline{\cD}_{\infty,\Z/m}^{\cK}.
\]
Hence boxing the equivalence 
\[
\cD_\infty^{\cK}\simeq \Cone(f^1\colon \cD_{0}^{\cK}\to \cD_1^{\cK})
\]
with ${}_{\cK}[\phi_{n,m}]^{\cK}$ yields the equivalence in the statement. 
\end{proof}

\appendix

\section{Positivity}
\label{sec:positivity}

We recall that the maps appearing in our homotopy equivalence
\[
\ve{\CF}^-(Y_{\Lambda}(L))\simeq \bX_{\Lambda}(L)
\]
potentially involve negative powers of $\scV$ because of our description of the complexes and maps in terms of local systems (cf. Remark~\ref{rem:negative-powers}). In this section, we show that the negative powers of $\scV$ are an artifact of our presentation, and in fact the maps that we construct involve only nonnegative powers of the variables $\scU$, $\scV$ and $U$.

 Since we can write $\scV^{-1}=U^{-1} \scU$, we can view negative powers of $\scV$ instead as negative powers of $U$. Our main result is the following:

 \begin{thm}\label{thm:positivity} For appropriately chosen Heegaard diagrams, the maps appearing in the homotopy equivalence
 \[
 \ve{\CF}^-(Y^3_{\Lambda}(L))\simeq \bX_{\Lambda}(Y,L) 
 \]
involve no negative powers of $U$. More generally, if $Y$ is a manifold with torus boundaries, and $Y'$ is obtained by gluing a solid torus to $Y$ so that the meridian of the torus is glued to the chosen longitude of $Y$, then the maps appearing in the homotopy equivalence
 \[
 \cX(Y)^{\frL}\boxtimes {}_{\frK} \frD_0\simeq \cX(Y')^{\frL'}
 \] 
involve no negative powers of $U$.
 \end{thm}

\subsection{Knot surgery}
\label{sec:knot-surgery-positivity}

In this section, we prove Theorem~\ref{thm:positivity} in the setting of the knot surgery formula for a knot $K$ in a closed 3-manifold $Y$. The case of links (or equivalently bordered manifolds with multiple boundary components) turns out to be not substantially different, and is discussed briefly in Section~\ref{sec:link-surgery-positivity}.

We assume that we have a meridianal Heegaard diagram $(\Sigma,\as,\bs_0,w,z)$ for a knot $K\subset Y$, which has a distinguished knot shadow, which we denote also by $K$. On this diagram, there is a distinguished punctured torus containing the union of $K$ and the meridian $\b_0\in \bs_0$. We form curves $\bs_1$ and $\bs_\lambda$ as in Section~\ref{sec:knot-surgery-formula}.

We will write $\cH$ and $\cH'$ for the Heegaard multi-diagrams
\[
\cH=(\Sigma, \as, \bs_{\lambda}, \bs_0, \bs_1, \bs_{\lambda}')\quad \text{and} \quad \cH'=(\Sigma, \as, \bs_0,\bs_1, \bs_{\lambda}, \bs_0',\bs_1').
\]
To simplify our arguments, we assume that the meridians $\b_0\in \bs_0$ and $\b_0'\in \bs_0'$ intersect near $w$ and $z$ as in Figure~\ref{fig:39}.

\begin{define} If $M$ and $N$ are free $\bF[U]$-modules and $f\colon M\to N$ is an $\bF[U]$-linear map, we set $w_U(f)\in [0,\infty]$ to be the infimum of $n$ such that there is an $\bF[U]$-linear $g\colon M\to N$ so that
\[
f=U^n g.
\]
If $f$ maps into $U^{-1} N$, we define $w_U(f)$ similarly, except that we allow values in $[-\infty,\infty]$. 
\end{define}

\begin{lem}
\label{lem:positivity-knot}
Consider a subdiagram of $\cH$ or $\cH'$ and a collection of input morphisms $\lb \xs_i, f_i\rb$, $i=1,\dots, n$ for this subdiagram. Assume that if $\lb \xs_i, f_i\rb\in \ve{\CF}^-(\bs_0^{E_0}, \bs_0'^{E_0})$, then $\lb \xs_i, f_i\rb$ is filtered. If $\lb \ys, f\rb$ is a summand of 
\[
\mu_n(\lb \xs_1,f_1\rb,\dots, \lb\xs_n, f_n\rb),
\]
then $w_U(f)\ge 0$ (i.e. no negative $U$-powers appear).
\end{lem}
\begin{proof}
When neither $\bs_0$ nor $\bs_0'$ is present, the claim is clear since each composite factor of $f$ has $w_U\ge 0$.

 We now consider the claim for subdiagrams of $\cH$ and $\cH'$ which contain exactly one of $\bs_0$ and $\bs_0'$. The argument in this case is essentially the same as Lemma~\ref{lem:positivity-simple}, above. If we write $\psi$ for the homology class of polygons contributing $\lb \ys, f\rb$, then $f$ will have a composite factor of $\scV^{\# \d_{\b_0}(\phi)\cap K}=\scV^{n_z(\phi)-n_w(\phi)}$, as well as an overall factor of $U^{n_w(\phi)}$. We can commute $U^{n_w(\phi)}$ past all maps in the composition, since they are $\bF[U]$-equivariant, to replace $\scV^{n_z(\phi)-n_w(\phi)}$ with \[
U^{n_w(\phi)}  \scV^{n_z(\phi)-n_w(\phi)}=\scU^{n_w(\phi)} \scV^{n_z(\phi)}.
\]
Since no other composite factors involve potentially negative powers of $U$, we conclude that the $\bF[U]$-module map $f$ associated with the composition has 
\[
w_U(f)\ge \min(n_w(\phi),n_z(\phi))\ge 0. 
\] 

We now prove the claim for subdiagrams of $\cH'$ which contain both $\bs_0$ and $\bs_0'$. Consider a polygon class $\phi$ on such a diagram. There are two intersection points of $\b_0\cap \b_0'$, and we let $\theta^+$ denote the higher graded intersection point. Write $n_w(\phi)$, $n_z(\phi)$, $N(\phi)$ and $M(\phi)$ for the multiplicities centered there, and assume that the subarc of $K$ goes through the region labeled $N(\phi)$. See Figure~\ref{fig:39}. The $\bF[U]$-module map in the output is a composition of maps, two of which are 
\begin{equation}
\scV^{N(\phi)-n_w(\phi)}\quad \text{and} \quad \scV^{n_z(\phi)-N(\phi)}.
\label{eq:b_0-twice-monodromy}
\end{equation}
 The two maps in Equation~\eqref{eq:b_0-twice-monodromy} might have $w_U<0$, though the other maps in the definition of $f$ have $w_U\ge 0$. Note that there is also an overall factor of $U^{n_w(\phi)}$ in the definition of $f$.

\begin{figure}[h]
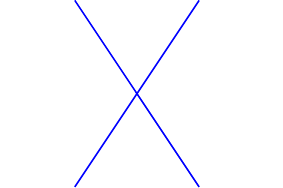
\caption{A neighborhood of the point $\theta^+\in \b_0\cap \b_0'$. Local multiplicities are labeled. The dashed line is the shadow $K$. }
\label{fig:39}
\end{figure}

Near $\theta^+$, we have
\begin{equation}
n_w(\phi)+n_z(\phi)=N(\phi)+M(\phi)+\delta, \label{eq:vertex-multiplicities}
\end{equation}
where $\delta$ is $1$ if $\theta^+$ is an output (but not an input) of $\phi$, $-1$ if $\theta^+$ is an input (but not an output), and $0$ otherwise.

We claim that if $ n_w(\phi)\ge N(\phi)$, then  $w_U(f)\ge 0$. To see this, we factor $U^{n_w(\phi)}$ as
\[
U^{n_w(\phi)}=U^{N(\phi)}\cdot U^{n_w(\phi)-N(\phi)},
\]
and then distribute the two factors $U^{N(\phi)}$ and $U^{n_w(\phi)-N(\phi)}$ into different parts of the composition as follows:
\[
U^{N(\phi)}\cdot \scV^{n_z(\phi)-N(\phi)}=\scU^{N(\phi)} \scV^{n_z(\phi)}\quad \text{and} \quad  U^{n_w(\phi)-N(\phi)}\scV^{N(\phi)-n_w(\phi)}=\scU^{n_w(\phi)-N(\phi)}.
\] 
The above expressions only involve nonnegative powers of $U$, so $w_U(f)\ge 0$. In fact, we have
\[
w_U(f)\ge \min(N(\phi), n_z(\phi)).
\] 

 Hence, it remains to consider the case that $N(\phi)>n_w(\phi)$. In this case, $\scV^{N(\phi)-n_w(\phi)}$ has $w_U\ge 0$. We use Equation~\eqref{eq:vertex-multiplicities} to rewrite
 \[
 U^{n_w(\phi)}\scV^{n_z(\phi)-N(\phi)}=\scU^{n_w(\phi)}\scV^{n_w(\phi)+n_z(\phi)-N(\phi)}=\scU^{n_w(\phi)}\scV^{M(\phi)+\delta}.
 \]
 
 We break the argument into further cases depending on the value of $\delta$. The case that $\delta\ge 0$ is clear, since then we have
 \[
 w_U(f)\ge \min(n_w(\phi), M(\phi)).
 \]
  It remains to consider the case that $\delta=-1$. Recall that $\delta=-1$ corresponds to the case that $\theta^+$ is an input of $\phi$, but not an output. Furthermore, we assumed as a hypothesis that the corresponding morphism $f_i$ associated to $\theta^+$ was filtered.   
 
 The expression for the output map contains the expression
 \[
 U^{n_w(\phi)} \circ \scV^{n_z(\phi)-N(\phi)} \circ f_i \circ \scV^{N(\phi)-n_w(\phi)}
 \]
 By Lemma~\ref{lem:filtered-morphism-compose-Vs}, we may rewrite this as
 \[
 U^{n_w(\phi)} \circ \scV^{n_z(\phi)-N(\phi)} \circ \scV^{N(\phi)-n_w(\phi)}\circ f_i'=\scU^{n_w(\phi)}\scV^{n_z(\phi)}\circ f_i'
 \]
 for some other filtered map $f_i'$. In particular, we have
 \[
 w_U(f)\ge \min(n_w(\phi), n_z(\phi))\ge 0,
 \]
 completing the proof. 
 \end{proof}
 
\begin{rem}
\label{rem:slightly-stronger-positivity} Note that the above proof showed something slightly stronger. It showed that if $\psi\in \pi_2(\xs_1,\dots, \xs_n, \ys)$ is a polygon class and $f_1,\dots, f_n$ is a sequence of input morphisms satisfying the assumptions in the lemma, and $f$ is the output morphism, then
\[
w_U(f)\ge \min(n_w(\phi), n_z(\phi), N(\phi), M(\phi) ),
\]
where $n_w(\phi)$, $n_z(\phi)$, $N(\phi)$ and $M(\phi)$ are the multiplicities labeled in Figure~\ref{fig:39}. 
\end{rem}

 We need one additional result about positivity, which concerns pure beta diagrams. We note that in our proof of Theorem~\ref{thm:exact-triangle}, we used the $A_\infty$-associativity relations for holomorphic polygons. For these to give  meaningful algebraic relations on the level of chain complexes, we need to know that any output in a pure beta diagram which lies in $\ve{\CF}^-(\bs_0^{E_0}, \bs_0'^{E_0})$ is also filtered in the sense of Definition~\ref{def:filtered-morphism}, since this is a hypothesis of the above lemma. We prove the following lemma:

 \begin{lem} Suppose that $(\Sigma,\bs_{\veps_0},\dots, \bs_{\veps_n})$ is a pure beta diagram where each $\bs_{\veps_i}$ is obtained from a small translation of one of $\bs_0$, $\bs_1$ or $\bs_{\lambda}$. We assume further that at most two of the $\bs_{\veps_i}$ are small translates of $\bs_0$.  Let $\lb \xs_1,f_1\rb,\dots, \lb \xs_n, f_n \rb$ be a collection of inputs which are Alexander grading preserving. Furthermore, we assume that if $\lb \xs_i, f_i\rb\in \ve{\CF}^-(\bs_0^{E_0}, \bs_0'^{E_0})$, then $\lb \xs_i,f_i\rb$ are filtered. If $\lb \ve{y}, f\rb$ is a summand of
 \[
 \mu_n(\lb \xs_1,f_1\rb,\dots, \lb \xs_n, f_n \rb ),
 \]
 then  $f$ is Alexander grading preserving. If $\lb \ys, f\rb\in \ve{\CF}^-(\bs_0^{E_0}, \bs_0'^{E_0})$, then it is also filtered.
 \end{lem}
Note that Lemma~\ref{lem:positivity-knot} shows that under the above assumptions, $w_U(f)\ge 0$. 
 \begin{proof} We observe that the Alexander grading of $f$ is
\[
\sum_{i=1}^{n} A(f_i)+\sum_{i=0}^n \# \d_{\b_{\veps_i}}(\psi)\cap K.
\]
We claim that the above is zero. To see this, observe that $A(f_i)=0$ by hypothesis, and that since diagram is a pure beta diagram, we have 
\[
\sum_{i=0}^n \# \d_{\b_{\veps_i}}(\psi)\cap K=\#\d(\psi)\cap K.
\] This is 0 since the $\d(\psi)$ is a boundary (as an integral 1-chain on $\Sigma$) and $K$ is a closed 1-chain.
To see that $f$ is filtered, we observe that for an Alexander grading preserving $\bF[U]$-module map $f\colon \bF[\scU,\scV]\to \bF[\scU,\scV]$, the inequality $w_U(f)\ge 0$ is equivalent to being filtered. The proof is complete.
 \end{proof}

\subsection{Link surgery}
\label{sec:link-surgery-positivity}

We now describe how to extend the work from Section~\ref{sec:knot-surgery-positivity} to prove Theorem~\ref{thm:positivity} in the setting of the link surgery formula (in particular, when there are multiple components). It is sufficient to consider the second claim of the theorem, which concerns tensoring a module $\cX_{\Lambda}(Y,L)^{\frL}$ with a single copy of ${}_{\cK} \cD_0$ (i.e. doing surgery on just one component of $L$). 

 In this case, we are interested in Heegaard diagrams with many basepoints. Furthermore, these diagrams have a special genus 1 region for each link component of $L$. The beta attaching curves are indexed by points in either $\bE_{\ell-1}\times \{\lambda,0,1,\lambda'\}$ or $\bE_{\ell-1} \times \{0,1,\lambda,0',1'\}$. Here, $\ell=|L|$  and we view $\{\lambda,0,1,\lambda'\}$ and $\{0,1,\lambda,0',1'\}$ as sets of symbols. Write $\bs_{\veps}$ for the beta curves. We write $\cH$ and $\cH'$ for the full Heegaard multi-diagrams, indexed by these sets.

We focus on the case of subdiagrams of $\cH'$, since the case of $\cH$ is simpler and follows from the same logic. Write $\bs_{\veps_0},\dots, \bs_{\veps_k}$ for the beta curves in such a subdiagram. Here, $\veps_i$ is an increasing sequence of points in $\bE_{\ell-1}\times \{0,1,\lambda,0',1'\}$ (where we order $0<1<\lambda<0'<1'$). We focus on a single link component at a time. Let $K_j$ be a link component for $j\in \{1,\dots, \ell-1\}$. By our construction, the curves in the special genus 1 region for $K_j$ form a sequence of the form $\b_0,\dots, \b_0,\b_1,\dots, \b_1$ (the case where all are $\b_0$ or all are $\b_1$ is allowed).

 Write $E_{\veps}$ for the tensor product over $\bF$ of the corresponding $E_{0}$, $E_1$ and $E_{\lambda}$ spaces. We view the maps $f_i\colon E_{\veps_i}\to E_{\veps_{i+1}}$ as sums of tensor products of maps between the corresponding $E_{0}$, $E_1$ and $E_{\lambda}$ spaces.

We assume that each morphism connecting two consecutive translates of $\b_0$ is a map of $\bF[U_j]$-modules from $\bF[\scU_j,\scV_j]$ to itself which is filtered in the sense of Definition~\ref{def:filtered-morphism}. Let $\b_{0,0},\dots, \b_{0,m}$ denote the copies of $\b_0$ and let $g_1,\dots, g_m$ denote the corresponding endomorphisms of $\bF[\scU_j,\scV_j]$ associated to the intersection points of copies of $\b_0$. We observe that the output map has tensor factor which contains the following map as a composite factor
\begin{equation}
\scV_j^{\# \d_{\b_{0,m}}(\psi) \cap K_m}\circ g_m\circ \scV_j^{\# \d_{\b_{0,m-1}}(\psi)\cap K}\circ g_{m-1} \cdots \circ g_1\circ \scV_j^{\# \d_{\b_{0,0}}(\psi)\cap K}.
\label{eq:regroup-1}
\end{equation}
(In our proof of the surgery formula, we in fact take $g_1=\cdots=g_m=\id_{E_0}$). Using Lemma~\ref{lem:filtered-morphism-compose-Vs}, we may reorder terms to rewrite the above map as
\begin{equation}
\scV_j^{\#\d_{\b_{0,m}+\cdots+ \b_{0,1}}(\psi) \cap K_m} \circ g'=\scV_j^{n_{z_j}(\psi)-n_{w_j}(\psi)}\circ g',
\label{eq:regroup-2}
\end{equation}
for some filtered map $g'$.  The additional overall factor of $U_j^{n_{w_j}(\phi)}$ transforms this to $\scU_j^{n_{w_j}(\phi)} \scV_j^{n_{z_j}(\phi)}\circ g'$ so we see the composition only involves nonnegative $U_j$-powers. 

Some extra care must be taken in the above argument for the knot component $K_\ell$ in the diagram $\cH'$, since we have $\b_0$, $\b_0'$, $\b_1$, and $\b_1'$, as well as small translations of these curves. To adapt the argument from the case of knots, we assume that copies of $\b_0$, $\b_0'$, $\b_1$ and $\b_1'$ are first fixed. To form the entire diagram $\cH'$, we need further small translations of these curves. We form these additional translations very small relative to the fixed initial copies of $\b_0$ and $\b_0'$. See Figure~\ref{fig:41}. With these choices, the argument for positivity is essentially identical to the case of the knot surgery formula from Lemma~\ref{lem:positivity-knot}, using the  strategy from Equations~\eqref{eq:regroup-1} and~\eqref{eq:regroup-2} to rearrange the maps and powers of $\scV_j$ for the $\b_0$ curves, and similarly for the $\b_0'$ curves.  With these comments we conclude the proof of Theorem~\ref{thm:positivity}.

\begin{figure}[ht]
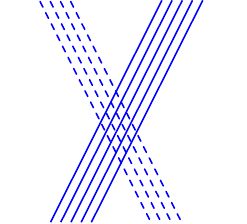
\caption{Configuration of the curves $\b_0$, $\b_0'$ and their small translations used in the construction of $\cH'$.}
\label{fig:41}
\end{figure}

\subsection{Endomorphism algebras}
\label{sec:positivity-endomorphism-algebras}

The ideas of the past two sections generalize easily to the setting of the endomorphism algebra $\End_{\Fil}(\b_0^{E_0}\oplus \b_1^{E_1})$ from Section~\ref{sec:endomorphisms}. The arguments of the last section adapt easily to give the following:

\begin{lem} Suppose that $\b_{\veps_0},\dots, \b_{\veps_n}$ is a sequence of copies of small translates of the curves $\b_0$ and $\b_1$ such that the copies of $\b_0$ and $\b_1$ are all suitable small translates of fixed copies of $\b_0$ and $\b_1$. Let $\lb \ve{x}_1,f_1\rb,\dots, \lb \ve{x}_n,f_n\rb$ be a sequence of Floer inputs which are filtered in the sense of Definition~\ref{def:filtered-morphism}. If $\lb \ve{y}, f\rb$ is a summand of $\mu_n(\lb \ve{x}_1,f_1\rb,\dots, \lb \ve{x}_n,f_n\rb)$ then $f$ is also filtered.
\end{lem}

The proof is straightforward using the regrouping technique of Section~\ref{sec:link-surgery-positivity}. Namely, we reorder the powers of $\scV$ associated to changes across the curves which are small translates of copies of $\b_0$ using Lemma~\ref{lem:filtered-morphism-compose-Vs}, which was the tool we used to transform Equation~\eqref{eq:regroup-1} into Equation~\eqref{eq:regroup-2}. Once we do this, if $f$ is the output of a holomorphic polygon $\psi$ with filtered inputs, then we see that 
\[
w_U(f)\ge \min(n_w(\phi),n_z(\phi))\ge 0. 
\]

\section{Admissibility and completions}
\label{sec:admissibility}

In this section, we prove that the maps appearing in the proof of the surgery formulas are continuous in the topologies we considered in this paper. Due to the slightly technical nature of these arguments we have decided to make them an appendix.

\begin{thm}\label{thm:admissibility} Suppose that $Y$ is a closed 3-manifold and $L\subset Y$ is a link with Morse framing $\Lambda$. The homotopy equivalence
\[
\ve{\CF}^-(Y_{\Lambda}(L))\simeq \bX_{\Lambda}(Y,L) 
\]
holds if we equip both complexes with the $U$-adic topologies. It also holds if we equip $\ve{\CF}^-$ with the $U$-adic topology, while equipping $\bX_{\Lambda}(Y,L)$ with the chiral topology. If instead $Y$ is a bordered manifold with torus boundary components, and $Y'$ is obtained by performing Dehn surgery on one component (compatibly with the boundary parametrization) then the homotopy equivalence
\[
\cX(Y)^{\frL}\boxtimes {}_{\frK} \frD_0\simeq \cX(Y')^{\frL'}
\] 
holds in the $U$-adic topology. The analogous statement  also hold with the chiral topologies.
\end{thm}

\subsection{\texorpdfstring{$\ve{f}$}{f}-admissibility}

In this section, we describe a helpful notion of admissibility. For various limiting arguments, it is helpful to allow real-valued powers of $\scU$, $\scV$, $U$ and $T$, instead of only integral powers.   We define 
\[
\scE_0^\infty, \scE_1^\infty, \scE_\lambda^\infty:=\bF[\R^2].
\]
We think of $\scE_0^\infty$ and $\scE_{\lambda}^\infty$ as being generated by monomials $e^{(x,y)}=\scU^{x}\scV^y$ where $x,y\in \R$. We think of $\scE_1^\infty$ and as being generated by monomials $e^{(x,y)}=U^x T^y$ for $x,y\in \R$.

 We define subspaces $\scE_{\veps}\subset \scE_\veps^\infty$ as follows. We define $\scE_0$ and $\scE_\lambda$ to be the span of monomials $\scU^x \scV^y$ with $x,y\ge 0$. We define $\scE_1$ to be generated by monomials $U^x T^y$ where $x\ge 0$. It is also natural to view $\bF[U]$ as lying inside of the ring $\cA_{\R}:=\bF[0,\infty)$, spanned by monomials $U^x$ where $x\ge 0$. 
We also consider $\cA_{\R}^\infty$, which is spanned by $U^x$ for $x\in \R$.

\begin{rem}
Note that it would seem more natural to define  $\scE_\lambda^\infty$ to be $\cA_{\R}^\infty$, but it will make various limiting arguments simpler to use the definition above.
\end{rem}

We define Alexander gradings on $\scE_0^\infty$, $\scE_1^\infty$ and $\scE_{\lambda}^\infty$ via the formula
\[
A(\scU^x\scV^y)=y-x,\quad \text{and} \quad A(U^x T^y)=y.
\]
On $\scE_0$, we also define gradings
\[
\gr_w(\scU^x\scV^y)=-2x\quad \text{and} \quad \gr_z(\scU^x\scV^y)=-2y. 
\]

If $\a>0$, let $r_{\a}\colon \bF[\R^2]\to \bF[\R^2]$ denote rescaling by a factor of $\a$, i.e. 
\[
r_{\a}\left(e^{(x,y)}\right)=e^{(\a x,\a y)}.
\]
\begin{rem}
Note that $r_{\a}$ is not equivariant with respect to $U$.  Rather, 
\[
r_{\a}^{-1}\circ  U^y \circ r_\a=U^{r/\a}.
\] 
\end{rem}

\begin{define}
\label{def:extended-morphism}
An \emph{extended morphism} from  $\scE_{\veps}^\infty$ to $\scE_{\veps'}^\infty$ consists of a finite sum of maps $f$ which each have the property that there is an  $A\colon \R^2\to \R^2$ such that 
\begin{equation}
f(e^{x})=e^{A(x)},\label{eq:continuous-morphism}
\end{equation}
and furthermore:
\begin{enumerate}
\item $f$ is $\cA_{\R}^\infty$-equivariant;
\item \label{condition:fundamental-equation-ext-morphism} $A$ is of the form $A(x)=B(x)+C$, where $C\in \scE_{\veps'}^{\infty}$ and $B$ is a continuous map satisfying $B(\a x)= \a B(x)$ for all $\a\ge 0$. 
\end{enumerate}
\end{define}

We make two additional definitions related to extended morphisms:
\begin{define}
\label{def:extended-morphisms-pos-fil}
 We say that an extended morphism $e^A$ from $\scE_{\veps}^\infty$ to $\scE_{\veps'}^{\infty}$ is \emph{positive} if $e^B$ maps $\scE_{\veps}$ to $\scE_{\veps'}$ and $e^{C}\in \scE_{\veps'}$. We say an extended morphism $f\colon \scE_0^\infty\to \scE_0^\infty$ is \emph{filtered} if $e^C\in \scE_{\veps'}$ and the map $e^B$ is non-increasing in the $\gr_w$ and $\gr_z$ gradings.
\end{define}

\begin{rem} Note that if $f= e^{A}=e^{B+C}$ is as in Definition~\ref{def:extended-morphism}, then 
\[
r_{\a}^{-1} \circ f \circ r_{\a}=e^{B+C/\a}.
\]
\end{rem}

\begin{example} 
\label{example:Pi}
The maps $\phi^\sigma$ and $\phi^\tau$ naturally induce extended morphisms from $\scE^\infty_0$ to $\scE^\infty_1$ since they may be written as $e^{L_\sigma}$ and $e^{L_\tau}$ for linear maps $L_\sigma$ and $L_\tau$. Note that the map $\Pi\colon E_1\to E_{\lambda}$ does not, since it fails the continuity condition. Instead, for admissibility it suffices to merely consider the map $\Pi'\colon \scE^\infty_{1}\to \scE^\infty_\lambda$, given by
\[
\Pi'(U^s T^t)=
\begin{cases}
 \scU^s \scV^{t+s}& \text{ if } t\ge 0\\
 \scU^{t-s} \scV^s &\text{ if } t\le 0
\end{cases}.
\]
The map coincides with $\Pi$ on Alexander grading zero, but has larger support.
It is straightforward to see that $\Pi'$ satisfies Definition~\ref{def:extended-morphism}. 
\end{example}

\begin{define}
\label{def:U-weight}
Given an extended morphism $f\colon \scE_\veps^\infty\to \scE_{\veps'}^\infty$, we define $w_U(f)\in [-\infty,\infty]$ as the supremum over $r\in \R$ such that $U^{-r}\cdot f$ maps $\scE_\veps$ into $\scE_{\veps'}$. 
\end{define}

Given a Heegaard tuple $(\Sigma,\gs_0,\dots, \gs_n,w)$, we define a \emph{weakly periodic domain} $P$ to be an integral 2-chain which has boundary equal to a linear combination of the curves in $\gs_0,\dots, \gs_n$. (We remark that unlike an ordinary periodic domain, we do not assume that $P$ has zero multiplicity at any basepoints).

Suppose now that each $\gs_j$ is equipped with a local system $\gs^{\scE_{\veps_j}}_j$, and we are given a sequence of extended morphisms $\ve{f}=(f_1,\dots, f_n)$, where $f_j\colon \scE_{\veps_{j-1}}\to \scE_{\veps_j}$. (We are only interested in the case that the local systems are of the form $\scE_0$, $\scE_1$ or $\scE_\lambda$, with the monodromies induced by intersection numbers with a closed  curve $K\subset \Sigma$). We may define a map
\[
\Map(P,\ve{f})\colon \scE_{\veps_0}^\infty\to \scE_{\veps_n}^\infty,
\]
using the same formula as for holomorphic polygons in Equation~\eqref{eq:local-systems}.

\begin{define} Suppose that 
\[
\cH=(\Sigma,\gs_0^{\scE_{\veps_0}},\dots, \gs_n^{\scE_{\veps_n}},w)
\]
 is a Heegaard diagram whose attaching curves are equipped with local systems, whose monodromies are determined by intersection numbers with a shadow $K\subset \Sigma$. Furthermore, assume each $\scE_{\veps_0}$ is either the trivial local system or one of $\scE_{0}$, $\scE_1$ or $\scE_\lambda$. Let $\ve{f}=(f_{1},\dots, f_{n})$ be a sequence of extended morphisms where $f_j\colon \scE_{\veps_{j-1}}^\infty\to \scE_{\veps_{j}}^{\infty}$. We say that $\cH$ is \emph{$\ve{f}$-admissible} if for all non-zero weakly periodic domains $P\ge 0$, we have
\[
w_U(\Map(P,\ve{f}))>0.
\]
\end{define}

We now prove several basic results about the functions $w_U$ and $\scF$:

\begin{lem}\label{lem:basic-facts-w_U}
\begin{enumerate}
\item Suppose $\a>0$, $\psi$ is a domain on a Heegaard diagram and $\rho_{\veps}$ is the monodromy map for the attaching curve $\bs_{\veps}$ labeled with $\scE_\veps^\infty$, where $\veps\in \{\lambda,0,1\}$. Then 
\[
r_\a^{-1}\circ  \rho_\veps(\psi) \circ   r_\a=\rho_\veps(\psi/\a).
\]
\item\label{item:continuity-D} Each of the monodromy maps considered in this paper takes the form $\rho(D)=e^{A(D)}$, where $D$ denotes a domain on the Heegaard diagram and $A(D)\colon \R^2\to \R^2$ is an affine map, depending on $D$. Furthermore, we can view $A$ as a continuous map from the set of domains on the Heegaard diagram (viewed as $\R^N$ where $N$ is the number of regions on the diagram) to the set of affine transformations of $\R^2$, equipped with the uniform metric.
\item \label{item:continuity-w_U} If $\{e^{A_i}\colon \scE_{\veps}^\infty\to \scE_{\veps'}^\infty\}_{i\in \N}$ is a sequence of $\bF$-linear maps, where  $A_i\colon \R^2\to \R^2$  is a sequence of functions which converges to some function $A\colon \R^2\to \R^2$ uniformly, then 
\[
w_U(e^{A_i})\to w_U(e^{A}).
\]
\item \label{item:w_U-rescale} Given $\a>0$, and a sequence of extended morphisms $\ve{f}=(f_1,\dots, f_n)$,  where $f_i=e^{B_i+C_i}$, let $\ve{f}_\a$ denote the morphism sequence $f_{i,\a}=e^{B_i+C_i/\a}$. Then
\[
r_{\a}^{-1}\circ \Map(\psi,\ve{f})\circ r_\a=\Map(\psi/\a,\ve{f}_\a).
\]
Additionally
\[
\frac{w_U(\Map(\psi,\ve{f}))}{\a}=w_U(\Map(\psi/\a,\ve{f}_\a)). 
\]
\end{enumerate}
\end{lem}
\begin{proof} For the first claim, $\rho_{\veps}(\psi)$ is either the identity (if $\veps=\lambda$) or is the map $e^x\mapsto e^{x+(0,s)}$ where $s$ is $\# (\d_{\b}(\psi)\cap K)$. Hence $r_{\a}^{-1}\circ  \rho_\veps(\psi)\circ  r_{\a}$ maps $e^x$ to $e^{x+(0,s/\a)}$, which agrees with the evaluation of $\rho_{\veps}(\psi/\a)$ on $e^x$. 

The second claim follows similarly and is completely straightforward.

We consider the third claim. Suppose for concreteness that $\veps'$ is $0$ or $\lambda$. The case that $\veps'=1$ is similar. We write $A_i(x)=(a_i(x),b_i(x))\in \R^2$ and $A(x)=(a(x),b(x))$. We observe
\[
w_U(f_i)=\sup \{\min(a_i(x), b_i(x)):e^x \in \scE_{\veps}\}.
\]
We observe that $\min(a_i,b_i)\to \min(a,b)$ uniformly, since $\min\colon \R^2\to \R$ is continuous. If $g_i$ are functions and $g_i\to g$ uniformly then $\sup_x g_i(x)\to \sup_x g(x)$. The main claim follows easily.

The fourth claim is proven as follows. First observe
\[
r_\a^{-1}\circ f_i\circ r_{\a}=f_{i,\a}.
\]
Therefore, by moving $r_{\a}$ from right to left we compute
\[
\begin{split}
&r_{\a}^{-1}\circ \Map(\psi,\ve{f})\circ  r_\a\\
=&r_{\a}^{-1}\circ U^{n_w(\psi)}\rho_n(\psi)\circ f_{n}\circ \cdots \circ f_{1} \circ\rho_0(\psi)\circ r_{\a} \\
=&U^{n_w(\psi)/\a} \rho_{n}(\psi/\a) \circ f_{n,\a}\circ\cdots \circ  f_{1,\a} \circ\rho_0(\psi/\a)\\
=&\Map(\psi/\a,\ve{f}_\a).
 \end{split}
\]
The equality involving $w_U$ is an immediate consequence of the above.
\end{proof}

\begin{lem}\label{lem:f-admissible}
 Let $\cH=(\Sigma,\gs_0^{\scE_{\veps_0}},\dots, \gs_n^{\scE_{\veps_n}},w)$ be a Heegaard diagram such that the attaching curves are equipped with local systems. Let $\ve{f}=(f_1,\dots, f_n)$ be a sequence of extended morphisms, where $f_i=e^{B_i+C_i}$, as in Definition~\ref{def:extended-morphism}. Let $\ve{f}'=(f_1',\dots, f_n')$ where $f_i'=e^{B_i}$. If $\cH$ is $\ve{f}'$-admissible, then for each $N$ there are only finitely many classes of polygons $\psi\ge 0$ with $w_U(\Map(\psi,\ve{f}))<N$.
\end{lem}
\begin{proof}
Suppose to the contrary that there is a sequence of distinct nonnegative classes $\psi_i$ and some $N$ so that 
\[
w_U(\Map(\psi_i, \ve{f}))<N.
\]
We must have $|\psi_i|_{\infty}\to \infty$, since $\psi_i$ have domains which are integral. By passing to a subsequence, we may assume that $\psi_i/|\psi_i|_\infty$ converges to a non-zero weakly periodic domain $P\ge 0$.

 We apply part~\eqref{item:w_U-rescale} of Lemma~\ref{lem:basic-facts-w_U} (with $\a=|\psi_i|_{\infty}$) and obtain
\begin{equation}
\frac{w_U(\Map(\psi_i,\ve{f}))}{|\psi_i|_\infty}= w_U \left(\Map(\psi_i/|\psi_i|_\infty,\ve{f}_{|\psi_i|_\infty})\right).
\label{eq:effect-rescale-sequence-psi}
\end{equation}
Parts ~\eqref{item:continuity-D} and~\eqref{item:continuity-w_U} of Lemma~\ref{lem:basic-facts-w_U} imply that
\[
w_U \left( \Map(\psi_i/|\psi_i|_\infty,\ve{f}_{|\psi_i|_\infty}) \right)\to w_U \left( \Map(P,\ve{f}')\right). 
\]
Since $ w_U(\Map(P,\ve{f}'))>0$ by $\ve{f}'$-admissibility, Equation~\eqref{eq:effect-rescale-sequence-psi} implies that $w_U(\Map(\psi_i,\ve{f}))\to \infty$, contradicting our earlier assumption.
\end{proof}

\subsection{Surgery admissibility}

In this section we describe a notion of admissibility for Heegaard diagrams which we will use to prove Theorem~\ref{thm:admissibility}. Our notion of admissibility is sufficient for both the $U$-adic and chiral topologies, so we call it just \emph{surgery admissibility}.

To understand the homotopy equivalence in Theorem~\ref{thm:admissibility}, we focus first on subdiagrams of the following two diagrams
\[
\cH=(\Sigma, \as, \bs_{\lambda}, \bs_0, \bs_1, \bs_{\lambda}')\quad \text{and} \quad \cH'=(\Sigma, \as, \bs_0,\bs_1, \bs_{\lambda}, \bs_0',\bs_1').
\]
In particular, this restricts us to the setting of the knot surgery formula. The extension of these arguments to links is considered in Section~\ref{sec:links}.

\begin{define}
\label{def:surgery-admissible} Suppose that $\cD=(\Sigma,\gs_0^{\scE_{i_0}},\dots, \gs_n^{\scE_{i_n}},w)$ is a Heegaard diagram whose attaching curves are small translates of $\as$, $\bs_0^{\scE_0}$, $\bs_1^{\scE_1}$ or $\bs_{\lambda}^{\scE_\lambda}$.
 We say that $\cD$ is \emph{surgery admissible} if all subdiagrams are $\ve{f}$-admissible for every sequence $\ve{f}$ of Alexander grading preserving extended morphisms which are positive and which have the property that each morphism $f_i$ from $\scE_0$ to $\scE_0$ is filtered  (Definition~\ref{def:extended-morphisms-pos-fil}).
\end{define}

In this section, we will prove the following:

\begin{prop}\label{prop:HH'-surgery-admissible} After performing a suitable isotopy to $\as$,  the diagrams $\cH$ and $\cH'$ are surgery admissible.
\end{prop}

We will break the proof into several lemmas. After establishing Proposition~\ref{prop:HH'-surgery-admissible}, we will how to prove the special case of Theorem~\ref{thm:admissibility} when $L$ has a single component.

\begin{lem}
\label{lem:surgery-admissible-weak-adm} Let $\cH_0$ be a subdiagram of $\cH$ or $\cH'$ which contains exactly one of $\b_0$ or $\b_0'$. Then $\cH_0$ is surgery admissible.
\end{lem}
\begin{proof}In the formula defining $\Map(P,\ve{f})$, we have both $U^{n_w(P)}$ and $\scV^{n_z(P)-n_w(P)}$, as well as several other factors. All other factors are $\bF[U]$-equivariant and $w_U\ge 0$. We rearrange the above to be $\scU^{n_w(P)} \scV^{n_z(P)}$, composed with some other other $\bF[U]$-equivariant maps which have $w_U\ge 0$. Hence
\[
w_U(\Map(P,\ve{f}))\ge \min(n_w(P), n_z(P)).
\]
By picking $\as$ so that the diagrams are weakly admissible at $w$ and $z$, it follows that if $P\ge 0$ and $P\neq 0$, then $\min(n_w(P), n_z(P))>0$, completing the proof. 
\end{proof}

We will write $\cD$ and $\cD'$ for the diagrams $\cH$ and $\cH'$ with $\as$ removed. We now prove the following:

\begin{prop}
\label{prop:D-surgery admissible}
 The diagrams $\cD$ and $\cD'$ are surgery admissible.
\end{prop}

We note first that it is straightforward to reduce the proof of Proposition~\ref{prop:D-surgery admissible} to the case when $g(\Sigma)=1$. We break the remainder of the proof into several lemmas.

\begin{lem}
\label{lem:moving-inequality}
 Suppose that $a,b\ge 0$ and that $f\colon \scE_0\to \scE_0$ is an $\cA_{\R}$-equivariant map of homogeneous Alexander grading $\delta\ge 0$. Then
\[
w_U(\scU^a \circ f\circ \scV^b)\ge \min(a,b).
\]
Symmetrically, if $\delta\le 0$, then
\[
w_U(\scV^a \circ f \circ \scU^b)\ge \min(a,b).
\]
\end{lem}

\begin{rem} In the above, we only require $f$ to be an $\cA_{\R}$-linear map which has Alexander grading $\delta$ and sends $\scE_0$ to $\scE_0$. In particular, we are not requiring $f$ be  an extended morphism in the sense of Definition~\ref{def:extended-morphism}, or to be filtered in the sense of Definition~\ref{def:filtered-morphism}. 
\end{rem}

\begin{proof} We will focus on the claim about $\scU^a \circ f\circ \scV^b$. Given $x\ge 0$, we define  auxiliary function $d_x\colon \R\to \R^{\ge 0}$ by the formula
\[
d_x(s)=
\begin{cases}0& \text{ if } s\le 0\\
s& \text{ if } 0\le s\le x\\
x& \text{ if } x\le s.
\end{cases}
\]
On Alexander grading $s\in \R$, the map $\scV^b$ will have $U$-weight $d_{b}(-s)$ while $\scU^a$ will have $U$-weight $d_a(s)$. Hence the composition $\scU^a \circ f\circ \scV^b$ will have $U$-weight which is at least the infimum of the function
\[
d_a(s+\delta+b)+d_{b}(-s),
\]
over $s\in \R$. This is a piecewise linear function which is nonnegative, and has only finitely many points of non-differentiability. Hence, its minimum must occur at a point of non-differentiability. These points are  $-\delta-b$, $a-\delta-b$, $-b$ and $0$. Hence $w_U(\scU^a\circ f\circ \scV^b)$ is at least the minimum of the four numbers
\[
d_a(0)+d_b(b+\delta),\quad  d_a(a)+d_b(b+\delta-a), \quad d_a(\delta)+d_b(b), \quad \text{and}\quad  d_a(b+\delta)+d_b(0).
\]
Since $\delta\ge 0$, each of these is at least $\min(a,b)$.

 The claim about $w_U(\scV^a\circ f\circ \scU^b)$ is proven similarly.
\end{proof}

\begin{lem}\label{lem:surgery-admissible} The diagrams $(\bT^2,\b_0, \b_1, \b_0')$ and $(\bT^2,\b_0, \b_1, \b_0',\b_1')$ are surgery admissible. 
\end{lem}
\begin{proof} We focus on $(\bT^2, \b_0,\b_1,\b_0',\b_1')$ since the same argument works for $(\bT^2, \b_0,\b_1,\b_0')$. Let $\ve{f}$ be a sequence of Alexander grading preserving inputs which are positive.
We define
\[
\delta_0=\#\d_{\b_0}(P)\cap K,\quad \delta_0'=\# \d_{\b_0'}(P)\cap K,  \]
\[
\delta_1=\#\d_{\b_1}(P)\cap K, \quad \text{and} \quad \delta_1'=\#\d_{\b_1'}(P)\cap K.
\]
Additionally, we write $a$ for the multiplicity in  one of the regions, as in Figure~\ref{fig:58}.

The map $\Map(P,\ve{f})$ can be written as
\begin{equation}
U^{n_w(P)} \cdot g\circ \scV^{\delta_0'}\circ f\circ \scV^{\delta_0}, \label{eq:output-morphism-9}
\end{equation}
where $f$ has Alexander grading $\delta_1$ and $g$ has Alexander grading $\delta_1'$. Furthermore, $w_U(f),w_U(g)\ge 0$. 

We observe firstly that $(\bT^2,\b_0,\b_1,\b_0',\b_1')$ is easily seen to be weakly admissible with respect to both $w$ and $z$.

 If $\delta_0,\delta_0'\ge 0$, then Equation~\eqref{eq:output-morphism-9} implies that $w_U(\Map(P,\ve{f}))\ge n_w(P)$. By weak admissibility, we must have $n_w(P)>0$ if $P\ge 0$ and $P\neq 0$. Symmetrically (cf. Section~\ref{sec:symmetry}) if $\delta_0,\delta_0'\le 0$, then $w_U(\Map(P,\ve{f}))\ge n_z(P)>0$. Hence, it suffices to consider the case that $\delta_0$ and $\delta_0'$ have opposite signs.   We focus on the case that $\delta_0>0$ and $\delta_0'<0$.
Using the relation $\scV^{\delta_0'}=U^{\delta_0'} \scU^{-\delta_0'}$, we may rewrite $\Map(P,\ve{f})$ as
\begin{equation}
U^{a+\delta_0+2\delta_0'} g\circ  \scU^{-\delta_0'} \circ f\circ \scV^{\delta_0}.\label{eq:output-morphism-10}
\end{equation}

We observe that $a+\delta_0+2\delta_0'$ is a multiplicity on the diagram (see Figure~\ref{fig:58}), so it is nonnegative. If $a+\delta_0+2\delta_0'>0$, we are done since $w_U(g\circ \scU^{-\delta_0'}\circ f\circ \scV^{\dt_0})\ge 0$, so that 
\[
w_U( \scF(P,\ve{f}))\ge a+\delta_0+2\delta_0'>0.
\]

 Hence, it remains to consider the case that 
\[
a+\delta_0+2\delta_0'=0.
\]
 We oberve that $a+\delta_0+2\delta_0'+\delta_1$ is also a multiplicity on the diagram. (See Figure~\ref{lem:surgery-admissible}).  We conclude therefore that $\delta_1\ge 0$. We now apply Lemma~\ref{lem:moving-inequality}, which implies that
\[
w_U(\scU^{-\delta_0'}\circ f\circ \scV^{\delta_0})\ge \min(-\delta_0',\delta_0).
\]
Therefore, we conclude in the present case that
\[
w_U(\Map(P,\ve{f}))\ge \min(-\delta_0',\delta_0).
\]
If $w_U(\Map(P,\ve{f}))=0$, then we conclude that $\delta_0=\delta_0'=0$ as well. Since $a+\delta_0+2\delta_0'=0$ by our earlier assumption, we conclude that $a=\delta_0=\delta_0'=0$. It is straightforward to see using multiplicities on the diagram that $\dt_1=\dt_1'=0$ as well, so we conclude that $P=0$. 

The case that $\delta_0<0$ and $\delta_0'>0$ is handled by a parallel argument.
\end{proof}

\begin{figure}[ht]
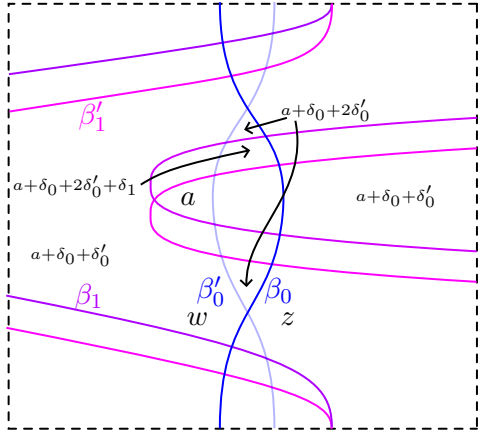
\caption{A semi-periodic domain on $(\bT^2, \b_0,\b_1,\b_0',\b_1')$ with several multiplicities labeled.}
\label{fig:58}
\end{figure}

We now complete our proof that $\cD$ and $\cD'$ are surgery admissible.

\begin{proof}[Proof of Proposition~\ref{prop:D-surgery admissible}]
Lemma~\ref{lem:surgery-admissible-weak-adm} implies $\cD$ is surgery admissible, as well as subdiagrams of $\cD'$ which contain only one of $\b_0$ and $\b_0'$. It remains to consider $\cD'=(\b_0,\b_1, \b_{\lambda}, \b_0',\b_1')$ and its subdiagrams. Observe that there are no weakly periodic domains on this diagram which have $\b_{\lambda}$ in their boundary, since $[\b_{\lambda}]$ is linearly independent in $H_1(\bT^2)$ from $[\b_0]=[\b_1]$. Therefore the main claim follows from Lemma~\ref{lem:surgery-admissible}.
\end{proof}

We now complete our proof that $\cH$ and $\cH'$ are surgery admissible:

\begin{proof}[Proof of Proposition~\ref{prop:HH'-surgery-admissible}] We note that in our present setting (of the knot surgery formula), the above Heegaard diagrams have only $w$ and $z$ as basepoints. In particular, we may pick a collection of curves $\g_1,\dots, \g_g$ on $\Sigma$ which are dual to the curves of $\as$. We wind the $\as$ curves in both directions parallel to $\g_1,\dots, \g_g$, using the standard procedure of Ozsv\'{a}th and Szab\'{o} \cite{OSDisks}*{Section~5}. By part~\eqref{item:w_U-rescale} of Lemma~\ref{lem:basic-facts-w_U}, it is sufficient to show that the above diagrams satisfy the surgery admissibility condition when applied to weakly periodic domains $P\ge 0$ with $|P|_\infty=1$.  For $m\in \N$ we write $\as_m$ for a copy of $\as$ which has been wound $m$-times in both directions parallel to $\g_1$, we consider a weakly-periodic domain $P_m$. If $\a\in \as_m$ appears in the boundary of $P_m$, then it must appear with coefficient $s\in [-1/m, 1/m]$, since there is a multiplicity $x$ on the diagram such that both $x+m\cdot s$ and $x-m\cdot s$ appear, due to the winding construction. In particular, as $m\to \infty$, we can extract a subsequence so that $P_m$ converges pointwise to a weakly periodic domain $P_\infty$ on $\cH$ and $\cH'$ which has no change along any of the $\as$ curves. It is straightforward to see that $P_\infty\neq 0$.

Let $\ve{f}$ be an input sequence of filtered, Alexander grading preserving maps which are positive and filtered. Using continuity of $w_U$ from Lemma~\ref{lem:basic-facts-w_U}, we see that $w_U(\Map( P_m,\ve{f}))\to w_U(\Map(P_\infty,\ve{f}))$. Proposition~\ref{prop:D-surgery admissible} (surgery admissibility of $\cD$ and $\cD'$) implies that $w_U(\Map(P_\infty,\ve{f} ))>0$, completing the proof. 
\end{proof}

We now prove Theorem~\ref{thm:admissibility} for the case of knots, i.e. we show that the maps in the homotopy equivalence $\bX_{\lambda}(Y,K)\iso \ve{\CF}^-(Y_{\lambda}(K))$ converge in the chiral and $U$-adic topologies:

\begin{proof}[Proof of Theorem~\ref{thm:admissibility} for knots]
Write $F$ for one of the component maps of the homotopy equivalence in the statement. The map $F$ counts holomorphic polygons with a fixed sequence of input morphisms $\ve{f}$ between the corresponding $E_\veps$ spaces. The maps $\ve{f}$ do not determine extended morphisms between the spaces $\scE_{\veps}^\infty$ because the map $\Pi$ does not satisfy the continuity hypothesis of Definition~\ref{def:extended-morphism}. Instead, we replace $\Pi$ with the map $\Pi'\colon \scE_1^\infty\to \scE_\lambda^\infty$ as in Example~\ref{example:Pi}. Let $\ve{f}'$ denote the input sequence, modified in this way and viewed as a map between the $\scE_\veps^\infty$ spaces. We observe that since each morphism in $\ve{f}$ is homogeneous with respect to the Alexander grading, the map $\Map(\psi,\ve{f})$ is a restriction of $\Map(\psi,\ve{f}')$ to certain Alexander gradings, and therefore
\[
w_U(\Map(\psi,\ve{f}))\ge w_U(\Map(\psi,\ve{f}')).
\]
Lemma~\ref{lem:f-admissible} therefore implies that there are only finitely many $\psi$ such that $w_U(\Map(\psi,\ve{f}))<N$. 

 We may therefore decompose $F$ as an infinite sum
\begin{equation}
F=\sum_{n=0}^\infty U^n F_n \label{eq:infinite-sums}
\end{equation}
where $F_n$ is $U^{-n}$ times the map which counts the (finitely many) holomorphic polygons with $w_U(\Map(\psi,\ve{f}))=n$. The $\bF[U]$-module morphism factor of an output from the representatives of a single polygon class $\psi$ is continuous since it is the composition of continuous maps. Hence, each $F_n$ is continuous. 

We recall the general and easily verifiable fact about linear topological spaces. If $X$ and $Y$ are linear topological spaces and $\{f_i\colon X\to Y\}_{i\in \N}$ is a sequence of continuous maps such that $f_i\to 0$ uniformly, then the infinite sum $\sum_{i=1}^\infty f_i$ converges uniformly to a continuous map between the completions of $X$ and $Y$.

In particular, an infinite sums as in Equation~\eqref{eq:infinite-sums} converges in both the chiral and $U$-adic topologies to a continuous map.
\end{proof}

\subsection{Weak admissibility}

In this section we describe how the simpler notion of weak admissibility may be used for many of the maps appearing in the surgery formula and its proof. 

We recall that a Heegaard diagram $(\Sigma,\gs_1,\dots, \gs_n,\ps)$ is called \emph{weakly admissible at $\ve{p}$} if each periodic domain $P$ with $n_{\ve{p}}(P)=0$ is either zero or has both positive and negative multiplicities. 

If $F\colon M\to N$ is a map between free $\bF[U_1,\dots, U_\ell]$-modules, we define $w_U(F)$ as follows. Given $n\in \N$, let $I_n\subset \bF[U_1,\dots, U_\ell]$ denote the ideal $(U_1,\dots, U_n)^n$. We define $w_U(F)$ to be the supremum over $n$ such that $F\in \Hom(M, I_n\cdot N)$.

For many of the maps involving the link surgery formula, it is sufficient to consider Heegaard multi-diagrams $(\Sigma,\gs_1,\dots, \gs_n, \ws, \zs)$ which are weakly admissible at each complete collection $\ve{p}\subset \ws\cup \zs$. For example, such diagrams are sufficient to construct the following maps:
\begin{enumerate}
\item The differential on the link surgery complex $\bX_{\Lambda}(Y,L)$;
\item The chain maps $\ve{\CF}^-(Y_{\Lambda}(L))\to \bX_{\Lambda}(Y,L)$ and $\bX_{\Lambda}(Y,L)\to \ve{\CF}^-(Y_{\Lambda}(L))$;
\item The chain homotopy $\ve{\CF}^-(Y_{\Lambda}(L))\to \ve{\CF}^-(Y_{\Lambda}(L))$ in the proof of the link surgery formula.
\item The $A_\infty$-algebra $\End_{\Fil}(\b_0^{E_0}\oplus \b_1^{E_1})$. 
\end{enumerate}
The reason that weak admissibility is sufficient in these cases because if $\psi$ is a class of polygons and $\ve{f}$ is an input sequence for $\psi$ (as counted by one of the above maps), then
\[
w_U(\scF(\psi, \ve{f}))\ge \min(n_{w_1}(\psi),n_{z_1}(\psi))+\cdots+ \min(n_{z_\ell}(\psi), n_{z_\ell}(\psi)). 
\]
Weak admissibility at all complete collections $\ve{p}\subset \ws\cup \zs$ implies that there are only finitely many classes $\psi$ with $\min(n_{w_1}(\psi),n_{z_1}(\psi))+\cdots+ \min(n_{z_\ell}(\psi), n_{z_\ell}(\psi))$ below a given number, implying that the maps above are defined by convergent series (in the uniform topology) of continuous maps.

Weak admissibility does not appear sufficient for defining the chain homotopy $\bX_{\Lambda}(Y,L)\to \bX_{\Lambda}(Y,L)$, appearing in the proof of the surgery formula. 

\subsection{Links and extra basepoints}
\label{sec:links}

In this section, we extend the ideas of the previous section to handle links and Heegaard diagrams with multiple basepoints. We will focus the claims necessary to construct the homotopy equivalence
\[
\cX(Y)^{\frL}\boxtimes {}_{\frK} \frD_0\simeq \cX(Y')^{\frL'}.
\]
By the discussion in the previous section, the only subtlety is in defining the chain homotopy from $\cX(Y)^{\frL}\boxtimes {}_{\frK} \frD_0$ to itself, since weak admissibility at all complete collections $\ve{p}\subset \ws\cup \zs$ suffices for all of the other maps.

We therefore consider a diagram of attaching curves $\scB'$ which is indexed by $\bE_{\ell-1}\times \{0,1,\lambda,0',1'\}$. The chain homotopy from $\cX(Y)^{\frL}\boxtimes {}_{\frK} \frD_0$ to itself which appears in the equivalence is obtained by counting holomorphic polygons on a subdiagrams of 
\[
\scH':=(\Sigma,\as,\scB').
\] 
In the above, we assume that we have $\ell=|L|$ special genus 1 regions where the curves $\scB'$ are small translates of the curves $\b_0$, $\b_1$ and $\b_\lambda$ shown in Figure~\ref{fig:torus_intro}.

Definition~\ref{def:surgery-admissible}, surgery admissibility, generalizes to our present setting with the understanding that we only consider morphism sequences $\ve{f}$ corresponding to componentwise increasing sequences in $\bE_{\ell-1}\times \{0,1,\lambda,0',1'\}$. It suffices to show that $(\Sigma,\as,\scB')$ is surgery admissible. 

Note that Lemma~\ref{prop:D-surgery admissible} adapts to show that $\scB'$ itself may be constructed to be surgery admissible. It remains to prove the following:

\begin{lem} The curves $\as$ and $\scB'$ may be chosen so that $(\Sigma,\as,\scB')$ is surgery admissible. 
\end{lem}

\begin{proof} Our method for winding $\as$ will be a variation of the standard winding procedure for attaching curves. Call a curve in $\as$ \emph{separating} if it lies in boundary of two components of $\Sigma\setminus \as$, and call it \emph{non-separating} it lies in the boundary of just one component. We first pick a collection of curves $\g_1,\dots, \g_n$ which are algebraically dual to the non-separating alpha curves. 

We index the link components so that $K_\ell$ (the final component) is the one that we are performing surgery on.

Additionally, for each component $D\subset\Sigma\setminus \as$, we pick a collection of pairwise disjoint arcs which connect the components of $\d D$ to the two basepoints in $D$. If $D$ contains $w_i$ and $z_i$ for $i\in \{1,\dots, \ell-1\}$, we pick two arcs from each separating component of $\d D$. We assume one arc connects $\d D$ to $w_i$ and the other connects $\d D$ to $z_i$. If instead $D$ contains $w_i$ and $z_i$, we pick a single arc for each separating component of $\d D$, and we assume this arc connects $\d D$ to $w_\ell$.  We assume these arcs are disjoint from the $\g_1,\dots, \g_n$ curves.

We wind the $\as$ curves in both directions parallel to the $\g_1,\dots, \g_n$, and additionally perform a finger move of $\as$ curves following the other arcs. If a region $D$ contains $w_i$ and $z_i$ for $i\in \{1,\dots, \ell-1\}$, then we assume that finger move is performed so that a small portion of the corresponding alpha curve lies in the same region as $w_i$ or $z_i$. If $D$ contains $w_\ell$ and $z_\ell$, we perform a finger move along all of the arcs to achieve the configuration as shown in Figure~\ref{fig:55}.  Compare \cite{OSLinks}*{Proposition~3.6}.

\begin{figure}[ht]
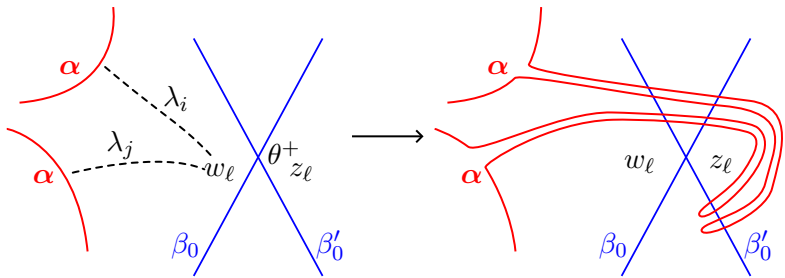
\caption{Winding the separating $\as$ curves along arcs $\lambda_i$ and $\lambda_j$.}
\label{fig:55}
\end{figure}

We claim that after the above manipulation, the diagram $(\Sigma,\as,\scB')$ is surgery admissible. We consider a weakly periodic domain $P_m$, (where $m$ denotes the number of times we have wound $\as$ along the $\g_i$). We let $\ve{f}$ be a sequence of Alexander grading preserinv extended morphisms which are positive and filtered.

 Analogously to Remark~\ref{rem:slightly-stronger-positivity}, we have 
\begin{equation}
\begin{split}
w_U(\Map(P_m,\ve{f}))\ge&\sum_{i=1}^{\ell-1} \min(n_{w_i}(P_m), n_{z_i}(P_m))\\
& +\min(n_{w_\ell}(P_m), n_{z_\ell}(P_m), N(P_m), M(P_m))
\end{split} \label{eq:U-weight-inequality}
\end{equation}
In the above, $n_{w_\ell}$, $n_{z_\ell}$, $N$ and $M$ denote the multiplicities near $w_\ell$ and $z_\ell$, as in Figure~\ref{fig:39}.  In a bit more detail, the proof of Lemma~\ref{lem:positivity-knot} adapts to show that only nonnegative powers of $U_\ell$ appear. In fact, as in Remark~\ref{rem:slightly-stronger-positivity}, the power of $U_\ell$ is in fact at least $\min(n_{w_\ell}(P_m), n_{z_\ell}(P_m), N(P_m), M(P_m))$. 
 The first part of the proof of Lemma~\ref{lem:positivity-knot} (i.e. when there is just one of $\bs_0$ or $\bs_0'$ present) shows that for $i\in \{1,\dots, \ell-1\}$, the overall $U_i$ power must be at least $\min(n_{w_i}(P_m), n_{z_i}(P_m))$. Combining these observations yields Equation~\eqref{eq:U-weight-inequality}.

We may assume without loss of generality that $|P_m|_\infty=1$. We suppose to the contrary that $w_U(\scF(P_m, \ve{f}))= 0$ for all $m$. If $\a\in \as$ is non-separating, then the coefficient $s$ of $\a$ in $\d P_m$ must be in $[-1/m, 1/m]$ since there is a number $c\in [0,1]$ such that $c+sm$ and $c-sm$ are multiplicities of $P_m$. On the other hand, if $\a\in \as$ is separating, then $\a$ has multiplicity 0 due to fact that the finger moves (shown in Figure~\ref{fig:55}) move a portion of $\a$ into each of the four regions adjacent to $\theta^+$ in Figure~\ref{fig:55}, combined with our lower bound for $w_U$ from Equation~\eqref{eq:U-weight-inequality}. 

 In particular $P_m$ has a subsequence which converges to a nonnegative weakly periodic domain $P_\infty$ on $\scB'$. It is straightforward to see that $P_\infty>0$, so $w_U(\Map(P_\infty,\ve{f}))>0$ by surgery admissibility. By continuity of $w_U$ (cf. part~\eqref{item:continuity-w_U} of Lemma~\ref{lem:basic-facts-w_U}), it follows that $w_U(\Map(P_m, \ve{f}))>0$ for large $m$. 
\end{proof}

\bibliographystyle{custom}
\def\MR#1{}
\bibliography{biblio}

\end{document}

%% file: fig_torus_intro.pdf_tex
\begingroup%
  \makeatletter%
  \providecommand\color[2][]{%
    \errmessage{(Inkscape) Color is used for the text in Inkscape, but the package 'color.sty' is not loaded}%
    \renewcommand\color[2][]{}%
  }%
  \providecommand\transparent[1]{%
    \errmessage{(Inkscape) Transparency is used (non-zero) for the text in Inkscape, but the package 'transparent.sty' is not loaded}%
    \renewcommand\transparent[1]{}%
  }%
  \providecommand\rotatebox[2]{#2}%
  \newcommand*\fsize{\dimexpr\f@size pt\relax}%
  \newcommand*\lineheight[1]{\fontsize{\fsize}{#1\fsize}\selectfont}%
  \ifx\svgwidth\undefined%
    \setlength{\unitlength}{123.75120752bp}%
    \ifx\svgscale\undefined%
      \relax%
    \else%
      \setlength{\unitlength}{\unitlength * \real{\svgscale}}%
    \fi%
  \else%
    \setlength{\unitlength}{\svgwidth}%
  \fi%
  \global\let\svgwidth\undefined%
  \global\let\svgscale\undefined%
  \makeatother%
  \begin{picture}(1,0.91515224)%
    \lineheight{1}%
    \setlength\tabcolsep{0pt}%
    \put(0,0){\includegraphics[width=\unitlength,page=1]{fig_torus_intro.pdf}}%
    \put(0.5186749,0.67109149){\color[rgb]{0,0,1}\makebox(0,0)[lt]{\lineheight{1.25}\smash{\begin{tabular}[t]{l}$\b_0$\end{tabular}}}}%
    \put(0.8075066,0.2941312){\color[rgb]{0.37647059,0,1}\makebox(0,0)[lt]{\lineheight{1.25}\smash{\begin{tabular}[t]{l}$\b_1$\end{tabular}}}}%
    \put(0.51925877,0.36285971){\makebox(0,0)[lt]{\lineheight{1.25}\smash{\begin{tabular}[t]{l}$\theta_\sigma^{+}$\end{tabular}}}}%
    \put(0.47784334,0.07623302){\makebox(0,0)[rt]{\lineheight{1.25}\smash{\begin{tabular}[t]{r}$\theta_\tau^+$\end{tabular}}}}%
    \put(0,0){\includegraphics[width=\unitlength,page=2]{fig_torus_intro.pdf}}%
    \put(0.43492168,0.18572002){\makebox(0,0)[rt]{\lineheight{1.25}\smash{\begin{tabular}[t]{r}$w$\end{tabular}}}}%
    \put(0.55646425,0.18572002){\makebox(0,0)[lt]{\lineheight{1.25}\smash{\begin{tabular}[t]{l}$z$\end{tabular}}}}%
    \put(0,0){\includegraphics[width=\unitlength,page=3]{fig_torus_intro.pdf}}%
    \put(0.67382746,0.47870199){\color[rgb]{1,0,0}\makebox(0,0)[lt]{\lineheight{1.25}\smash{\begin{tabular}[t]{l}$\b_\lambda$\end{tabular}}}}%
    \put(0,0){\includegraphics[width=\unitlength,page=4]{fig_torus_intro.pdf}}%
    \put(0.78432684,0.11805983){\makebox(0,0)[lt]{\lineheight{1.25}\smash{\begin{tabular}[t]{l}$K$\end{tabular}}}}%
  \end{picture}%
\endgroup%

%% file: fig54.pdf_tex
\begingroup%
  \makeatletter%
  \providecommand\color[2][]{%
    \errmessage{(Inkscape) Color is used for the text in Inkscape, but the package 'color.sty' is not loaded}%
    \renewcommand\color[2][]{}%
  }%
  \providecommand\transparent[1]{%
    \errmessage{(Inkscape) Transparency is used (non-zero) for the text in Inkscape, but the package 'transparent.sty' is not loaded}%
    \renewcommand\transparent[1]{}%
  }%
  \providecommand\rotatebox[2]{#2}%
  \newcommand*\fsize{\dimexpr\f@size pt\relax}%
  \newcommand*\lineheight[1]{\fontsize{\fsize}{#1\fsize}\selectfont}%
  \ifx\svgwidth\undefined%
    \setlength{\unitlength}{109.62697886bp}%
    \ifx\svgscale\undefined%
      \relax%
    \else%
      \setlength{\unitlength}{\unitlength * \real{\svgscale}}%
    \fi%
  \else%
    \setlength{\unitlength}{\svgwidth}%
  \fi%
  \global\let\svgwidth\undefined%
  \global\let\svgscale\undefined%
  \makeatother%
  \begin{picture}(1,0.91515196)%
    \lineheight{1}%
    \setlength\tabcolsep{0pt}%
    \put(0,0){\includegraphics[width=\unitlength,page=1]{fig54.pdf}}%
    \put(0.39438558,0.47294454){\makebox(0,0)[rt]{\lineheight{1.25}\smash{\begin{tabular}[t]{r}$w$\end{tabular}}}}%
    \put(0,0){\includegraphics[width=\unitlength,page=2]{fig54.pdf}}%
    \put(0.77783425,0.36149513){\color[rgb]{1,0,0}\makebox(0,0)[lt]{\lineheight{1.25}\smash{\begin{tabular}[t]{l}$\g_0$\end{tabular}}}}%
    \put(0.5897117,0.43327574){\color[rgb]{0,0.77254902,0}\makebox(0,0)[rt]{\lineheight{1.25}\smash{\begin{tabular}[t]{r}$\g_2$\end{tabular}}}}%
    \put(0.5944003,0.15795228){\color[rgb]{0,0,1}\makebox(0,0)[t]{\lineheight{1.25}\smash{\begin{tabular}[t]{c}$\g_1$\end{tabular}}}}%
  \end{picture}%
\endgroup%

%% file: fig45.pdf_tex
\begingroup%
  \makeatletter%
  \providecommand\color[2][]{%
    \errmessage{(Inkscape) Color is used for the text in Inkscape, but the package 'color.sty' is not loaded}%
    \renewcommand\color[2][]{}%
  }%
  \providecommand\transparent[1]{%
    \errmessage{(Inkscape) Transparency is used (non-zero) for the text in Inkscape, but the package 'transparent.sty' is not loaded}%
    \renewcommand\transparent[1]{}%
  }%
  \providecommand\rotatebox[2]{#2}%
  \newcommand*\fsize{\dimexpr\f@size pt\relax}%
  \newcommand*\lineheight[1]{\fontsize{\fsize}{#1\fsize}\selectfont}%
  \ifx\svgwidth\undefined%
    \setlength{\unitlength}{154.17924126bp}%
    \ifx\svgscale\undefined%
      \relax%
    \else%
      \setlength{\unitlength}{\unitlength * \real{\svgscale}}%
    \fi%
  \else%
    \setlength{\unitlength}{\svgwidth}%
  \fi%
  \global\let\svgwidth\undefined%
  \global\let\svgscale\undefined%
  \makeatother%
  \begin{picture}(1,0.63803695)%
    \lineheight{1}%
    \setlength\tabcolsep{0pt}%
    \put(0,0){\includegraphics[width=\unitlength,page=1]{fig45.pdf}}%
    \put(0.61773721,0.4356565){\makebox(0,0)[rt]{\lineheight{1.25}\smash{\begin{tabular}[t]{r}$\gs_1$\end{tabular}}}}%
    \put(0.67549482,0.23127121){\makebox(0,0)[rt]{\lineheight{1.25}\smash{\begin{tabular}[t]{r}$\gs_2$\end{tabular}}}}%
    \put(0.27178672,0.25069965){\makebox(0,0)[lt]{\lineheight{1.25}\smash{\begin{tabular}[t]{l}$\gs_3$\end{tabular}}}}%
    \put(0.35085055,0.45623135){\makebox(0,0)[lt]{\lineheight{1.25}\smash{\begin{tabular}[t]{l}$\gs_5$\end{tabular}}}}%
    \put(0.49407584,0.13719729){\makebox(0,0)[rt]{\lineheight{1.25}\smash{\begin{tabular}[t]{r}$\gs_3$\end{tabular}}}}%
    \put(0.65926453,0.4635342){\makebox(0,0)[lt]{\lineheight{1.25}\smash{\begin{tabular}[t]{l}$\omega(e^{m_1})$\end{tabular}}}}%
    \put(0.70106257,0.18934679){\makebox(0,0)[lt]{\lineheight{1.25}\smash{\begin{tabular}[t]{l}$\omega(e^{m_2})$\end{tabular}}}}%
    \put(0.43188132,0.02676243){\makebox(0,0)[t]{\lineheight{1.25}\smash{\begin{tabular}[t]{c}$\omega(e^{m_3})$\end{tabular}}}}%
    \put(0.239687,0.22615877){\makebox(0,0)[rt]{\lineheight{1.25}\smash{\begin{tabular}[t]{r}$\omega(e^{m_4})$\end{tabular}}}}%
    \put(0.37230424,0.52678686){\makebox(0,0)[rt]{\lineheight{1.25}\smash{\begin{tabular}[t]{r}$\omega(e^{m_5})$\end{tabular}}}}%
    \put(0.82918671,0.36537004){\makebox(0,0)[lt]{\lineheight{1.25}\smash{\begin{tabular}[t]{l}$\lb\ve{x}_{1,2},\phi_{1,2}\rb$\end{tabular}}}}%
    \put(0.6415561,0.00521917){\makebox(0,0)[lt]{\lineheight{1.25}\smash{\begin{tabular}[t]{l}$\lb\ve{x}_{2,3},\phi_{2,3}\rb$\end{tabular}}}}%
    \put(0.24598141,0.04509864){\makebox(0,0)[rt]{\lineheight{1.25}\smash{\begin{tabular}[t]{r}$\lb\ve{x}_{3,4},\phi_{3,4}\rb$\end{tabular}}}}%
    \put(0.17074945,0.42855307){\makebox(0,0)[rt]{\lineheight{1.25}\smash{\begin{tabular}[t]{r}$\lb\ve{x}_{4,5},\phi_{4,5}\rb$\end{tabular}}}}%
  \end{picture}%
\endgroup%

%% file: fig46.pdf_tex
\begingroup%
  \makeatletter%
  \providecommand\color[2][]{%
    \errmessage{(Inkscape) Color is used for the text in Inkscape, but the package 'color.sty' is not loaded}%
    \renewcommand\color[2][]{}%
  }%
  \providecommand\transparent[1]{%
    \errmessage{(Inkscape) Transparency is used (non-zero) for the text in Inkscape, but the package 'transparent.sty' is not loaded}%
    \renewcommand\transparent[1]{}%
  }%
  \providecommand\rotatebox[2]{#2}%
  \newcommand*\fsize{\dimexpr\f@size pt\relax}%
  \newcommand*\lineheight[1]{\fontsize{\fsize}{#1\fsize}\selectfont}%
  \ifx\svgwidth\undefined%
    \setlength{\unitlength}{119.02111852bp}%
    \ifx\svgscale\undefined%
      \relax%
    \else%
      \setlength{\unitlength}{\unitlength * \real{\svgscale}}%
    \fi%
  \else%
    \setlength{\unitlength}{\svgwidth}%
  \fi%
  \global\let\svgwidth\undefined%
  \global\let\svgscale\undefined%
  \makeatother%
  \begin{picture}(1,0.42759897)%
    \lineheight{1}%
    \setlength\tabcolsep{0pt}%
    \put(0,0){\includegraphics[width=\unitlength,page=1]{fig46.pdf}}%
    \put(0.48748632,0.21416549){\makebox(0,0)[rt]{\lineheight{1.25}\smash{\begin{tabular}[t]{r}$w$\end{tabular}}}}%
    \put(0.54857101,0.21416569){\makebox(0,0)[lt]{\lineheight{1.25}\smash{\begin{tabular}[t]{l}$z$\end{tabular}}}}%
    \put(0.54107791,0.31224269){\color[rgb]{0,0,1}\makebox(0,0)[lt]{\lineheight{1.25}\smash{\begin{tabular}[t]{l}$\b_0$\end{tabular}}}}%
    \put(0.26426818,0.10391035){\color[rgb]{1,0,0}\makebox(0,0)[rt]{\lineheight{1.25}\smash{\begin{tabular}[t]{r}$\a$\end{tabular}}}}%
  \end{picture}%
\endgroup%

%% file: fig47.pdf_tex
\begingroup%
  \makeatletter%
  \providecommand\color[2][]{%
    \errmessage{(Inkscape) Color is used for the text in Inkscape, but the package 'color.sty' is not loaded}%
    \renewcommand\color[2][]{}%
  }%
  \providecommand\transparent[1]{%
    \errmessage{(Inkscape) Transparency is used (non-zero) for the text in Inkscape, but the package 'transparent.sty' is not loaded}%
    \renewcommand\transparent[1]{}%
  }%
  \providecommand\rotatebox[2]{#2}%
  \newcommand*\fsize{\dimexpr\f@size pt\relax}%
  \newcommand*\lineheight[1]{\fontsize{\fsize}{#1\fsize}\selectfont}%
  \ifx\svgwidth\undefined%
    \setlength{\unitlength}{115.18321052bp}%
    \ifx\svgscale\undefined%
      \relax%
    \else%
      \setlength{\unitlength}{\unitlength * \real{\svgscale}}%
    \fi%
  \else%
    \setlength{\unitlength}{\svgwidth}%
  \fi%
  \global\let\svgwidth\undefined%
  \global\let\svgscale\undefined%
  \makeatother%
  \begin{picture}(1,0.90638695)%
    \lineheight{1}%
    \setlength\tabcolsep{0pt}%
    \put(0,0){\includegraphics[width=\unitlength,page=1]{fig47.pdf}}%
    \put(0.51477588,0.6810825){\color[rgb]{0,0,1}\makebox(0,0)[lt]{\lineheight{1.25}\smash{\begin{tabular}[t]{l}$\b_0$\end{tabular}}}}%
    \put(0.81918578,0.27853692){\color[rgb]{0.37647059,0,1}\makebox(0,0)[lt]{\lineheight{1.25}\smash{\begin{tabular}[t]{l}$\b_1$\end{tabular}}}}%
    \put(0.52841397,0.36031568){\makebox(0,0)[lt]{\lineheight{1.25}\smash{\begin{tabular}[t]{l}$\theta_\sigma^{+}$\end{tabular}}}}%
    \put(0.47174178,0.0502421){\makebox(0,0)[rt]{\lineheight{1.25}\smash{\begin{tabular}[t]{r}$\theta_\tau^+$\end{tabular}}}}%
    \put(0,0){\includegraphics[width=\unitlength,page=2]{fig47.pdf}}%
    \put(0.42650522,0.18215764){\makebox(0,0)[rt]{\lineheight{1.25}\smash{\begin{tabular}[t]{r}$w$\end{tabular}}}}%
    \put(0.55460331,0.18215764){\makebox(0,0)[lt]{\lineheight{1.25}\smash{\begin{tabular}[t]{l}$z$\end{tabular}}}}%
    \put(0,0){\includegraphics[width=\unitlength,page=3]{fig47.pdf}}%
    \put(0.82640425,0.11773934){\makebox(0,0)[lt]{\lineheight{1.25}\smash{\begin{tabular}[t]{l}$K$\end{tabular}}}}%
  \end{picture}%
\endgroup%

%% file: fig37.pdf_tex
\begingroup%
  \makeatletter%
  \providecommand\color[2][]{%
    \errmessage{(Inkscape) Color is used for the text in Inkscape, but the package 'color.sty' is not loaded}%
    \renewcommand\color[2][]{}%
  }%
  \providecommand\transparent[1]{%
    \errmessage{(Inkscape) Transparency is used (non-zero) for the text in Inkscape, but the package 'transparent.sty' is not loaded}%
    \renewcommand\transparent[1]{}%
  }%
  \providecommand\rotatebox[2]{#2}%
  \newcommand*\fsize{\dimexpr\f@size pt\relax}%
  \newcommand*\lineheight[1]{\fontsize{\fsize}{#1\fsize}\selectfont}%
  \ifx\svgwidth\undefined%
    \setlength{\unitlength}{318.85819851bp}%
    \ifx\svgscale\undefined%
      \relax%
    \else%
      \setlength{\unitlength}{\unitlength * \real{\svgscale}}%
    \fi%
  \else%
    \setlength{\unitlength}{\svgwidth}%
  \fi%
  \global\let\svgwidth\undefined%
  \global\let\svgscale\undefined%
  \makeatother%
  \begin{picture}(1,0.45361771)%
    \lineheight{1}%
    \setlength\tabcolsep{0pt}%
    \put(0,0){\includegraphics[width=\unitlength,page=1]{fig37.pdf}}%
    \put(0.19587026,0.09546049){\makebox(0,0)[t]{\lineheight{1.25}\smash{\begin{tabular}[t]{c}$w$\end{tabular}}}}%
    \put(0.26776716,0.09546049){\makebox(0,0)[t]{\lineheight{1.25}\smash{\begin{tabular}[t]{c}$z$\end{tabular}}}}%
    \put(0.74425877,0.09580412){\makebox(0,0)[t]{\lineheight{1.25}\smash{\begin{tabular}[t]{c}$w$\end{tabular}}}}%
    \put(0.80980012,0.09580412){\makebox(0,0)[t]{\lineheight{1.25}\smash{\begin{tabular}[t]{c}$z$\end{tabular}}}}%
    \put(0.23427269,0.23809875){\makebox(0,0)[lt]{\lineheight{1.25}\smash{\begin{tabular}[t]{l}$\xs_0$\end{tabular}}}}%
    \put(0.17088654,0.23809875){\makebox(0,0)[lt]{\lineheight{1.25}\smash{\begin{tabular}[t]{l}$\xs_1$\end{tabular}}}}%
    \put(0.78132226,0.23691026){\makebox(0,0)[lt]{\lineheight{1.25}\smash{\begin{tabular}[t]{l}$\xs_0$\end{tabular}}}}%
    \put(0.71793613,0.23691026){\makebox(0,0)[lt]{\lineheight{1.25}\smash{\begin{tabular}[t]{l}$\xs_1$\end{tabular}}}}%
    \put(0.23895138,0.17128294){\makebox(0,0)[lt]{\lineheight{1.25}\smash{\begin{tabular}[t]{l}$\theta_\sigma^+$\end{tabular}}}}%
    \put(0.7656729,0.05122693){\makebox(0,0)[rt]{\lineheight{1.25}\smash{\begin{tabular}[t]{r}$\theta_\tau^+$\end{tabular}}}}%
    \put(0,0){\includegraphics[width=\unitlength,page=2]{fig37.pdf}}%
    \put(0.34958678,0.23834595){\makebox(0,0)[lt]{\lineheight{1.25}\smash{\begin{tabular}[t]{l}$\a$\end{tabular}}}}%
    \put(0.23561083,0.32645011){\color[rgb]{0,0,1}\makebox(0,0)[lt]{\lineheight{1.25}\smash{\begin{tabular}[t]{l}$\b_0$\end{tabular}}}}%
    \put(0.39457953,0.33693867){\color[rgb]{0.39215686,0.39215686,1}\makebox(0,0)[rt]{\lineheight{1.25}\smash{\begin{tabular}[t]{r}$\b_1$\end{tabular}}}}%
  \end{picture}%
\endgroup%

%% file: fig51.pdf_tex
\begingroup%
  \makeatletter%
  \providecommand\color[2][]{%
    \errmessage{(Inkscape) Color is used for the text in Inkscape, but the package 'color.sty' is not loaded}%
    \renewcommand\color[2][]{}%
  }%
  \providecommand\transparent[1]{%
    \errmessage{(Inkscape) Transparency is used (non-zero) for the text in Inkscape, but the package 'transparent.sty' is not loaded}%
    \renewcommand\transparent[1]{}%
  }%
  \providecommand\rotatebox[2]{#2}%
  \newcommand*\fsize{\dimexpr\f@size pt\relax}%
  \newcommand*\lineheight[1]{\fontsize{\fsize}{#1\fsize}\selectfont}%
  \ifx\svgwidth\undefined%
    \setlength{\unitlength}{147.39985663bp}%
    \ifx\svgscale\undefined%
      \relax%
    \else%
      \setlength{\unitlength}{\unitlength * \real{\svgscale}}%
    \fi%
  \else%
    \setlength{\unitlength}{\svgwidth}%
  \fi%
  \global\let\svgwidth\undefined%
  \global\let\svgscale\undefined%
  \makeatother%
  \begin{picture}(1,1.00423747)%
    \lineheight{1}%
    \setlength\tabcolsep{0pt}%
    \put(0,0){\includegraphics[width=\unitlength,page=1]{fig51.pdf}}%
    \put(0.44143611,0.24358744){\makebox(0,0)[t]{\lineheight{1.25}\smash{\begin{tabular}[t]{c}$w$\end{tabular}}}}%
    \put(0.61622933,0.24358744){\makebox(0,0)[t]{\lineheight{1.25}\smash{\begin{tabular}[t]{c}$z$\end{tabular}}}}%
    \put(0,0){\includegraphics[width=\unitlength,page=2]{fig51.pdf}}%
    \put(0.5088142,0.31334057){\makebox(0,0)[rt]{\lineheight{1.25}\smash{\begin{tabular}[t]{r}$\theta_\sigma^+$\end{tabular}}}}%
    \put(0,0){\includegraphics[width=\unitlength,page=3]{fig51.pdf}}%
    \put(0.76470583,0.64152794){\color[rgb]{0.39215686,0.39215686,1}\makebox(0,0)[lt]{\lineheight{1.25}\smash{\begin{tabular}[t]{l}$\b_1$\end{tabular}}}}%
    \put(0.31092411,0.64606572){\color[rgb]{1,0,0}\makebox(0,0)[lt]{\lineheight{1.25}\smash{\begin{tabular}[t]{l}$\a$\end{tabular}}}}%
    \put(0.55142861,0.50539316){\color[rgb]{0,0,1}\makebox(0,0)[lt]{\lineheight{1.25}\smash{\begin{tabular}[t]{l}$\b_0$\end{tabular}}}}%
    \put(0,0){\includegraphics[width=\unitlength,page=4]{fig51.pdf}}%
    \put(0.34768619,0.16656927){\makebox(0,0)[rt]{\lineheight{1.25}\smash{\begin{tabular}[t]{r}$\zs_1^+$\end{tabular}}}}%
    \put(0.56730777,0.73489752){\makebox(0,0)[lt]{\lineheight{1.25}\smash{\begin{tabular}[t]{l}$\ys_0^+$\end{tabular}}}}%
    \put(0,0){\includegraphics[width=\unitlength,page=5]{fig51.pdf}}%
  \end{picture}%
\endgroup%

%% file: fig17.pdf_tex
\begingroup%
  \makeatletter%
  \providecommand\color[2][]{%
    \errmessage{(Inkscape) Color is used for the text in Inkscape, but the package 'color.sty' is not loaded}%
    \renewcommand\color[2][]{}%
  }%
  \providecommand\transparent[1]{%
    \errmessage{(Inkscape) Transparency is used (non-zero) for the text in Inkscape, but the package 'transparent.sty' is not loaded}%
    \renewcommand\transparent[1]{}%
  }%
  \providecommand\rotatebox[2]{#2}%
  \newcommand*\fsize{\dimexpr\f@size pt\relax}%
  \newcommand*\lineheight[1]{\fontsize{\fsize}{#1\fsize}\selectfont}%
  \ifx\svgwidth\undefined%
    \setlength{\unitlength}{165.75117746bp}%
    \ifx\svgscale\undefined%
      \relax%
    \else%
      \setlength{\unitlength}{\unitlength * \real{\svgscale}}%
    \fi%
  \else%
    \setlength{\unitlength}{\svgwidth}%
  \fi%
  \global\let\svgwidth\undefined%
  \global\let\svgscale\undefined%
  \makeatother%
  \begin{picture}(1,0.90950288)%
    \lineheight{1}%
    \setlength\tabcolsep{0pt}%
    \put(0,0){\includegraphics[width=\unitlength,page=1]{fig17.pdf}}%
    \put(0.51657864,0.68393331){\color[rgb]{0,0,1}\makebox(0,0)[lt]{\lineheight{1.25}\smash{\begin{tabular}[t]{l}$\b_0$\end{tabular}}}}%
    \put(0.82265419,0.27910381){\color[rgb]{0.37647059,0,1}\makebox(0,0)[lt]{\lineheight{1.25}\smash{\begin{tabular}[t]{l}$\b_1$\end{tabular}}}}%
    \put(0.51719749,0.35729954){\makebox(0,0)[lt]{\lineheight{1.25}\smash{\begin{tabular}[t]{l}$\theta_\sigma^{+}$\end{tabular}}}}%
    \put(0.47330937,0.0535606){\makebox(0,0)[rt]{\lineheight{1.25}\smash{\begin{tabular}[t]{r}$\theta_\tau^+$\end{tabular}}}}%
    \put(0,0){\includegraphics[width=\unitlength,page=2]{fig17.pdf}}%
    \put(0.42782527,0.17717467){\makebox(0,0)[rt]{\lineheight{1.25}\smash{\begin{tabular}[t]{r}$w$\end{tabular}}}}%
    \put(0.55662422,0.17717467){\makebox(0,0)[lt]{\lineheight{1.25}\smash{\begin{tabular}[t]{l}$z$\end{tabular}}}}%
    \put(0,0){\includegraphics[width=\unitlength,page=3]{fig17.pdf}}%
    \put(0.68099427,0.47737587){\color[rgb]{1,0,0}\makebox(0,0)[lt]{\lineheight{1.25}\smash{\begin{tabular}[t]{l}$\b_\lambda$\end{tabular}}}}%
    \put(0,0){\includegraphics[width=\unitlength,page=4]{fig17.pdf}}%
    \put(0.05096066,0.12612823){\makebox(0,0)[lt]{\lineheight{1.25}\smash{\begin{tabular}[t]{l}$K$\end{tabular}}}}%
  \end{picture}%
\endgroup%

%% file: fig18.pdf_tex
\begingroup%
  \makeatletter%
  \providecommand\color[2][]{%
    \errmessage{(Inkscape) Color is used for the text in Inkscape, but the package 'color.sty' is not loaded}%
    \renewcommand\color[2][]{}%
  }%
  \providecommand\transparent[1]{%
    \errmessage{(Inkscape) Transparency is used (non-zero) for the text in Inkscape, but the package 'transparent.sty' is not loaded}%
    \renewcommand\transparent[1]{}%
  }%
  \providecommand\rotatebox[2]{#2}%
  \newcommand*\fsize{\dimexpr\f@size pt\relax}%
  \newcommand*\lineheight[1]{\fontsize{\fsize}{#1\fsize}\selectfont}%
  \ifx\svgwidth\undefined%
    \setlength{\unitlength}{359.79253163bp}%
    \ifx\svgscale\undefined%
      \relax%
    \else%
      \setlength{\unitlength}{\unitlength * \real{\svgscale}}%
    \fi%
  \else%
    \setlength{\unitlength}{\svgwidth}%
  \fi%
  \global\let\svgwidth\undefined%
  \global\let\svgscale\undefined%
  \makeatother%
  \begin{picture}(1,0.30408874)%
    \lineheight{1}%
    \setlength\tabcolsep{0pt}%
    \put(0,0){\includegraphics[width=\unitlength,page=1]{fig18.pdf}}%
    \put(0.83817767,0.22852539){\color[rgb]{0,0,1}\makebox(0,0)[lt]{\lineheight{1.25}\smash{\begin{tabular}[t]{l}$\b_0$\end{tabular}}}}%
    \put(0.94045971,0.09342838){\color[rgb]{0.37647059,0,1}\makebox(0,0)[lt]{\lineheight{1.25}\smash{\begin{tabular}[t]{l}$\b_1$\end{tabular}}}}%
    \put(0.83838434,0.11650339){\makebox(0,0)[lt]{\lineheight{1.25}\smash{\begin{tabular}[t]{l}$\theta_\sigma^{+}$\end{tabular}}}}%
    \put(0,0){\includegraphics[width=\unitlength,page=2]{fig18.pdf}}%
    \put(0.80852074,0.06545337){\makebox(0,0)[rt]{\lineheight{1.25}\smash{\begin{tabular}[t]{r}$w$\end{tabular}}}}%
    \put(0.85155884,0.06545337){\makebox(0,0)[lt]{\lineheight{1.25}\smash{\begin{tabular}[t]{l}$z$\end{tabular}}}}%
    \put(0,0){\includegraphics[width=\unitlength,page=3]{fig18.pdf}}%
    \put(0.50554207,0.22852533){\color[rgb]{0,0,1}\makebox(0,0)[lt]{\lineheight{1.25}\smash{\begin{tabular}[t]{l}$\b_0$\end{tabular}}}}%
    \put(0.60781954,0.09342838){\color[rgb]{0.37647059,0,1}\makebox(0,0)[lt]{\lineheight{1.25}\smash{\begin{tabular}[t]{l}$\b_1$\end{tabular}}}}%
    \put(0.49108358,0.01888356){\makebox(0,0)[rt]{\lineheight{1.25}\smash{\begin{tabular}[t]{r}$\theta_\tau^+$\end{tabular}}}}%
    \put(0,0){\includegraphics[width=\unitlength,page=4]{fig18.pdf}}%
    \put(0.47588496,0.06545337){\makebox(0,0)[rt]{\lineheight{1.25}\smash{\begin{tabular}[t]{r}$w$\end{tabular}}}}%
    \put(0.51892325,0.06545337){\makebox(0,0)[lt]{\lineheight{1.25}\smash{\begin{tabular}[t]{l}$z$\end{tabular}}}}%
    \put(0,0){\includegraphics[width=\unitlength,page=5]{fig18.pdf}}%
    \put(0.56048213,0.15959451){\color[rgb]{1,0,0}\makebox(0,0)[lt]{\lineheight{1.25}\smash{\begin{tabular}[t]{l}$\b_\lambda$\end{tabular}}}}%
    \put(0,0){\includegraphics[width=\unitlength,page=6]{fig18.pdf}}%
    \put(0.17290323,0.22852533){\color[rgb]{0,0,1}\makebox(0,0)[lt]{\lineheight{1.25}\smash{\begin{tabular}[t]{l}$\b_0$\end{tabular}}}}%
    \put(0.27517937,0.09342838){\color[rgb]{0.37647059,0,1}\makebox(0,0)[lt]{\lineheight{1.25}\smash{\begin{tabular}[t]{l}$\b_1$\end{tabular}}}}%
    \put(0,0){\includegraphics[width=\unitlength,page=7]{fig18.pdf}}%
    \put(0.14324594,0.06545337){\makebox(0,0)[rt]{\lineheight{1.25}\smash{\begin{tabular}[t]{r}$w$\end{tabular}}}}%
    \put(0.18628449,0.06545337){\makebox(0,0)[lt]{\lineheight{1.25}\smash{\begin{tabular}[t]{l}$z$\end{tabular}}}}%
    \put(0,0){\includegraphics[width=\unitlength,page=8]{fig18.pdf}}%
  \end{picture}%
\endgroup%

%% file: fig19.pdf_tex
\begingroup%
  \makeatletter%
  \providecommand\color[2][]{%
    \errmessage{(Inkscape) Color is used for the text in Inkscape, but the package 'color.sty' is not loaded}%
    \renewcommand\color[2][]{}%
  }%
  \providecommand\transparent[1]{%
    \errmessage{(Inkscape) Transparency is used (non-zero) for the text in Inkscape, but the package 'transparent.sty' is not loaded}%
    \renewcommand\transparent[1]{}%
  }%
  \providecommand\rotatebox[2]{#2}%
  \newcommand*\fsize{\dimexpr\f@size pt\relax}%
  \newcommand*\lineheight[1]{\fontsize{\fsize}{#1\fsize}\selectfont}%
  \ifx\svgwidth\undefined%
    \setlength{\unitlength}{240.11365319bp}%
    \ifx\svgscale\undefined%
      \relax%
    \else%
      \setlength{\unitlength}{\unitlength * \real{\svgscale}}%
    \fi%
  \else%
    \setlength{\unitlength}{\svgwidth}%
  \fi%
  \global\let\svgwidth\undefined%
  \global\let\svgscale\undefined%
  \makeatother%
  \begin{picture}(1,0.45565448)%
    \lineheight{1}%
    \setlength\tabcolsep{0pt}%
    \put(0,0){\includegraphics[width=\unitlength,page=1]{fig19.pdf}}%
    \put(0.75751736,0.34242833){\color[rgb]{0,0,1}\makebox(0,0)[lt]{\lineheight{1.25}\smash{\begin{tabular}[t]{l}$\b_0$\end{tabular}}}}%
    \put(0.91077263,0.1399955){\color[rgb]{0.37647059,0,1}\makebox(0,0)[lt]{\lineheight{1.25}\smash{\begin{tabular}[t]{l}$\b_1$\end{tabular}}}}%
    \put(0.73585238,0.02829564){\makebox(0,0)[rt]{\lineheight{1.25}\smash{\begin{tabular}[t]{r}$\theta_\tau^+$\end{tabular}}}}%
    \put(0,0){\includegraphics[width=\unitlength,page=2]{fig19.pdf}}%
    \put(0.7193254,0.09807708){\makebox(0,0)[rt]{\lineheight{1.25}\smash{\begin{tabular}[t]{r}$w$\end{tabular}}}}%
    \put(0.78381499,0.09807708){\makebox(0,0)[lt]{\lineheight{1.25}\smash{\begin{tabular}[t]{l}$z$\end{tabular}}}}%
    \put(0.81079913,0.25870648){\color[rgb]{1,0,0}\makebox(0,0)[lt]{\lineheight{1.25}\smash{\begin{tabular}[t]{l}$\b_\lambda$\end{tabular}}}}%
    \put(0.51207808,0.25881413){\color[rgb]{1,0.19607843,0.19607843}\makebox(0,0)[lt]{\lineheight{1.25}\smash{\begin{tabular}[t]{l}$\b_\lambda'$\end{tabular}}}}%
    \put(0,0){\includegraphics[width=\unitlength,page=3]{fig19.pdf}}%
    \put(0.25908268,0.34242833){\color[rgb]{0,0,1}\makebox(0,0)[lt]{\lineheight{1.25}\smash{\begin{tabular}[t]{l}$\b_0$\end{tabular}}}}%
    \put(0,0){\includegraphics[width=\unitlength,page=4]{fig19.pdf}}%
    \put(0.21464343,0.09807707){\makebox(0,0)[rt]{\lineheight{1.25}\smash{\begin{tabular}[t]{r}$w$\end{tabular}}}}%
    \put(0.27913356,0.09807707){\makebox(0,0)[lt]{\lineheight{1.25}\smash{\begin{tabular}[t]{l}$z$\end{tabular}}}}%
    \put(0,0){\includegraphics[width=\unitlength,page=5]{fig19.pdf}}%
    \put(0.44435151,0.24021902){\makebox(0,0)[t]{\lineheight{1.25}\smash{\begin{tabular}[t]{c}$\theta^+$\end{tabular}}}}%
  \end{picture}%
\endgroup%

%% file: fig20.pdf_tex
\begingroup%
  \makeatletter%
  \providecommand\color[2][]{%
    \errmessage{(Inkscape) Color is used for the text in Inkscape, but the package 'color.sty' is not loaded}%
    \renewcommand\color[2][]{}%
  }%
  \providecommand\transparent[1]{%
    \errmessage{(Inkscape) Transparency is used (non-zero) for the text in Inkscape, but the package 'transparent.sty' is not loaded}%
    \renewcommand\transparent[1]{}%
  }%
  \providecommand\rotatebox[2]{#2}%
  \newcommand*\fsize{\dimexpr\f@size pt\relax}%
  \newcommand*\lineheight[1]{\fontsize{\fsize}{#1\fsize}\selectfont}%
  \ifx\svgwidth\undefined%
    \setlength{\unitlength}{240.68059083bp}%
    \ifx\svgscale\undefined%
      \relax%
    \else%
      \setlength{\unitlength}{\unitlength * \real{\svgscale}}%
    \fi%
  \else%
    \setlength{\unitlength}{\svgwidth}%
  \fi%
  \global\let\svgwidth\undefined%
  \global\let\svgscale\undefined%
  \makeatother%
  \begin{picture}(1,0.45701115)%
    \lineheight{1}%
    \setlength\tabcolsep{0pt}%
    \put(0,0){\includegraphics[width=\unitlength,page=1]{fig20.pdf}}%
    \put(0.909805,0.14084343){\color[rgb]{0.37647059,0,1}\makebox(0,0)[lt]{\lineheight{1.25}\smash{\begin{tabular}[t]{l}$\b_1$\end{tabular}}}}%
    \put(0.71577695,0.03098226){\makebox(0,0)[rt]{\lineheight{1.25}\smash{\begin{tabular}[t]{r}$\theta_\tau^+$\end{tabular}}}}%
    \put(0,0){\includegraphics[width=\unitlength,page=2]{fig20.pdf}}%
    \put(0.71257651,0.09902374){\makebox(0,0)[rt]{\lineheight{1.25}\smash{\begin{tabular}[t]{r}$w$\end{tabular}}}}%
    \put(0.77691428,0.09902374){\makebox(0,0)[lt]{\lineheight{1.25}\smash{\begin{tabular}[t]{l}$z$\end{tabular}}}}%
    \put(0,0){\includegraphics[width=\unitlength,page=3]{fig20.pdf}}%
    \put(0.83904041,0.23975497){\color[rgb]{1,0,0}\makebox(0,0)[lt]{\lineheight{1.25}\smash{\begin{tabular}[t]{l}$\b_\lambda$\end{tabular}}}}%
    \put(0,0){\includegraphics[width=\unitlength,page=4]{fig20.pdf}}%
    \put(0.41998817,0.09721674){\color[rgb]{0.37647059,0,1}\makebox(0,0)[rt]{\lineheight{1.25}\smash{\begin{tabular}[t]{r}$\b_1$\end{tabular}}}}%
    \put(0,0){\includegraphics[width=\unitlength,page=5]{fig20.pdf}}%
    \put(0.2153156,0.09902374){\makebox(0,0)[rt]{\lineheight{1.25}\smash{\begin{tabular}[t]{r}$w$\end{tabular}}}}%
    \put(0.27965378,0.09902374){\makebox(0,0)[lt]{\lineheight{1.25}\smash{\begin{tabular}[t]{l}$z$\end{tabular}}}}%
    \put(0,0){\includegraphics[width=\unitlength,page=6]{fig20.pdf}}%
    \put(0.27269026,0.24131273){\color[rgb]{0,0,1}\makebox(0,0)[lt]{\lineheight{1.25}\smash{\begin{tabular}[t]{l}$\b_0$\end{tabular}}}}%
    \put(0.22347852,0.24131273){\color[rgb]{0.39215686,0.39215686,1}\makebox(0,0)[rt]{\lineheight{1.25}\smash{\begin{tabular}[t]{r}$\b_0'$\end{tabular}}}}%
  \end{picture}%
\endgroup%

%% file: fig42.pdf_tex
\begingroup%
  \makeatletter%
  \providecommand\color[2][]{%
    \errmessage{(Inkscape) Color is used for the text in Inkscape, but the package 'color.sty' is not loaded}%
    \renewcommand\color[2][]{}%
  }%
  \providecommand\transparent[1]{%
    \errmessage{(Inkscape) Transparency is used (non-zero) for the text in Inkscape, but the package 'transparent.sty' is not loaded}%
    \renewcommand\transparent[1]{}%
  }%
  \providecommand\rotatebox[2]{#2}%
  \newcommand*\fsize{\dimexpr\f@size pt\relax}%
  \newcommand*\lineheight[1]{\fontsize{\fsize}{#1\fsize}\selectfont}%
  \ifx\svgwidth\undefined%
    \setlength{\unitlength}{359.84954647bp}%
    \ifx\svgscale\undefined%
      \relax%
    \else%
      \setlength{\unitlength}{\unitlength * \real{\svgscale}}%
    \fi%
  \else%
    \setlength{\unitlength}{\svgwidth}%
  \fi%
  \global\let\svgwidth\undefined%
  \global\let\svgscale\undefined%
  \makeatother%
  \begin{picture}(1,0.30404062)%
    \lineheight{1}%
    \setlength\tabcolsep{0pt}%
    \put(0,0){\includegraphics[width=\unitlength,page=1]{fig42.pdf}}%
    \put(0.59653901,0.23836618){\color[rgb]{0.77254902,0,1}\makebox(0,0)[rt]{\lineheight{1.25}\smash{\begin{tabular}[t]{r}$\b_1'$\end{tabular}}}}%
    \put(0.87977402,0.06266532){\makebox(0,0)[lt]{\lineheight{1.25}\smash{\begin{tabular}[t]{l}$\theta_\tau^+$\end{tabular}}}}%
    \put(0.2220756,0.16104635){\makebox(0,0)[lt]{\lineheight{1.25}\smash{\begin{tabular}[t]{l}$\theta_\sigma^+$\end{tabular}}}}%
    \put(0.80839263,0.08211643){\makebox(0,0)[rt]{\lineheight{1.25}\smash{\begin{tabular}[t]{r}$w$\end{tabular}}}}%
    \put(0.85142392,0.08211643){\makebox(0,0)[lt]{\lineheight{1.25}\smash{\begin{tabular}[t]{l}$z$\end{tabular}}}}%
    \put(0,0){\includegraphics[width=\unitlength,page=2]{fig42.pdf}}%
    \put(0.50546197,0.22848913){\color[rgb]{0,0,1}\makebox(0,0)[lt]{\lineheight{1.25}\smash{\begin{tabular}[t]{l}$\b_0'$\end{tabular}}}}%
    \put(0.58355738,0.06868506){\color[rgb]{0.37647059,0,1}\makebox(0,0)[rt]{\lineheight{1.25}\smash{\begin{tabular}[t]{r}$\b_1$\end{tabular}}}}%
    \put(0,0){\includegraphics[width=\unitlength,page=3]{fig42.pdf}}%
    \put(0.4758095,0.08211643){\makebox(0,0)[rt]{\lineheight{1.25}\smash{\begin{tabular}[t]{r}$w$\end{tabular}}}}%
    \put(0.51884096,0.08211643){\makebox(0,0)[lt]{\lineheight{1.25}\smash{\begin{tabular}[t]{l}$z$\end{tabular}}}}%
    \put(0,0){\includegraphics[width=\unitlength,page=4]{fig42.pdf}}%
    \put(0.56039338,0.15956917){\color[rgb]{1,0,0}\makebox(0,0)[lt]{\lineheight{1.25}\smash{\begin{tabular}[t]{l}$\b_\lambda$\end{tabular}}}}%
    \put(0,0){\includegraphics[width=\unitlength,page=5]{fig42.pdf}}%
    \put(0.14322327,0.08211643){\makebox(0,0)[rt]{\lineheight{1.25}\smash{\begin{tabular}[t]{r}$w$\end{tabular}}}}%
    \put(0.186255,0.08211643){\makebox(0,0)[lt]{\lineheight{1.25}\smash{\begin{tabular}[t]{l}$z$\end{tabular}}}}%
    \put(0,0){\includegraphics[width=\unitlength,page=6]{fig42.pdf}}%
    \put(0.07506382,0.16685868){\makebox(0,0)[rt]{\lineheight{1.25}\smash{\begin{tabular}[t]{r}$\theta^+$\end{tabular}}}}%
    \put(0,0){\includegraphics[width=\unitlength,page=7]{fig42.pdf}}%
  \end{picture}%
\endgroup%

%% file: fig22.pdf_tex
\begingroup%
  \makeatletter%
  \providecommand\color[2][]{%
    \errmessage{(Inkscape) Color is used for the text in Inkscape, but the package 'color.sty' is not loaded}%
    \renewcommand\color[2][]{}%
  }%
  \providecommand\transparent[1]{%
    \errmessage{(Inkscape) Transparency is used (non-zero) for the text in Inkscape, but the package 'transparent.sty' is not loaded}%
    \renewcommand\transparent[1]{}%
  }%
  \providecommand\rotatebox[2]{#2}%
  \newcommand*\fsize{\dimexpr\f@size pt\relax}%
  \newcommand*\lineheight[1]{\fontsize{\fsize}{#1\fsize}\selectfont}%
  \ifx\svgwidth\undefined%
    \setlength{\unitlength}{240.11365035bp}%
    \ifx\svgscale\undefined%
      \relax%
    \else%
      \setlength{\unitlength}{\unitlength * \real{\svgscale}}%
    \fi%
  \else%
    \setlength{\unitlength}{\svgwidth}%
  \fi%
  \global\let\svgwidth\undefined%
  \global\let\svgscale\undefined%
  \makeatother%
  \begin{picture}(1,0.97130463)%
    \lineheight{1}%
    \setlength\tabcolsep{0pt}%
    \put(0,0){\includegraphics[width=\unitlength,page=1]{fig22.pdf}}%
    \put(0.9107727,0.65557094){\color[rgb]{0.37647059,0,1}\makebox(0,0)[lt]{\lineheight{1.25}\smash{\begin{tabular}[t]{l}$\b_1$\end{tabular}}}}%
    \put(0.71628647,0.55169737){\makebox(0,0)[rt]{\lineheight{1.25}\smash{\begin{tabular}[t]{r}$\theta_\tau^+$\end{tabular}}}}%
    \put(0,0){\includegraphics[width=\unitlength,page=2]{fig22.pdf}}%
    \put(0.41979935,0.61184124){\color[rgb]{0.37647059,0,1}\makebox(0,0)[rt]{\lineheight{1.25}\smash{\begin{tabular}[t]{r}$\b_1$\end{tabular}}}}%
    \put(0,0){\includegraphics[width=\unitlength,page=3]{fig22.pdf}}%
    \put(0.21464342,0.61365245){\makebox(0,0)[rt]{\lineheight{1.25}\smash{\begin{tabular}[t]{r}$w$\end{tabular}}}}%
    \put(0.27913346,0.61365245){\makebox(0,0)[lt]{\lineheight{1.25}\smash{\begin{tabular}[t]{l}$z$\end{tabular}}}}%
    \put(0,0){\includegraphics[width=\unitlength,page=4]{fig22.pdf}}%
    \put(0.71307846,0.61365245){\makebox(0,0)[rt]{\lineheight{1.25}\smash{\begin{tabular}[t]{r}$w$\end{tabular}}}}%
    \put(0.77756824,0.61365245){\makebox(0,0)[lt]{\lineheight{1.25}\smash{\begin{tabular}[t]{l}$z$\end{tabular}}}}%
    \put(0,0){\includegraphics[width=\unitlength,page=5]{fig22.pdf}}%
    \put(0.91077277,0.1399954){\color[rgb]{0.37647059,0,1}\makebox(0,0)[lt]{\lineheight{1.25}\smash{\begin{tabular}[t]{l}$\b_1$\end{tabular}}}}%
    \put(0,0){\includegraphics[width=\unitlength,page=6]{fig22.pdf}}%
    \put(0.71307851,0.09807712){\makebox(0,0)[rt]{\lineheight{1.25}\smash{\begin{tabular}[t]{r}$w$\end{tabular}}}}%
    \put(0.77756828,0.09807712){\makebox(0,0)[lt]{\lineheight{1.25}\smash{\begin{tabular}[t]{l}$z$\end{tabular}}}}%
    \put(0,0){\includegraphics[width=\unitlength,page=7]{fig22.pdf}}%
    \put(0.41979943,0.09626522){\color[rgb]{0.37647059,0,1}\makebox(0,0)[rt]{\lineheight{1.25}\smash{\begin{tabular}[t]{r}$\b_1$\end{tabular}}}}%
    \put(0,0){\includegraphics[width=\unitlength,page=8]{fig22.pdf}}%
    \put(0.21464337,0.09807712){\makebox(0,0)[rt]{\lineheight{1.25}\smash{\begin{tabular}[t]{r}$w$\end{tabular}}}}%
    \put(0.2791336,0.09807712){\makebox(0,0)[lt]{\lineheight{1.25}\smash{\begin{tabular}[t]{l}$z$\end{tabular}}}}%
    \put(0,0){\includegraphics[width=\unitlength,page=9]{fig22.pdf}}%
  \end{picture}%
\endgroup%

%% file: fig43.pdf_tex
\begingroup%
  \makeatletter%
  \providecommand\color[2][]{%
    \errmessage{(Inkscape) Color is used for the text in Inkscape, but the package 'color.sty' is not loaded}%
    \renewcommand\color[2][]{}%
  }%
  \providecommand\transparent[1]{%
    \errmessage{(Inkscape) Transparency is used (non-zero) for the text in Inkscape, but the package 'transparent.sty' is not loaded}%
    \renewcommand\transparent[1]{}%
  }%
  \providecommand\rotatebox[2]{#2}%
  \newcommand*\fsize{\dimexpr\f@size pt\relax}%
  \newcommand*\lineheight[1]{\fontsize{\fsize}{#1\fsize}\selectfont}%
  \ifx\svgwidth\undefined%
    \setlength{\unitlength}{359.79250805bp}%
    \ifx\svgscale\undefined%
      \relax%
    \else%
      \setlength{\unitlength}{\unitlength * \real{\svgscale}}%
    \fi%
  \else%
    \setlength{\unitlength}{\svgwidth}%
  \fi%
  \global\let\svgwidth\undefined%
  \global\let\svgscale\undefined%
  \makeatother%
  \begin{picture}(1,0.3049033)%
    \lineheight{1}%
    \setlength\tabcolsep{0pt}%
    \put(0,0){\includegraphics[width=\unitlength,page=1]{fig43.pdf}}%
    \put(0.83817765,0.22933986){\color[rgb]{0,0,1}\makebox(0,0)[lt]{\lineheight{1.25}\smash{\begin{tabular}[t]{l}$\b_0'$\end{tabular}}}}%
    \put(0.58128668,0.11960823){\color[rgb]{0.54901961,0,1}\makebox(0,0)[lt]{\lineheight{1.25}\smash{\begin{tabular}[t]{l}$\b_1'$\end{tabular}}}}%
    \put(0.80852072,0.0787748){\makebox(0,0)[rt]{\lineheight{1.25}\smash{\begin{tabular}[t]{r}$w$\end{tabular}}}}%
    \put(0.85155883,0.0787748){\makebox(0,0)[lt]{\lineheight{1.25}\smash{\begin{tabular}[t]{l}$z$\end{tabular}}}}%
    \put(0,0){\includegraphics[width=\unitlength,page=2]{fig43.pdf}}%
    \put(0.50554215,0.22933986){\color[rgb]{0,0,1}\makebox(0,0)[lt]{\lineheight{1.25}\smash{\begin{tabular}[t]{l}$\b_0'$\end{tabular}}}}%
    \put(0.58098207,0.07084283){\color[rgb]{0.37647059,0,1}\makebox(0,0)[rt]{\lineheight{1.25}\smash{\begin{tabular}[t]{r}$\b_1$\end{tabular}}}}%
    \put(0,0){\includegraphics[width=\unitlength,page=3]{fig43.pdf}}%
    \put(0.47588498,0.07460577){\makebox(0,0)[rt]{\lineheight{1.25}\smash{\begin{tabular}[t]{r}$w$\end{tabular}}}}%
    \put(0.51892363,0.07460577){\makebox(0,0)[lt]{\lineheight{1.25}\smash{\begin{tabular}[t]{l}$z$\end{tabular}}}}%
    \put(0,0){\includegraphics[width=\unitlength,page=4]{fig43.pdf}}%
    \put(0.14324599,0.07460577){\makebox(0,0)[rt]{\lineheight{1.25}\smash{\begin{tabular}[t]{r}$w$\end{tabular}}}}%
    \put(0.18628452,0.07460577){\makebox(0,0)[lt]{\lineheight{1.25}\smash{\begin{tabular}[t]{l}$z$\end{tabular}}}}%
    \put(0,0){\includegraphics[width=\unitlength,page=5]{fig43.pdf}}%
  \end{picture}%
\endgroup%

%% file: fig44.pdf_tex
\begingroup%
  \makeatletter%
  \providecommand\color[2][]{%
    \errmessage{(Inkscape) Color is used for the text in Inkscape, but the package 'color.sty' is not loaded}%
    \renewcommand\color[2][]{}%
  }%
  \providecommand\transparent[1]{%
    \errmessage{(Inkscape) Transparency is used (non-zero) for the text in Inkscape, but the package 'transparent.sty' is not loaded}%
    \renewcommand\transparent[1]{}%
  }%
  \providecommand\rotatebox[2]{#2}%
  \newcommand*\fsize{\dimexpr\f@size pt\relax}%
  \newcommand*\lineheight[1]{\fontsize{\fsize}{#1\fsize}\selectfont}%
  \ifx\svgwidth\undefined%
    \setlength{\unitlength}{359.79248208bp}%
    \ifx\svgscale\undefined%
      \relax%
    \else%
      \setlength{\unitlength}{\unitlength * \real{\svgscale}}%
    \fi%
  \else%
    \setlength{\unitlength}{\svgwidth}%
  \fi%
  \global\let\svgwidth\undefined%
  \global\let\svgscale\undefined%
  \makeatother%
  \begin{picture}(1,0.304323)%
    \lineheight{1}%
    \setlength\tabcolsep{0pt}%
    \put(0,0){\includegraphics[width=\unitlength,page=1]{fig44.pdf}}%
    \put(0.55795901,0.23446487){\color[rgb]{0.83137255,0,1}\makebox(0,0)[rt]{\lineheight{1.25}\smash{\begin{tabular}[t]{r}$\b_1'$\end{tabular}}}}%
    \put(0.80852074,0.07796066){\makebox(0,0)[rt]{\lineheight{1.25}\smash{\begin{tabular}[t]{r}$w$\end{tabular}}}}%
    \put(0.85155891,0.07796066){\makebox(0,0)[lt]{\lineheight{1.25}\smash{\begin{tabular}[t]{l}$z$\end{tabular}}}}%
    \put(0,0){\includegraphics[width=\unitlength,page=2]{fig44.pdf}}%
    \put(0.5957515,0.19274415){\color[rgb]{0.37647059,0,1}\makebox(0,0)[lt]{\lineheight{1.25}\smash{\begin{tabular}[t]{l}$\b_1$\end{tabular}}}}%
    \put(0,0){\includegraphics[width=\unitlength,page=3]{fig44.pdf}}%
    \put(0.47588498,0.07796066){\makebox(0,0)[rt]{\lineheight{1.25}\smash{\begin{tabular}[t]{r}$w$\end{tabular}}}}%
    \put(0.51892345,0.07796066){\makebox(0,0)[lt]{\lineheight{1.25}\smash{\begin{tabular}[t]{l}$z$\end{tabular}}}}%
    \put(0,0){\includegraphics[width=\unitlength,page=4]{fig44.pdf}}%
    \put(0.14324597,0.07796066){\makebox(0,0)[rt]{\lineheight{1.25}\smash{\begin{tabular}[t]{r}$w$\end{tabular}}}}%
    \put(0.1862845,0.07796066){\makebox(0,0)[lt]{\lineheight{1.25}\smash{\begin{tabular}[t]{l}$z$\end{tabular}}}}%
    \put(0,0){\includegraphics[width=\unitlength,page=5]{fig44.pdf}}%
    \put(0.51613966,0.15666561){\color[rgb]{0,0,1}\makebox(0,0)[lt]{\lineheight{1.25}\smash{\begin{tabular}[t]{l}$\b_0$\end{tabular}}}}%
    \put(0.48321985,0.15666561){\color[rgb]{0.39215686,0.39215686,1}\makebox(0,0)[rt]{\lineheight{1.25}\smash{\begin{tabular}[t]{r}$\b_0'$\end{tabular}}}}%
    \put(0,0){\includegraphics[width=\unitlength,page=6]{fig44.pdf}}%
  \end{picture}%
\endgroup%

%% file: fig28.pdf_tex
\begingroup%
  \makeatletter%
  \providecommand\color[2][]{%
    \errmessage{(Inkscape) Color is used for the text in Inkscape, but the package 'color.sty' is not loaded}%
    \renewcommand\color[2][]{}%
  }%
  \providecommand\transparent[1]{%
    \errmessage{(Inkscape) Transparency is used (non-zero) for the text in Inkscape, but the package 'transparent.sty' is not loaded}%
    \renewcommand\transparent[1]{}%
  }%
  \providecommand\rotatebox[2]{#2}%
  \newcommand*\fsize{\dimexpr\f@size pt\relax}%
  \newcommand*\lineheight[1]{\fontsize{\fsize}{#1\fsize}\selectfont}%
  \ifx\svgwidth\undefined%
    \setlength{\unitlength}{194.33285637bp}%
    \ifx\svgscale\undefined%
      \relax%
    \else%
      \setlength{\unitlength}{\unitlength * \real{\svgscale}}%
    \fi%
  \else%
    \setlength{\unitlength}{\svgwidth}%
  \fi%
  \global\let\svgwidth\undefined%
  \global\let\svgscale\undefined%
  \makeatother%
  \begin{picture}(1,0.52790243)%
    \lineheight{1}%
    \setlength\tabcolsep{0pt}%
    \put(0,0){\includegraphics[width=\unitlength,page=1]{fig28.pdf}}%
    \put(0.21776759,0.19899104){\makebox(0,0)[lt]{\lineheight{1.25}\smash{\begin{tabular}[t]{l}$\xs$\end{tabular}}}}%
    \put(0.71355043,0.21333149){\makebox(0,0)[lt]{\lineheight{1.25}\smash{\begin{tabular}[t]{l}$\ys$\end{tabular}}}}%
    \put(0,0){\includegraphics[width=\unitlength,page=2]{fig28.pdf}}%
    \put(0.58327909,0.16099216){\makebox(0,0)[lt]{\lineheight{1.25}\smash{\begin{tabular}[t]{l}$\Sigma$\end{tabular}}}}%
    \put(0.76231482,0.28341072){\color[rgb]{1,0,0}\makebox(0,0)[rt]{\lineheight{1.25}\smash{\begin{tabular}[t]{r}$\a$\end{tabular}}}}%
    \put(0.84929237,0.18441311){\color[rgb]{0,0,1}\makebox(0,0)[rt]{\lineheight{1.25}\smash{\begin{tabular}[t]{r}$\b$\end{tabular}}}}%
    \put(0,0){\includegraphics[width=\unitlength,page=3]{fig28.pdf}}%
  \end{picture}%
\endgroup%

%% file: fig24.pdf_tex
\begingroup%
  \makeatletter%
  \providecommand\color[2][]{%
    \errmessage{(Inkscape) Color is used for the text in Inkscape, but the package 'color.sty' is not loaded}%
    \renewcommand\color[2][]{}%
  }%
  \providecommand\transparent[1]{%
    \errmessage{(Inkscape) Transparency is used (non-zero) for the text in Inkscape, but the package 'transparent.sty' is not loaded}%
    \renewcommand\transparent[1]{}%
  }%
  \providecommand\rotatebox[2]{#2}%
  \newcommand*\fsize{\dimexpr\f@size pt\relax}%
  \newcommand*\lineheight[1]{\fontsize{\fsize}{#1\fsize}\selectfont}%
  \ifx\svgwidth\undefined%
    \setlength{\unitlength}{120.43265382bp}%
    \ifx\svgscale\undefined%
      \relax%
    \else%
      \setlength{\unitlength}{\unitlength * \real{\svgscale}}%
    \fi%
  \else%
    \setlength{\unitlength}{\svgwidth}%
  \fi%
  \global\let\svgwidth\undefined%
  \global\let\svgscale\undefined%
  \makeatother%
  \begin{picture}(1,0.90846504)%
    \lineheight{1}%
    \setlength\tabcolsep{0pt}%
    \put(0,0){\includegraphics[width=\unitlength,page=1]{fig24.pdf}}%
    \put(0.51654816,0.68271946){\color[rgb]{0,0,1}\makebox(0,0)[lt]{\lineheight{1.25}\smash{\begin{tabular}[t]{l}$\b_0$\end{tabular}}}}%
    \put(0.28706091,0.53289035){\color[rgb]{0.37647059,0,1}\makebox(0,0)[rt]{\lineheight{1.25}\smash{\begin{tabular}[t]{r}$\b_1$\end{tabular}}}}%
    \put(0,0){\includegraphics[width=\unitlength,page=2]{fig24.pdf}}%
    \put(0.44040261,0.22045215){\makebox(0,0)[rt]{\lineheight{1.25}\smash{\begin{tabular}[t]{r}$w$\end{tabular}}}}%
    \put(0.5689793,0.22045215){\makebox(0,0)[lt]{\lineheight{1.25}\smash{\begin{tabular}[t]{l}$z$\end{tabular}}}}%
    \put(0,0){\includegraphics[width=\unitlength,page=3]{fig24.pdf}}%
    \put(0.53535102,0.35528717){\makebox(0,0)[lt]{\lineheight{1.25}\smash{\begin{tabular}[t]{l}$\theta_\sigma^+$\end{tabular}}}}%
    \put(0.48285997,0.04634582){\makebox(0,0)[rt]{\lineheight{1.25}\smash{\begin{tabular}[t]{r}$\theta_\tau^+$\end{tabular}}}}%
    \put(0,0){\includegraphics[width=\unitlength,page=4]{fig24.pdf}}%
  \end{picture}%
\endgroup%

%% file: fig57.pdf_tex
\begingroup%
  \makeatletter%
  \providecommand\color[2][]{%
    \errmessage{(Inkscape) Color is used for the text in Inkscape, but the package 'color.sty' is not loaded}%
    \renewcommand\color[2][]{}%
  }%
  \providecommand\transparent[1]{%
    \errmessage{(Inkscape) Transparency is used (non-zero) for the text in Inkscape, but the package 'transparent.sty' is not loaded}%
    \renewcommand\transparent[1]{}%
  }%
  \providecommand\rotatebox[2]{#2}%
  \newcommand*\fsize{\dimexpr\f@size pt\relax}%
  \newcommand*\lineheight[1]{\fontsize{\fsize}{#1\fsize}\selectfont}%
  \ifx\svgwidth\undefined%
    \setlength{\unitlength}{249.14346868bp}%
    \ifx\svgscale\undefined%
      \relax%
    \else%
      \setlength{\unitlength}{\unitlength * \real{\svgscale}}%
    \fi%
  \else%
    \setlength{\unitlength}{\svgwidth}%
  \fi%
  \global\let\svgwidth\undefined%
  \global\let\svgscale\undefined%
  \makeatother%
  \begin{picture}(1,0.44148747)%
    \lineheight{1}%
    \setlength\tabcolsep{0pt}%
    \put(0,0){\includegraphics[width=\unitlength,page=1]{fig57.pdf}}%
    \put(0.66396176,0.20439839){\color[rgb]{0.77254902,0,1}\makebox(0,0)[lt]{\lineheight{1.25}\smash{\begin{tabular}[t]{l}$\b_1'$\end{tabular}}}}%
    \put(0.61458913,0.20431327){\color[rgb]{0.37647059,0,1}\makebox(0,0)[rt]{\lineheight{1.25}\smash{\begin{tabular}[t]{r}$\b_1$\end{tabular}}}}%
    \put(0,0){\includegraphics[width=\unitlength,page=2]{fig57.pdf}}%
    \put(0.72654561,0.08680454){\makebox(0,0)[rt]{\lineheight{1.25}\smash{\begin{tabular}[t]{r}$w$\end{tabular}}}}%
    \put(0.78941917,0.08680454){\makebox(0,0)[lt]{\lineheight{1.25}\smash{\begin{tabular}[t]{l}$z$\end{tabular}}}}%
    \put(0,0){\includegraphics[width=\unitlength,page=3]{fig57.pdf}}%
    \put(0.76862237,0.21389376){\color[rgb]{0,0,1}\makebox(0,0)[lt]{\lineheight{1.25}\smash{\begin{tabular}[t]{l}$\b_0$\end{tabular}}}}%
    \put(0,0){\includegraphics[width=\unitlength,page=4]{fig57.pdf}}%
    \put(0.43221682,0.33135086){\color[rgb]{0.37647059,0,1}\makebox(0,0)[rt]{\lineheight{1.25}\smash{\begin{tabular}[t]{r}$\b_1$\end{tabular}}}}%
    \put(0,0){\includegraphics[width=\unitlength,page=5]{fig57.pdf}}%
    \put(0.20800097,0.08680454){\makebox(0,0)[rt]{\lineheight{1.25}\smash{\begin{tabular}[t]{r}$w$\end{tabular}}}}%
    \put(0.27015367,0.08680454){\makebox(0,0)[lt]{\lineheight{1.25}\smash{\begin{tabular}[t]{l}$z$\end{tabular}}}}%
    \put(0,0){\includegraphics[width=\unitlength,page=6]{fig57.pdf}}%
    \put(0.21585898,0.21116464){\color[rgb]{0,0,1}\makebox(0,0)[rt]{\lineheight{1.25}\smash{\begin{tabular}[t]{r}$\b_0$\end{tabular}}}}%
    \put(0.26064399,0.21116464){\color[rgb]{0.39215686,0.39215686,1}\makebox(0,0)[lt]{\lineheight{1.25}\smash{\begin{tabular}[t]{l}$\b_0'$\end{tabular}}}}%
    \put(0,0){\includegraphics[width=\unitlength,page=7]{fig57.pdf}}%
    \put(0.25395124,0.32014176){\makebox(0,0)[lt]{\lineheight{1.25}\smash{\begin{tabular}[t]{l}$\theta^+$\end{tabular}}}}%
    \put(0.66478102,0.25686505){\makebox(0,0)[rt]{\lineheight{1.25}\smash{\begin{tabular}[t]{r}$\theta^+$\end{tabular}}}}%
    \put(0,0){\includegraphics[width=\unitlength,page=8]{fig57.pdf}}%
  \end{picture}%
\endgroup%

%% file: fig50.pdf_tex
\begingroup%
  \makeatletter%
  \providecommand\color[2][]{%
    \errmessage{(Inkscape) Color is used for the text in Inkscape, but the package 'color.sty' is not loaded}%
    \renewcommand\color[2][]{}%
  }%
  \providecommand\transparent[1]{%
    \errmessage{(Inkscape) Transparency is used (non-zero) for the text in Inkscape, but the package 'transparent.sty' is not loaded}%
    \renewcommand\transparent[1]{}%
  }%
  \providecommand\rotatebox[2]{#2}%
  \newcommand*\fsize{\dimexpr\f@size pt\relax}%
  \newcommand*\lineheight[1]{\fontsize{\fsize}{#1\fsize}\selectfont}%
  \ifx\svgwidth\undefined%
    \setlength{\unitlength}{170.03093176bp}%
    \ifx\svgscale\undefined%
      \relax%
    \else%
      \setlength{\unitlength}{\unitlength * \real{\svgscale}}%
    \fi%
  \else%
    \setlength{\unitlength}{\svgwidth}%
  \fi%
  \global\let\svgwidth\undefined%
  \global\let\svgscale\undefined%
  \makeatother%
  \begin{picture}(1,0.36611742)%
    \lineheight{1}%
    \setlength\tabcolsep{0pt}%
    \put(0,0){\includegraphics[width=\unitlength,page=1]{fig50.pdf}}%
    \put(0.14144242,0.30169347){\makebox(0,0)[t]{\lineheight{1.25}\smash{\begin{tabular}[t]{c}$\cH_1$\end{tabular}}}}%
    \put(0.42294143,0.30205826){\makebox(0,0)[t]{\lineheight{1.25}\smash{\begin{tabular}[t]{c}$\cH_2$\end{tabular}}}}%
    \put(0,0){\includegraphics[width=\unitlength,page=2]{fig50.pdf}}%
  \end{picture}%
\endgroup%

%% file: fig56.pdf_tex
\begingroup%
  \makeatletter%
  \providecommand\color[2][]{%
    \errmessage{(Inkscape) Color is used for the text in Inkscape, but the package 'color.sty' is not loaded}%
    \renewcommand\color[2][]{}%
  }%
  \providecommand\transparent[1]{%
    \errmessage{(Inkscape) Transparency is used (non-zero) for the text in Inkscape, but the package 'transparent.sty' is not loaded}%
    \renewcommand\transparent[1]{}%
  }%
  \providecommand\rotatebox[2]{#2}%
  \newcommand*\fsize{\dimexpr\f@size pt\relax}%
  \newcommand*\lineheight[1]{\fontsize{\fsize}{#1\fsize}\selectfont}%
  \ifx\svgwidth\undefined%
    \setlength{\unitlength}{293.92748416bp}%
    \ifx\svgscale\undefined%
      \relax%
    \else%
      \setlength{\unitlength}{\unitlength * \real{\svgscale}}%
    \fi%
  \else%
    \setlength{\unitlength}{\svgwidth}%
  \fi%
  \global\let\svgwidth\undefined%
  \global\let\svgscale\undefined%
  \makeatother%
  \begin{picture}(1,0.23431107)%
    \lineheight{1}%
    \setlength\tabcolsep{0pt}%
    \put(0,0){\includegraphics[width=\unitlength,page=1]{fig56.pdf}}%
    \put(0.09093652,0.03887566){\makebox(0,0)[lt]{\lineheight{1.25}\smash{\begin{tabular}[t]{l}$\Sigma_0$\end{tabular}}}}%
    \put(0.25607912,0.02300196){\color[rgb]{0.39215686,0.39215686,0.39215686}\makebox(0,0)[lt]{\lineheight{1.25}\smash{\begin{tabular}[t]{l}$\mu_i$\end{tabular}}}}%
    \put(0.36528718,0.12793195){\color[rgb]{0.39215686,0.39215686,0.39215686}\makebox(0,0)[lt]{\lineheight{1.25}\smash{\begin{tabular}[t]{l}$\lambda_i$\end{tabular}}}}%
    \put(0.29861311,0.06273812){\makebox(0,0)[lt]{\lineheight{1.25}\smash{\begin{tabular}[t]{l}$\d \Sigma_0$\end{tabular}}}}%
    \put(0,0){\includegraphics[width=\unitlength,page=2]{fig56.pdf}}%
    \put(0.08326186,0.17293274){\color[rgb]{1,0,0}\makebox(0,0)[lt]{\lineheight{1.25}\smash{\begin{tabular}[t]{l}$\as$\end{tabular}}}}%
    \put(0,0){\includegraphics[width=\unitlength,page=3]{fig56.pdf}}%
    \put(0.62578102,0.03873228){\makebox(0,0)[lt]{\lineheight{1.25}\smash{\begin{tabular}[t]{l}$\Sigma_0$\end{tabular}}}}%
    \put(0.90876742,0.12793196){\color[rgb]{0.39215686,0.39215686,0.39215686}\makebox(0,0)[lt]{\lineheight{1.25}\smash{\begin{tabular}[t]{l}$K_i$\end{tabular}}}}%
    \put(0,0){\includegraphics[width=\unitlength,page=4]{fig56.pdf}}%
    \put(0.61810636,0.17278936){\color[rgb]{1,0,0}\makebox(0,0)[lt]{\lineheight{1.25}\smash{\begin{tabular}[t]{l}$\as$\end{tabular}}}}%
    \put(0.80487652,0.02425335){\color[rgb]{0,0,1}\makebox(0,0)[lt]{\lineheight{1.25}\smash{\begin{tabular}[t]{l}$\b_{0,i}$\end{tabular}}}}%
    \put(0.75429475,0.11311502){\makebox(0,0)[t]{\lineheight{1.25}\smash{\begin{tabular}[t]{c}$w$\end{tabular}}}}%
    \put(0.84863743,0.11311502){\makebox(0,0)[t]{\lineheight{1.25}\smash{\begin{tabular}[t]{c}$z$\end{tabular}}}}%
  \end{picture}%
\endgroup%

%% file: fig29.pdf_tex
\begingroup%
  \makeatletter%
  \providecommand\color[2][]{%
    \errmessage{(Inkscape) Color is used for the text in Inkscape, but the package 'color.sty' is not loaded}%
    \renewcommand\color[2][]{}%
  }%
  \providecommand\transparent[1]{%
    \errmessage{(Inkscape) Transparency is used (non-zero) for the text in Inkscape, but the package 'transparent.sty' is not loaded}%
    \renewcommand\transparent[1]{}%
  }%
  \providecommand\rotatebox[2]{#2}%
  \newcommand*\fsize{\dimexpr\f@size pt\relax}%
  \newcommand*\lineheight[1]{\fontsize{\fsize}{#1\fsize}\selectfont}%
  \ifx\svgwidth\undefined%
    \setlength{\unitlength}{120.43265382bp}%
    \ifx\svgscale\undefined%
      \relax%
    \else%
      \setlength{\unitlength}{\unitlength * \real{\svgscale}}%
    \fi%
  \else%
    \setlength{\unitlength}{\svgwidth}%
  \fi%
  \global\let\svgwidth\undefined%
  \global\let\svgscale\undefined%
  \makeatother%
  \begin{picture}(1,0.91650113)%
    \lineheight{1}%
    \setlength\tabcolsep{0pt}%
    \put(0,0){\includegraphics[width=\unitlength,page=1]{fig29.pdf}}%
    \put(0.51654816,0.68913879){\color[rgb]{0,0,1}\makebox(0,0)[lt]{\lineheight{1.25}\smash{\begin{tabular}[t]{l}$\b_0$\end{tabular}}}}%
    \put(0.28706091,0.53930969){\color[rgb]{0.37647059,0,1}\makebox(0,0)[rt]{\lineheight{1.25}\smash{\begin{tabular}[t]{r}$\b_1$\end{tabular}}}}%
    \put(0,0){\includegraphics[width=\unitlength,page=2]{fig29.pdf}}%
    \put(0.45285766,0.20196116){\makebox(0,0)[rt]{\lineheight{1.25}\smash{\begin{tabular}[t]{r}$w$\end{tabular}}}}%
    \put(0.55652425,0.20196116){\makebox(0,0)[lt]{\lineheight{1.25}\smash{\begin{tabular}[t]{l}$z$\end{tabular}}}}%
    \put(0,0){\includegraphics[width=\unitlength,page=3]{fig29.pdf}}%
    \put(0.53535102,0.3617065){\makebox(0,0)[lt]{\lineheight{1.25}\smash{\begin{tabular}[t]{l}$\theta_\sigma^+$\end{tabular}}}}%
    \put(0.52431707,0.03069796){\makebox(0,0)[lt]{\lineheight{1.25}\smash{\begin{tabular}[t]{l}$\theta_\tau^+$\end{tabular}}}}%
    \put(0,0){\includegraphics[width=\unitlength,page=4]{fig29.pdf}}%
    \put(0.36999814,0.6375469){\makebox(0,0)[rt]{\lineheight{1.25}\smash{\begin{tabular}[t]{r}$[\mu]$\end{tabular}}}}%
    \put(0,0){\includegraphics[width=\unitlength,page=5]{fig29.pdf}}%
  \end{picture}%
\endgroup%

%% file: fig49.pdf_tex
\begingroup%
  \makeatletter%
  \providecommand\color[2][]{%
    \errmessage{(Inkscape) Color is used for the text in Inkscape, but the package 'color.sty' is not loaded}%
    \renewcommand\color[2][]{}%
  }%
  \providecommand\transparent[1]{%
    \errmessage{(Inkscape) Transparency is used (non-zero) for the text in Inkscape, but the package 'transparent.sty' is not loaded}%
    \renewcommand\transparent[1]{}%
  }%
  \providecommand\rotatebox[2]{#2}%
  \newcommand*\fsize{\dimexpr\f@size pt\relax}%
  \newcommand*\lineheight[1]{\fontsize{\fsize}{#1\fsize}\selectfont}%
  \ifx\svgwidth\undefined%
    \setlength{\unitlength}{249.56907519bp}%
    \ifx\svgscale\undefined%
      \relax%
    \else%
      \setlength{\unitlength}{\unitlength * \real{\svgscale}}%
    \fi%
  \else%
    \setlength{\unitlength}{\svgwidth}%
  \fi%
  \global\let\svgwidth\undefined%
  \global\let\svgscale\undefined%
  \makeatother%
  \begin{picture}(1,0.49162764)%
    \lineheight{1}%
    \setlength\tabcolsep{0pt}%
    \put(0,0){\includegraphics[width=\unitlength,page=1]{fig49.pdf}}%
    \put(0.26478821,0.00452996){\makebox(0,0)[t]{\lineheight{1.25}\smash{\begin{tabular}[t]{c}$\b_1+\b_{-1}+\sum_{s\ge 1} U^{s(s-1)/2}$\end{tabular}}}}%
    \put(0.68730422,0.11274895){\makebox(0,0)[t]{\lineheight{1.25}\smash{\begin{tabular}[t]{c}$B_s^-$\end{tabular}}}}%
    \put(0.92708375,0.11274895){\makebox(0,0)[t]{\lineheight{1.25}\smash{\begin{tabular}[t]{c}$B_s^+$\end{tabular}}}}%
    \put(0.43447301,0.44622984){\makebox(0,0)[lt]{\lineheight{1.25}\smash{\begin{tabular}[t]{l}\reflectbox{$\ddots$}\end{tabular}}}}%
  \end{picture}%
\endgroup%

%% file: fig39.pdf_tex
\begingroup%
  \makeatletter%
  \providecommand\color[2][]{%
    \errmessage{(Inkscape) Color is used for the text in Inkscape, but the package 'color.sty' is not loaded}%
    \renewcommand\color[2][]{}%
  }%
  \providecommand\transparent[1]{%
    \errmessage{(Inkscape) Transparency is used (non-zero) for the text in Inkscape, but the package 'transparent.sty' is not loaded}%
    \renewcommand\transparent[1]{}%
  }%
  \providecommand\rotatebox[2]{#2}%
  \newcommand*\fsize{\dimexpr\f@size pt\relax}%
  \newcommand*\lineheight[1]{\fontsize{\fsize}{#1\fsize}\selectfont}%
  \ifx\svgwidth\undefined%
    \setlength{\unitlength}{135.58346362bp}%
    \ifx\svgscale\undefined%
      \relax%
    \else%
      \setlength{\unitlength}{\unitlength * \real{\svgscale}}%
    \fi%
  \else%
    \setlength{\unitlength}{\svgwidth}%
  \fi%
  \global\let\svgwidth\undefined%
  \global\let\svgscale\undefined%
  \makeatother%
  \begin{picture}(1,0.66317566)%
    \lineheight{1}%
    \setlength\tabcolsep{0pt}%
    \put(0,0){\includegraphics[width=\unitlength,page=1]{fig39.pdf}}%
    \put(0.25697682,0.40078329){\makebox(0,0)[lt]{\lineheight{1.25}\smash{\begin{tabular}[t]{l}$n_w$\end{tabular}}}}%
    \put(0.64069774,0.40078329){\makebox(0,0)[lt]{\lineheight{1.25}\smash{\begin{tabular}[t]{l}$n_z$\end{tabular}}}}%
    \put(0.45023433,0.53501762){\makebox(0,0)[lt]{\lineheight{1.25}\smash{\begin{tabular}[t]{l}$N$\end{tabular}}}}%
    \put(0.44883704,0.1094361){\makebox(0,0)[lt]{\lineheight{1.25}\smash{\begin{tabular}[t]{l}$M$\end{tabular}}}}%
    \put(0,0){\includegraphics[width=\unitlength,page=2]{fig39.pdf}}%
    \put(0.65930242,0.53093103){\color[rgb]{0,0,1}\makebox(0,0)[lt]{\lineheight{1.25}\smash{\begin{tabular}[t]{l}$\b_0'$\end{tabular}}}}%
    \put(0.31531363,0.53093103){\color[rgb]{0,0,1}\makebox(0,0)[rt]{\lineheight{1.25}\smash{\begin{tabular}[t]{r}$\b_0$\end{tabular}}}}%
    \put(0.5127906,0.31176195){\makebox(0,0)[lt]{\lineheight{1.25}\smash{\begin{tabular}[t]{l}$\theta^+$\end{tabular}}}}%
  \end{picture}%
\endgroup%

%% file: fig41.pdf_tex
\begingroup%
  \makeatletter%
  \providecommand\color[2][]{%
    \errmessage{(Inkscape) Color is used for the text in Inkscape, but the package 'color.sty' is not loaded}%
    \renewcommand\color[2][]{}%
  }%
  \providecommand\transparent[1]{%
    \errmessage{(Inkscape) Transparency is used (non-zero) for the text in Inkscape, but the package 'transparent.sty' is not loaded}%
    \renewcommand\transparent[1]{}%
  }%
  \providecommand\rotatebox[2]{#2}%
  \newcommand*\fsize{\dimexpr\f@size pt\relax}%
  \newcommand*\lineheight[1]{\fontsize{\fsize}{#1\fsize}\selectfont}%
  \ifx\svgwidth\undefined%
    \setlength{\unitlength}{118.30371591bp}%
    \ifx\svgscale\undefined%
      \relax%
    \else%
      \setlength{\unitlength}{\unitlength * \real{\svgscale}}%
    \fi%
  \else%
    \setlength{\unitlength}{\svgwidth}%
  \fi%
  \global\let\svgwidth\undefined%
  \global\let\svgscale\undefined%
  \makeatother%
  \begin{picture}(1,0.90634263)%
    \lineheight{1}%
    \setlength\tabcolsep{0pt}%
    \put(0,0){\includegraphics[width=\unitlength,page=1]{fig41.pdf}}%
    \put(0.70362643,0.19935733){\color[rgb]{0,0,1}\makebox(0,0)[lt]{\lineheight{1.25}\smash{\begin{tabular}[t]{l}$\b_0'$\end{tabular}}}}%
    \put(0.25311784,0.20319041){\color[rgb]{0,0,1}\makebox(0,0)[rt]{\lineheight{1.25}\smash{\begin{tabular}[t]{r}$\b_0$\end{tabular}}}}%
    \put(0.14512921,0.61313416){\makebox(0,0)[lt]{\lineheight{1.25}\smash{\begin{tabular}[t]{l}$w$\end{tabular}}}}%
    \put(0.78389655,0.61846479){\makebox(0,0)[lt]{\lineheight{1.25}\smash{\begin{tabular}[t]{l}$z$\end{tabular}}}}%
    \put(0,0){\includegraphics[width=\unitlength,page=2]{fig41.pdf}}%
  \end{picture}%
\endgroup%

%% file: fig58.pdf_tex
\begingroup%
  \makeatletter%
  \providecommand\color[2][]{%
    \errmessage{(Inkscape) Color is used for the text in Inkscape, but the package 'color.sty' is not loaded}%
    \renewcommand\color[2][]{}%
  }%
  \providecommand\transparent[1]{%
    \errmessage{(Inkscape) Transparency is used (non-zero) for the text in Inkscape, but the package 'transparent.sty' is not loaded}%
    \renewcommand\transparent[1]{}%
  }%
  \providecommand\rotatebox[2]{#2}%
  \newcommand*\fsize{\dimexpr\f@size pt\relax}%
  \newcommand*\lineheight[1]{\fontsize{\fsize}{#1\fsize}\selectfont}%
  \ifx\svgwidth\undefined%
    \setlength{\unitlength}{178.38540867bp}%
    \ifx\svgscale\undefined%
      \relax%
    \else%
      \setlength{\unitlength}{\unitlength * \real{\svgscale}}%
    \fi%
  \else%
    \setlength{\unitlength}{\svgwidth}%
  \fi%
  \global\let\svgwidth\undefined%
  \global\let\svgscale\undefined%
  \makeatother%
  \begin{picture}(1,0.90741184)%
    \lineheight{1}%
    \setlength\tabcolsep{0pt}%
    \put(0,0){\includegraphics[width=\unitlength,page=1]{fig58.pdf}}%
    \put(0.37038666,0.4760841){\makebox(0,0)[lt]{\lineheight{1.25}\smash{\begin{tabular}[t]{l}$a$\end{tabular}}}}%
    \put(0.91047815,0.48554859){\makebox(0,0)[rt]{\lineheight{1.25}\smash{\begin{tabular}[t]{r}$\Sss{a+\delta_0+\delta_0'}$\end{tabular}}}}%
    \put(0.05346327,0.36178512){\makebox(0,0)[lt]{\lineheight{1.25}\smash{\begin{tabular}[t]{l}$\Sss{a+\delta_0+\delta_0'}$\end{tabular}}}}%
    \put(0.58181676,0.66748821){\makebox(0,0)[lt]{\lineheight{1.25}\smash{\begin{tabular}[t]{l}$\Sss{a+\delta_0+2\delta_0'}$\end{tabular}}}}%
    \put(0.27978511,0.50762403){\makebox(0,0)[rt]{\lineheight{1.25}\smash{\begin{tabular}[t]{r}$\Sss{a+\delta_0+2\delta_0'+\dt_1}$\end{tabular}}}}%
    \put(0.54653348,0.27955954){\color[rgb]{0,0,1}\makebox(0,0)[lt]{\lineheight{1.25}\smash{\begin{tabular}[t]{l}$\b_0$\end{tabular}}}}%
    \put(0.46489415,0.27955954){\color[rgb]{0,0,1}\makebox(0,0)[rt]{\lineheight{1.25}\smash{\begin{tabular}[t]{r}$\b_0'$\end{tabular}}}}%
    \put(0.42913286,0.22295497){\makebox(0,0)[rt]{\lineheight{1.25}\smash{\begin{tabular}[t]{r}$w$\end{tabular}}}}%
    \put(0.58121312,0.22295497){\makebox(0,0)[lt]{\lineheight{1.25}\smash{\begin{tabular}[t]{l}$z$\end{tabular}}}}%
    \put(0.14624443,0.26795222){\color[rgb]{0.71764706,0,1}\makebox(0,0)[lt]{\lineheight{1.25}\smash{\begin{tabular}[t]{l}$\b_1$\end{tabular}}}}%
    \put(0.1543141,0.64878129){\color[rgb]{1,0,1}\makebox(0,0)[lt]{\lineheight{1.25}\smash{\begin{tabular}[t]{l}$\b_1'$\end{tabular}}}}%
    \put(0,0){\includegraphics[width=\unitlength,page=2]{fig58.pdf}}%
  \end{picture}%
\endgroup%

%% file: fig55.pdf_tex
\begingroup%
  \makeatletter%
  \providecommand\color[2][]{%
    \errmessage{(Inkscape) Color is used for the text in Inkscape, but the package 'color.sty' is not loaded}%
    \renewcommand\color[2][]{}%
  }%
  \providecommand\transparent[1]{%
    \errmessage{(Inkscape) Transparency is used (non-zero) for the text in Inkscape, but the package 'transparent.sty' is not loaded}%
    \renewcommand\transparent[1]{}%
  }%
  \providecommand\rotatebox[2]{#2}%
  \newcommand*\fsize{\dimexpr\f@size pt\relax}%
  \newcommand*\lineheight[1]{\fontsize{\fsize}{#1\fsize}\selectfont}%
  \ifx\svgwidth\undefined%
    \setlength{\unitlength}{293.38182128bp}%
    \ifx\svgscale\undefined%
      \relax%
    \else%
      \setlength{\unitlength}{\unitlength * \real{\svgscale}}%
    \fi%
  \else%
    \setlength{\unitlength}{\svgwidth}%
  \fi%
  \global\let\svgwidth\undefined%
  \global\let\svgscale\undefined%
  \makeatother%
  \begin{picture}(1,0.34625036)%
    \lineheight{1}%
    \setlength\tabcolsep{0pt}%
    \put(0,0){\includegraphics[width=\unitlength,page=1]{fig55.pdf}}%
    \put(0.29512809,0.12890348){\makebox(0,0)[rt]{\lineheight{1.25}\smash{\begin{tabular}[t]{r}$w_\ell$\end{tabular}}}}%
    \put(0.36414011,0.12890348){\makebox(0,0)[lt]{\lineheight{1.25}\smash{\begin{tabular}[t]{l}$z_\ell$\end{tabular}}}}%
    \put(0.83433793,0.13912893){\makebox(0,0)[rt]{\lineheight{1.25}\smash{\begin{tabular}[t]{r}$w_\ell$\end{tabular}}}}%
    \put(0.90335003,0.13912893){\makebox(0,0)[lt]{\lineheight{1.25}\smash{\begin{tabular}[t]{l}$z_\ell$\end{tabular}}}}%
    \put(0.39830191,0.03055549){\color[rgb]{0,0,1}\makebox(0,0)[lt]{\lineheight{1.25}\smash{\begin{tabular}[t]{l}$\b_0'$\end{tabular}}}}%
    \put(0.24955676,0.03055549){\color[rgb]{0,0,1}\makebox(0,0)[rt]{\lineheight{1.25}\smash{\begin{tabular}[t]{r}$\b_0$\end{tabular}}}}%
    \put(0.336256,0.14400306){\makebox(0,0)[lt]{\lineheight{1.25}\smash{\begin{tabular}[t]{l}$\theta^+$\end{tabular}}}}%
    \put(0,0){\includegraphics[width=\unitlength,page=2]{fig55.pdf}}%
    \put(0.09697276,0.25960748){\color[rgb]{1,0,0}\makebox(0,0)[rt]{\lineheight{1.25}\smash{\begin{tabular}[t]{r}$\as$\end{tabular}}}}%
    \put(0.06305517,0.1196984){\color[rgb]{1,0,0}\makebox(0,0)[rt]{\lineheight{1.25}\smash{\begin{tabular}[t]{r}$\as$\end{tabular}}}}%
    \put(0.20211895,0.21850213){\makebox(0,0)[lt]{\lineheight{1.25}\smash{\begin{tabular}[t]{l}$\lambda_i$\end{tabular}}}}%
    \put(0.16625454,0.16129458){\makebox(0,0)[rt]{\lineheight{1.25}\smash{\begin{tabular}[t]{r}$\lambda_j$\end{tabular}}}}%
    \put(0,0){\includegraphics[width=\unitlength,page=3]{fig55.pdf}}%
    \put(0.94792679,0.03055549){\color[rgb]{0,0,1}\makebox(0,0)[lt]{\lineheight{1.25}\smash{\begin{tabular}[t]{l}$\b_0'$\end{tabular}}}}%
    \put(0.79918179,0.03055549){\color[rgb]{0,0,1}\makebox(0,0)[rt]{\lineheight{1.25}\smash{\begin{tabular}[t]{r}$\b_0$\end{tabular}}}}%
    \put(0,0){\includegraphics[width=\unitlength,page=4]{fig55.pdf}}%
    \put(0.64659786,0.25960748){\color[rgb]{1,0,0}\makebox(0,0)[rt]{\lineheight{1.25}\smash{\begin{tabular}[t]{r}$\as$\end{tabular}}}}%
    \put(0.61268021,0.1196984){\color[rgb]{1,0,0}\makebox(0,0)[rt]{\lineheight{1.25}\smash{\begin{tabular}[t]{r}$\as$\end{tabular}}}}%
    \put(0,0){\includegraphics[width=\unitlength,page=5]{fig55.pdf}}%
  \end{picture}%
\endgroup%